\documentclass[oneside]{amsart}
\usepackage[lmargin=1in,rmargin=1in, bmargin=1in, tmargin=1in]{geometry}
\usepackage{amsmath, amsthm, amssymb}
\usepackage{tikz}
\usepackage{tikz-cd}
\usepackage{todonotes}
\usepackage{mathrsfs}
\usepackage{extarrows}
\usepackage{graphicx}
\usepackage{colonequals}
\usepackage{stmaryrd}

\usepackage[breaklinks, pagebackref]{hyperref}
\usepackage{IEEEtrantools}
\usepackage{mathrsfs}
\usepackage[utf8]{inputenc}
\usepackage[english]{babel}
 
\usepackage{comment}
\usepackage{csquotes}

\usepackage[shortlabels]{enumitem}

\usepackage{bbm}

\hypersetup{
    colorlinks=true,
    linkcolor=black,
    filecolor=magenta, 
    citecolor=red,     
    urlcolor=cyan,
    }
	
\newcommand{\QQ}{\mathbb{Q}}

\newcommand{\ZZ}{\mathbb{Z}}

\newcommand{\defeq}{\colonequals}

\newcommand{\Iw}{\mathrm{Iw}}

\newcommand{\ch}{\mathrm{ch}}   
\newcommand{\Vol}{\mathrm{Vol}}   
   
\newcommand{\delT}{\mathbb{S}}   
\newcommand{\CC}{\mathbb{C}}
\newcommand{\GG}{\mathbb{G}}
    
\newcommand{\Sh}{\mathrm{Sh}}    
\newcommand{\Ab}{\mathbb{A}}   
    
\newcommand{\pr}{\mathrm{pr}}

\renewcommand{\tilde}{\widetilde}

\DeclareMathOperator{\Gal}{Gal}

\numberwithin{equation}{section}

\newtheorem{prop}[equation]{Proposition}
\newtheorem{theorem}[equation]{Theorem}
\newtheorem{proposition}[equation]{Proposition}
\newtheorem{lemma}[equation]{Lemma}
\newtheorem{corollary}[equation]{Corollary}
\newtheorem*{corollary*}{Corollary}

\newtheorem{assumption}[equation]{Assumption}

\newtheorem{theoremx}{Theorem}

\theoremstyle{definition}

\newtheorem{definition}[equation]{Definition}
\newtheorem{notation}[equation]{Notation}
\newtheorem*{convention*}{Convention}

\theoremstyle{remark}
\newtheorem{remark}[equation]{Remark}
\newtheorem{example}[equation]{Example}

\DeclareFontFamily{U}{wncy}{}
\DeclareFontShape{U}{wncy}{m}{n}{<->wncyr10}{}
\DeclareSymbolFont{mcy}{U}{wncy}{m}{n}
\DeclareMathSymbol{\sha}{\mathord}{mcy}{"58}

\newcommand{\ide}[1]{\mathfrak{#1}}
\newcommand{\mbb}[1]{\mathbb{#1}}
\newcommand{\cohom}[2]{\mathrm{H}^{#1}_{#2}}
\newcommand{\ordd}{\mathcal{O}}

\newcommand{\invs}[1]{\mathcal{#1}}

\newcommand{\opn}[1]{\operatorname{#1}}

\newcommand{\dcris}[1]{\mathbf{D}_{\mathrm{cris}}(#1)}
\newcommand{\mbf}[1]{\mathbf{#1}}

\newcommand{\Qbar}{\overline{\mbb{Q}}}
\newcommand{\Gt}{\widetilde{\mbf{G}}}
\newcommand{\gt}{\widetilde{G}}           
                 
\newcommand{\Gb}{\mathbf{G}}   
\newcommand{\Tb}{\mathbf{T}}    
   
\newcommand{\Zb}{\mathbf{Z}}

\newcommand{\Hb}{\mathbf{H}}

\newcommand{\GU}{\mathrm{GU}}   
\newcommand{\U}{\mathrm{U}}           
\newcommand{\Art}{\mathrm{Art}}
\newcommand{\Res}{\mathrm{Res}}   
    
\newcommand{\Nm}{\mathscr{N}}

\newcommand{\tbyt}[4]{\left( \begin{array}{cc} #1 & #2 \\ #3 & #4 \end{array} \right)}

\newcommand{\Addresses}{{%
  \bigskip
  \footnotesize

  (Graham) \textsc{Department of Mathematics, South Kensington Campus, Imperial College London, London SW7 2AZ, UK}\par\nopagebreak
  \textit{E-mail address}: \texttt{a.graham17@imperial.ac.uk}

  \medskip

  (Shah) \textsc{Department of Mathematics, Harvard University, 1 Oxford Street, Cambridge, MA 02138, U.S.A.}\par\nopagebreak
  \textit{E-mail address}: \texttt{swshah@math.harvard.edu}

}}

\DeclareMathOperator{\Hom}{Hom}

\newcommand{\OO}{\mathcal{O}}
\newcommand{\GL}{\mathrm{GL}}

\newcommand{\et}{\mathrm{\acute{e}t}}

\newcommand{\Qpb}{\overline{\QQ}_p}
\newcommand{\RR}{\mathbb{R}}

\DeclareMathOperator{\Spec}{Spec}    

\newcommand{\bincoeff}[2]{\genfrac(){0pt}{0}{#1}{#2}}

\title{Anticyclotomic Euler systems for unitary groups}
\author{Andrew Graham and Syed Waqar Ali Shah}
\date{}

\begin{document}
\begin{abstract}
    Let $n \geq 1$ be an odd integer. We construct an anticyclotomic Euler system for certain cuspidal automorphic representations of unitary groups with signature $(1, 2n-1)$.
\end{abstract}

\maketitle

\tableofcontents


\section{Introduction}

In \cite{KolyvaginHP}, Kolyvagin constructs an \emph{anticyclotomic} Euler system for (modular) elliptic curves over $\mbb{Q}$ which satisfy the so-called ``Heegner hypothesis''. One can view the classes in this construction as images under the modular parameterisation of certain divisors on modular curves arising from the embedding of reductive groups
\begin{equation} \label{HeegnerPair}
\opn{Res}_{E/\mbb{Q}}\mbb{G}_m \hookrightarrow \opn{GL}_2
\end{equation}
where $E$ is an imaginary quadratic number field. These divisors (Heegner points) are defined over ring class fields of $E$ and satisfy trace compatibility relations as one varies the field of definition. In particular, Kolyvagin shows that if the bottom Heegner point is non-torsion, then the group of $E$-rational points has rank equal to one. Combining this so-called ``Kolyvagin step'' with the Gross--Zagier formula that relates the height of this point to the derivative of the $L$-function at the central value, one obtains instances of the Birch--Swinnerton-Dyer conjecture in the analytic rank one case. The above construction has been generalised to higher weight modular forms and to situations where a more general hypothesis is placed on the modular form (see \cite{schoen1986}, \cite{Nekovar1992}, \cite{bertolini2013} and \cite{SZhang}).

In this paper, we consider a possible generalisation of this setting; namely,  we construct an anticyclotomic Euler system for the $ p $-adic Galois representations attached to certain regular algebraic conjugate self-dual cuspidal automorphic representations of $\opn{GL}_{2n}/E$. More precisely, we consider the following ``symmetric pair''
\begin{equation} \label{SymPair}
\opn{U}(1, n-1) \times \opn{U}(0, n) \hookrightarrow \opn{U}(1, 2n-1).
\end{equation}
Note that both groups are outer forms of the groups appearing in (\ref{HeegnerPair}) when $ n = 1 $. Let $\pi_0$ be a cuspidal automorphic representation of $\opn{U}(1, 2n-1)$ and let $\pi$ denote a lift to the group of unitary similitudes. Under certain reasonable assumptions on $\pi_0$ (for example, non-endoscopic and tempered with stable $L$-packet), there exists a cuspidal automorphic representation $\Pi$ of $\opn{GL}_{2n}(\mbb{A}_E) \times \opn{GL}_1(\mbb{A}_E)$ which is locally isomorphic to the base-change of $\pi$ at all but finitely many primes. Let $\rho_{\pi}$ denote the representation of the absolute Galois group of $ E $  attached to $\Pi$, as constructed by Chenevier and Harris \cite{ChenevierHarris}.

We impose the following assumptions. We require that the lift $ \pi $ of $\pi_{0} $ as above is  cohomological and the Galois-automorphic piece $\rho_{\pi} \otimes \pi_f$ appears in the middle degree geometric \'{e}tale cohomology of the Shimura variety for $\opn{GU}(1, 2n-1)$ -- see Assumption \ref{KeyAssumption}. If this assumption is satisfied, we say that $\rho_{\pi} \otimes \pi_f$ admits a ``modular parameterisation''. Our second assumption is an analogue of the Heegner hypothesis: we require that $\pi_{0,f}$ admits a $\mbf{H}_0$-linear model, i.e. one has 
\[
\opn{Hom}_{\mbf{H}_0(\mbb{A}_f)}\left( \pi_{0,f}, \mbb{C} \right) \neq 0
\]
where $\mbf{H}_0$ is the subgroup $\opn{U}(1, n-1) \times \opn{U}(0, n)$.  Finally, we require that $ p $ is split in the quadratic extension $ E $, and $ \pi_{0} $ (or equivalently, $ \pi $) has a fixed eigenvector with respect to a certain Hecke operator associated with the Siegel parahoric level at the prime $ p $. We  refer to this assumption as ``Siegel ordinarity".

We now state our main result. Let $S$ be a finite set of primes containing all primes where $\pi_0$ is ramified, and all primes that ramify in $E/\mbb{Q}$. For $p \notin S$, let $\invs{R}$ denote the set of \emph{square-free} positive integers divisible only by primes that lie outside $S \cup \{p\}$ and split in $E/\mbb{Q}$. For $m \in \invs{R}$ and an integer $r \geq 0$, we let $E[mp^r]$ denote the ring class field of $E$ of conductor $mp^r$.
\begin{theoremx} \label{TheoremA}
Let $n \geq 1$ be an odd integer and $p \notin S$ an odd prime that splits in $E/\mbb{Q}$. Let $\pi_0$ be a cuspidal automorphic representation of $\opn{U}(1, 2n-1)$ as above. We impose the following conditions:
\begin{itemize}
    \item $\rho_{\pi}$ is absolutely irreducible and $\rho_{\pi} \otimes \pi_f$ admits a ``modular parameterisation'' (see Assumption \ref{KeyAssumption}).
    \item $\pi_{0, f}$ admits a $\mbf{H}_0$-linear model.
    \item $\pi_0$ is ``Siegel ordinary'' (see Assumption \ref{MainAssumption}).
\end{itemize}
Let $T_{\pi}$ be a Galois stable lattice inside $\rho_{\pi}$. Then there exists a split anticyclotomic Euler system for $T_{\pi}^* (1 - n)$, i.e. for any collection of primes of $E$ lying above primes in $\invs{R}$, and for $m \in \invs{R}$, $r \geq 0$, there exist classes $c_{mp^r} \in \opn{H}^1\left(E[mp^r], T_{\pi}^*(1-n) \right)$ satisfying
\[
\opn{cores}^{E[\ell m p^r]}_{E[mp^r]} c_{\ell m p^r} = \left\{ \begin{array}{cc} P_\lambda \left( \opn{Fr}_\lambda^{-1} \right) \cdot c_{m p^r} & \text{ if } \ell \neq p \text{ and } \ell m \in \invs{R} \\ c_{m p^r} & \text{ if } \ell = p  \end{array} \right.
\]
where $\lambda$ is the fixed prime of $E$ lying above $\ell$, $P_\lambda (X) = \opn{det}\left(1 - \opn{Frob}_\lambda^{-1}X | T_{\pi}(n) \right)$ is the reverse characteristic polynomial of a choice of geometric Frobenius element $ \opn{Frob}_{\lambda}^{-1} \in \Gal(\overline{E}/E) $, and $\opn{Fr}_{\lambda} \in \Gal\left(E[m p^r]/E \right)$ denotes the arithmetic Frobenius at $\lambda$. 
\end{theoremx}

\begin{remark}
Evidently, we can take $c_{mp^r} = 0$ for all $m, r$, and the above theorem is vacuous. However, the Euler system classes we construct arise from special cycles on Shimura varieties, and the non-vanishing of these classes is expected to be related to the behaviour of the $L$-function attached to $\pi_0$ (more precisely, the non-vanishing of the derivative of the $L$-function at its central value). Proving such a relation constitutes an \emph{explicit reciprocity law}, and we will pursue this in future work.
\end{remark}

The above definition of a split anticyclotomic Euler system originates from the forthcoming work of Jetchev, Nekov\'{a}\v{r} and Skinner, in which a general machinery for bounding Selmer groups attached to conjugate self-dual representations is developed. An interesting feature of this work is that one only needs norm relations for primes that split in the CM extension, rather than a ``full'' anticyclotomic Euler system. As a consequence of this, we expect to obtain the following corollary:
\begin{corollary*}
Let $\pi_0$ be as in Theorem \ref{TheoremA}, and suppose that the Galois representation $\rho_{\pi}$ satisfies the following hypotheses:
\begin{itemize}
    \item The primes for which $\rho_{\pi}$ is ramified lie above primes that split in $E/\mbb{Q}$.
    \item There exists an element $\sigma \in \Gal\left(\overline{E}/E[1](\mu_{p^{\infty}})\right)$ such that 
    \[
    \opn{rank} \rho_{\pi}/(\sigma - 1)\rho_{\pi} = 1.
    \]
\end{itemize}
Then, if $c \defeq \opn{cores}^{E[1]}_E c_1$ is non-torsion, the Bloch--Kato Selmer group $\opn{H}^1_f\left(E, \rho_{\pi}(n) \right)$ is one-dimensional. 
\end{corollary*}

\begin{remark}
The existence of such an element $\sigma$ in the above corollary is expected to follow if the image of the Galois representation $\rho_{\pi} \colon \Gal(\overline{E}/E) \to \opn{GL}_{2n}(\Qpb)$ is sufficiently large. In particular, this will exclude automorphic representations of ``CM type''. 
\end{remark}

\begin{remark}
In this paper, we have chosen to focus on the case where $E$ is an imaginary quadratic number field -- however,  we expect the results to also hold for general CM fields. Moreover, it should be possible to construct similar classes for certain inner forms of the groups appearing in (\ref{SymPair}). Another direction worth exploring is the case of inert primes which is closer in spirit to Kolyvagin's original construction -- these questions will also be investigated further in future work.  
\end{remark}

\subsection{Outline of the paper}

Let $\pi_0$ be a cuspidal automorphic representation for the unitary group $\mbf{G}_0$ of signature $(1, 2n-1)$ (defined in section \ref{TheGroups}), such that $\pi_{0, \infty }$ lies in the discrete series. Discrete series $L$-packets of $\mbf{G}_0$ are parameterised by irreducible algebraic representations of $\opn{GL}_{2n}/E$; we let $V_0$ denote the algebraic representation corresponding to the $L$-packet containing $\pi_{0, \infty}$. As explained in the previous section, we assume that $\pi_{0,f}$ admits a $\mbf{H}_0$-linear model, i.e. we have 
\[
\opn{Hom}_{\mbf{H}_0(\mbb{A}_f)}(\pi_{0,f}, \mbb{C}) \neq 0 .
\]
In particular, this implies that $V_0$ is self-dual and the central character of $\pi_0$ is trivial. Let $\pi$ denote a lift of $\pi_0$ to the unitary similitude group $\mbf{G} = \GU(1,2n-1) $, and let $\Pi = \omega \boxtimes \Pi_0$ denote a weak base-change to $\opn{GL}_1(\mbb{A}_E) \times \opn{GL}_{2n}(\mbb{A}_E)$ as explained in Proposition \ref{WeakBCProp}. We assume that $\Pi$ is cuspidal and, since the central character of $\pi_0$ is trivial, we can choose $\pi$ such that $\omega$ is equal to the trivial character (Lemma \ref{TrivialCentralLiftLemma}). Furthermore, $\pi_{\infty}$ lies in the discrete series $L$-packet corresponding to the trivial extension of $V_0$ to an algebraic representation of $\opn{GL}_1/E \times \opn{GL}_{2n}/E$, which we denote by $V$ (so in particular, $V$ is also self-dual). 

The representation $\Pi$ decomposes as a restricted tensor product $\otimes_{v}' \Pi_v$ over the places of $E$, and for each (finite) prime $\lambda$ where $\Pi_{\lambda}$ is unramified, we let $L(\Pi_{\lambda}, s)$ denote the standard local $L$-factor attached to $\Pi_{\lambda}$ (see Definition \ref{StandardLFactorDef}). We use similar notation for the representation $\Pi_0$. Combining the results of several people (see Theorem \ref{ExistenceOfGaloisRep}), Chenevier and Harris have shown that there exists a semisimple continuous Galois representation $\rho_{\pi} \colon \Gal(\overline{E}/E)  \to \opn{GL}_{2n}(\Qpb)$ satisfying the following unramified local--global compatibility: for all but finitely many rational primes $\ell$ and for all $\lambda | \ell$
\[
\opn{det}\left(1 - \opn{Frob}_{\lambda}^{-1} (\opn{Nm} \lambda)^{-s} | \rho_{\pi} \right) = L(\Pi_\lambda, s + 1/2 - n)^{-1} = L(\Pi_{0,\lambda}, s + 1/2 - n)^{-1} 
\]
where $\opn{Frob}_{\lambda}$ is an arithmetic Frobenius at $\lambda$, $\opn{Nm}\lambda$ is the absolute norm of $\lambda$, and the last equality follows because $\omega = 1$. We assume that $\rho_{\pi} \otimes \pi_f$ appears in the middle degree geometric \'{e}tale cohomology of the Shimura variety $\opn{Sh}_{\mbf{G}}(K)$ associated with $\mbf{G}$ (with appropriate level $K \subset \mbf{G}(\mbb{A}_f)$). More precisely, we assume that we have a ``modular parameterisation''
\[
\opn{H}^{2n-1}_{\et}\left( \opn{Sh}_{\mbf{G}, \overline{\mbb{Q}}}, \mathscr{V}(n) \right) \defeq \varinjlim_L \opn{H}^{2n-1}_{\et}\left( \opn{Sh}_{\mbf{G}}(L)_{\overline{\mbb{Q}}}, \mathscr{V}(n) \right)  \twoheadrightarrow \pi_f^{\vee} \otimes \rho_{\pi}^*(1-n).
\]
where the limit is over all sufficiently small compact open subgroups $L \subset \mbf{G}(\mbb{A}_f)$, and $\mathscr{V}$ denotes the $p$-adic sheaf associated with the representation $V$. In analogy with the Heegner point case, the Euler system will be obtained as the images of certain classes under this modular parameterisation. We also require that $\pi = \otimes'_{\ell} \pi_\ell$ satisfies a Siegel ordinarity condition at $p$. Briefly, this means that the representation $\pi_p$ contains a vector $\varphi_p$ that is an eigenvector for the Hecke operator $\invs{U}_S \defeq {\mu(\tau)}[J \tau J]$, where 
\[
\tau = 1 \times \tbyt{p}{}{}{1} \in \opn{GL}_1(\mbb{Q}_p) \times \opn{GL}_{2n}(\mbb{Q}_p) = \mbf{G}(\mbb{Q}_p)
\] 
and $J$ is the parahoric subgroup associated with the upper-triangular Siegel parabolic subgroup (i.e. the standard parabolic subgroup corresponding to the partition $(n, n)$). Here $\mu \colon \left(\mbb{Q}_p^{\times}\right)^{1+2n} \to \mbb{Q}_p^{\times}$ is the highest weight of $V$ viewed as a character of the standard torus of $\mbf{G}(\mbb{Q}_p) = \opn{GL}_1(\mbb{Q}_p) \times \opn{GL}_{2n}(\mbb{Q}_p)$. 

We consider the following extended group $\Gt \defeq \mbf{G} \times \mbf{T}$ where $\mbf{T} = \opn{U}(1)$ is the torus controlling the variation in the anti-cyclotomic tower (see Lemma \ref{TorusVariationLemma}) -- the group $\mbf{H} = \opn{G}\left( \opn{U}(1, n-1) \times \opn{U}(0, n) \right)$ then lifts to a subgroup of $\Gt$ via the map
\[
(h_1, h_2) \mapsto \left( \tbyt{h_1}{}{}{h_2}, \frac{\opn{det}h_2}{\opn{det}h_1} \right) .
\]
In particular, the dimensions of the associated Shimura varieties\footnote{Strictly speaking, the group $\mbf{H}$ does not give rise to a Shimura variety but rather what we call a \emph{Shimura--Deligne variety} (the axiom -- typically denoted (SV3) -- is not satisfied). In practice this does not affect the arguments in the paper, so the reader can safely pretend that it gives rise to a Shimura variety in the usual sense. We provide justifications for this viewpoint in Appendix \ref{PureAppendix}.} are $ \dim \Sh_{\Hb} = n - 1$ and $ \dim \Sh_{\Gt} = 2n-1$ respectively, and the inclusion  $ \Hb \hookrightarrow  \tilde{\Gb} $ provides us with a rich supply of codimension $ n $ cycles on the target variety, via Gysin morphisms. These cycles carry an action of both $ \Gt(\Ab_{f}) $ and $ \Gal(\bar{E}/E)$, and via Shimura reciprocity, the Galois action can be translated into an action of $ \Tb(\Ab_{f}) $. This will allow us to freely switch between the automorphic and Galois viewpoints.

By considering the images under the cycle class map, we obtain classes in the degree $ 2n $ absolute \'{e}tale cohomology of $ \Sh_{\tilde{\Gb}} $. A consequence of the Siegel ordinarity condition is that we can modify these classes so that they are universal norms in the anticyclotomic $\mbb{Z}_p$-extension, and since this tower is a $p$-adic Lie extension of positive dimension, this will force our classes to be cohomologically trivial. We then obtain classes in the first group cohomology of the \'{e}tale cohomology of $\opn{Sh}_{\mbf{G}}$ by applying the Abel--Jacobi map (an edge map in the Hochschild--Serre spectral sequence). The reason that these classes are universal norms will follow from the results of \cite{Loeffler19}, using the fact that the pair of subgroups in (\ref{SymPair}) gives rise to a spherical variety at primes which split in the imaginary quadratic extension. This norm-compatibility will also imply the Euler system relations in the $p$-direction. We discuss these relations in section \ref{VerticalNormRelations}. For the convenience of the reader, we have provided the constructions of the Gysin morphisms and Abel--Jacobi maps for continuous \'{e}tale cohomology in Appendix \ref{TheAppendix}. 

The key technique introduced in \cite{LSZ17} is to rephrase the horizontal Euler system relations as a statement in local representation theory. The underlying idea is to consider the above construction for varying levels of the source and target Shimura varieties, and then  package this data into a map of $ \Hb(\Ab_{f}^{p}) \times \tilde{\Gb}(\Ab_{f}^{p}) $ representations. In section \ref{TheInterludeSection}, we construct this map in an abstract setting using the language of cohomology functors. The ideas used here are essentially the same as those that appear in \emph{op.cit.} but we hope that this formalism will be useful in proving Euler system relations for general pairs of reductive groups. We also note that these abstract cohomology functors appear in certain areas of representation theory, under the name ``Mackey functors'' (see \cite{Dress} or \cite{WebbGuideToMackey} for example), so these techniques can be considered as an application of Mackey theory.

Recall that $E[mp^r]$ denotes the ring class field of $E$ corresponding to the order $\mbb{Z} + mp^r \ordd_E$ and let $D_{mp^r} \subset \mbf{T}(\Ab_{f})$ denote the compact open subgroup corresponding to this ring class field under the Artin reciprocity map. By composing the Gysin morphisms, the Abel--Jacobi map, the modular parameterisation above and passing to the ``completion" (see Proposition \ref{LemmEin}) for varying levels, we construct a collection of maps
\[
\invs{H}\left( \Gt(\mbb{A}_f) \right)^{1 \times D_{mp^r}} \xrightarrow{\mathcal{L}_{mp^{r}}} \pi_f^\vee \otimes \opn{H}^1\left(E[mp^r], \rho_{\pi}^*(1-n) \right)
\]
which satisfy a certain equivariance property under the action of $\Gt(\mbb{A}_f) \times \mbf{H}(\mbb{A}_f)$. Here $\invs{H}(-)^C$ denotes the elements of the associated Hecke algebra with rational coefficients which are invariant under right-translation by elements in $C$ (and the convolution product is with respect to the fixed Haar measure in \S \ref{TheGroups}).

Let $\varphi$ be an element of $\pi_f$ which is fixed by the Siegel parahoric level at $p$. We define our Euler system to be 
\[
c_{mp^r} \approx \left(v_{\varphi} \circ \; \mathcal{L}_{mp^{r}} \right) (\phi^{(m)})
\]
for suitable test data $\phi^{(m)}$, where $v_{\varphi}$ denotes the evaluation map at the the vector $\varphi$. Here,  by approximately we mean up to some suitable volume factors which ensure that the classes $c_{mp^r}$ land in the cohomology of a Galois stable lattice. To prove the tame norm relation at $\ell \nmid m$, we choose a character $\chi \colon \Gal(E[mp^r]/E) \to \mbb{C}^{\times}$, which can naturally be viewed as a character of $\mbf{T}(\mbb{A}_f)$ via the Artin reciprocity map, and consider the $\chi$-component $\mathcal{L}_{mp^{r}}^{\chi}$ of the map $\mathcal{L}_{mp^{r}}$. It is enough to prove the tame norm relation locally at $\ell$ and, by applying Frobenius reciprocity, the map can naturally be viewed as an element 
\[
\left(\mathcal{L}_{mp^{r}}^{\chi}\right)_{\ell} \in \opn{Hom}_{\mbf{H}_0(\mbb{Q}_\ell)}(\pi_{0, \ell}, \chi^{-1} \boxtimes \chi ).
\]
This reduces the tame norm relation to a calculation in local representation theory, which constitutes the main part of this paper. Once this relation is established, we then  sum over all $\chi$-components to obtain the full norm relation. 

We now describe the ingredients that go into the local computation described above, all of which take place in section \ref{TheHorizontalRelations} of the paper. If $\ell$ is a rational prime that splits in $E/\mbb{Q}$, then there exists a prime $\lambda | \ell$ and an isomorphism $\pi_{0, \ell} \cong \Pi_{0,\lambda}$ compatible with the identification $\mbf{G}_0(\mbb{Q}_\ell) \cong \opn{GL}_{2n}(\mbb{Q}_\ell)$. Furthermore, since $\pi_{0, f}$ admits a $\mbf{H}_0$-linear model, we also have 
\[
\opn{Hom}_{\mbf{H}_0(\mbb{Q}_\ell)}(\pi_{0, \ell}, \mbb{C}) \neq 0 .
\]
The work of Matringe (Proposition \ref{DistImpliesShalika}) implies that $\pi_{0, \ell}$ has a local Shalika model and as a consequence of this, we prove the following theorem:
\begin{theoremx} \label{TheoremB}
Let $\chi \colon \mbb{Q}_\ell^{\times} \to \mbb{C}^\times$ be an unramified finite-order character. Then
\begin{enumerate}
    \item $\opn{dim}_{\mbb{C}} \opn{Hom}_{\mbf{H}_0(\mbb{Q}_\ell)}(\pi_{0, \ell}, \chi^{-1} \boxtimes \chi) = 1$, where $\chi^{-1} \boxtimes \chi$ denotes the character of $\mbf{H}_0(\mbb{Q}_\ell)$ given by sending a pair of matrices $(h_1, h_2)$ to $\chi(\opn{det}h_2/\opn{det}h_1)$. 
    \item There exists a compactly-supported locally constant function $\phi \colon \mbf{G}_0(\mbb{Q}_\ell) \to \mbb{Z}$ (independent of $\chi$) such that, for all $\ide{z} \in \opn{Hom}_{\mbf{H}_0(\mbb{Q}_\ell)}(\pi_{0, \ell}, \chi^{-1} \boxtimes \chi)$ and all spherical vectors $\varphi_0 \in \pi_{0, \ell}$
    \begin{equation} \label{FundamentalRel}
    \ide{z}(\phi \cdot \varphi_0 ) = \frac{\ell^{n^2}}{\ell - 1} L(\pi_{0, \ell} \otimes \chi, 1/2)^{-1} \ide{z}(\varphi_0) .
    \end{equation}
    \item The function $ \phi $ satisfies the following integrality condition: for all $ g \in \Gb_{0}(\QQ_{\ell})$,
    \begin{equation} \label{FundamentalRelintegrality} (\ell-1) \cdot \phi(g) \cdot [\mathbf{H}_{0}(\ZZ_{\ell}) : V_{1,g}]^{-1} \in \ZZ 
    \end{equation}
    where $ V_{1,g} = g \, \mbf{G}_0(\mbb{Z}_\ell) g^{-1} \cap \mathbf{H}_{0}(\mbb{Z}_{\ell}) \cap \nu^{-1} ( 1  + \ell \ZZ_{\ell})$ and $\nu \colon \mbf{H}_0(\mbb{Q}_\ell) \to \mbf{T}(\mbb{Q}_\ell)$ is the map given by $(h_1, h_2) \mapsto \frac{\opn{det}h_2}{\opn{det}h_1}$.    
\end{enumerate}
\end{theoremx}
The first part of this theorem is a result of Jacquet and Rallis \cite{ULP} in the case that $\chi$ is trivial, and a result of Chen and Sun \cite{Chen2019} in the general case. Since $\pi_{0, \ell}$ admits a local Shalika model, one obtains an explicit basis of the hom-space by considering an associated zeta integral, and the proof of part (2) then involves manipulating this zeta integral by applying several $U_\ell$-operators to its input. By taking a suitable linear combination of these operators, we are able to produce the relation appearing in (\ref{FundamentalRel}) and the integrality condition (\ref{FundamentalRelintegrality}). The latter of these ensures the integrality of the resulting    Euler system classes. Such a linear combination arises naturally from an inclusion-exclusion principle on flag varieties associated with $\opn{GL}_m$, for $m=1, \dots, n$.

\subsection{Notation and conventions} \label{NotationSection}

We fix the following notation and conventions throughout the paper:
\begin{itemize}
    \item An embedding $\overline{\mbb{Q}} \hookrightarrow \mbb{C}$ and an imaginary quadratic number field $E \subset \overline{\QQ}$. We assume that the discriminant of $E$ satisfies $\opn{disc}(E) \neq -3, -4$. In particular, this implies that $\ordd_E^{\times} = \{\pm 1 \}$.
    \item For any rational prime $\ell$, we fix an embedding $\overline{\mbb{Q}} \hookrightarrow \overline{\mbb{Q}}_\ell$ and an isomorphism
    \[
    \overline{\mbb{Q}}_{\ell} \cong \mbb{C}
    \]
    compatible with the fixed embeddings $\left(\overline{\mbb{Q}} \hookrightarrow \mbb{C}, \overline{\mbb{Q}} \hookrightarrow \overline{\mbb{Q}}_\ell\right)$. In particular, for each $\ell$, we have a choice of $ \lambda $ above $ \ell $. 
    \item A rational prime $p > 2$ that splits in $E$.
    \item $\mbb{A}_F$ will denote the adeles of a number field $F$ and $\mbb{A}_{F, f}$ the finite adeles; if $F = \mbb{Q}$ we simply write $\mbb{A}$ and $\mbb{A}_f$ respectively. For any integer (or finite set of primes) $S$, we write $\mbb{A}^S$ (resp. $\mbb{A}_f^S$) to mean the adeles (resp. finite adeles) away from $S$ (i.e. at all primes not dividing $S$/not in $S$), and $\mbb{A}_S = \prod_{\ell \in S} \mbb{Q}_{\ell}$ for the adeles at $S$. We also set $\mbb{Z}_S = \prod_{\ell \in S} \mbb{Z}_\ell$. 
    \item For a ring $R$ and a connected reductive group $\mbf{G}$ over $\mbb{Q}$ with a fixed Haar measure on $\mbf{G}(\mbb{A}_f)$, we let
    \[
    \invs{H}\left( \mbf{G}(\mbb{A}_f), R \right) \defeq \left\{ \phi \colon \mbf{G}(\mbb{A}_f) \to R : \begin{array}{c} \phi \text{ is locally constant } \\ \text{ and compactly-supported } \end{array} \right\}
    \]
    denote the Hecke algebra, which carries an action of $\mbf{G}(\mbb{A}_f)$ given by right-translation of the argument. We will omit the ring $R$ from the notation when it is clear from the context. Additionally, if $K \subset \mbf{G}(\mbb{A}_f)$ is a compact open subgroup, then $\invs{H}\left( K \backslash \mbf{G}(\mbb{A}_f) / K \right)$ will denote the subset of $K$-biinvariant functions. We also have similar notation for the $\mbb{Q}_\ell$-points of $\mbf{G}$. The Haar measures for the specific groups we work with throughout the article are fixed in Notation \ref{NotationForFixedHaarMeasures}.
    \item If $X \subset \mbf{G}(\mbb{A}_f)$ is a compact open subset, then we write $\opn{ch}(X) \in \invs{H}(\mbf{G}(\mbb{A}_f), \mbb{Z})$ for the indicator function of $X$.
    \item For a number field $F$, we let $G_F$ denote the absolute Galois group of $F$. Furthermore, the global Artin reciprocity map 
    \[
    \opn{Art}_F \colon \mbb{A}_F^{\times} \to G_F^{\mathrm{ab}}
    \]
    is defined geometrically, i.e. it takes a uniformiser to the associated geometric Frobenius.
    \item For a prime $v$ of $F$, we let $\opn{Frob}_v$ denote an arithmetic Frobenius in $G_F$. If $L/F$ is an abelian extension of number fields that is unramified at $v$, we will sometimes write $\opn{Fr}_v$ for the arithmetic Frobenius in $\Gal(L/F)$ associated with the prime $v$.
    \item For an integer $m \geq 1$, we let $E[m]$ denote the ring class field of conductor $m$, i.e. the finite abelian extension with norm subgroup given by $E^{\times} \cdot \widehat{\ordd}_m^{\times}$ where
    \[
    \ordd_m = \mbb{Z} + m \ordd_E .
    \]
    \item The smooth dual of a representation will be denoted by $(-)^\vee$ and the linear dual by $(-)^*$.
    \item For a smooth quasi-projective scheme $X$ over a characteristic zero field, and a number field $F$, we let $\opn{CHM}(X)_F$ denote the category of relative Chow motives over $X$ with an $F$-structure, as defined in \cite[Definition 2.4]{Torzewski2019}. Let $\opn{DM}_{B, c}(X)$ denote the category of \emph{constructible Beilinson motives}, as in \cite[Definition 15.1.1]{CDMixed}. Then, for any Chow motive $(Y, e, n) \in \opn{CHM}(X)_{\mbb{Q}}$, one can associate a constructible Beilinson motive $eM_X(Y)(n)$ as the mapping fibre
    \[
    eM_X(Y)(n) \defeq \opn{Cone}\left( M_X(Y)(n) \xrightarrow{1 - e} M_X(Y)(n) \right) [-1]
    \]
    where $M_X(Y)$ is as in \cite[\S 11.1.2]{CDMixed}. This association assembles into a functor $\opn{CHM}(X)_{\mbb{Q}} \to \opn{DM}_{B, c}(X)$, so we can naturally view a Chow motive as a constructible Beilinson motive. If $(Y, e, n)$ is equipped with an $F$-structure, then so is $eM_X(Y)(n)$ (i.e. one has a $\mbb{Q}$-linear homomorphism $F \hookrightarrow \opn{End}_{\opn{DM}_{B, c}(X)}\left(eM_X(Y)(n)\right)$ ). 
    \item The motivic cohomology of $X$ with values in a relative Chow motive $\mathscr{F} \in \opn{CHM}(X)_F$ is the $F$-vector space given by
    \[
    \opn{H}^\bullet_{\mathrm{mot}}\left(X, \mathscr{F} \right) \defeq \opn{Hom}_{\opn{DM}_{B, c}(X)}\left( \mbf{1}_X, \mathscr{F}[\bullet] \right)
    \]
    where $\mbf{1}_X$ denotes the trivial motive. This definition is compatible with the one arising from algebraic $K$-theory by \cite[Corollary 14.2.14]{CDMixed}, and since the categories $\opn{DM}_{B, c}(-)$ satisfy the six functor formalism, we can define pushforwards/pullbacks between these cohomology groups (see \cite[\S 4.1]{Lemma17} and the references therein).
    \item All \'{e}tale cohomology groups refer to continuous \'{e}tale cohomology in the sense of Jannsen (\cite{Jannsen1988}). A subscript $ c $ will denote the compactly supported version.   
    \item If $H$ and $G$ are locally profinite groups, and $\sigma$, $\pi$ smooth representations of $H$ and $G$ respectively, then $\sigma \boxtimes \pi$ will denote the tensor product representation of $H \times G$. 
    \item If $\mbf{G}$ is an algebraic group over $\mbb{Q}$ and $F$ is a characteristic zero field, then we write $\opn{Rep}_F(\mbf{G})$ for the category of algebraic representations of $\mbf{G}_F$.  
    \item Unless specified otherwise, all automorphic representations are assumed to be unitary. 
    \item Unless specified otherwise, all modules over a group ring will be \emph{left} modules.
    \item For $\Phi = \Qpb$ or a finite extension of $\mbb{Q}_p$, with ring of integers $\ordd$, and $T$ a finite-free $\ordd$-module with a continuous action of $\Gal(\overline{\mbb{Q}}/E)$, we set
    \[
    \opn{H}^1_{\opn{Iw}}\left( E[mp^{\infty}], T \right) \defeq \varprojlim_{r \geq 1} \opn{H}^1\left(E[mp^r], T \right)
    \]
    to be the Iwasawa cohomology of $T$ (in the anticyclotomic tower), where the transitions maps are given by corestriction.
    \item Let $\ell$ be a rational prime. For a positive integer $n$ we write $[n]_\ell \defeq (\ell^n - 1)(\ell - 1)^{-1}$. This definition extends to $[n]_\ell ! \defeq [n]_\ell \cdot [n-1]_\ell \cdots [1]_\ell$ and 
    \[
    \genfrac[]{0pt}{0}{n}{m}_\ell \defeq \frac{[n]_\ell !}{[m]_\ell !\cdot [(n-m)]_\ell !}
    \]
    with the obvious convention that $[0]_\ell ! = 1$.
\end{itemize}

\subsection{Acknowledgements} 

We are greatly indebted to David Loeffler, Chris Skinner and Sarah Zerbes -- a large amount of this paper owes its existence to the ideas introduced in \cite{LSZ17} on constructing Euler systems in automorphic settings. We are also very grateful to Wei Zhang for suggesting this problem to us. The second named author would like to extend his gratitude to Barry Mazur for his constant guidance and to Sasha Petrov, Ziquan Yang and David Yang for many enlightening discussions. The authors would like to thank the anonymous referees for their feedback, which has led to significant improvements in the presentation of this article. Finally, we would like to thank the organizers of the conference ``Automorphic $p$-adic $L$-functions and regulators'' held in Lille, France in October 2019. This conference provided a pleasant opportunity for the authors to work together in person.

The first named author was supported by the Engineering and Physical Sciences Research Council [EP/L015234/1], the EPSRC Centre for Doctoral Training in Geometry and Number Theory (The London School of Geometry and Number Theory), University College London and Imperial College London.     


\section{Preliminaries I: Cohomology functors} \label{TheInterludeSection}

In \cite{Loeffler19}, a class of functors on subgroups of locally profinite groups is defined in an abstract setting that captures  the key properties of the cohomology of locally symmetric spaces. In this section, we study these functors in more detail, and develop some machinery that will be useful later on. 

\subsection{Generalities} \label{GeneralitiesSection}

Let $G$ be a locally profinite topological group, $\Sigma \subset G$ an open submonoid and let $\Upsilon$ be a non-empty collection of compact open subgroups of $G$ contained in $ \Sigma $. We will impose the following conditions on $(G, \Sigma, \Upsilon) $    
\begin{itemize} 
\item[(T1)] For all $g \in \Sigma \cup \Sigma^{-1}$ and $K \in \Upsilon$, $gKg^{-1} \subset \Sigma$ implies $gKg^{-1} \in \Upsilon$.
\item[(T2)] For all $K \in \Upsilon$, $\Upsilon$ contains a topological basis of open normal subgroups of $K$.
\item [(T3)] For all $K, L \in \Upsilon$ and $g \in \Sigma \cup \Sigma^{-1}$, we have $K \cap gLg^{-1} \in \Upsilon$.
\end{itemize} 
To such a triplet,  we associate a \emph{category of compact opens} $\mathcal{P}(G, \Sigma, \Upsilon)$ whose objects are elements of $\Upsilon$ and whose morphisms are given by 
\[
\Hom_{\mathcal{P}(G, \Sigma, \Upsilon)}(L, K) \defeq \left\{g \in \Sigma | \, g^{-1} L g \subset K\right\} 
\]
for $L , K \in \Upsilon $, with compositions given by $( L \xrightarrow{g} K ) \circ ( L'  \xrightarrow { h }  L ) = ( L' \xrightarrow {hg} K )$.   We will denote a morphism $ (L \xrightarrow{g} K  )   $ by $[g]_{L,K}$; if we have $L \subset K$ and $g = 1$, we shall also denote the corresponding morphism by $\pr_{L, K}$. If no confusion can arise, we will suppress the subscripts from these morphisms.  Omission of $\Upsilon$ from the triplet $ (G, \Sigma , \Upsilon) $ will mean that we take all compact open subgroups of $\Sigma$, and omission of $\Sigma$ means $\Sigma = G$.

\begin{definition} \label{CohoFunc} A \emph{cohomology functor} $ M $ on a triplet $ (G,\Sigma , \Upsilon) $ as above and valued in $ R $\textbf{-Mod} for a commutative ring $ R $ is a pair of covariant functors 
\[
M^{*} \colon \mathcal{P}(G, \Sigma, \Upsilon)^{\text{op}} \to R\textbf{-Mod} \qquad \quad M_{*} \colon \mathcal{P}(G, \Sigma^{-1}, \Upsilon) \to R\textbf{-Mod} 
\]
satisfying the following:
\begin{enumerate}
 \item [(C1)] $M_{*}(K) = M^{*}(K)$ for all $K \in \Upsilon$. We denote this common value by $M(K)$.
 \item [(C2)] If $\phi$ is a morphism then we set $\phi_* = M_*(\phi)$ and $\phi^* = M^*(\phi)$. We require that
 \[
 [g]^{*}_{L,K} = [g^{-1}]_{K,L, *} \in \Hom_R(M(K), M(L)) 
 \]
 for all $g \in \Sigma$ and $L, K \in \Upsilon$ satisfying $g^{-1}Lg = K$. 
 \item [(C3)]   $[\gamma]_{K, K, *} = \mathrm{id}$ for all $\gamma \in K$ and $K \in \Upsilon$. 
\end{enumerate}    
We will denote the cohomology functor as $  M : (G, \Sigma, \Upsilon) \to R $\textbf{-Mod}, and  sometimes abusively as $ M : \mathcal{P}_{G} \to  R$\textbf{-Mod} if the categories of compact opens are clear from the context.   
\end{definition}

Our cohomology functors will often be enhanced with the following additional  axioms. 

\begin{definition}  \label{DefOfAdditionalProperties}
Let $ M : (G, \Sigma,  \Upsilon) \to  R$-\textbf{Mod} be a  cohomology functor.      
\begin{enumerate}
\item[(G)]We say $ M $ is \emph{Galois} if for all $ K , L  \in \Upsilon   $ with $ L \triangleleft K $, $$   \pr_{L,K } ^ { * } : M(K) \xrightarrow {\sim} M(L) ^ { K/L}   $$
where the (left) action of $ \gamma \in K/L $ on $  x \in M(L) $ is via $ x \mapsto [\gamma]_{L, L}^{*} (x)  $.  
\item[(Co)]   We say that $ M $ is a \emph{covering functor}  if for all $  L , K  \in \Upsilon $, $  L \subset K $,   
\[
( L  \xrightarrow {\pr}  K )_{*} \circ ( L \xrightarrow{\pr} K )^{*} =   ( M(K) \xrightarrow{[ K: L ]}  M(K) )      
\]
i.e. the composition is multiplication by the index $ [K:L] $ on  $ M(K) $.     
  \item[(M)]
We say $ M $ is \emph{Cartesian}   
if for all $K, L, L' \in \Upsilon$ with $L ,L' \subset K$, we have a commutative diagram 
\begin{center}
 \begin{tikzcd} 
 \bigoplus_{\gamma} M(L_{\gamma}) \arrow[r,"\sum\pr_{*}"] & M(L) \\
 M(L') \arrow[r,"\pr_{*}",swap] \arrow[u,"\sum\gamma^{*}"] & M(K) \arrow[u,"\pr^{*}",swap] 
\end{tikzcd}
\end{center} 
where the direct sum in the top left corner is over a fixed choice of coset representatives $\gamma \in L \backslash K / L'$ and $L_{\gamma} = \gamma L' \gamma^{-1} \cap L   \in    \Upsilon       $. This   condition is independent of the representatives, since if $ \gamma $ are replaced with $ \gamma ' =  \delta  \gamma \delta ' $ for $ \delta \in  L $, $ \delta ' \in  L' $,  $  L_{\gamma } $ is replaced with $  L_{\gamma ' } =  \delta L _ { \gamma } \delta ^{-1}$, and $$ (L _ { \gamma }  \xrightarrow { \pr } L  ) _ { * }    \circ   (L   _ { \gamma    }   \xrightarrow{ \gamma } L)^{*}   = (  L _ { \gamma '}   \xrightarrow  { \pr} L )_{* }  \circ  (L_{\gamma' }   \xrightarrow  { \gamma' }  L ) ^{*}   $$
\end{enumerate}   
If $ M $ satisfies both (M) and (Co), we will say that $M  $ is \emph{CoMack}. It is easily seen that if $ M $ is CoMack and $ R $ is a $ \QQ $-algebra, then $ M $ is Galois.    
\end{definition}

\begin{remark} 
If one takes $\Upsilon$ to be all compact opens in $\Sigma$, then one obtains the definition of a cohomology functor given in \cite{Loeffler19}. The addition of $\Upsilon$ is made since we only want to vary over sufficiently small compact open subgroups of $\Gb(\Ab_{f})$ for a reductive group $\Gb$ (Definition \ref{DefinitionOfSufficientlySmall}), and the conditions (T1)--(T3)  on $ \Upsilon $ are included so that taking products is well-behaved.  Axiom (G) is inspired by Galois descent in \'{e}tale cohomology and  axiom (Co) reflects the corresponding property for covering maps in \'{e}tale cohomology  \cite[Tome 3, Expose IX, \S 5]{SGA4}.  Axiom (M) is based on Mackey's \emph{decomposition formula}, e.g. see \cite{Dress}. In \cite{Loeffler19}, this property was named \emph{Cartesian} after the Cartesian condition for proper/smooth base change in \'{e}tale cohomology. The terminology `CoMack' is standard, e.g.  see  \cite{Webb}.
\end{remark}    

\begin{example} 
Let $X$ be a set with a right action of $\Sigma$, and $\Upsilon$ any collection as above. Let $M(K) = C^{\infty}_{c}( X / K ;  R )$ be the set of all functions $\zeta \colon X  \to  R $ which factor through $ X/K $ and are non-zero only on finitely many orbits. For $ g \in \Sigma $, we denote by  $g \star \zeta$ the composition $ X  \xrightarrow {g} X \xrightarrow{\zeta}  R $, which gives a \emph{left} action on the set of functions from $X$ to $R$. Then, for elements $\sigma , \tau^{-1} \in \Sigma$ and morphisms $\sigma \colon L \to K$, $\tau \colon L' \to K  '  $, we set
\begin{align*} 
[\sigma]^{*} & \colon M(K) \to M(L) \quad \quad \quad \zeta \mapsto \sigma \star \zeta  \\ 
[\tau]_{*} & \colon M(L') \to M(K') \quad \quad  \, \, \,   \zeta  \mapsto  \sum _ { \gamma }   \left( \gamma \tau ^ { - 1 } \right) \star  \zeta 
\end{align*}
where the sum runs over representatives in $K'$ of the coset space $K'/(\tau ^{-1} L ' \tau)$. Then $M$ is CoMack.   
\end{example}

\begin{example}   \label{shimuracohomology}     
The  prototypical example of a cohomology functor for us will be the cohomology of Shimura--Deligne varieties with coefficients in an equivariant \'{e}tale sheaf (c.f. Propositions \ref{PropRationalCohFunctor} and \ref{PropIntegralCohFunc}).
\end{example}

\subsection{Completions}   Let $ M  :  (G,  \Sigma,  \Upsilon)  \to  R $\textbf{-Mod} be a cohomology functor.    
For any compact subgroup $C \subset \Sigma$ of $G$, we can define two completions 
\[
\overline{M}(C) = \varinjlim_{K\supset C} M(K), \quad M_{\Iw}(C) = \varprojlim_{K \supset C} M(K)
\]
where the limits are taken with respect to the morphisms $\pr^{*}$ and $\pr_{*}$ respectively over all $K \in \Upsilon$ containing $C$.   We  call  these the  \emph{inductive} and \emph{Iwasawa} completions  respectively. When $ C = \left \{1 \right \}$, we  also  denote the inductive completion by $ \widehat{M} $. We denote by $
j_K \colon M(K) \to \overline {M}(C) $ the natural map. 
If $C$ is central (i.e. contained in the center of $\Sigma$), the space $\overline{M}(C)$ naturally carries a smooth action of $\Sigma$ as follows. For an element $g \in \Sigma$ and $x \in \overline{M}(C)$, choose a compact open $K \in \Upsilon$ such that there exists $x_K \in M(K)$ with $j_{K}(x_{K}) = x$. By replacing $K$ with a subgroup of $K \cap g^{-1} K g$ contained in $\Upsilon$, we may assume that $g K g^{-1} \in \Upsilon$. We then define $g \cdot x$ to be the image of $x_K$ under the composition 
\[
M(K) \xrightarrow {[g]^{*}} M(gKg^{-1}) \to \overline{M}(C).
\]
It is a routine check  that this is well-defined, and the action is smooth by property (C3) in Definition \ref{CohoFunc}.

\begin{lemma} 
Suppose $M$ is a Galois functor, $C$ is central. Then $ j_{K} : M(K)  \xrightarrow{\sim} \overline{M}(C)^{K} $ for any $K \supset C$ and $K \in \Upsilon$. In particular, if we choose a left invariant  Haar measure on $ G $ giving $ K $ measure one and $ \Sigma = G $, then  $M(K)$ is a  (left)   module  over the  Hecke algebra  $\mathcal{H}(K \backslash G / K )$ with  the action of the Hecke operator  $   \mathrm{ch}( K \sigma K)  $ on $ x \in M(K)  \hookrightarrow  \overline{M}(C)  $  given  by   $$  \ch  (  K\sigma K ) \cdot x   =  \sum _ { \alpha \in K  \sigma K /  K  }  \alpha \cdot j_{K}(x) $$
\end{lemma} 
\begin{proof}
Let $x \in \overline{M}(C)^{K}$. Then $x$ is in the image of the map $j_L$ for some $L \in \Upsilon$ such that $L \triangleleft K$ and $L \supset C$. This implies that 
\[
x \in M(L)^{K/L} \xleftarrow{\sim}  M(K)
\]
by (G).   The  explicit  action of the  Hecke operator can then be read off via the integral   in   \cite[Eqn 4.2.2]{Bushnell}.   
\end{proof}

\subsection{Hecke correspondences} \label{HeckeCorrs}
Let $M \colon   (G, \Sigma, \Upsilon) \to R\textbf{-Mod}$ be a cohomology functor. For every $K, K' \in \Upsilon$ and $\sigma \in \Sigma$, we have a diagram
\begin{center}
 \begin{tikzcd}
 & K \cap \sigma^{-1} K' \sigma \arrow[dl, "\pr" , swap ] & \sigma K \sigma^{-1} \cap K' \arrow[l,"{[\sigma]}",swap] \arrow[dr, "\pr" ] & \\ 
 K & & & K'
 \end{tikzcd} 
\end{center} 
This induces a map 
\[
[K' \sigma K] \colon M(K) \xrightarrow {\pr^{*}} M(K \cap \sigma^{-1} K' \sigma) \xrightarrow {[\sigma]^{*}} M(\sigma K \sigma^{-1} \cap K') \xrightarrow {\pr_{*}} M(K')
\]
which we refer to as the \emph{Hecke correspondence} induced by $\sigma$. It is straightforward to verify that $[K' \sigma K] $ only depends on the double coset $K ' \sigma K$. 

\begin{lemma} 
Suppose that $ M \colon   (G,\Sigma, \Upsilon) \to R$\textbf{-Mod} is Cartesian. Let $ K , K' , K'' \in \Upsilon $ and $ \sigma, \tau \in  \Sigma   $. 
\begin{enumerate}[(a)]
 \item If we write $K' \sigma K = \sqcup \; \alpha K$, then we have
 \[
 j_{K'} \circ [K' \sigma K] = \sum \alpha \cdot j_{K}. 
 \]
 \item The composition $ j_{K''} \circ [K'' \tau K'] \circ [K' \sigma K] $ is given by the convolution product of double cosets, i.e. if we write $K' \sigma K = \sqcup \; \alpha K$ and $K'' \tau K' = \sqcup \; \beta K'$ then
 \[
 j_{K''} \circ [K'' \tau K'] \circ [K' \sigma K] = \sum (\beta \alpha) \cdot j_K .
 \]
 \item If $ M$ is moreover Galois,   $  \Sigma =  G   $      and a left invariant   Haar measure on $ G $ is chosen giving $ K $ measure $ 1 $, then the actions on $ M(K) $  of the correspondence $ [K \sigma K] $ and the operator $ \mathrm{ch}(K\sigma K) $ defined in the previous subsection   agree. 
\end{enumerate}
\end{lemma} 
\begin{proof} 
Let $L' \defeq \sigma K \sigma^{-1} \cap K'$ which is an element of $\Upsilon$. We can (and do) choose a compact open subgroup $L \in \Upsilon$ satisfying $L \triangleleft K'$ and $L \subset L'$, and let $ \{\gamma\}_{\gamma \in I}$ denote any set of representatives in $K'$ of the coset space $K' / L' = L \backslash K' / L'$; for any such $\gamma$ we set $L_{\gamma} = \gamma L' \gamma^{-1} \cap L = L $. By the Cartesian property, we obtain the following commutative diagram
\begin{center}
\begin{tikzcd}
 & \bigoplus_{\gamma} M(L) \arrow[r, "\sum \mathrm{pr} _ { * }"] & M(L) \\
M(K) \arrow[r, "{[\sigma]^{*}}"] \arrow[rr, "{[K' \sigma K]}"', bend right, swap] & M(L') \arrow[r, "\mathrm{pr}_{*}"] \arrow[u, "{\sum [\gamma]^{*}}"] & M(K') \arrow[u, "\mathrm{pr}^{*}"']
\end{tikzcd}
\end{center} 
which implies that $\pr_{L,K'}^{*} \circ [K' \sigma K] = \sum_{\gamma} [\gamma \sigma]^{*}$. It is then easily verified that $ \gamma \sigma $ are distinct representatives of $ K ' \sigma K / K $, and that the action so defined is  independent of the choice  of  representatives. This completes the proof of part (a).  For (b), we note that the composition of the Hecke correspondences in the limit is given by
\[
\begin{tikzcd}
M(K) \arrow[r, "{\sum [\alpha]^{*}}"] & M(L) \arrow[r, "{\sum [\beta] ^ { * } }"]& M (J ) 
\end{tikzcd} 
\]
where $ J \in \Upsilon $ is such that $ J \triangleleft K'' $ and $ J \subset \tau L \tau^{-1} $. Part (c) is  immediate.  
\end{proof}

The next lemma is a useful criterion for when two Hecke correspondences commute with each other.

\begin{lemma} \label{Heckecommutes} 
For $i=1, 2$ set $ \mathcal{P}_{i} =  \mathcal{P}(G_{i}, \Sigma_{i} ,  \Upsilon_{i}) $ and $ \mathcal{P} = \mathcal{P}(G,\Sigma, \Upsilon) = \mathcal{P}_{1} \times \mathcal{P}_{2} $. Let $ M \colon \mathcal{P} \to R \textbf{-Mod} $ be a Cartesian cohomology functor, and take
\begin{itemize} 
\item $ \sigma \in \Sigma_{2} $
\item $ J \in \Upsilon_{2} $
\item $ K = K_{1} \times J $ and $L = L_{1} \times  J$ both elements in $\Upsilon$
\item $ g = (g_{1}, 1) \colon L \to K $ a morphism in $ \mathcal{P} $. 
\end{itemize}
Then we have 
\[
[ g ] ^ { *} \circ [ K \sigma K ] =  [L \sigma L ] \circ [ g ] ^{* } . 
\]
A similar result holds for pushforwards. 
\end{lemma}    

\begin{proof}
Let $ K^{g} =  g K g^{-1} $. Then we have $ g (\sigma K \sigma^{-1} \cap K ) g^{-1} = \sigma K^{g} \sigma^{-1} \cap K^{g}  $. It suffices to prove that 
\begin{center}
    \begin{tikzcd}  M( \sigma L \sigma ^{-1}  \cap L ) \arrow[r, "\pr_{*}"]  &  M ( L )  \\             
    M ( \sigma K^{g}  \sigma ^{-1}  \cap  K^{g}) \arrow[u,"\pr^{*}"]        \arrow[r, "\pr_{*}"]  &  M( K^{g})    \arrow[u,"\pr^{*}",swap]
    \end{tikzcd}
\end{center}
But this is immediate from the Cartesian property since 
\[
L \backslash  K^{g} / ( \sigma K^{g} \sigma^{-1} \cap K^{g} )  = ( L_ {1} \times J  ) \backslash  ( K_{1}^{g_{1}} \times J ) /   K_{1}^{g_{1}} \times (\sigma J \sigma^{-1} \cap J)  =  \left \{ \mathrm{id}_{K} \right \}   
\]
and $ L \cap ( \sigma K^{g} \sigma^{-1} \cap K^{g} ) = \sigma L \sigma ^{-1} \cap  L $.
\end{proof}

\subsection{Pushforwards between cohomology functors} 

Let $ \iota \colon H \hookrightarrow G $ be a closed subgroup, and suppose that $ M _ { H } , M_{G} $ are cohomology functors for the triplets $(H, \Sigma_H, \Upsilon _ { H } )$ and $(G, \Sigma_G , \Upsilon _ { G })$ respectively. We  require  $ \iota (  \Sigma _ { H }  )  \subset  \Sigma _{G} $ and  that   for all $ U \in  \Upsilon _ { H } $, $ K \in  \Upsilon _ { G } $,  we have   $ U  \cap K \in \Upsilon _ { H }  $.  We say that  an   element $ (U,K) \in \Upsilon _ { H } \times \Upsilon _ { G } $ forms a \emph{compatible} pair if $ U \subset K $.  A     \emph{morphism} of  compatible pairs $ (V,L) \to (U,K) $ is a pair of morphisms $ [h] \colon V \to U $, $ [h] \colon L \to K $ for $ h \in \Sigma_H $. 

\begin{definition} \label{push}  A \emph{pushforward} $ \iota _ { * } \colon M_{H } \to M_{G } $ is a collection of maps $ M_{H}(U ) \to M_{G}(K) $ for all compatible pairs $ (U,K) $, which commute with the pushforward maps induced by morphisms  of compatible pairs coming  from  $ \Sigma _  {  H  }  ^  { - 1   }  $.     A \emph{Cartesian pushforward} $ \iota _ { * } \colon M_{H} \to M_{G } $ is a pushforward such that for all $ U \in \Upsilon _ { H } $, $ L , K \in \Upsilon _ { G } $ and $ U , L \subset K $, we have a commutative diagram 
\begin{center}
 \begin{tikzcd} 
 \bigoplus _ { \gamma } M_H ( U _ { \gamma } ) \arrow[r, "{\sum [\gamma]_{*}}"]& M_G ( L ) \\
 M_H ( U ) \arrow[u, " \sum \pr^{ * } " ] \arrow [ r , "\iota_ { * } " ] &    M_G( K ) \arrow[ u,"\pr ^ { * } ",swap] 
\end{tikzcd} 
\end{center} 
where $ \gamma \in U \backslash K / L $ is a fixed set of representatives, $ U _ { \gamma } = \gamma L \gamma ^ { - 1 } \cap U     $  and $ [ \gamma ] _  { * } \colon M_H(U_{\gamma}   ) \to M_G(L) $ is the composition $ M_H(U_{\gamma} ) \to M_G(\gamma L \gamma^{-1} )   \xrightarrow{} M_G(L)  $.
\end{definition}    

\begin{remark} 
When $ G = H $   and $ \iota _ {*} = \pr_{*} $, one  recovers the conditions for $ M_G $ to be Cartesian.  We  point out that what we refer to as a Cartesian pushforward here is  simply called a pushforward in \cite{Loeffler19}.     
\end{remark}

Assume for  the rest of this section that $ H$ is unimodular, $ \mathcal{P}(H, \Sigma_{H},  \Upsilon_{H}) $, $ \mathcal{P} (G , \Sigma_{G}, \Upsilon_{G} )  $ are  categories of compact opens  with $ \Sigma_{H} = H $, $ \Sigma _{G}  = G  $  and  $ R $ is a $ \QQ $-algebra. We  fix  left  invariant    Haar measures $dh $, $dg $ on $ H $, $ G  $ respectively that take values in $ \QQ $ on the  respective  compact opens, and  denote by  $ \Vol( -  ) $ the volume   of the corresponding  compact  opens.  We consider $ \mathcal{H}(G, R) $ as an algebra under the usual convolution 
product $ * $ (see \cite[\S 1.4.1]{Bushnell}) and a  left representation of $ G $ with $ g \in G $ acting on $  \xi \in \mathcal{H}(G, R) $ via $ \rho_{g}(\xi) : x \mapsto  \xi ( x g) $. The action $ \rho $ of $ G $ induces an action of the Hecke algebra on itself  \cite[1.4.2]{Bushnell}, which we denote by $ \xi * _{\rho } \zeta $ for $  \xi \in \mathcal{H}(G,R) $ (considered as a ring element) and $ \zeta \in \mathcal{H}(G,R) $ (considered as a module element).      Finally, for $ \xi \in  \mathcal{H}(G, R) $, we let $ \xi ^{t} $ denote its transpose $ x  \mapsto    \xi (  x^{-1} )  $.  

\begin{definition}  \label{DefinitionOfIntertwiningMap}
Given  smooth  representations $ \tau $ of $ H $, $ \sigma $ of $ G $ and $ C $ a central compact subgroup of $ G $ which acts trivially on $ \sigma  $, we consider  $ \tau \otimes \mathcal{H}(G,R) ^{C}  $ and $ \sigma $ smooth representations  of $ H \times G $ by the following  \emph{extended action}. 
\begin{itemize} 
\item $ (h, g) \in H \times G $ acts on $ x \otimes \xi \in \tau \otimes     \mathcal{H}(G, R ) ^{C} $ via $ x \otimes \xi \mapsto  h x  \otimes  \xi ( h  ^ { - 1  } ( - )  g )   $.   
\item $ H \times  G  $   acts  on $ \sigma $  via projection to $ G $ which acts via its natural action.   
\end{itemize}          
An    \emph{intertwining map} $ \mathfrak{Z} : \tau  \otimes \mathcal{H}(G,R)^{C} \to \sigma $   is just a morphism of $ H \times G $ representations. 
\end{definition}

\begin{lemma}   \label{Frobenius}         Let $ \tau, \sigma $ and $ \mathfrak{Z} $ be as above.     
\begin{enumerate}[(a)]
    \item   
    For  all $  x \in  \tau   $, $ \xi_{1} , \xi_{2}  \in  \mathcal{H}(G,R) ^ { C } $,  $$   \mathfrak{Z}( x \otimes \xi_{1} * \xi_{2} )  =   \xi_{2} ^ { t}         \cdot           \mathfrak{Z} ( x \otimes \xi_{1} )  . $$    
\item  (Frobenius  Reciprocity)  Let $ \sigma^\vee $ denote the smooth dual of $ \sigma $  and  denote by $ \langle \cdot , \cdot  \rangle : \sigma^\vee \times \sigma \to  R $ the induced pairing. Consider $  \tau  \otimes  \sigma  ^  {  \vee    } $ as a smooth $ H $-representation via $  h \cdot ( x \otimes \varphi )  =   h \cdot x  \otimes    \iota ( h ) \varphi $ for $ h \in  H $, $ x \in \tau $, $ \varphi \in  \sigma^\vee  $. There is a unique  element  $  \mathfrak{z}  \in   \Hom _ { H  }   (  \tau \otimes \sigma^\vee ,  R )  $ (depending on the choice of $ dg $) such that $$  \langle   \varphi ,  \mathfrak{Z}  ( x \otimes \xi )  \rangle   =  \mathfrak{z} ( x  \otimes  \xi  \cdot \varphi ) $$
for all $ \varphi \in \sigma^\vee $, $ x \in \tau $, $ \xi  \in  \mathcal{H}(G, R) ^{C} $. 
\end{enumerate}    
\end{lemma}   
\begin{proof} For (a),  note that the $ G $-equivariance  of  $ \mathfrak{Z} $ tells us that $   \mathfrak{Z}  ( x \otimes  ( \xi_{2} ^{t}   * _{\rho} \xi_1 )  )  =   \xi_{2} ^{t}   \cdot  \mathfrak{Z }    ( x  \otimes  \xi_{1} )  $  \cite[Proposition 1, p.35]{Bushnell}. Since $ \xi_{1}  *   \xi_{2} = \xi_{2} ^{t}    * _ { \rho }  \xi_{1}   $, we have the first claim.  \\

\noindent (b)        Given  $  x \in  \tau   $,   $  \varphi \in  \sigma^\vee    $, choose  a compact open $ K  \supset  C $  such that the idempotent $ e_{K} =  \Vol(K;dg) ^{-1}  \mathrm{ch}(K)   \in  \mathcal{H}(G, R) ^ { C } $ fixes $ \varphi $,  and    set $  \mathfrak{z} (  x  \otimes  \varphi  )   :  =   \langle   \varphi ,   \mathfrak{Z}    ( x \otimes e_{K} )  \rangle  $.  It is easily verified that $ \mathfrak{z}  $ is  well-defined, satisfies the  property  claimed  above and is uniquely  determined by its  properties.                            
\end{proof}    
We now construct a ``completed pushforward"  from a given pushforward $ M _ { H } \to M_{G} $ i.e. an intertwining  map of  $ H \times  G  $ representations  as    above. This is a straightforward generalisation of the ``Lemma--Eisenstein'' map in \cite[Lemma 8.2.1]{LSZ17}.

\begin{proposition}[Completed Pushforward] \label{LemmEin} 
Suppose  $ M_ { H } $, $ M_{G} $ are  cohomology functors  with $ M_{H} $ CoMack and $ M_{G} $ satisfying (M).  Consider $ \widehat{M}_{H}   \otimes  \mathcal{H}(G, R ) ^ { C }   $ and $  \overline{M}_{G}(C) $  as smooth $ H \times G $ representations via the extended action.  Then for any pushforward $  \iota_{*}  : M_{H}  \to  M_{G} $, there is a unique  (depending on the choice of $ dh $)    intertwining map of $ H  \times  G  $ representations 
\begin{align*} 
\hat{\iota} _ { * } : \widehat{M}_{H} \otimes \mathcal{H}(G  ,  R ) ^ {  C  }  \to   \overline { M } _{G} ( C )  
\end{align*} satisfying the following  compatibility    condition:  for all compatible pairs $ (U,K)  \in  \Upsilon _{H} \times \Upsilon_{G}  $ with  $ K \supset C $,  $   x \in M_{H}(U) $, $ y \in M_{G}(K) $ such that $ \iota _{ *}  ( x  ) = y  $,  we  have $  \hat { \iota } _{*} ( j_{U}(x) \otimes \ch(K) ) =  \Vol(U)  \cdot  j_{K} (   y    ) $.  
\end{proposition} 
\begin{proof}
We  first  assume   that $ C = \left \{  1  \right  \}   $.        Elements of $ \widehat{M}_{H} \otimes \mathcal{H}(G,R)   $ are spanned by pure tensors of the form $ x \otimes \ch ( g K ) $ for $ x \in \widehat{M}_{ H } $, $ g \in G $ and $ K \in \Upsilon _ { G } $. Indeed, for any compact open subgroup $K$ of $G$, there exists a compact open subgroup such that $K' \subset K$ and $K' \in \Upsilon$, because we require that $\Upsilon$ contains a basis of the identity. We will define the map on these pure tensors and extend linearly.

Recall that we are assuming $R$ is a $\mbb{Q}$-algebra, hence $M_H$ is also Galois.  Choose a compact open subgroup $ U \in \Upsilon _ { H } $ such that $ U $ fixes $ x $, $ U \subset g K g ^{-1} \cap H $, and let $ x_{U} \in M _{ H } (U) $ be an  element that maps to $ x $ under $ j_{U} $. Then, we define $ \hat{\iota} _ { * } \left( x \otimes \opn{ch}(gK)\right) $ to be the image of $ \Vol(U) x_{U} $ under the composition
\[
M _ { H } ( U ) \xrightarrow { \iota_ { * } } M _ { G } ( g K g ^{-1} ) \xrightarrow { [ g ] _ { * } } M _ { G } ( K ) \xrightarrow { j _{ K } } \widehat{M}_{G } 
\]
To see that this is independent of choice of $ x_{U} $ (and $ U $),   assume that $ U $ is replaced by another open  compact subgroup $ V $ and $ x_{U} $ by $ x_{V} \in M_H ( V )   $. We can then choose an even smaller compact open $ V ' \in \Upsilon_{H} $, $ V'  \subset U \cap V $ such that $ x_{U} , x_{V} \mapsto x_{V'} $ under pullbacks,   and we may therefore assume that $ V \subset U $, and  $ x_{U} \mapsto x_{V} $ under $ \pr ^ { * } _  { V , U } $. Since $ M_{H} $ is a covering functor, $ x_{V } \mapsto [U:V] x _ { U } $ under  $ \pr _ { V,U , *} $.    
Hence both $ \Vol(U) x_{U} $ and $ \Vol(V) x_{V} $ map to the same element in $ \widehat{M}_{G} $. 

It remains to show that that the various relations among elements of $ \mathcal{H} ( G,  R) $ do not give conflicting images on the sums of these simple tensors. To this end, let $ L, K \in \Upsilon _ { G } $ such that $ L $ is a normal subgroup of $ K $. We want to verify that
\[
\hat{\iota}_ {* } ( x \otimes \ch ( K ) ) = \sum _ { \gamma \in K / L } \hat{\iota}_ { * } ( x \otimes \ch ( \gamma L ) ) .
\]
The general case for two different representations of $ \xi \in \mathcal{H}( G,  R) $ as sums of characteristic functions of left cosets can be reduced to this case by successively choosing normal subgroup for pairs and twisting by the action of $ G $, by establishing $ G $-equivariance of such a (a priori hypothetical) map first.

Choose $ U \in \Upsilon _ { H } $ such that $ U $ fixes $ x $ and $ U \subset L \cap H $. Note that as $ L \triangleleft K $, $ \gamma L \gamma ^ { - 1 } \cap H = L \cap H $ for any $ \gamma \in K $. As before, let $ x_{U } \in M_{H}(U) $ be an element mapping to $ x $. By definition
\[
\hat { \iota } _ { * } ( x \otimes \ch ( K ) ) = j _{ K } \circ \iota_{ U , K , * } ( \Vol ( U ) x_{U } ) , \quad \quad \hat { \iota } _ { * } ( x \otimes \ch ( \gamma L ) ) = j_{ L } \circ [\gamma ] _{ * } \circ \iota_{ U , L , * } ( \Vol ( U ) x _ { U } ) . 
\]
As $ j _ { K } = j _ { L } \circ \pr _ { L , K } ^ { *} $, it suffices to prove that 
\[
\pr ^ { * } _ { L , K } \circ \iota _ { U , K , * } = \sum _ { \gamma } [ \gamma ] _ { * } \circ \iota _ { U , L , * } = \sum _ { \gamma } [ \gamma ] ^ { * } \circ \iota _ { U , L , * } . 
\]
as maps $ M _{ H } ( U ) \to M_ { G } ( L ) $, where the last equality follows from the fact that $L$ is normal in $K$, so we can replace the set of representatives by their inverses. This equality is then justified by replacing $ \iota _ { U , K , * } = \pr _ { L , K , * } \circ \iota _ { U , L , * } $ and invoking the axiom (M) for $ M _ { G } $. Therefore, $ \hat { \iota } _ {* } $ is well-defined.

We now check that $ \hat { \iota } _ { * } $ is $ H \times G $ equivariant with the said actions: let 
\[
(h,g) \in H \times G , \quad \quad v = x \otimes \ch(g_{1}K) \in \widehat { M } _{H } \otimes \mathcal{H}(G,  R) . 
\]
Then, 
\[
(h,g) \cdot v = h x \otimes \ch ( h g ' K ' ) 
\]
where $ g ' = g_{1} g ^{-1} $, $ K ' = g K g ^ { - 1 } $. Choose $ U \in \Upsilon _ { H } $ such that $ U $ fixes $ x $, $ U \subset g_{1} K g _{1}^{-1} \cap H $ and $ x_{U} \in M_{H} ( U ) $ that maps to $ x $. Then, $ h U h ^{-1} \subset h g_{1} K (h g_{1}) ^{ -1 } \cap H = hg ' K ' ( hg ' ) ^{-1} \cap H $, and $ [ h ] ^ { * } x _ { U } \in M_{H}( h U h ^{-1} ) $ maps to $ h x $. Since $[h]^* = [h^{-1}]_*$, we obtain a commutative diagram 
\begin{center}
\begin{tikzcd} 
M_{H}( U ) \arrow[r , "\iota _ { * }" ] \arrow[d, "{[h]^{*}}", swap] & M_{G}( g_{1} K g_{1} ^ { - 1 } ) \arrow[d, "{[h]^{*}}"] \\
M_{H} ( h U h ^ { - 1 } ) \arrow[r, "\iota _ { * }" ] & M_{G} ( h g ' K ' ( h g ' ) ^ { - 1 } ) 
\end{tikzcd}
\end{center}
Now, 
\begin{align*} (h,g) \cdot \hat { \iota } _{*} ( v ) & = \Vol ( U ) \cdot \left[ g \cdot \left ( j_{ K } \circ [ g_{1} ] _ { * } \circ \iota_{*}( x_{U} ) \right ) \right] \\ \hat {\iota } _ { * } ( ( h , g ) \cdot v ) & = \Vol(hUh^{-1} ) \cdot \left ( j_{K'} \circ [ h g ' ]_{ * } \circ \iota _ { * } ( [ h ] ^ { * } x_{U} ) \right ) . 
\end{align*} 
As $ H $ is  unimodular,  $ \Vol ( U ) = \Vol ( h U h ^ { - 1 } ) $ and it therefore suffices to verify that $ [ g ^{-1} ] _ { * } \circ [ g_{1} ] _{*} \circ \iota _ { * } ( x_{U } ) = [ h g' ] _ { * } \circ \iota _ { * } ( [ h ] ^ { * } x_{ U } ) $ 
as elements of $ M _ { G} ( K ' ) $. But this follows immediately from the commutativity of the above diagram and the composition law $ [ h g ] _ { * } = [ g ]_{ * } [ h ] _ { * } $. This finishes the checking of equivariance. The statement  involving $ C $ is obtained by taking $ C $ invariants on both sides, and noticing that the map factors through $ \overline { M } ( C ) \subset \widehat { M } ^ { C } $. It is easily seen that this map is uniquely  determined with the prescribed actions of $ H \times G $ and the compatibility condition with $ \iota _ { * }  $.
\end{proof}   

Ideally one would like an integral version of the above proposition but such a map doesn't exist in general due to the presence of the volume factors in the definition. It is however possible to relate the pushforward map on finite level to an integral version, as follows. 

Let $\ordd$ be an integral domain with field of fractions $\Phi$, and let $M_{\ordd}$ and $M_{\Phi}$ be cohomology functors for $(G, \Sigma, \Upsilon)$ valued in $\ordd$ and $\Phi$ respectively. We say that $M_{\ordd}$ and $M_{\Phi}$ are \emph{compatible under base change} if there exists a collection of $\ordd$-linear maps $M_{\ordd}(K) \to M_{\Phi}(K)$ which are natural in $K$ and commute with pullbacks and pushforwards.

Let $\iota \colon H \hookrightarrow G$ be as above, and let $(M_{H, \ordd}, M_{H, \Phi})$ and $(M_{G, \ordd}, M_{G, \Phi})$ be pairs of compatible cohomology functors. Suppose that we have pushforwards $\iota_* \colon M_{H, \Phi} \to M_{G, \Phi}$ and $\iota_* \colon M_{H, \mathcal{O}} \to M_{G, \mathcal{O}}$ which are compatible with each other via the $\mathcal{O}$-linear maps $M_{H, \mathcal{O}}(-) \to M_{H, \Phi}(-)$ and $M_{G, \mathcal{O}}(-) \to M_{G, \Phi}(-)$. Then we obtain the following corollary.
\begin{corollary} \label{CommDiagramCor}
Let $\Sigma_G = G$ and $\Sigma_H = H$. Keeping with the same notation as above, let $U \in \Upsilon_H$ and $K \in \Upsilon_G$ be compact open subgroups, and suppose that we have an element $g \in G$ such that $U \subset g K g^{-1} \cap H$. Then we have a commutative diagram:
\[
\begin{tikzcd}
{\widehat{M}_{H, \Phi}} \arrow[rr, "x \mapsto \hat{\iota}_{*} \left( x \otimes \opn{ch}(gK) \right)"] & & {\widehat{M}_{G, \Phi}} \\
{M_{H, \ordd}(U)} \arrow[u, "j_U"] \arrow[r, "\iota_{*}"] & {M_{G, \ordd}(gKg^{-1})} \arrow[r, "{[g]_{*}}"] & {M_{G, \ordd}(K)} \arrow[u, "\Vol(U) \cdot j_K"']
\end{tikzcd}
\]
\end{corollary}


\section{Preliminaries II: Unitary groups}

\subsection{The groups} \label{TheGroups}

In this section, we will define the algebraic groups and their associated Shimura data that will appear throughout the paper. All groups and morphisms will be defined over $ \QQ $ unless specified otherwise, and maps of algebraic groups will be defined on their $R$-points for $ \QQ $-algebras $ R  $. For an arbitrary algebraic group $  \mathbf{K} $, we denote by $ \Tb_{\mathbf{K}} $ its maximal abelian quotient, and $ \mathbf{Z}_{\mathbf{K}} $ its centre.  

Let $J_{r, s}$ denote the $(r+s) \times (r+s)$ diagonal matrix $\opn{diag}(1, \dots, 1, -1, \dots, -1)$ comprising of $r$ copies of $1$ and $s$ copies of $-1$. We define $\opn{U}(p, q)$ (resp. $\opn{GU}(p, q)$) to be the unitary (resp. unitary similitude) groups whose $ R $-points for a $ \QQ $-algebra $ R $ are given by
\begin{eqnarray}
\opn{U}(p,q)(R) & \defeq & \left\{ g \in \opn{GL}_{p+q}(E\otimes_{\QQ} R) \colon {^t \bar{g}} J_{p, q} g = J_{p, q} \right\} \nonumber \\
\opn{GU}(p, q)(R) & \defeq & \left\{ g \in \opn{GL}_{p+q, E}(E \otimes_{\QQ} R) \colon \text{ there exists } \lambda \in \mbb{G}_m(R) \text{ s.t. } {^t \bar{g}} J_{p, q} g = \lambda J_{p, q} \right\} \nonumber. 
\end{eqnarray} The group $  \GU(p,q) $ comes with two morphisms 
\[
c  \colon  \GU(p,q) \to \mbb{G}_{m} \quad \quad  \det \colon \GU(p,q) \to \Res_{E/\QQ} \mbb{G}_{m} 
\]
where $ c(g) = \lambda $ is the similitude factor and $ \det(g) $ is the determinant, satisfying the relation $\overline{\opn{det}(g)} \opn{det}(g) = c(g)^{p+q}$. To ease notation, we will omit $q$ if it is zero. 
 
\begin{remark} These groups arise as the generic fibre of group schemes over $ \ZZ $ by replacing $ E $ with $ \OO_{E} $ in the definition. If $ \ell $ is unramified in $ E $, the pullback to $ \QQ_{\ell} $ is unramified and the $ \ZZ_{\ell} $-points form a hyperspecial maximal compact subgroup. 
\end{remark} 

Let $n \geq 1$ be an integer. We define:
\begin{itemize}
    \item $\mbf{G}_0 \defeq \opn{U}(1, 2n-1)$.
    \item $\mbf{G} \defeq \opn{GU}(1, 2n-1)$.
    \item $\mbf{H}_0 \defeq \opn{U}(1, n-1) \times \opn{U}(0, n)$.
    \item $\mbf{H} \defeq \opn{GU}(1, n-1) \times_c \opn{GU}(0, n)$ where the product is fibred over the similitude map. 
\end{itemize}
We have natural inclusions $\iota \colon \mbf{H}_0 \hookrightarrow \mbf{G}_0$ and $\iota \colon \mbf{H} \hookrightarrow \mbf{G}$, both given by sending an element $(h_1, h_2)$ to the $2n \times 2n$ block matrix 
\[
\iota(h_1, h_2) = \tbyt{h_1}{}{}{h_2}. 
\]
Throughout, we will let $\mbf{T}$ denote the algebraic torus $\opn{U}(1) =\{x \in  \mathrm{Res}_{E/\QQ} \mbb{G}_{m, E} : \bar{x} x = 1 \} $ and we note that
\begin{itemize} 
\item $ \mathbf{Z}_{\Gb_{0}} = \Tb, \,    \mathbf{Z}_{\Hb_{0}} =  \Tb \times \Tb $
\item $ \mathbf{Z}_{\Gb} =  \GU(1), \,  \mathbf{Z}_{\Hb} \cong  \Tb \times  \GU(1)    $
\item $   \Tb_{\Hb} \cong \Tb  \times \opn{GU}(1)$ if $ n $ is odd, and $  \Tb_{\Hb} \cong \Tb \times \Tb \times \mathbb{G}_{m} $ if $ n $ is even. 
\item $\opn{Res}_{E/\mbb{Q}}\mbb{G}_{m, E} \cong \opn{GU}(1)$.
\end{itemize}
When $ n $ is odd, the quotient map is given explicitly by  
\[
\Hb \to \Tb_{\Hb} , \quad \quad  (h_{1}, h_{2})  \mapsto  ( \det h_{2} / \det h_{1}  , c(h_{1}) ^ { (n+1)/2}   \det ( h_{1}  ^{-1}  )   ).
\]
\begin{remark} 
This quotient in fact contains more information than we are interested in -- to control the anticyclotomic variation of our Euler system we only need the fact that $\mbf{T}$ is a quotient of $\mbf{H}$. 
\end{remark} 

We consider the group $ \tilde{\Gb} = \Gb \times \Tb $. The group $\mbf{H}$ extends to a subgroup of $\Gt$ via the map $(h_1, h_2) \mapsto (\iota(h_1, h_2) , \nu(h_1, h_2) )$, where $\nu(h_1, h_2) = \opn{det}h_2/\opn{det}h_1$.

\begin{remark} \label{RemIdentifyGroups}
If $R$ is an $E$-algebra, then one has an identification $\opn{U}(p, q)_R = \opn{GL}_{p+q, R}$ given by sending a matrix $g \in \opn{U}(p, q)(R)$ to the matrix $g' \in \opn{GL}_{p+q}(R)$ obtained from projecting the entries of $g$ to the first component in the identification
\begin{align*}
    E \otimes_{\mbb{Q}} R &= R \oplus R \\
    \lambda \otimes x &\mapsto (\lambda x, \overline{\lambda} x) .
\end{align*}
Similarly we have an identification $\opn{GU}(p, q)_R = \opn{GL}_{1, R} \times \opn{GL}_{p+q, R}$ given by sending a matrix $g \in \opn{GU}(p, q)(R)$ to the pair $(c(g); g')$. We will frequently use these identifications in the cases when $R = E$ or $R = \mbb{Q}_p$ for a prime $p$ which splits in $E/\mbb{Q}$, with the $E$-algebra structure on $\mbb{Q}_p$ given by the fixed embedding $\overline{\mbb{Q}} \hookrightarrow \Qpb$. In this case, the groups $\mbf{H}_{0, R}$, $\mbf{G}_{0, R}$, $\mbf{H}_R$, $\mbf{G}_R$ and $\Gt_R$ are identified with $\opn{GL}_{n, R} \times \opn{GL}_{n, R}$, $\opn{GL}_{2n, R}$, $\opn{GL}_{1, R} \times \opn{GL}_{n, R} \times \opn{GL}_{n, R}$, $\opn{GL}_{1, R} \times \opn{GL}_{2n, R}$ and $\opn{GL}_{1, R} \times \opn{GL}_{2n, R} \times \opn{GL}_{1, R}$ respectively. 
\end{remark}

We will need the following two lemmas later on:     
\begin{lemma}  \label{ACCFT}    
The map 
\begin{eqnarray} 
\Nm \colon \mbb{A}_{E, f}^{\times} & \to &  \Tb (\Ab_{f}) \nonumber \\
 z & \mapsto & \bar{z}/z \nonumber 
\end{eqnarray}
is open and surjective. 
\end{lemma}

\begin{proof}   
A  continuous  surjective  homomorphism of $ \sigma $-countable locally  compact  groups is open, and so it suffices to prove that the map is surjective. We have an exact sequence $ 1 \to \mbb{G}_{m}  \to \opn{GU}(1)  \xrightarrow{\Nm}  \Tb \to  1   $
of algebraic tori over $ \QQ $, which gives an exact sequence 
\[
1 \to   \Ab_{f}   ^{\times}  \to  \Ab_{E,f}^{\times}  \to    \Tb (  \Ab_{f} ) . 
\]
For each finite place $ \ell $, the map $ (E \otimes \mbb{Q}_\ell)^{\times}  \to \Tb(\QQ_{\ell}) $ is surjective, since $ \opn{H}^{1}(\QQ_{\ell}, \mbb{G}_{m}) = 0 $ by Hilbert's theorem 90. For $ \ell \nmid  \mathrm{disc}(E) $, both $ \opn{GU}(1) $ and $ \Tb $ have smooth models over $ \ZZ_{\ell}  $ and the map extends to these models. This induces a map $ \opn{GU}(1)(\mathbb{F}_{\ell})  \to \Tb  (\mathbb{F}_{\ell}) $. Since $ \opn{H}^{1}(\mathbb{F}_{\ell},\mbb{G}_{m}) = 0 $ by Lang's lemma, this map is surjective and a lifting argument shows that $ \opn{GU}(1) (\ZZ_{\ell} )  \to  \Tb (\ZZ_{\ell}) $ is also surjective. The claim for surjectivity on adelic points now follows.   
\end{proof}

\begin{lemma} \label{center} Consider the following groups
\[Z_{0} \defeq Z_{\Gb_{0}}(\Ab) \; \; \; \text{ and } \; \; \;  Z_{1} \defeq Z_{\Gb}(\Ab) \Gb(\QQ) \cap \Gb_{0}(\Ab) \]
Then $Z_{1} = Z_{0} \Gb_{0}(\QQ)$.  
\end{lemma}    
\begin{proof} 
The inclusion $Z_{0} \mbf{G}_0(\QQ) \subset Z_{1}$ is obvious. Let $z \gamma \in Z_{1}$ where $z \in Z_{\Gb}(\Ab)$ and $\gamma \in \Gb (\QQ)$. Then we have
\[
c(z \gamma) = z \overline {z} c(\gamma) = 1 
\]
which implies that $c(\gamma)$ is a norm locally everywhere. By the Hasse norm theorem, we must have that $c(\gamma)$ is a global norm, i.e. there exists an element $\zeta \in E^{\times } = Z_{\Gb}(\QQ)$ such that $c(\gamma) = c(\zeta)$. Therefore, we can replace $z$ with $z \zeta \in Z_{0}$ and $\gamma$ with $\zeta ^{-1} \gamma \in \Gb_{0}(\QQ)$.  
\end{proof}                 

\begin{notation} \label{NotationForFixedHaarMeasures}
We fix left Haar measures $dg$, $dh$ and $dt$ on $\mbf{G}(\mbb{Q}_\ell)$, $\mbf{H}(\mbb{Q}_\ell)$ and $\mbf{T}(\mbb{Q}_\ell)$ respectively. As these groups are unimodular \cite[p. 58]{Renard}, these are also right Haar measures. We normalise them so that the volume of any compact open subgroup lies in $\mbb{Q}$ and the volume of any hyperspecial subgroup is $1$ (when applicable). In particular, the products of these measures induce Haar measures on the $\mbb{A}_f$-points of the corresponding groups, which we may also denote by the same letters.
\end{notation}

\begin{remark} 
If $n$ is odd then the above groups are quasi-split at all finite places. Indeed, the groups $\mbf{H}_0$ and $\mbf{H}$ are automatically quasi-split at all finite places because there is only one unitary group (up to isomorphism) attached to Hermitian spaces of fixed odd dimension over a non-archimedean local field (see \cite[p. 152]{Morel}). For $\mbf{G}_0$ and $\mbf{G}$, we follow the argument in \emph{loc.cit.}, where the local cohomological invariant is calculated (the group is quasi-split when the invariant is $0$). 

We note that these groups are automatically quasi-split at primes $\ell$ which do not divide the discriminant of $E$, because the group is unramified. If $\ell$ divides the discriminant of $E$, then the cohomological invariant is $0$ if $-1$ is a norm in $\mbb{Q}_\ell$. Otherwise, it is equal to $2n-1 + n = 3n-1$ modulo $2$. But this quantity is even, under the assumption that $n$ is odd.  

Quasi-splitness at all finite places is used in the results invoked on weak base-change in section \ref{CSV_section}.   However, we expect such results to hold without this hypothesis.
\end{remark}

\subsection{The Shimura--Deligne data} \label{TheShimuraData}

We now define the Shimura--Deligne data associated with $\mbf{H}$ and $\mbf{G}$. Since the datum attached to the group $\mbf{H}$ doesn't satisfy all the usual axioms of a Shimura datum, we have extended the definition to that of a Shimura--Deligne datum (Appendix \ref{PureAppendix}). In this section, we will freely use notation and results from this appendix. 

Let $\mbb{S} \defeq \opn{Res}_{\mbb{C}/\mbb{R}}\mbb{G}_m$ denote the Deligne torus and let $h_{\mbf{G}}$ and $h_{\mbf{H}}$ be the following algebraic homomorphisms:
\begin{eqnarray}
h_{\mbf{G}} \colon \mbb{S} & \rightarrow & \mbf{G}_{\mbb{R}} \nonumber \\
z & \mapsto & \left( \opn{diag}(z, \bar{z}, \dots, \bar{z}) \right) \nonumber \\
h_{\mbf{H}} \colon \mbb{S} & \rightarrow & \mbf{H}_{\mbb{R}} \nonumber \\
z & \mapsto & \left( \opn{diag}(z, \bar{z}, \dots, \bar{z}), \opn{diag}(\bar{z}, \dots, \bar{z}) \right). \nonumber 
\end{eqnarray}
Let $X_{\mbf{G}}$ (resp. $X_{\mbf{H}}$) denote the $\mbf{G}(\mbb{R})$ (resp. $\mbf{H}(\mbb{R})$) conjugacy class of homomorphisms containing $h_{\mbf{G}}$ (resp. $h_{\mbf{H}}$). Then the pairs $(\mbf{G}, X_{\mbf{G}})$ and $(\mbf{H}, X_{\mbf{H}})$ both satisfy the axioms of a Shimura--Deligne datum as in Definition \ref{SD} (axioms \cite[2.1.1.1-2.1.1.2]{DeligneSVs}). The datum $(\mbf{G}, X_{\mbf{G}})$ also satisfies $2.1.1.3$ in \emph{loc.cit.},  so it is a Shimura datum in the usual sense. In particular, if $K \subset \mbf{G}(\mbb{A}_f)$ is a sufficiently small compact open subgroup\footnote{In this paper, the phrases ``sufficiently small compact open subgroup'' and ``neat compact open subgroup'' are synonymous.} (Definition \ref{DefinitionOfSufficientlySmall}), we let $\opn{Sh}_{\mbf{G}}(K)$ denote the associated Shimura--Deligne variety whose complex points equal 
\[
\opn{Sh}_{\mbf{G}}(K)(\mbb{C}) = \mbf{G}(\mbb{Q}) \backslash [X_{\mbf{G}} \times \mbf{G}(\mbb{A}_f)/K ]
\]
and similarly for $\mbf{H}$. The reflex field for the datum $(\mbf{G}, X_{\mbf{G}})$ is the imaginary quadratic number field $E$ fixed in the previous section for $n > 1$, and $\mbb{Q}$ otherwise (see Example \ref{ReflexFieldUnitaryExample}). However, we will always consider the associated Shimura--Deligne varieties as schemes over $E$, unless stated otherwise. If the level $K$ is small enough, the embedding $\iota$ induces a closed embedding on the associated Shimura--Deligne varieties. More precisely, we have the following. 

\begin{definition} \label{HSmallDef}
Let $w$ equal the $2n \times 2n$ block diagonal matrix
\[
w \defeq \iota(1, -1) = \tbyt{1}{}{}{-1} \; \in \; \mbf{G}(\mbb{Q}).
\]
We say that a compact open subgroup $K \subset \mbf{G}(\mbb{A}_f)$ is $\mbf{H}$\emph{-small} if there exists a compact open subgroup $K' \subset \mbf{G}(\mbb{A}_f)$, containing both $K$ and $wKw$, such that for any $g \in \mbf{G}(\mbb{A}_f)$, $gK'g^{-1} \cap \mbf{G}(\mbb{Q})$ has no non-trivial stabilisers for its action on $X_{\mbf{G}}$. 
\end{definition}
\begin{proposition}
Let $K \subset \mbf{G}(\mbb{A}_f)$ be an $\mbf{H}$-small compact open subgroup. Then the natural map
\[
\opn{Sh}_{\mbf{H}}(K \cap \mbf{H}(\mbb{A}_f)) \xrightarrow{\iota} \opn{Sh}_{\mbf{G}}(K)
\]
is a closed immersion.
\end{proposition}
\begin{proof}
This follows the same argument in  \cite[Proposition 5.3.1]{LSZ17}. Indeed, one has an injective map $\iota \colon \opn{Sh}_{\mbf{H}} \hookrightarrow \opn{Sh}_{\mbf{G}}$ on the infinite level Shimura--Deligne varieties. So it suffices to prove that for any $u \in K - K \cap \mbf{H}(\mbb{A}_f)$,   one has $\opn{Sh}_{\mbf{H}}u \cap \opn{Sh}_{\mbf{H}} = \varnothing$ as subsets of $\opn{Sh}_{\mbf{G}}$. If $Q$ and $Qu$ are both in $\opn{Sh}_{\mbf{H}}$, we must have that $u' = u \cdot w u^{-1} w$ fixes $Q$. By the conditions on $K'$, we have that $u' = 1$. This implies that $u$ is in the centraliser of $w$, which is precisely $\mbf{H}(\mbb{A}_f)$.   
\end{proof}

We also consider the trivial Shimura--Deligne datum $(\mbf{T}, X_{\mbf{T}})$, where $X_{\mbf{T}} $ is the singleton $h_{\mbf{T}}(z) = \bar{z}/z$ (considered as a $ \mbf{T}(\RR)$-conjugacy class of homomorphisms $ \mathbb{S} \to  \mbf{T} $) and the datum $(\Gt, X_{\Gt}) = (\Gt, X_{\mbf{G}} \times X_{\mbf{T}})$. Then the data associated with the groups $\mbf{H}$ and $\Gt$ are clearly compatible with respect to the inclusion $(\iota, \nu)$, and if $U = K \times C \subset \Gt(\mbb{A}_f)$ is a sufficiently small compact open subgroup, we have a finite unramified map
\[
\opn{Sh}_{\mbf{H}}\left(U \cap \mbf{H}(\mbb{A}_f) \right) \to \opn{Sh}_{\Gt}\left( U \right) \cong \opn{Sh}_{\mbf{G}}(K) \times \opn{Sh}_{\mbf{T}}(C).
\]
which is a closed embedding when $ K $ is $ \Hb $-small. The reflex field for $(\Gt, X_{\Gt})$ is equal to the imaginary quadratic field $E$.

\begin{lemma} \label{LemSD5issatisfied}
The groups $\mbf{H}$, $\mbf{G}$, $\Gt$, and $\mbf{T}$ satisfy axiom (SD5), i.e. $\mbf{Z}_{\invs{G}}(\mbb{Q})$ is discrete in $\mbf{Z}_{\invs{G}}(\mbb{A}_f)$ for $\invs{G} \in \{\mbf{H}, \mbf{G}, \Gt, \mbf{T} \}$.
\end{lemma}
\begin{proof}
Note that $ \mathbf{Z}_{\Hb} =\GU(1) \times \U(1) $, $ \mathbf{Z}_{\Gb}  =  \GU(1) $, $ \mathbf{Z}_{\tilde{\Gb}} = \GU(1) \times  \U(1) $ and since $ \GU(1)(\ZZ) $ and $ \U(1)(\ZZ) $ are finite, the $\ZZ$-points of all these central tori are finite too. As a result $ \GU(1)(\QQ) $ (resp. $  \U(1)(\QQ) $) is discrete in  $\GU(1)(\Ab_{f}) $ (resp. $ \U(1)(\Ab_{f}) $), as required.
\end{proof}

\subsection{Dual groups, base-change, and $L$-factors} \label{DualGrpsBC}

In this section, we describe the dual groups attached to the groups defined in the previous section and several maps between them. Let $W_{F}$  denote the Weil group of a local or global field $F$. In each of the following examples,  we describe the connected component of the $L$-group. The Weil group $W_E$ will always act trivially on this component, so to describe the full $L$-group, it will be enough to specify how complex conjugation acts.

As in \cite[\S 2.2]{Skinner}, let $\Phi_m = \left( (-1)^{i+1} \delta_{i, m+1 - j} \right)_{i, j}$. The dual groups are as follows:
\begin{itemize}
    \item $\widehat{\mbf{G}} = \opn{GL}_1(\mbb{C}) \times \opn{GL}_{2n}(\mbb{C})$ and complex conjugation acts by 
    \[
    (\lambda; g) \mapsto (\lambda \cdot \opn{det}g; \Phi_{2n}^{-1} {^t g^{-1}} \Phi_{2n} ).
    \]
    \item $\widehat{\mbf{H}} = \opn{GL}_1(\mbb{C}) \times \opn{GL}_n(\mbb{C}) \times \opn{GL}_n(\mbb{C})$ and complex conjugation acts as
    \[
    (\lambda; h_1, h_2) \mapsto (\lambda \cdot \opn{det}h_1 \cdot \opn{det}h_2 ; \Phi_n^{-1} {^t h_1^{-1}} \Phi_n, \Phi_n^{-1} {^t h_2^{-1}} \Phi_n )
    \]
    \item ${\widehat{\mbf{G}}_0} = \opn{GL}_{2n}(\mbb{C})$ and ${\widehat{\mbf{H}}_0} = \opn{GL}_n(\mbb{C}) \times \opn{GL}_n(\mbb{C})$ and complex conjugation acts via the same formulae above, but omitting the $\opn{GL}_1$ factor. 
    \item $\widehat{\mbf{T}} = \opn{GL}_1(\mbb{C})$ and complex conjugation acts by sending an element $x$ to its inverse $x^{-1}$.
    \item ${\widehat{\widetilde{\mbf{G}}}} = {\widehat{\mbf{G}} } \times {\widehat{\mbf{T}}}$ and complex conjugation acts diagonally.
\end{itemize}
We also note that we can choose a splitting such that the dual group of $R\opn{GL}_m \defeq \opn{Res}_{E/\mbb{Q}} \opn{GL}_m$ is equal to $\opn{GL}_m(\mbb{C}) \times \opn{GL}_m(\mbb{C})$ and complex conjugation acts by $(g_1, g_2) \mapsto \left( \Phi_m^{-1} {^t g_2^{-1}} \Phi_m, \Phi_m^{-1} {^t g_1^{-1}} \Phi_m \right)$. This implies that the natural diagonal embedding $\widehat{\mbf{G}}_0 \to \widehat{R\opn{GL}}_{2n}$ extends to a map of $L$-groups, which we will denote by $\opn{BC}$. Similarly,  we have a map $\opn{BC} \colon {^L \mbf{G}} \to {^L (R\opn{GL}_1 \times R\opn{GL}_{2n})}$. We also have a natural map ${^L \mbf{G}} \to {^L \mbf{G}_0}$ given by projecting to the second component, and this is compatible with base-change in the sense that we have a commutative diagram: 
\[
\begin{tikzcd}
{^L \mbf{G}} \arrow[r, "\opn{BC}"] \arrow[d] & {^L (R\opn{GL}_1 \times R\opn{GL}_{2n})} \arrow[d] \\
{^L \mbf{G}_0} \arrow[r, "\opn{BC}"]         & {^L R\opn{GL}_{2n} }                             
\end{tikzcd}
\]
where the right-hand vertical arrow is projection to the second component. 

The map ${^L \mbf{G}} \to {^L \mbf{G}_0}$ corresponds on the automorphic side to restricting an automorphic representation of $\mbf{G}$ to $\mbf{G}_0$. It turns out that all irreducible constituents in this restriction lie in the same $L$-packet, and one can always lift a representation of $\mbf{G}_0$ to one of $\mbf{G}$. We summarise this in the following proposition. 

\begin{proposition} \label{UnitaryBaseChangeProp}
Let $\pi$ be a cuspidal automorphic representation of $\mbf{G}(\mbb{A})$. Then all irreducible constituents of $\pi|_{\mbf{G}_0}$ are cuspidal and lie in the same $L$-packet with $L$-parameter obtained by post-composing the $L$-parameter for $\pi$ with the natural map ${^L \mbf{G}} \to {^L \mbf{G}_0}$. Conversely, if $\pi_0$ is a cuspidal automorphic representation of $\mbf{G}_0(\mbb{A})$, then $\pi_0$ lifts to a cuspidal automorphic representation of $\mbf{G}(\mbb{A})$.

Moreover, the same statements hold if we work over a local field $F$, that is to say: if $\pi$ is an irreducible admissible representation of $\mbf{G}(F) $, then all irreducible constituents of $\pi|_{\mbf{G}_0}$ lie in the same local $L$-packet (with parameter obtained by composing with ${^L \mbf{G}_F} \to {^L \mbf{G}_{0, F} }$), with cuspidal representations corresponding to each other, and every representation of $\mbf{G}_0(F)$ lifts to one of $\mbf{G}(F)$. If $F = \mbb{R}$ and $\pi_0$ lies in the discrete series, then we can lift $\pi_0$ to a discrete series representation of $\mbf{G}(\mbb{R})$.  

If $\omega_{\pi_0}$ denotes the central character of a  representation $\pi_0$, and $\omega$ is a (unitary) character extending $\omega_{\pi_0}$ to $Z_{\mbf{G}}(\mbb{Q}) \backslash Z_{\mbf{G}}(\mbb{A})$, then there exists a lift $\pi$ of $\pi_0$ as above with central character $\omega$. A similar statement holds for local fields. 
\end{proposition}
\begin{proof}
See Proposition 1.8.1 and the discussion preceding it in \cite{HarrisLabesse} (which relies on \cite[Lemme 5.6]{Clozel}) and Theorem 1.1.1 in \cite{LabesseSchwermer}. The last part follows from the techniques used in \cite[Theorem 5.2.2]{LabesseSchwermer} and Lemma \ref{center}. 
\end{proof}

We will frequently need to talk about local $L$-factors attached to unramified representations of the groups defined in \S \ref{TheGroups}, so for completeness, we recall the definition here. We follow \S 1.11, \S 1.14 and \S 5.1 Example (b) in \cite{BlasiusRogawski}.

\begin{definition} \label{StandardLFactorDef}
Let $\ell$ be a prime which is unramified in $E/\mbb{Q}$ and for which $\mbf{G}_{0, \mbb{Q}_\ell}$ is quasi-split. For a positive integer $m$, let $\opn{Std}_m$ denote the standard representation of $\opn{GL}_m$.
\begin{enumerate}[(a)]
    \item Let $F$ be a finite extension of $\mbb{Q}_\ell$. For $\mathscr{G} \in \{\opn{GL}_{2n}, \opn{GL}_1 \times \opn{GL}_{2n} \}$, let $r_{\mathrm{std}} \colon {^L \mathscr{G}_F} \to \opn{GL}_{2n}(\mbb{C})$ denote the following representation:
    \[
    r_{\mathrm{std}}(g, x) = \left\{ \begin{array}{cc}
\opn{Std}_{2n}(g) & \text{ if } \mathscr{G} = \opn{GL}_{2n} \text{ with } {^L \mathscr{G}_F} = \opn{GL}_{2n}(\mbb{C}) \times W_F \\
\left(\opn{Std}_1 \boxtimes \opn{Std}_{2n}\right)(g) & \text{ if } \mathscr{G} = \opn{GL}_1 \times \opn{GL}_{2n} \text{ with } {^L \mathscr{G}_F} = \opn{GL}_1(\mbb{C}) \times \opn{GL}_{2n}(\mbb{C}) \times W_F 
\end{array} \right.
\]
Let $q$ denote the order of the residue field of $F$ and $\opn{Frob}$ an arithmetic Frobenius in $W_F$. For any unramified representation $\sigma$ of $\mathscr{G}(F)$, with corresponding (unramified) Langlands parameter $\varphi_{\sigma} \colon W_F \to {^L \mathscr{G}_F}$, the \emph{standard} local $L$-factor is defined to be
\[
L(\sigma, s) \defeq \opn{det}\left(1 - q^{-s}(r_{\mathrm{std}} \circ \varphi_{\sigma})(\opn{Frob}^{-1})\right)^{-1}
\]
where $s \in \mbb{C}$.  
\item If $\ell$ splits in $E/\mbb{Q}$, then we have an isomorphism $\mbf{G}_{0, \mbb{Q}_\ell} \cong \opn{GL}_{2n, \mbb{Q}_\ell}$, and for any unramified representation $\sigma$ of $\mbf{G}_0(F)$, we define the local $L$-factor of $\sigma$ to be as in part (a), viewing $\sigma$ as a representation of $\opn{GL}_{2n}(F) \cong \mbf{G}_0(F)$. We have a similar definition for representations of $\mbf{G}(F)$.
\end{enumerate}
\end{definition}

\subsection{Discrete series representations} \label{DiscreteSeriesReps}

Recall from Remark \ref{RemIdentifyGroups} that one has an identification $\mbf{G}_E = \opn{GL}_{1, E} \times \opn{GL}_{2n, E}$. We let $\mbb{G}_m^{1+2n} \subset \mbf{G}_{E}$ denote the standard torus with respect to this identification. Then any algebraic character of this torus is of the form 
\begin{eqnarray}
\mbb{G}_m^{1+ 2n} & \rightarrow & \mbb{G}_m \nonumber \\
(t_0; t_1, \dots, t_{2n}) & \mapsto & t_0^{c_0} \cdot \prod_i t_i^{c_i} \nonumber 
\end{eqnarray}
where $\mbf{c} = (c_0; c_1, \dots, c_{2n}) \in \mbb{Z}^{1+2n}$. Such a character is called \emph{dominant} if $c_1 \geq \cdots \geq c_{2n}$ and these characters classify all irreducible algebraic representations of $\mbf{G}_E$. Indeed, any such representation corresponds to a representation with highest weight given by $\mbf{c}$. There is a similar description for $\mbf{G}_0$ by omitting the first $\mbb{G}_m$-factor. 

The $L$-packets of discrete series representations of $\mbf{G}(\mbb{R})$ are described by the above algebraic representations. More precisely, following \cite[\S 3.3]{Clozel} (the extension to similitude groups is immediate), if $\pi_{\infty}$ lies in the discrete series, the Langlands parameter $\varphi_{\infty} \colon W_{\mbb{R}} \to {^L \mbf{G} }$ associated with $\pi_{\infty}$, restricted to $W_{\mbb{C}} = \mbb{C}^{\times}$, is equivalent to a parameter of the form
\[
z \mapsto \left((z/\bar{z})^{p_0}; (z/\bar{z})^{p_1 + \frac{2n-1}{2}}, \dots, (z/\bar{z})^{p_{2n} + \frac{2n-1}{2}} \right) 
\]
with $p_i \in \mbb{Z}$ and $p_1 > \cdots > p_{2n}$. Then $\pi_{\infty}$ corresponds to an irreducible algebraic representation of $\mbf{G}_E$ with highest weight $\mbf{c} = (c_0; c_1, \dots, c_{2n})$ satisfying $c_0 = p_0$ and $c_i = p_i + (i-1)$, for $i=1, \dots, 2n$. We have a similar description for discrete series $L$-packets for $\mbf{G}_0(\mbb{R})$. 

\begin{lemma} \label{TrivialCentralLiftLemma}
Let $\pi_0$ be a cuspidal automorphic representation of $\mbf{G}_0(\mbb{A})$ such that $\pi_{0, \infty}$ lies in the discrete series $L$-packet corresponding to the highest weight $\mbf{c} = (c_1, \dots, c_{2n})$. Suppose that the central character of $\pi_0$ is trivial. Then there exists a cuspidal automorphic representation $\pi$ of $\mbf{G}(\mbb{A})$ such that:
\begin{itemize}
    \item $\pi$ is a lift of $\pi_0$.
    \item $\pi$ has trivial central character.
    \item $\pi_{\infty}$ lies in the discrete series $L$-packet corresponding to the weight $(0; c_1, \dots, c_{2n})$.
\end{itemize}
\end{lemma}
\begin{proof}
By looking at the infinitesimal character of $\pi_{0, \infty}$, we must have $c_1 + \cdots + c_{2n} = 0$ because the central character is trivial. By Proposition \ref{UnitaryBaseChangeProp}, there exists a cuspidal automorphic representation $\pi$ of $\mbf{G}(\mbb{A})$ lifting $\pi_0$ such that $\pi_{\infty}$ lies in the discrete series. Furthermore, we can arrange it so that the central character of $\pi$ is trivial. In particular, if $(c_0; c_1, \dots, c_{2n})$ is the weight corresponding to the $L$-packet of $\pi_{\infty}$, then since $\pi$ has trivial central character, we must have 
\[
c_0 + \sum_{i=1}^{2n} c_i = 0.
\]
This implies that $c_0 = 0$.
\end{proof}


\section{Coefficient sheaves for Shimura--Deligne varieties}

In this section, we describe how to associate to algebraic representations of $\mbf{G}$ an appropriate coefficient sheaf for motivic or \'{e}tale cohomology. We follow \cite{Pink1992}, \cite{Ancona2015}, \cite{Torzewski2019} and \cite{LSZ17} closely (although note that the conventions are slightly different in the latter to that of the first three -- we follow conventions in the first three references). In particular, this will mean that the multiplier character of $\mbf{G}$ is sent to the Tate motive under the functor $\opn{Anc}_{\mbf{G}}$. Since the group $\mbf{H}$ does not give rise to a Shimura datum in the usual sense, we have provided justifications for the results in this section in Appendix \ref{AnconaAppendix}. 

Throughout this section we let $(\invs{G}, X_{\invs{G}})$ denote one of the Shimura--Deligne data $(\mbf{H}, X_{\mbf{H}})$, $(\mbf{G}, X_{\mbf{G}})$, $(\Gt, X_{\Gt})$ introduced in \S \ref{TheShimuraData}. Recall that we have fixed an odd prime $p$ that splits in $E/\mbb{Q}$. Let $\ide{p}$ denote the prime of $E$ lying above $p$ which is fixed by the embedding $E \hookrightarrow \Qpb$ (so $E_{\ide{p}}$ is identified with $\mbb{Q}_p$).

\subsection{The canonical construction for \'{e}tale cohomology} \label{TheCanonicalConstructionEtaleCoh}

Let $T_n$ be a finite representation of a (sufficiently small) compact open subgroup $K \subset \invs{G}(\mbb{A}_f)$ with coefficients in $\mbb{Z}/p^n \mbb{Z}$. Since $T_n$ is finite, there exists a finite-index normal subgroup $L \subset K$ that acts trivially on $T_n$. 

\begin{definition}
We let $\mu_{\invs{G}, K, n}(T_n)$ be the locally constant \'{e}tale sheaf of abelian groups on $\opn{Sh}_{\invs{G}}(K)$ corresponding to 
\[
\left(\opn{Sh}_{\invs{G}}(L) \times T_n\right)/\Gamma \to \opn{Sh}_{\invs{G}}(K)
\]
where $\Gamma \defeq K/L$ acts via $(x, t)\cdot h = (xh, h^{-1}t)$. This construction is independent of the choice of $L$. 
\end{definition}

If $x_0 = \opn{Spec}(k)$ is a point on $\opn{Sh}_{\invs{G}}(K)$ and $x = \opn{Spec}(\bar{k})$ is the associated geometric point obtained from fixing a separable closure of $k$, then we can explicitly describe the action of $\Gal(\bar{k}/k)$ on the stalk $\left(\mu_{\invs{G}, K, n}(T_n) \right)_x$. Indeed, lift $x$ to a geometric point $\tilde{x}$ on $\opn{Sh}_{\invs{G}}(L)$. Then for $\sigma \in \Gal(\bar{k}/k) $,   there exists a unique element $\psi(\sigma) \in \Gamma$ such that
\[
\sigma \cdot \tilde{x} = \tilde{x}^{\psi(\sigma)}.
\]
Since the Galois action commutes with that of $\Gamma$, the map $\psi$ defines a homomorphism $\psi \colon \Gal(\bar{k}/k) \to \Gamma$. The stalk $\left(\mu_{\invs{G}, K, n}(T_n) \right)_x$ is isomorphic to $T_n$ with the Galois action given by $\psi$.  

\begin{lemma}
The functors $\mu_{\invs{G}, K, n}$ are compatible as $n$ varies. In particular, if we let $T$ be a finitely-generated continuous representation of $K$ with coefficients in $\mbb{Z}_p$, then
\[
\mathscr{T}_K = \mu_{\invs{G}, K}(T) \defeq \left( \mu_{\invs{G}, K, n}(T/p^nT) \right)_{n \geq 1}
\]
defines a lisse sheaf on $\opn{Sh}_{\invs{G}}(K)$.
\end{lemma}
\begin{proof}
This follows immediately from the fact that for any two integers $n, n'$ one can pick the same normal subgroup $L \subset K$ in the definitions of $\mu_{\invs{G}, K, n}(T/p^nT)$ and $\mu_{\invs{G}, K, n'}(T/p^{n'}T)$.
\end{proof}

By considering the category of \'{e}tale sheaves on $\opn{Sh}_{\invs{G}}(K)$ with coefficients in $\mbb{Z}_p$ up to isogeny, this construction extends to finite-dimensional continuous representations of $K$ with coefficients in $\mbb{Q}_p$. In particular, any algebraic representation of $\invs{G}_{\mbb{Q}_p}$ gives rise to a continuous representation of $K$ via the natural map $K \to \invs{G}(\mbb{Q}_p)$, with coefficients in $\mbb{Q}_p$. Therefore, one obtains an additive tensor functor 
\[
\mu_{\invs{G}, K} \colon \opn{Rep}_{\mbb{Q}_p}(\invs{G}) \to \opn{\acute{E}t}(\opn{Sh}_{\invs{G}}(K) )_{\mbb{Q}_p}
\]
from the category of algebraic representations of $\invs{G}_{\mbb{Q}_p}$ to the category of \'{e}tale sheaves on $\opn{Sh}_{\invs{G}}(K)$ with coefficients in $\mbb{Q}_p$ (see Definition \ref{DefOfEtCategory}).

\subsection{Equivariant coefficient sheaves} \label{EquivariantCoefficientSheaves}

If we wish to study the sheaves $\mathscr{V}_K = \mu_{\invs{G}, K}(V)$ as $K$ varies, we need to keep track of the action of $\invs{G}(\mbb{A}_f)$. More precisely, for any sufficiently small compact open subgroups $K, L \subset \invs{G}(\mbb{A}_f)$ and any element $\sigma \in \invs{G}(\mbb{A}_f)$ satisfying $\sigma^{-1}L\sigma \subset K$, one obtains a finite \'{e}tale morphism of Shimura--Deligne varieties
\[
\opn{Sh}_{\invs{G}}(L) \xrightarrow{\sigma} \opn{Sh}_{\invs{G}}(K) 
\]
given by right-translation by $\sigma$.

\begin{definition} \label{DefOfEquivariantCat}
Let $\opn{\acute{E}t}\left( \opn{Sh}_{\invs{G}} \right)_{\mbb{Q}_p}$ denote the category of equivariant \'{e}tale sheaves -- that is to say -- the objects are collections of sheaves $\mathscr{F}_K \in \opn{\acute{E}t}\left( \opn{Sh}_{\invs{G}}(K) \right)_{\mbb{Q}_p}$, where $K$ varies over all sufficiently small compact open subgroups of $\invs{G}(\mbb{A}_f)$, and isomorphisms $\varphi_{\sigma} \colon \sigma^* \mathscr{F}_{K} \cong \mathscr{F}_L$ for any $\sigma \in \invs{G}(\mbb{A}_f)$ satisfying $\sigma^{-1} L \sigma \subset K$, such that the following diagram is commutative
\[
\begin{tikzcd}
                                & \tau^*\sigma^*\mathscr{F}_K = (\tau \sigma)^*\mathscr{F}_K \arrow[ld, "\tau^*\varphi_{\sigma}", swap] \arrow[rd, "\varphi_{\tau \sigma}"] &                  \\
\tau^* \mathscr{F}_L \arrow[rr, "\varphi_{\tau}"] &                                                                                  & \mathscr{F}_{L'}
\end{tikzcd}
\]
where $\sigma, \tau \in \invs{G}(\mbb{A}_f)$ and $\sigma^{-1} L \sigma \subset K$, $\tau^{-1}L'\tau \subset L$.

The morphisms in this category are morphisms $\mathscr{F}_K \to \mathscr{G}_K$ in $\opn{\acute{E}t}\left( \opn{Sh}_{\invs{G}}(K) \right)_{\mbb{Q}_p}$ which are compatible with the collection of isomorphisms in the natural way. This is an additive category with tensor products.
\end{definition}

We have the following proposition which collects together the construction in the previous subsection for varying $K$.

\begin{proposition} \label{PropMuEquiv}
The functors $\mu_{\invs{G}, K}$ assemble into an additive tensor functor:
\[
\mu_{\invs{G}} \colon \opn{Rep}_{\mbb{Q}_p}(\invs{G}) \to \opn{\acute{E}t}\left( \opn{Sh}_{\invs{G}} \right)_{\mbb{Q}_p}.
\]
Explicitly, for an algebraic representation $V$ of $\invs{G}_{\mbb{Q}_p}$, the equivariant sheaf $\mathscr{V} = \mu_{\invs{G}}(V)$ is the collection of lisse sheaves $\mathscr{V}_K$ with isomorphisms $\sigma^* \mathscr{V}_K \xrightarrow{\sim} \mathscr{V}_L$ induced from the action of $\sigma$ on $V$.
\end{proposition}

Keeping track of this extra data also allows us to define pushforwards and pullbacks between the cohomology of these sheaves. Indeed, for an equivariant sheaf $\mathscr{F} \in \opn{\acute{E}t}\left( \opn{Sh}_{\invs{G}} \right)_{\mbb{Q}_p}$ we define the pullback $[\sigma]^*$ as the composition
\[
[\sigma]^* \colon \opn{H}^i_{\et}\left(\opn{Sh}_{\invs{G}}(K), \mathscr{F}_K \right) \xrightarrow{\sigma^*} \opn{H}^i_{\et}\left(\opn{Sh}_{\invs{G}}(L), \sigma^* \mathscr{F}_K \right) \xrightarrow{\sim} \opn{H}^i_{\et}\left(\opn{Sh}_{\invs{G}}(L), \mathscr{F}_L \right)
\]
where the last isomorphism is induced from $\varphi_{\sigma} \colon \sigma^* \mathscr{F}_K \cong \mathscr{F}_L$. Similarly, we define the pushforward $[\sigma]_*$ as the composition
\[
[\sigma]_* \colon \opn{H}^i_{\et}\left(\opn{Sh}_{\invs{G}}(L), \mathscr{F}_L \right) \xrightarrow{\varphi_{\sigma}^{-1}} \opn{H}^i_{\et}\left(\opn{Sh}_{\invs{G}}(L), \sigma^* \mathscr{F}_K \right) \xrightarrow{\opn{Tr}_{\sigma}} \opn{H}^i_{\et}\left(\opn{Sh}_{\invs{G}}(K), \mathscr{F}_K \right)
\]
where the last map is the trace map, and we have used the fact that the morphism $\sigma$ is finite \'{e}tale.

\begin{proposition} \label{PropRationalCohFunctor}
Let $\mathscr{F} \in \opn{\acute{E}t}\left( \opn{Sh}_{\invs{G}} \right)_{\mbb{Q}_p}$ be an equivariant sheaf and let $\Upsilon$ denote the collection of all sufficiently small compact open subgroups of $\invs{G}(\mbb{A}_f)$. Then the data 
\[
\left( M_{\invs{G}}(K) = \opn{H}^i_{\et}\left( \opn{Sh}_{\invs{G}}(K), \mathscr{F}_K \right), [-]^*, [-]_* \right)
\]
defines a cohomology functor $M_{\invs{G}} \colon \invs{P}(\invs{G}(\mbb{A}_f), \Upsilon) \to \mbb{Q}_p\text{-Mod}$ in the sense of Definitions \ref{CohoFunc} and \ref{DefOfAdditionalProperties}. Since the Shimura--Deligne datum $(\invs{G}, X_{\invs{G}})$ satisfies (SD5), this cohomology functor is CoMack.
\end{proposition}
\begin{proof}
The fact that these cohomology groups and pullbacks/pushforwards form a cohomology functor follows from the commutative diagram in Definition \ref{DefOfEquivariantCat}. For example, for $g \in \invs{G}(\mbb{A}_f)$ and $L, K \in \Upsilon$ satisfying $g^{-1}Lg = K$, axiom (C2) follows by taking $\sigma = g$ and $\tau = g^{-1}$ in the commutative diagram and using various naturality properties of the unit $1 \to g_*g^*$ and counit/trace $g_*g^* \to 1$. If $L=K$ and $g \in K$, then axiom (C3) follows from the fact that under the identification $g^* \mathscr{F}_K = \mathscr{F}_K$, the morphism $\varphi_g \colon g^*\mathscr{F}_K \to \mathscr{F}_K$ is identified with the identity morphism.

The covering axiom follows from the analogous property for \'{e}tale morphisms (see \cite[Tome 3, Expose IX, \S 5]{SGA4}) and Lemma \ref{LemSD5index}. Finally, since the Shimura--Deligne datum satisfies (SD5), in the notation of Definition \ref{DefOfAdditionalProperties} one has a Cartesian square
\[
\begin{tikzcd}
\bigsqcup_{\gamma} \opn{Sh}_{\invs{G}}(L_{\gamma}) \arrow[d, "\sqcup_{\gamma} \gamma"'] \arrow[r, "\sqcup_{\gamma} \opn{pr}"] & \opn{Sh}_{\invs{G}}(L) \arrow[d, "\opn{pr}"] \\
\opn{Sh}_{\invs{G}}(L') \arrow[r, "\opn{pr}"]                                                                              & \opn{Sh}_{\invs{G}}(K)                      
\end{tikzcd}
\]
Axiom (M) then follows from the fact that pushforwards and pullbacks commute for Cartesian squares (c.f. \cite[p. 6]{Loeffler19}).
\end{proof}

\begin{notation}
For an equivariant sheaf $\mathscr{F} \in \opn{\acute{E}t}\left(\opn{Sh}_{\invs{G}} \right)_{\mbb{Q}_p}$ we set
\[
\opn{H}^i_{\et}\left(\opn{Sh}_{\invs{G}}, \mathscr{F} \right) \defeq \varinjlim_{K \in \Upsilon} \opn{H}^i_{\et}\left(\opn{Sh}_{\invs{G}}(K), \mathscr{F}_K \right)
\]
where the limit is with respect to the maps $\opn{pr}_{L, K}^*$ for $L \subset K$. This carries a smooth action of $\invs{G}(\mbb{A}_f)$. 
\end{notation}

\subsection{An integral version} \label{IntegralEquivariantCoefficientSheaves}

We will also need to talk about equivariant \'{e}tale $\mbb{Z}_p$-sheaves. Let $\Sigma \subset \invs{G}(\mbb{A}_f)$ denote an open submonoid.

\begin{definition} \label{DefOfEquivIntegral}
We let $\opn{\acute{E}t}\left( \opn{Sh}_{\invs{G}} \right)_{\mbb{Z}_p, \Sigma}$ denote the category of collections of \'{e}tale $\mbb{Z}_p$-sheaves $\mathscr{F}_K$, where $K$ runs over all sufficiently small compact open subgroups of $\Sigma$, and \emph{morphisms} $\sigma^* \mathscr{F}_K \to \mathscr{F}_L$, where $\sigma \in \Sigma$ such that $\sigma^{-1}L \sigma \subset K$, satisfying the same commutative diagram in Definition \ref{DefOfEquivariantCat}. We impose the additional condition that $1^* \mathscr{F}_K \to \mathscr{F}_L$ is an isomorphism, for any $L \subset K$. The morphisms in this category are defined in the same way as in Definition \ref{DefOfEquivariantCat}.
\end{definition}

Let $T$ denote a finite-dimensional continuous $\Sigma$-representation with coefficients in $\mbb{Z}_p$, i.e. $T$ is a finite-free $\mbb{Z}_p$-module with a continuous homomorphism (of monoids) 
\[
\Sigma \to \opn{End}_{\mbb{Z}_p,\mathrm{cont}}(T).
\]
As in the previous section, any such representation gives rise to an equivariant \'{e}tale $\mbb{Z}_p$-sheaf, denoted $\mathscr{T} = \mu_{\invs{G}}(T)$.

Let $\mathscr{F} \in \opn{\acute{E}t}\left( \opn{Sh}_{\invs{G}} \right)_{\mbb{Z}_p, \Sigma}$. Then we can define pullbacks and pushforwards on the cohomology of these sheaves in a similar fashion to the previous subsection. Indeed, if $\sigma \in \Sigma$ and $K, L \subset \Sigma$ are compact open subgroups satisfying $\sigma^{-1}L \sigma \subset K$, then we can define the pullback $[\sigma]^*$ as the composition
\[
[\sigma]^* \colon \opn{H}^i_{\et}\left(\opn{Sh}_{\invs{G}}(K), \mathscr{F}_K \right) \xrightarrow{\sigma^*} \opn{H}^i_{\et}\left(\opn{Sh}_{\invs{G}}(L), \sigma^* \mathscr{F}_K \right) \to \opn{H}^i_{\et}\left(\opn{Sh}_{\invs{G}}(L), \mathscr{F}_L \right)
\]
where the last map is induced from $\sigma^* \mathscr{F}_K \to \mathscr{F}_L$. Similarly, if $\sigma \in \Sigma^{-1}$ such that $\sigma^{-1}L\sigma \subset K$, we define the pushforward $[\sigma]_*$ as the composition
\begin{align*} 
[\sigma]_* \colon \opn{H}^i_{\et}\left(\opn{Sh}_{\invs{G}}(L), \mathscr{F}_L \right) &\xrightarrow{[\sigma^{-1}]^*} \opn{H}^i_{\et}\left(\opn{Sh}_{\invs{G}}(\sigma^{-1} L \sigma), \mathscr{F}_{\sigma^{-1}L\sigma} \right) \\
&\xrightarrow{\sim} \opn{H}^i_{\et}\left(\opn{Sh}_{\invs{G}}(\sigma^{-1} L \sigma), 1^*\mathscr{F}_{K} \right) \\
&\xrightarrow{\opn{Tr}_1} \opn{H}^i_{\et}\left(\opn{Sh}_{\invs{G}}(K), \mathscr{F}_K \right)
\end{align*}
where the middle map is induced from the inverse of the isomorphism $1^*\mathscr{F}_K \xrightarrow{\sim} \mathscr{F}_{\sigma^{-1}L\sigma}$ and the last map is the trace map.

\begin{proposition} \label{PropIntegralCohFunc}
Let $\mathscr{F} \in \opn{\acute{E}t}\left( \opn{Sh}_{\invs{G}} \right)_{\mbb{Z}_p, \Sigma}$ and let $\Upsilon$ denote the collection of all sufficiently small compact open subgroups of $\Sigma$. Then the data 
\[
\left( M_{\invs{G}}(K) = \opn{H}^i_{\et}\left( \opn{Sh}_{\invs{G}}(K), \mathscr{F}_K \right), [-]^*, [-]_* \right)
\]
defines a cohomology functor $M_{\invs{G}} \colon \invs{P}(\invs{G}(\mbb{A}_f), \Sigma, \Upsilon) \to \mbb{Z}_p\text{-Mod}$ in the sense of Definitions \ref{CohoFunc} and \ref{DefOfAdditionalProperties}. Since the Shimura--Deligne datum $(\invs{G}, X_{\invs{G}})$ satisfies (SD5), this cohomology functor is CoMack.
\end{proposition}
\begin{proof}
The proof is essentially the same as that of Proposition \ref{PropRationalCohFunctor}, however as pushforwards are defined in a somewhat ad-hoc manner, we provide some justifications. Let $\sigma \in \Sigma^{-1}$ and $L, K \in \Upsilon$ such that $\sigma^{-1}L \sigma \subset K$. Then the commutative diagram
\[
\begin{tikzcd}
\opn{Sh}_{\invs{G}}(L_{\sigma}) \arrow[d, "\sigma^{-1}"'] \arrow[r, "\opn{pr}"] & \opn{Sh}_{\invs{G}}(K) \arrow[d, "\sigma^{-1}"] \\
\opn{Sh}_{\invs{G}}(L) \arrow[r, "\opn{pr}"]                                    & \opn{Sh}_{\invs{G}}(\sigma K \sigma^{-1})      
\end{tikzcd}
\]
is Cartesian (the vertical arrows are isomorphisms), where $L_{\sigma} \defeq \sigma^{-1}L \sigma$. This implies that the pushforward $[\sigma]_*$ can be expressed in two equivalent ways, namely the pushforward is equal to the composition:
\[
\left( L_{\sigma} \xrightarrow{\opn{pr}} K \right)_* \circ \left( L_{\sigma} \xrightarrow{\sigma^{-1}} L \right)^* = \left(K \xrightarrow{\sigma^{-1}} \sigma K \sigma^{-1} \right)^* \circ \left(L \xrightarrow{\opn{pr}} \sigma K \sigma^{-1} \right)_* .
\]
One can check that, using this description, $M_{\invs{G}}$ does indeed define a cohomology functor. This functor is CoMack for the same reasons as in the proof of Proposition \ref{PropRationalCohFunctor}. 
\end{proof}

\begin{remark}
Note that if $\sigma \in \Sigma \cap \Sigma^{-1}$ and the morphism $\sigma^*\mathscr{F}_K \to \mathscr{F}_L$ is in fact an isomorphism, then the definition of $[\sigma]_*$ coincides with the one in the previous subsection. More generally, if $\mathscr{G} \in \opn{\acute{E}t}\left( \opn{Sh}_{\invs{G}} \right)_{\mbb{Q}_p}$ then we can (non-canonically) think of $\mathscr{G}$ as an equivariant $\mbb{Z}_p$-\'{e}tale sheaf, denoted $\mathscr{F} \in \opn{\acute{E}t}\left( \opn{Sh}_{\invs{G}} \right)_{\mbb{Z}_p, \Sigma}$, and one has commutative diagrams:
\[
\begin{tikzcd}
{\opn{H}^i\left(\opn{Sh}_{\invs{G}}(L), \mathscr{G}_L \right)} \arrow[r, "{[\sigma]^*}"]           & {\opn{H}^i\left(\opn{Sh}_{\invs{G}}(K), \mathscr{G}_K \right)}           &  & {\opn{H}^i\left(\opn{Sh}_{\invs{G}}(L), \mathscr{G}_L \right)} \arrow[r, "{[\sigma]_*}"]           & {\opn{H}^i\left(\opn{Sh}_{\invs{G}}(K), \mathscr{G}_K \right)}           \\
{\opn{H}^i\left(\opn{Sh}_{\invs{G}}(L), \mathscr{F}_L \right)} \arrow[u] \arrow[r, "{[\sigma]^*}"] & {\opn{H}^i\left(\opn{Sh}_{\invs{G}}(K), \mathscr{F}_K \right)} \arrow[u] &  & {\opn{H}^i\left(\opn{Sh}_{\invs{G}}(L), \mathscr{F}_L \right)} \arrow[u] \arrow[r, "{[\sigma]_*}"] & {\opn{H}^i\left(\opn{Sh}_{\invs{G}}(K), \mathscr{F}_K \right)} \arrow[u]
\end{tikzcd}
\]
whenever the maps make sense. This follows from property (C2) of the cohomology functor in Proposition \ref{PropRationalCohFunctor}.
\end{remark}

As in the previous subsection, one can pass to the limit over all compact open subgroups in $\Upsilon$ to obtain a smooth $\Sigma$-representation $\opn{H}^i_{\et}\left(\opn{Sh}_{\invs{G}}, \mathscr{F} \right)$.

\subsection{Functoriality} \label{FunctorialitySubSec}

In this section, let $(\invs{G}, X_{\invs{G}})$ denote either $(\mbf{G}, X_{\mbf{G}})$ or $(\Gt, X_{\Gt})$ and let $\iota \colon (\mbf{H}, X_{\mbf{H}}) \hookrightarrow (\invs{G}, X_{\invs{G}})$ denote the morphism of Shimura--Deligne data induced from $\iota \colon \mbf{H} \hookrightarrow \mbf{G}$ or $(\iota, \nu) \colon \mbf{H} \hookrightarrow \Gt$ respectively. Recall that we are considering the Shimura--Deligne variety $\opn{Sh}_{\invs{G}}(K)$ as a variety over $E$, so we obtain finite unramified morphisms $\iota \colon \opn{Sh}_{\mbf{H}}(U) \to \opn{Sh}_{\invs{G}}(K)$ where $U \subset \mbf{H}(\mbb{A}_f)$ and $K \subset \invs{G}(\mbb{A}_f)$ are any sufficiently small compact open subgroups satisfying $U \subset K$.

\begin{lemma}
The morphism $\iota \colon \opn{Sh}_{\mbf{H}}(U) \to \opn{Sh}_{\invs{G}}(K)$ forms an \'{e}tale smooth $\Spec E$-pair of codimension $n$, in the sense of Definition \ref{pushschemes}.
\end{lemma}
\begin{proof}
This follows from \cite[Proposition 1.15]{DeligneTS}. More precisely, \emph{loc.cit.} implies that there exists a compact open subgroup $L$ containing $U$ such that $\opn{Sh}_{\mbf{H}}(U) \to \opn{Sh}_{\invs{G}}(L)$ is a closed immersion. This map factors through the morphism $\opn{Sh}_{\mbf{H}}(U) \to \opn{Sh}_{\invs{G}}(L \cap K)$, which by the cancellation property, is also a closed immersion. The desired factorisation is then given by $\opn{Sh}_{\mbf{H}}(U) \to \opn{Sh}_{\invs{G}}(L \cap K) \to \opn{Sh}_{\invs{G}}(K)$.
\end{proof}

Let $\Sigma_{\invs{G}} \subset \invs{G}(\mbb{A}_f)$ denote an open submonoid and set $\Sigma_{\mbf{H}} \defeq \Sigma_{\invs{G}} \cap \mbf{H}(\mbb{A}_f)$. Let $\mathscr{G} \in \opn{\acute{E}t}\left( \opn{Sh}_{\invs{G}} \right)_{\mbb{Z}_p, \Sigma_{\invs{G}}}$ and $\mathscr{F} \in \opn{\acute{E}t}\left( \opn{Sh}_{\mbf{H}} \right)_{\mbb{Z}_p, \Sigma_{\mbf{H}}}$ be equivariant \'{e}tale $\mbb{Z}_p$-sheaves on $\opn{Sh}_{\invs{G}}$ and $\opn{Sh}_{\mbf{H}}$ respectively, as in Definition \ref{DefOfEquivIntegral}.

\begin{notation} \label{PullbackOfEquivSheafAdhoc}
    For any compatible pair $(U, K)$ (i.e. sufficiently small compact open subgroups $U \subset \Sigma_{\mbf{H}}$, $K \subset \Sigma_{\mathcal{G}}$ with $U \subset K$), let $\iota_{U, K} \colon \opn{Sh}_{\mbf{H}}(U) \to \opn{Sh}_{\mathcal{G}}(K)$ denote the induced morphism. Then, by abuse of notation, we write $\phi \colon \mathscr{F} \to \iota^*\mathscr{G}$ to mean a collection of morphisms $\phi_{U, K} \colon \mathscr{F}_U \to \iota_{U, K}^* \mathscr{G}_K$ for any compatible pair $(U, K)$ that commute with the equivariant structures coming from $\mathscr{F}$ and $\mathscr{G}$, i.e. for any morphism $[h] \colon (V, L) \to (U, K)$ of compatible pairs with $h \in \Sigma_{\mbf{H}}$, we have a commutative diagram:
    \[
\begin{tikzcd}
{[h]_{V, U}^*\mathscr{F}_U} \arrow[d] \arrow[r, "{[h]^*_{V, U} \phi_{U, K}}"] & {[h]_{V, U}^*\iota^*_{U, K}\mathscr{G}_K} \arrow[r, equals] & {\iota^*_{V, L}[h]^*_{L, K}\mathscr{G}_K} \arrow[d] \\
\mathscr{F}_V \arrow[rr, "{\phi_{V ,L}}"]                                     &                                                                                    & {\iota^*_{V, L}\mathscr{G}_L}                      
\end{tikzcd}
    \]
    where the vertical arrows are induced from the equivariant structures of $\mathscr{F}$ and $\mathscr{G}$. We will often omit the compatible pair from the notation.\footnote{It is possible to make sense of $\iota^*\mathscr{G}$ as an equivariant sheaf on $\opn{Sh}_{\mbf{H}}$ by defining $(\iota^*\mathscr{G})_U \defeq \varinjlim_K \iota_{U, K}^*\mathscr{G}_K$ where the colimit is over all $K$ such that $(U, K)$ is a compatible pair (and the transition morphisms are given by the equivariant structure induced from the identity, which are all isomorphisms). The cocycle condition then implies we have an induced equivariant structure obtained from pulling back the equivariant structure on $\mathscr{G}$. We leave the details to the interested reader.}
\end{notation}

Suppose that we have a morphism $\phi \colon \mathscr{F} \to \iota^* \mathscr{G}$ as in Notation \ref{PullbackOfEquivSheafAdhoc}.

\begin{proposition} \label{PropPushCohFunc}
Let $\Upsilon_{\invs{G}}$ (resp. $\Upsilon_{\mbf{H}}$) denote the collection of sufficiently small compact open subgroups $K \subset \Sigma_{\invs{G}}$ (resp. $U \subset \Sigma_{\mbf{H}}$), and set
\begin{align*}
    M_{\invs{G}} \colon \invs{P}\left(\invs{G}(\mbb{A}_f), \Sigma_{\invs{G}}, \Upsilon_{\invs{G}} \right) &\to \mbb{Z}_p\text{-Mod} \\
    K &\mapsto \opn{H}^{i+2n}_{\et}\left( \opn{Sh}_{\invs{G}}(K), \mathscr{G}(n) \right) \\
     M_{\mbf{H}} \colon \invs{P}\left(\mbf{H}(\mbb{A}_f), \Sigma_{\mbf{H}}, \Upsilon_{\mbf{H}} \right) &\to \mbb{Z}_p\text{-Mod} \\
     U &\mapsto \opn{H}^{i}_{\et}\left( \opn{Sh}_{\mbf{H}}(U), \mathscr{F} \right)
\end{align*}
to be the cohomology functors as in Proposition \ref{PropIntegralCohFunc}.

Then, for pairs $U \subset \Sigma_{\mbf{H}}$, $K \subset \Sigma_{\invs{G}}$ with $U \subset K$, the pushforward
\[
\opn{H}^{i}_{\et}\left( \opn{Sh}_{\mbf{H}}(U), \mathscr{F}_U \right) \xrightarrow{\phi_{U, K}} \opn{H}^{i}_{\et}\left( \opn{Sh}_{\mbf{H}}(U), \iota_{U, K}^*\mathscr{G}_K \right) \xrightarrow{\iota_{U, K,*}} \opn{H}^{i+2n}_{\et}\left( \opn{Sh}_{\invs{G}}(K), \mathscr{G}_K(n) \right)
\]
defines a pushforward of cohomology functors $M_{\mbf{H}} \to M_{\invs{G}}$ as in Definition \ref{push}. Since both $(\invs{G}, X_{\invs{G}})$ and $(\mbf{H}, X_{\mbf{H}})$ satisfy (SD5), this pushforward is Cartesian.
\end{proposition}
\begin{proof}
Let $[h]\colon (V, L) \to (U, K)$ be a morphism of compatible pairs. Then one has a commutative diagram 
\[
\begin{tikzcd}
\opn{Sh}_{\mbf{H}}(U) \arrow[r, "\iota"]                                       & \opn{Sh}_{\invs{G}}(K)                                     \\
\opn{Sh}_{\mbf{H}}(h^{-1}Vh) \arrow[r, "\iota"] \arrow[u] \arrow[d, "h^{-1}"'] & \opn{Sh}_{\invs{G}}(h^{-1}Lh) \arrow[u] \arrow[d, "h^{-1}"] \\
\opn{Sh}_{\mbf{H}}(V) \arrow[r, "\iota"]                                       & \opn{Sh}_{\invs{G}}(L)                                     
\end{tikzcd}
\]
where the bottom square is Cartesian because the morphisms $h^{-1}$ are isomorphisms. The fact that $\iota_*$ defines a pushforward follows from Proposition \ref{contpush} and the fact that pushforwards commute for any commutative square, and pushforwards and pullbacks commute for Cartesian squares.

Since both Shimura--Deligne data satisfy (SD5), in the notation of Definition \ref{push} we have a commutative square
\[
\begin{tikzcd}
\bigsqcup_{\gamma} \opn{Sh}_{\mbf{H}}(U_{\gamma}) \arrow[d, "\sqcup_{\gamma} \opn{pr}"'] \arrow[r, "\sqcup_{\gamma} \iota_{\gamma}"] & \opn{Sh}_{\invs{G}}(L) \arrow[d, "\opn{pr}"] \\
\opn{Sh}_{\mbf{H}}(U) \arrow[r, "\iota"]                                                                                             & \opn{Sh}_{\invs{G}}(K)                      
\end{tikzcd}
\]
where $\iota_{\gamma}$ is the composition $\opn{Sh}_{\mbf{H}}(U_{\gamma}) \xrightarrow{\iota} \opn{Sh}_{\invs{G}}(\gamma L \gamma^{-1}) \xrightarrow{\gamma} \opn{Sh}_{\invs{G}}(L)$, which is Cartesian because the vertical maps are finite \'{e}tale of the same degree (namely $\sum_{\gamma} [U : U_{\gamma}] = [K : L]$). The Cartesian property follows.
\end{proof}

\begin{remark}
We have an analogous statement for equivariant \'{e}tale sheaves over $\mbb{Q}_p$. 
\end{remark}

\begin{example}
We are chiefly interested in the following special case of Proposition \ref{PropPushCohFunc}; namely, suppose $V$ (resp. $W$) is a continuous representation of $\Sigma_{\invs{G}}$ (resp. $\Sigma_{\mbf{H}}$) and we have a $\Sigma_{\mbf{H}}$-equivariant map $W \xrightarrow{\phi} V$. Then we obtain a morphism $\mathscr{W} \xrightarrow{\phi} \iota^*\mathscr{V}$ of equivariant sheaves as in Notation \ref{PullbackOfEquivSheafAdhoc}. Such examples arise naturally from algebraic representations of $\invs{G}_{\mbb{Q}_p}$ and $\mbf{H}_{\mbb{Q}_p}$.  
\end{example}

\subsection{Motivic sheaves}

In this subsection we let $(\invs{G}, X_{\invs{G}})$ denote either $(\mbf{G}, X_{\mbf{G}})$ or $(\mbf{H}, X_{\mbf{H}})$ (so in particular $(\invs{G}, X_{\invs{G}})$ is a PEL-type Shimura--Deligne datum that satisfies (SD5)). Recall that we have identified $E_{\ide{p}}$ with $\mbb{Q}_p$. The following result is due to Ancona and Torzewski:

\begin{theorem} \label{ThmAncAndEt}
Let $K \subset \invs{G}(\mbb{A}_f)$ be a sufficiently small compact open subgroup. There exists an additive tensor functor 
\[
\opn{Anc}_{\invs{G}, K} \colon \opn{Rep}_E(\invs{G}) \to \opn{CHM}\left( \opn{Sh}_{\invs{G}}(K) \right)_E
\]
where $\opn{CHM}\left( - \right)_E$ denotes the category of relative Chow motives over $E$ (see \S \ref{NotationSection}) which fits into the following commutative diagram (up to natural equivalence)
\[
\begin{tikzcd}
\opn{Rep}_E(\invs{G}) \arrow[d, "- \otimes \mbb{Q}_p"'] \arrow[r, "{\opn{Anc}_{\invs{G}, K}}"] & \opn{CHM}\left(\opn{Sh}_{\invs{G}}(K) \right)_E \arrow[d, "r_{\et}"] \\
\opn{Rep}_{\mbb{Q}_p}(\invs{G}) \arrow[r, "{\mu_{\invs{G}, K}}"]                               & \opn{\acute{E}t}\left( \opn{Sh}_{\invs{G}}(K) \right)_{\mbb{Q}_p}         
\end{tikzcd}
\]
where $r_{\et}$ denotes the $\ide{p}$-adic realisation of a motive. 
\end{theorem}
\begin{proof}
See \cite{Ancona2015} and \cite[\S 10]{Torzewski2019}. Since we are working with Shimura--Deligne data which do not necessarily satisfy (SD3), justifications have been provided in Appendix \ref{AnconaAppendix}.
\end{proof}

We also have a similar compatibility property for the functors $\opn{Anc}_{\invs{G}, K}$ for varying $K$, as in Proposition \ref{PropMuEquiv}. 

\begin{lemma} \label{LemmaEquivariantAncFunc}
Let $K, L \subset \invs{G}(\mbb{A}_f)$ be sufficiently small compact open subgroups, and $\sigma \in \invs{G}(\mbb{A}_f)$ such that $\sigma^{-1}L \sigma \subset K$. Then there are natural isomorphisms 
\[
\sigma^*\opn{Anc}_{\invs{G}, K} \cong \opn{Anc}_{\invs{G}, L}
\]
compatible with composition.
\end{lemma}
\begin{proof}
This follows the same argument as in the proof of \cite[Proposition 6.2.4]{LSZ17} using the fact that the canonical model $\opn{Sh}_{\invs{G}}(K)$ represents a moduli problem of abelian varieties with extra structure (Example \ref{UnitaryPELModEx}).
\end{proof}

As in \S \ref{EquivariantCoefficientSheaves}, one can interpret this lemma as saying that the functors $\opn{Anc}_{\invs{G}, K}$ assemble into a functor $\opn{Anc}_{\invs{G}}$ valued in \emph{equivariant relative Chow motives} -- in particular one can define pullbacks $[\sigma]^*$ and pushforwards $[\sigma]_*$ in exactly the same way as in \S \ref{EquivariantCoefficientSheaves}. One also has the following functoriality statement with respect to the Shimura--Deligne datum:

\begin{proposition} \label{PropFunctorialityAnc}
One has a commutative diagram (up to natural isomorphism):
\[
\begin{tikzcd}
\opn{Rep}_E(\mbf{G}) \arrow[d, "{\iota^*}"'] \arrow[r, "\opn{Anc}_{\mbf{G}, K}"] & \opn{CHM}\left( \opn{Sh}_{\mbf{G}}(K) \right)_E \arrow[d, "{\iota^*}"] \\
\opn{Rep}_E(\mbf{H}) \arrow[r, "\opn{Anc}_{\mbf{H}, U}"]                        & \opn{CHM}\left(\opn{Sh}_{\mbf{H}}(U) \right)_E                         
\end{tikzcd}
\]
compatible with the natural isomorphisms in Lemma \ref{LemmaEquivariantAncFunc}, where $\iota$ denotes the embedding $\mbf{H} \hookrightarrow \mbf{G}$ in section \ref{TheGroups}.
\end{proposition}
\begin{proof}
This follows from \cite[Theorem 9.7]{Torzewski2019} because the morphism of associated PEL data is ``admissible'' (Definition 9.1 in \emph{op.cit.}). 
\end{proof}

\begin{remark}
Let $\opn{Rep}_E(\Gt)_0$ denote the full subcategory of $\opn{Rep}_E(\Gt)$ consisting of representations of the form $V \boxtimes \mbf{1}$. Let $\tilde{K} \subset \Gt(\mbb{A}_f)$ be a sufficiently small compact open subgroup and let $K$ denote its projection to $\mbf{G}(\mbb{A}_f)$ (which is also sufficiently small). Then we have an induced map $p \colon \opn{Sh}_{\Gt}(\tilde{K}) \twoheadrightarrow \opn{Sh}_{\mbf{G}}(K)$ and we can extend the construction of $\opn{Anc}_{\mbf{G}, K}$ to a functor on $\opn{Rep}_E(\Gt)_0$ by defining
\[
\opn{Anc}_{\Gt, \tilde{K}}(V \boxtimes \mbf{1}) \defeq p^*\opn{Anc}_{\mbf{G}, K}(V) .
\]
In particular, Theorem \ref{ThmAncAndEt} and Lemma \ref{LemmaEquivariantAncFunc} both hold for this functor and one has an analogous result to Proposition \ref{PropFunctorialityAnc} with respect to the embedding $(\iota, \nu) \colon \mbf{H} \hookrightarrow \Gt$. Note that $(\Gt, X_{\Gt})$ does not give rise to a PEL-type Shimura--Deligne datum, so we cannot immediately apply the results of \cite{Ancona2015} and \cite{Torzewski2019}.
\end{remark}

Since relative purity also holds for constructible Beilinson motives (see \cite[Theorem 2.4.50(3) and \S 14.2.15]{CDMixed}), one can define pushforwards
\[
(\iota, \nu)_* \colon \opn{H}^i_{\mathrm{mot}}\left( \opn{Sh}_{\mbf{H}}(U), \mathscr{F} \right) \to \opn{H}^{i+2n}_{\mathrm{mot}}\left( \opn{Sh}_{\Gt}(\tilde{K}), \mathscr{G}(n) \right)
\]
for any morphism $\phi \colon \mathscr{F} \to (\iota, \nu)^*\mathscr{G}$ of relative Chow motives, similar to subsection \ref{FunctorialitySubSec} and Proposition \ref{contpush}. 


\section{Branching laws} \label{BranchingLawsAndReps}

Recall that one has an identification $\mbf{G}_E = \opn{GL}_{1, E} \times \opn{GL}_{2n, E}$ (Remark \ref{RemIdentifyGroups}) and irreducible algebraic representations of $\mbf{G}_E$ are classified by tuples $\mbf{c} = (c_0; c_1, \dots, c_{2n}) \in \mbb{Z}^{1+2n}$ which satisfy $c_1 \geq \cdots \geq c_{2n}$ (see \S \ref{DiscreteSeriesReps}). We are interested in algebraic representations $V$ of $\mbf{G}_E$ with highest weight $\mbf{c} = (0; c_1, \dots, c_{2n})$ that are self dual: in this case, we must have $\mbf{c} = -\mbf{c}'$, where $\mbf{c}' = (0; c_{2n}, \dots, c_1)$. We  have the following ``branching law":
\begin{lemma} \label{Branching}
Let $V$ be a self-dual algebraic representation of $\mbf{G}_{E}$ with highest weight $\mbf{c} = (0; c_1, \dots, c_{2n})$ as above and let $\iota^*V$ denote its restriction to an algebraic representation of $\mbf{H}_E$. Then $\iota^* V$ contains the trivial representation as a direct summand, with multiplicity one. The same statement holds for representations of $\Gt_E$ of the form $\tilde{V} \defeq V \boxtimes \mbf{1}$, with respect to the embedding $(\iota, \nu)$. 
\end{lemma}
\begin{proof}
This is well-known and follows from Proposition 6.3.1 in \cite{GR2014} for example.
\end{proof}

These self-dual representations will correspond to the coefficient sheaves that we will consider throughout the paper. We will also need to consider certain lattices inside these algebraic representations (after base-changing to $\mbb{Q}_p$). For this we introduce the following subgroups and elements.

\begin{definition}
Let $Q_0 \subset \mbf{G}_{0, \mbb{Q}_p} = \opn{GL}_{2n, \mbb{Q}_p}$ denote the standard parabolic subgroup
\[
Q_0 \defeq \tbyt{*}{*}{}{*}
\]
where the blocks are of size $(n \times n)$. Set $Q = \opn{GL}_{1, \mbb{Q}_p} \times Q_0 \subset \mbf{G}_{\mbb{Q}_p} = \opn{GL}_{1, \mbb{Q}_p} \times \mbf{G}_{0, \mbb{Q}_p}$ and $\tilde{Q} = Q \times \opn{GL}_{1, \mbb{Q}_p} \subset \Gt_{\mbb{Q}_p} = \mbf{G}_{\mbb{Q}_p} \times \opn{GL}_{1, \mbb{Q}_p}$. Let $\tau_0$ denote the following $(2n \times 2n)$-block matrix
\[
\tau_0 \defeq \tbyt{p}{}{}{1} \in \mbf{G}_0(\mbb{Q}_p)
\]
and set $\tau = 1 \times \tau_0 \in \mbf{G}(\mbb{Q}_p)$ and $\tilde{\tau} = \tau \times 1 \in \Gt(\mbb{Q}_p)$. 

We let $J_0 \subset \mbf{G}_0(\mbb{Z}_p)$ denote the parahoric subgroup of elements which land in $Q_0$ modulo $p$, and let $J \defeq \mbb{Z}_p^{\times} \times J_0 \subset \mbf{G}(\mbb{Z}_p)$ and $\tilde{J} \defeq J \times \mbb{Z}_p^{\times} \subset \Gt(\mbb{Z}_p)$.
\end{definition}

Let $V$ be an algebraic representation of $\mbf{G}_{E}$ that is self-dual of weight $\mbf{c} = (0; c_1, \dots, c_{2n})$. If we let $\overline{B}$ denote the standard opposite Borel subgroup in $\mbf{G}_E$, then this representation can be viewed as the algebraic induction:
\[
V = \opn{Ind}_{\overline{B}}^{\mbf{G}_E} \mu \defeq \left\{ f \colon \mbf{G}_E \to \mbb{A}^1_E : \begin{array}{c} f \text{ is algebraic } \\ f(b x) = \mu(b)f(x) \text{ for all } b \in \overline{B} \end{array}\right\}
\]
where $\mu$ is (the inflation to $\overline{B}$ of) the character which sends an element $\opn{diag}(t_0; t_1, \dots, t_{2n}) \in \mbb{G}_m^{1+2n}$ to $t_1^{c_1} \cdots t_{2n}^{c_{2n}}$. The action of $\mbf{G}_E$ on $\opn{Ind}_{\overline{B}}^{\mbf{G}_E} \mu$ is given by right-translation of the argument, i.e. $g \cdot f(x) = f(xg)$.

\begin{definition} \label{SpotLattice}
Let $T \subset V_{\mbb{Q}_p}$ denote the $\mbf{G}(\mbb{Z}_p)$-stable lattice defined as the subspace of all functions $f \in \opn{Ind}_{\overline{B}_{\mbb{Q}_p}}^{\mbf{G}_{\mbb{Q}_p}} \mu$ satisfying $f\left(\mbf{G}(\mbb{Z}_p) \right) \subset \mbb{Z}_p$. This extends trivially to a $\Gt(\mbb{Z}_p)$-stable lattice $\tilde{T} \subset \tilde{V}_{\mbb{Q}_p}$.
\end{definition}

We introduce the following monoids and actions on the lattices introduced in Definition \ref{SpotLattice}.

\begin{lemma} \label{SpotAction}
Consider the following sets:
\begin{itemize} 
\item $\Sigma_{0, G_p} \defeq J_0 \cdot \{ \tau_0^{-r} : r \in \mbb{Z}_{\geq 0} \} \cdot J_0$ 
\item $\Sigma_{G_p} \defeq \mbb{Q}_p^\times \times \Sigma_{0, G_p}$ as a subset of $\mbf{G}(\mbb{Q}_p)$
\item $\Sigma_{{\gt}_p} \defeq \Sigma_{G_p} \times \mbb{Q}_p^{\times}$ as a subset of $\Gt(\mbb{Q}_p)$. 
\end{itemize}
Then $\Sigma_{G_p}$ and $\Sigma_{\gt_p}$ are open submonoids of $\mbf{G}(\mbb{Q}_p)$ and $\Gt(\mbb{Q}_p)$ respectively, and the actions of $J$ and $\tilde{J}$ extend to actions $\bullet$ of $\Sigma_{G_p}$ and $\Sigma_{\gt_p}$ on $T$ and $\tilde{T}$ which factor through $\Sigma_{0, G_p}$ and satisfy:
\[
\tau^{-1} \bullet v = \mu(\tau)(\tau^{-1} \cdot v) \quad \quad \quad \tilde{\tau}^{-1} \bullet \tilde{v} = \mu(\tau)(\tilde{\tau}^{-1} \cdot \tilde{v})
\]
for all $v \in T$ and $\tilde{v} \in \tilde{T}$.
\end{lemma}
\begin{proof}
We have the following Iwahori decomposition
\[
J = \overline{N}_1 \cdot \mbf{H}(\mbb{Z}_p) \cdot N(\mbb{Z}_p) = N(\mbb{Z}_p) \cdot \mbf{H}(\mbb{Z}_p) \cdot \overline{N}_1
\]
where $N \subset Q$ (resp. $\overline{N} \subset \overline{Q}$) is the unipotent radical (resp. opposite unipotent radical) and $\overline{N}_r = \tau^{-r}\overline{N}(\mbb{Z}_p)\tau^r$. Since $\tau^{-1} \overline{N}_1 \tau \subset \overline{N}_1$ and $\tau N(\mbb{Z}_p) \tau^{-1} \subset N(\mbb{Z}_p)$, we see that $\Sigma_{0, G_p}$, $\Sigma_{G_p}$ and $\Sigma_{\gt_p}$ do indeed form monoids. Furthermore, one can easily check that the action of $J$ on $V_{\mbb{Q}_p}$ extends to a unique action $\bullet$ of $\Sigma_{G_p}$ on $\tilde{V}_{\mbb{Q}_p}$ satisfying the property in the statement of the proposition. Therefore, we are required to check that $\bullet$ stabilises the lattice $T$. 

Via the description in terms of an algebraic induction, we are required to prove the statement: for any $f \in T$, one has
\[
\mu(\tau) f(x \tau^{-1}) \in \mbb{Z}_p
\]
for any $x \in \mbf{G}(\mbb{Z}_p)$. For this, we let $I$ denote the standard Iwahori subgroup of $\mbf{G}(\mbb{Q}_p)$ and note that we have the decomposition
\[
\mbf{G}(\mbb{Z}_p) = \bigsqcup_w \overline{B}(\mbb{Z}_p) \cdot w \cdot I
\]
where the disjoint union runs over any representatives of the Weyl group of $\mbf{G}_{\mbb{Q}_p}$. For any element $i \in I$ and representative $w$, one has
\[
\mu(\tau) f(wi \tau^{-1}) = \mu(\tau) f(w \tau^{-1} w^{-1} w \tau i \tau^{-1}) = \mu(\tau w \tau^{-1} w^{-1}) f(w\tau i \tau^{-1}).
\]
Since $\mu$ is dominant and $w \tau i \tau^{-1} \in \mbf{G}(\mbb{Z}_p)$, this gives the claim. The statement for $\tilde{T}$ follows the same argument.
\end{proof}

We have the following compatibility between these lattices and their rational versions on the level of cohomology.

\begin{lemma} \label{LemmaCompatibleAction}
Let $\mathscr{T} = \mu_{\mbf{G}}(T)$, $\tilde{\mathscr{T}} = \mu_{\Gt}(\tilde{T})$, $\mathscr{V} = \mu_{\mbf{G}}(V_{\mbb{Q}_p})$ and $\tilde{\mathscr{V}} = \mu_{\Gt}(\tilde{V}_{\mbb{Q}_p})$ denote the equivariant sheaves constructed in \S \ref{EquivariantCoefficientSheaves}--\ref{IntegralEquivariantCoefficientSheaves}, where the first two are defined with respect to the monoids $\mbf{G}(\mbb{A}_f^p) \times \Sigma_{G_p}$ and $\Gt(\mbb{A}_f^p) \times \Sigma_{\gt_p}$ respectively. The natural maps 
\[
\opn{H}^i_{\et}\left(\opn{Sh}_{\mbf{G}}, \mathscr{T}(j) \right) \to \opn{H}^i_{\et}\left(\opn{Sh}_{\mbf{G}}, \mathscr{V}(j) \right) \quad \quad \opn{H}^i_{\et}\left(\opn{Sh}_{\Gt}, \tilde{\mathscr{T}}(j) \right) \to \opn{H}^i_{\et}\left(\opn{Sh}_{\Gt}, \tilde{\mathscr{V}}(j) \right)
\]
are equivariant for the actions of $\mbf{G}(\mbb{A}_f^p) \times J$ and $\Gt(\mbb{A}_f^p) \times \tilde{J}$ respectively. Furthermore, the action of $\tau^{-1}$ (resp. $\tilde{\tau}^{-1}$) on the left-hand side is intertwined with the action of $\mu(\tau)\tau^{-1}$ (resp. $\mu(\tau)\tilde{\tau}^{-1}$) on the right-hand side.
\end{lemma}
\begin{proof}
The first part follows from the fact that $T \subset V_{\mbb{Q}_p}$ and $\tilde{T} \subset \tilde{V}_{\mbb{Q}_p}$ are equivariant embeddings for the actions of $J$ and $\tilde{J}$ respectively. The last assertion follows from the fact that $\tilde{\tau}^{-1} \bullet \tilde{v} = \mu(\tau)(\tilde{\tau}^{-1} \cdot \tilde{v})$ for all $\tilde{v} \in \tilde{T}$. 
\end{proof}

\begin{notation} \label{NoteFixedEmbedding}
We fix once and for all a $\mbf{H}_E$-equivariant embedding $\opn{br} \colon E \hookrightarrow \tilde{V}$ which restricts to a $\mbf{H}(\mbb{Z}_p)$-equivariant embedding $\mbb{Z}_p \hookrightarrow \tilde{T}$. Such an embedding exists by Lemma \ref{Branching}.
\end{notation}

Let $\Sigma_{H_p} \defeq \Sigma_{G_p} \cap \mbf{H}(\mbb{Q}_p) = \Sigma_{\gt_p} \cap \mbf{H}(\mbb{Q}_p)$. This monoid can be described as the collection of triples $(\alpha; p^{-r} h_1, h_2)$ where $\alpha \in \mbb{Q}_p^{\times}$, $h_1, h_2 \in \opn{GL}_n(\mbb{Z}_p)$ and $r \in \mbb{Z}_{\geq 0}$. In particular, if we let $\lambda = \sum_{i=1}^n c_i$, then we have a well-defined (continuous) character
\begin{align*} 
|\opn{det}|_p^{\frac{\lambda}{n}} \colon \quad \quad \quad \quad \Sigma_{H_p}  &\to  \mbb{Z}_p - \{0\}  \\
 (\alpha; p^{-r}h_1, h_2)  &\mapsto  p^{\lambda r} . 
\end{align*}
Let $\mbb{Z}_p[\lambda / n]$ denote the free $\mbb{Z}_p$-module of rank one with a continuous action of $\Sigma_{H_p}$ given by the character $|\opn{det}|_p^{\frac{\lambda}{n}}$. Then the $\mbf{H}(\mbb{Z}_p)$-equivariant embedding $\mbb{Z}_p \hookrightarrow \tilde{T}$ in Notation \ref{NoteFixedEmbedding} extends to a $\Sigma_{H_p}$-equivariant embedding $\mbb{Z}_p[\lambda/n] \hookrightarrow \tilde{T}$.

By abuse of notation, we also let $\mbb{Z}_p[\lambda/n]$ denote the equivariant sheaf $\mu_{\mbf{H}}(\mbb{Z}_p[\lambda/n])$ (as in section \ref{IntegralEquivariantCoefficientSheaves} with respect to the monoid $\mbf{H}(\mbb{A}_f^p) \times \Sigma_{H_p}$). We obtain the following corollary:

\begin{corollary}
Let $\lambda = \sum_{i=1}^n c_i$. The morphism $\opn{br} \colon \mbb{Q}_p \to (\iota, \nu)^* \tilde{\mathscr{V}}$ restricts to a morphism
\[
\mbb{Z}_p[\lambda/n] \to (\iota, \nu)^* \tilde{\mathscr{T}}
\]
of $\mbf{H}(\mbb{A}_f^p) \times \Sigma_{H_p}$-equivariant sheaves.
\end{corollary}

\subsection{The cohomology functors} \label{FunctorsFromEt}

We now introduce the cohomology functors that will be used throughout this article. For any prime $\ell$, we denote 
\[
G_\ell \defeq \mbf{G}(\mbb{Q}_\ell), \quad G^\ell\defeq \mbf{G}(\mbb{A}_f^\ell), \quad G_{f} \defeq \Gb(\Ab_{f}) 
\]
and similarly for the groups $\mbf{H}$ and $\mbf{T}$. Recall $\Sigma_{G_{p}} \subset G_{p}$ is the open submonoid appearing in Lemma \ref{SpotAction}, and we have set $\Sigma _ { H _ { p } } = \Sigma_ { G _ {p }} \cap H_p $. We let $ \Upsilon _ { G _{ p } } $ be the collection of all compact open subgroups in $\Sigma _ {G_{p}} $ (or equivalently, all compact open subgroups in $J$) and take $ \Upsilon _ { G _ { p } } ^ { \text{tot} } $ to be the collection of all compact open subgroups in $ G_ { p } $. 

\begin{lemma} \label{MaximalLemma}
There exists a (non-empty) maximal collection $ \Upsilon _ { G ^ {p }} $ of compact open subgroups of $ G ^ { p } $, closed under inclusions and conjugation by $ G^{p} $, such that the product $ \Upsilon _ { G ^ { p } } \times \Upsilon_{G_{p} } ^ { \mathrm{tot} } $ consists of sufficiently small subgroups of $ G _ {f } $. Moreover, for any $\ell \neq p$ and compact open subgroup $K_\ell \subset G_\ell$, there exists an element in $\Upsilon_{G^p}$ whose $\ell$-adic component is equal to $K_\ell$.
\end{lemma} 
\begin{proof} 
Let $K_p'$ denote any representative of the unique conjugacy class of maximal compact open subgroups in $G_p = \mbb{Q}_p^{\times} \times \opn{GL}_{2n}(\mbb{Q}_p)$. Then there exists a compact open subgroup $K^p$ such that $K^pK_p'$ is sufficiently small and therefore, $ K ^ { p } K _ { p } $ is sufficiently small for any $ K _ { p } \subset G _ {p } $. We take $ \Upsilon _ { G ^ { p } } $ to be the collection of all such groups $ K ^ { p } $. This collection is closed under conjugation by $ G ^ { p } $ and contains all compact open subgroups of any element $ K ^ { p } \in \Upsilon _ { G ^ { p } } $. By deepening the level at a finite set of places $ S $ not containing $ \ell $ and $p$, we can arrange that $ K ^ { p } = K_\ell K_S K^{S \cup \{\ell, p\}} \in \Upsilon_{G^p}$ for any arbitrary $K_\ell$.
\end{proof} 

We now define the monoids $\Sigma_{\square}$ and collections $\Upsilon_{\square}$ that will be used throughout the paper. Firstly, we deal with the ``tame level'': set
\begin{itemize}
 \item $\Sigma_{G^p}, \Sigma_{H^p}, \Sigma_{T^p}$ and $\Sigma_{\tilde{G}^p}$ to be $G^p, H^p, T^p$ and $\tilde{G}^p$ respectively. 
 \item $\Upsilon_{G^p}$ is the collection appearing in Lemma \ref{MaximalLemma}. 
 \item $\Upsilon_{H^p}$ is defined similarly to $\Upsilon_{G^p}$. 
 \item $\Upsilon_{T^p} = \Upsilon_{T^p}^{\mathrm{tot}}$ denotes the collection of all compact open subgroups of $T^p$.
 \item $\Upsilon_{\tilde{G}^p} = \Upsilon_{G^p} \times \Upsilon_{T^p}$ (we will only be interested in compact open subgroups of $\tilde{G}^p$ of the form $K^p \times C^p$). 
\end{itemize}
At the prime $p$ we fix the following data:
\begin{itemize}
 \item $\Sigma_{G_p}$ and $\Sigma_{H_p}$ are as in the start of this section. We set $\Sigma_{T_p} = T_p$ and 
 \[
 \Sigma_{\tilde{G}_p} = \Sigma_{G_p} \times \Sigma_{T_p}
 \]
 as in Lemma \ref{SpotAction}.
 \item For $\square_p = G_p, H_p$ we set $\Upsilon_{\square_p}$ to be the collection of all compact open subgroups contained in $\Sigma_{\square_p}$. We set $\Upsilon_{T_p}$ to be the collection of compact open subgroups of $1+p\mbb{Z}_p \subset \mbf{T}(\mbb{Z}_p)$, and $\Upsilon_{\gt_p}$ to be the collection of compact open subgroups of $\Sigma_{G_p} \times (1+p\mbb{Z}_p) \subset \Gt(\mbb{Z}_p)$ (so in particular, elements of $\Upsilon_{\tilde{G}_p}$ are not necessarily of the form $K_p \times C_p$). 
\end{itemize}

Finally, for $\square = G, H, T$ or $\tilde{G}$, we set $\Sigma_{\square} = \Sigma_{\square^p} \times \Sigma_{\square_p}$ and $\Upsilon_{\square} = \Upsilon_{\square^p} \times \Upsilon_{\square_p}$. For brevity, we set 
\[
\invs{P}_{\square} \defeq \invs{P}\left( \square_f, \Sigma_{\square}, \Upsilon_{\square} \right)
\]
and similarly for $\invs{P}_{\square^p}$ and $\invs{P}_{\square_p}$. 

\begin{remark}
Although not explicitly included in the notation, the collections $\Upsilon_{\square^p}$ depend on the group $\square_f$. The subscript simply refers to the ambient group of which the elements are compact open subgroups. 
\end{remark}

\begin{remark}
To help orient the reader, we will give a more down-to-earth description of the compact open subgroups in the collection $\Upsilon_{\gt}$. Any such compact open subgroup is of the form
\[
\tilde{K} = (K^p \times C^p) \cdot \tilde{K}_p
\]
where 
\begin{enumerate}
    \item $C^p \subset \mbf{T}(\mbb{A}_f^p)$ is any compact open subgroup and $K^p \subset \mbf{G}(\mbb{A}_f^p)$ is any open compact subgroup such that, for any compact open subgroup $K'_p \subset \mbf{G}(\mbb{Q}_p)$, the subgroup $K^pK'_p \subset \mbf{G}(\mbb{A}_f)$ is sufficiently small.
    \item $\tilde{K}_p \subset \Gt(\mbb{Q}_p)$ is any compact open subgroup contained in $J \times (1+p\mbb{Z}_p)$. 
\end{enumerate}
In particular, any compact open subgroup in $\Upsilon_{\gt}$ is sufficiently small. Informally, for our cohomology functors we want to vary over compact open subgroups of $G^p$ and $T^p$ independently, and over all compact open subgroups of $J \times (1+p\mbb{Z}_p)$. The additional constraints are imposed to ensure that we always work with a sufficiently small compact open subgroup. More generally, we have defined $\Upsilon_{\square}$ such that any subgroup in this collection is sufficiently small in $\square_f$.  
\end{remark}

The following lemma will allow us to apply the results of section \ref{TheInterludeSection}.

\begin{lemma}
For a group $\square \in \{G, H, T, \tilde{G} \}$ and a pair $(\Sigma, \Upsilon) \in \{ (\Sigma_{\square}, \Upsilon_{\square}), (\Sigma_{\square_p}, \Upsilon_{\square_p}), (\Sigma_{\square^p}, \Upsilon_{\square^p}) \}$ the axioms (T1)--(T3) at the start of section \ref{GeneralitiesSection} are satisfied. Moreover, for any elements $U \in \Upsilon_H$ and $\tilde{K} \in \Upsilon_G$, the intersection satisfies $U \cap \tilde{K} \in \Upsilon_H$.
\end{lemma}
\begin{proof}
By construction, if properties (T1)--(T3) are satisfied for both $(\Sigma_{\square_p}, \Upsilon_{\square_p})$ and $(\Sigma_{\square^p}, \Upsilon_{\square^p})$ then they are also satisfied for the pair $(\Sigma_{\square}, \Upsilon_{\square})$. For the ``tame level'', the properties follow from the fact that the collection in Lemma \ref{MaximalLemma} is closed under inclusions and conjugation. For the level at $p$, the properties follow from the fact that $\Upsilon$ is taken to be the collection of all compact open subgroups of an appropriate submonoid of $\Sigma$.

The last part of the lemma follows from the analogous statement at levels away and at $p$; the former is true because the collection in Lemma \ref{MaximalLemma} is closed under inclusion, the latter is true because $\Upsilon_{H_p}$ is the collection of all compact open subgroups of $\Sigma_{H_p}$, which is also closed under inclusion.
\end{proof}

Let $V$ be a self-dual algebraic representation of $\mbf{G}_E$ with highest weight $\mbf{c} = (0; c_1, \dots, c_{2n})$. With notation as above, we define the following cohomology functors on $ \mathcal { P } _ { \square } $ where $ \square $ is to be replaced by the symbol of the group appearing in the subscript of the cohomology functor:

\vspace{0.2cm}
\begin{center}
    \begin{tabular}{c c c} \vspace{0.2cm}
       $M_{H, \mbb{Z}_p}\left( - \right) \defeq \opn{H}^0_{\et}\left( \opn{Sh}_{\mbf{H}}(-), \mbb{Z}_p[\lambda/n] \right) $  & $\quad$ &  $M_{H, \mbb{Q}_p}\left( - \right) \defeq \opn{H}^0_{\et}\left( \opn{Sh}_{\mbf{H}}(-), \mbb{Q}_p \right) $ \\ \vspace{0.2cm}
       $M_{\tilde{G}, \mbb{Z}_p}\left( - \right) \defeq \opn{H}^{2n}_{\et}\left(\opn{Sh}_{\Gt}( - ), \tilde{\mathscr{T}}(n) \right)$  & $\quad$ & $M_{\tilde{G}, \mbb{Q}_p}\left( - \right) \defeq \opn{H}^{2n}_{\et}\left(\opn{Sh}_{\Gt}( - ), \tilde{\mathscr{V}}(n) \right)$  \\ \vspace{0.2cm}
       $M_{G, \et, \mbb{Z}_p}\left( - \right) \defeq \opn{H}^{2n-1}_{\et}\left(\opn{Sh}_{\mbf{G}}( - )_{\overline{E}}, \mathscr{T}(n) \right) $ & $\quad$ & $M_{G, \et, \mbb{Q}_p}\left( - \right) \defeq \opn{H}^{2n-1}_{\et}\left(\opn{Sh}_{\mbf{G}}( - )_{\overline{E}}, \mathscr{V}(n) \right) $ \\ \vspace{0.2cm}
       $M_{\tilde{G}, \et, \mbb{Z}_p}(-) \defeq \opn{H}^{2n-1}_{\et}\left(\opn{Sh}_{\Gt}(-)_{\overline{E}}, \tilde{\mathscr{T}}(n) \right)$ & $\quad$ & $M_{\tilde{G}, \et, \mbb{Q}_p}(-) \defeq \opn{H}^{2n-1}_{\et}\left(\opn{Sh}_{\Gt}(-)_{\overline{E}}, \tilde{\mathscr{V}}(n) \right)$
    \end{tabular}
\end{center}
where $\mbb{Z}_p[\lambda/n]$ denotes the equivariant sheaf associated with the character $|\opn{det}|_p^{\lambda/n}$. If we do not wish to specify whether we are working integrally or rationally, we will simply leave the ring out of the subscript. 

\begin{lemma}
The functors $M_H, M_{\tilde{G}}, M_{G, \et}$ and $M_{\tilde{G}, \et}$ are CoMack functors on their respective categories. Moreover, we have a  Cartesian pushforward $ \iota_{*} \colon M_{H} \to M_{\tilde{G}}$. 
\end{lemma}
\begin{proof}
The functors $M_{H}$ and $M_{\gt}$ are simply the restriction of the cohomology functors appearing in Propositions \ref{PropRationalCohFunctor} and \ref{PropIntegralCohFunc} to the categories $\invs{P}_H$ and $\invs{P}_{\gt}$ respectively, and since the Shimura--Deligne data we're considering satisfy (SD5) (Lemma \ref{LemSD5issatisfied}) these cohomology functors are CoMack. Furthermore, the same proofs of Propositions \ref{PropRationalCohFunctor} and \ref{PropIntegralCohFunc} imply that $M_{G, \et}$ and $M_{\gt, \et}$ are CoMack too. The last part of the lemma follows from Proposition \ref{PropPushCohFunc}. 
\end{proof}

\begin{remark}
Note that $(M_{H, \mbb{Z}_p}, M_{H, \mbb{Q}_p})$ and $(M_{\gt, \mbb{Z}_p}, M_{\gt, \mbb{Q}_p})$ do not form pairs of compatible cohomology functors as defined in the discussion preceding Corollary \ref{CommDiagramCor}, because the natural maps are only $\mbf{H}(\mbb{A}_f^p) \times \mbf{H}(\mbb{Z}_p)$-equivariant and $\Gt(\mbb{A}_f^p) \times \tilde{J}$-equivariant respectively.  
\end{remark}


\section{Cohomology of unitary Shimura varieties} \label{CSV_section}

In this section we recall the construction of Galois representations attached to cuspidal automorphic representations for the group $\mbf{G} = \opn{GU}(1, 2n-1)$. To apply the results of \cite{Morel}, we assume that $n$ is \emph{odd} throughout this section.

Let $\pi_0$ be a cuspidal automorphic representation of $\mbf{G}_0(\mbb{A})$ such that $\pi_{0, \infty}$ lies in the discrete series, and let $\pi$ denote a lift to $\mbf{G}(\mbb{A})$ (see Proposition \ref{UnitaryBaseChangeProp}). If $\pi_0$ has trivial central character, then we take $\pi$ to be a lift with trivial central character. Assume that $\pi$ is cohomological, i.e. there exists:
\begin{itemize}
    \item A sufficiently small level $K \subset \mbf{G}(\mbb{A}_f)$ such that $\pi_f^K \neq \{0\}$.
    \item An irreducible algebraic representation $V$ of $\mbf{G}_E$ and an integer $i \in \mbb{Z}$ such that
    \[
    \opn{H}^i\left(\ide{g}_{\mbb{C}}, K_{\infty}'; \pi_{\infty} \otimes V_{\mbb{C}} \right) \neq \{0\}
    \]
    where $\ide{g} = \opn{Lie}(\mbf{G}(\mbb{R}) )$ and $K_{\infty}' = K_{\infty}Z_{\mbf{G}}(\mbb{R})^\circ$, with $K_{\infty}$ a maximal compact subgroup of $\mbf{G}(\mbb{R})$. 
\end{itemize}
We say that a representation $\Pi = \psi \boxtimes \Pi_0$ of $\opn{GL}_1(\mbb{A}_E) \times \opn{GL}_{2n}(\mbb{A}_E)$ is $\theta$-stable if $\Pi_0^c \cong \Pi_0^\vee$ and $\psi = \psi^c \omega_{\Pi_0}^c$, where $\omega_{\Pi_0}$ is the central character of $\Pi_0$. Recall that our choice of Hermitian pairing implies that the group $\mbf{G}$ is quasi-split at all finite places. Then we have the following result due to Morel:
\begin{proposition} \label{WeakBCProp}
Let $S$ be the set of places of $\mbb{Q}$ containing $\infty$, all primes which ramify in $E$ and all places where $K$ is not hyperspecial (so $\pi$ is unramified outside $S$). 
\begin{enumerate}
    \item There exists an irreducible admissible representation $\Pi$ of $\opn{Res}_{E/\mbb{Q}}\mbf{G}_E (\mbb{A})$ such that for all but finitely many primes
    \begin{equation} \label{BCequivalence}
    \varphi_{\Pi_{\ell}} \cong \opn{BC}\left( \varphi_{\pi_\ell} \right)
    \end{equation}
    where $\varphi_{\Pi_\ell}$ and $\varphi_{\pi_\ell}$ are the local Langlands parameters of $\Pi_\ell$ and $\pi_\ell$ respectively, and $\opn{BC}$ is the local base-change morphism defined in section \ref{DualGrpsBC}. We say $\Pi$ is a weak base-change of $\pi$.
    \item If the representation $\Pi$ above is cuspidal then any weak base-change of $\pi$ is a $\theta$-stable regular algebraic cuspidal automorphic representation and (\ref{BCequivalence}) holds for all $\ell \notin S$. Furthermore, the infinitesimal character of $\Pi_\infty$ is the same as that of $(V \otimes V^\theta)^*$, where $V^\theta$ is the representation with highest weight $(c_0 + \dots + c_{2n}; -c_{2n}, \dots, -c_1)$.
    \item Suppose that $\Pi$ is cuspidal and $\pi$ has trivial central character (for example, if $\pi_{0, f}$ admits a $\mbf{H}_0$-linear model). Then under the isomorphism $\opn{Res}_{E/\mbb{Q}}\mbf{G}_E(\mbb{A}_{\mbb{Q}}) \cong \opn{GL}_1(\mbb{A}_E) \times \opn{GL}_{2n}(\mbb{A}_E)$, the representation $\Pi$ is of the form $\mbf{1} \boxtimes \Pi_0$ for some conjugate self-dual regular algebraic cuspidal automorphic representation $\Pi_0$ of $\opn{GL}_{2n}(\mbb{A}_E)$ satisfying $\varphi_{\Pi_{0, \ell}} = \opn{BC}\left(\varphi_{\pi_{0, \ell}}\right)$ for all $\ell \notin S$.\footnote{Here $\Pi_{0, \ell} = \bigotimes_{\lambda | \ell} \Pi_{0, \lambda}$.}
\end{enumerate}
\end{proposition}
\begin{proof}
Most of the proposition follows from Corollary 8.5.3, Remark 8.5.4 and the proof of Lemma 8.5.6 in \cite{Morel}. The reason that (\ref{BCequivalence}) holds for all $\ell \notin S$ follows from the fact that $\mbf{G}$ is quasi-split at all finite places. Moreover, if $\pi_{0, f}$ admits a $\mbf{H}_0$-linear model then the assertion that $\Pi$ is of the form $\mbf{1} \boxtimes \Pi_0$ follows from the commutative diagram in section \ref{DualGrpsBC} and the fact that the central character of $\pi_{f}$ is trivial. Indeed, if we write $\Pi = \psi \boxtimes \Pi_0$ then it is enough to show that $\psi_\ell = 1$ for all primes $\ell \notin S$, and this follows from a direct computation involving Satake parameters. For the assertion about the infinitesimal character see the proofs of \cite[Lemma 8.5.6]{Morel} and \cite[Theorem 9]{Skinner}. 
\end{proof}

\begin{remark}
If the highest weight of the representation $V$ is regular then the above proposition follows from Theorem A in \cite{Skinner} (which also builds on the work of Morel). However we are only interested in the case where $\Pi$ is cuspidal which does not require this assumption. 
\end{remark}

From now on we assume that any weak base change of $\pi$ is \emph{cuspidal}. In this case, thanks to the work of many people (see for example \cite{ChenevierHarris}, \cite{HarrisTaylor}, \cite{Morel}, \cite{Shin2011} and \cite{Skinner}), there exists a semisimple Galois representation attached to $\pi$. 

\begin{theorem} \label{ExistenceOfGaloisRep}
Let $\pi$ be as above and assume that the central character of $\pi$ is trivial. Suppose that a weak base-change $\Pi = \mbf{1} \boxtimes \Pi_0$ of $\pi$ is cuspidal. Then there exists a continuous semisimple Galois representation 
\[
\rho_{\pi} \colon G_E \to \opn{GL}(W_{\pi}) \cong \opn{GL}_{2n}(\overline{\mbb{Q}}_p)
\]
such that 
\begin{enumerate}
    \item $\rho_{\pi}$ is conjugate self-dual (up to a twist), i.e. one has an isomorphism 
    \[
    \rho_{\pi}^c \cong \rho_{\pi}^*(1-2n)
    \]
    where $\rho_{\pi}^c$ denotes the representation given by conjugating the argument by complex conjugation in $G_{\mbb{Q}}$ (recall that we have fixed an embedding $\overline{\mbb{Q}} \hookrightarrow \mbb{C}$). 
    \item For all $\ell \notin S \cup \{p\}$ and $\lambda | \ell$, the representation $\rho_{\pi}|_{{E_\lambda}}$ is unramified and satisfies 
    \[
    Q_\lambda(\opn{Nm}\lambda^{-s}) \defeq \opn{det}\left(1 - \opn{Frob}_{\lambda}^{-1} (\opn{Nm}\lambda)^{-s} | \; \rho_{\pi} \right) = L\left( \Pi_\lambda, s + (1 - 2n)/2 \right)^{-1}.
    \]
    In particular, if $\ell$ splits in $E/\mbb{Q}$ and $\lambda$ is the prime lying above $\ell$ corresponding to the fixed embedding $\overline{\mbb{Q}} \hookrightarrow \overline{\mbb{Q}}_\ell$, then $Q_\lambda(\ell^{-n}) = L(\pi_\ell, 1/2)^{-1}$.
    \item If $p \notin S$ (so $\pi_p$ is unramified),  then for any prime $v$ dividing $p$ the representation $\rho_v \defeq \rho_{\pi}|_{E_v}$ is crystalline with jumps in the Hodge--Tate filtration occurring at
    \[
    \left\{ \begin{array}{ccc} c_i + 2n - i & i=1, \dots, 2n & \text{ if } v \text{ is the place fixed by the embedding } \overline{\mbb{Q}} \hookrightarrow \Qpb \\ -c_{2n+1 - i} + 2n - i & i=1, \dots, 2n & \text{ otherwise } \end{array} \right.
    \]
    Furthermore, via the fixed isomorphism $\mbb{C} \cong \Qpb$, the characteristic polynomial of the crystalline Frobenius $\varphi$ on $\dcris{\rho_v}$ is equal to the characteristic polynomial of $\varphi_{\Pi_{0,v} \otimes |\cdot |_v^{1/2 - n}}(\opn{Frob}_v^{-1})$, where
    \[
    \varphi_{\Pi_{0,v} \otimes |\cdot |_v^{1/2 - n}} \colon W_{E_v} \to \opn{GL}_{2n}(\mbb{C})
    \]
    is the local Langlands parameter attached to the representation $\Pi_{0, v} \otimes |\cdot |_v^{1/2 - n}$.
\end{enumerate}
\end{theorem}
\begin{proof}
See Theorem 3.2.3 in \cite{ChenevierHarris}. The third part follows from the explicit recipe of (1.5) in \emph{op.cit.}. Note that the representation $\rho_{\pi}$, which a priori is associated with $\Pi$, only depends on $\pi$. This is because the Galois representation constructed by Chenevier--Harris is semisimple (so is determined by its behaviour at all but finitely many primes, for any finite set of primes), and $\Pi$ is a weak-base change of $\pi$ (so its local Langlands parameters are determined by $\pi$ for all but finitely many primes).
\end{proof}

In \cite[Corollary 8.4.6]{Morel}, Morel shows that, up to multiplicity, this representation appears in the intersection cohomology of the Baily--Borel compactification of $\opn{Sh}_{\mbf{G}}$. However, to be able to construct an Euler system, we would like this representation to occur in the (compactly-supported) cohomology of the open Shimura variety. In fact to obtain the correct Euler factor, we work with the dual of the representation and assume that it appears as a quotient of the (usual) \'{e}tale cohomology. 

We would also like $ \rho_{\pi} $ to be absolutely irreducible (which is expected because we have assumed that $\Pi$ is cuspidal) -- however this is currently not known in general. We summarise this in the following assumption:

\begin{assumption} \label{KeyAssumption}
Let $\pi_0$ be a cuspidal automorphic representation of $\mbf{G}_0(\mbb{A})$ such that $\pi_{0, \infty}$ lies in the discrete series, and let $\pi$ be a lift to $\mbf{G}(\mbb{A})$. Suppose that $\pi$ is cohomological with respect to the representation $V$, and suppose that any weak base-change of $\pi$ is cuspidal. Then we assume that: 
\begin{itemize}
    \item $\rho_{\pi}$ is absolutely irreducible. 
    \item There exists a $G_E \times \mbf{G}(\mbb{A}_f)$-equivariant surjective map
    \[
    \opn{pr}_{\pi^\vee} \colon \opn{H}^{2n-1}_\et \left(\opn{Sh}_{\mbf{G}, \overline{\mbb{Q}}},\mathscr{V}^* (n) \right) \otimes \overline {\QQ}_{p} \twoheadrightarrow \pi_f^\vee \otimes W_{\pi}^*(1-n).
    \]
    (Here, we are identifying $\Qpb$ with $\mbb{C}$ via the fixed isomorphism in section \ref{NotationSection}). 
\end{itemize}
\end{assumption}

\begin{remark}
If $\pi$ is tempered, non-endoscopic and the global $L$-packet containing $\pi$ is stable, then we expect that the above assumptions hold true for the following reasons. Since $\pi$ is tempered the cohomology should vanish outside the middle degree and $\pi$ is not a CAP representation.\footnote{CAP = cuspidal but associated to a parabolic.} In particular, the latter should imply that $\pi_f$ does not contribute to the cohomology of the boundary strata of the Baily--Borel compactification of $\opn{Sh}_{\mbf{G}}$. Since the $L$-packet containing $\pi$ is stable, this will imply that any weak base-change of $\pi$ is cuspidal, and finally, since $\pi$ is non-endoscopic then (at least for $p$ large enough) the representation $\rho_{\pi}$ should be absolutely irreducible of dimension $2n$.
\end{remark}   

\begin{remark}
If we assume that $\Pi_{0, v}$ is supercuspidal at some non-archimedean place $v$ of $E$, then the Galois representation $\rho_{\pi}$ will be absolutely irreducible for local reasons. In fact, by choosing $p$ large enough, we can even ensure that $\rho_{\pi}$ is residually absolutely irreducible (see \cite[Proposition 4.2.3]{LTXZZDeformations21}).  
\end{remark}


\section{Horizontal norm relations} \label{TheHorizontalRelations}

This section comprises of the local calculations needed to prove the ``tame norm relations'' for our anti-cyclotomic Euler system at split primes. In \cite{LSZ17}, the key input for proving the norm relations is the existence of a Bessel model and a zeta integral that computes the local $L$-factor of the automorphic representation. The analogous notion in our case will be that of a Shalika model. This section is entirely local and independent of the rest of the paper -- in particular, we can relax the assumption that $n$ is odd.

For this section only, we consider 
\begin{align*}
 & G =   \mathrm{GL}_{1} \times \mathrm{GL}_{2n}    &&  H =    \opn{GL}_1 \times \opn{GL}_n \times \opn{GL}_n  \\    
&G_0 = \opn{GL}_{2n} &&  H_{0} = \opn{GL}_n \times \opn{GL}_n 
\end{align*}
which we view as algebraic groups over $\mbb{Z}_\ell$ and come with diagonal embeddings $ H \hookrightarrow  G $ and $  H_{0} \hookrightarrow G_{0} $. We fix an irreducible, smooth, admissible, unramified and \emph{generic} representation $ \sigma $ of $ G_{0}(\mbb{Q}_\ell) $. We also assume that $\sigma$ is unitary up to a (potentially non-unitary) twist. These conditions are satisfied for any unramified local component (at a split prime above $\ell$) of a cuspidal automorphic representation of $\opn{GL}_{2n}(\mbb{A}_E)$, which is the setting we are primarily interested in. 

Let $ \pi = \mbf{1} \boxtimes \sigma$ denote the extension of $\sigma$ to $ G(\mbb{Q}_\ell) $. For our applications later on, it will be necessary to phrase everything in terms of the extended group $\tilde{G} \defeq G \times \opn{GL}_1$. The group $H$ lifts to a subgroup of $\tilde{G}$ via the map 
\[ 
(\alpha, h_{1} , h_{2} ) \mapsto \left(\alpha,  \tbyt{h_1}{}{}{h_2}, \,  \nu(h_{1},h_{2}) \right )
\]
where $\nu(h_{1},h_{2}) = \opn{det}h_2/\opn{det}h_1$. In particular, for any character $\chi \colon \mbb{Q}_\ell^{\times} \to \mbb{C}^\times$, the product $\pi \boxtimes \chi^{-1}$ is a representation of $\tilde{G}(\mbb{Q}_\ell)$ where the action of the subgroup $H(\mbb{Q}_\ell)$ is now twisted by $\chi^{-1} \circ \nu$. 

\begin{convention*}
By a \emph{spherical vector} in $\sigma$ (resp. $\pi$, resp. $\pi \boxtimes \chi^{-1}$),    we mean a non-zero element which is fixed by $G_0(\mbb{Z}_\ell)$ (resp. $G(\mbb{Z}_\ell)$, resp. $\tilde{G}(\mbb{Z}_\ell)$). 
\end{convention*}

Using the theory of Shalika models, we will prove the following theorem:

\begin{theorem} \label{MainTheorem}
Let $\pi = \mbf{1} \boxtimes \sigma$ be as above and suppose that $\opn{Hom}_{H_0}\left( \sigma, \mbb{C} \right) \neq 0$. Then there exists a locally constant compactly-supported function $\phi \colon \gt(\mbb{Q}_\ell) \to \mbb{Z}$ satisfying the following:
\begin{enumerate}
    \item For any choice of finite-order unramified character $\chi \colon \QQ_{\ell}^{\times } \to \mathbb{C}^{\times}$, any element $\ide{z} \in \opn{Hom}_H(\pi \boxtimes \chi^{-1}, \mbb{C})$ and any spherical vector $\varphi \in \pi \boxtimes \chi^{-1}$, we have the relation
    \[
    \ide{z}(\phi \cdot \varphi) = \frac{\ell^{n^2}}{\ell-1} L(\sigma \otimes \chi, 1/2)^{-1} \ide{z}(\varphi)
    \]
    where $L(\sigma \otimes \chi, s)$ is the standard local $L$-factor attached to $\sigma \otimes \chi$. 
    \item The function $\phi$ is an (integral) linear combination of indicator functions $\phi = \sum_{r=0}^n b_r \opn{ch}\left( (g_r, 1) K \right)$ where 
    \[
    g_r = 1 \times \tbyt{1}{\ell^{-1} X_r}{}{1}
    \]
    and $X_r = \opn{diag}(1, \dots, 1, 0, \dots, 0)$ is the diagonal matrix having $1$ as its first $r$ entries (we set $X_0 = 0$) and $K = \gt(\mbb{Z}_\ell)$. 
    \item If we set $K_{(1)} = G(\mbb{Z}_\ell) \times (1 + \ell \mbb{Z}_\ell)$ and 
    \[
    V_{1, r} \defeq (g_r, 1) \cdot K_{(1)} \cdot (g_r, 1)^{-1} \cap H(\mbb{Q}_\ell)
    \]
    then $V_{1, r}$ is contained in $H(\mbb{Z}_\ell)$ and the coefficients in (2) satisfy 
    \[
    (\ell - 1) \cdot b_r \cdot  [H(\mbb{Z}_\ell) : V_{1, r}]^{-1} \in \mbb{Z}.
    \]
\end{enumerate}
\end{theorem}

\subsection{Shalika models}

Let $Q$ denote the Siegel parabolic subgroup of $G$ given by 
\[
Q = \opn{GL}_1 \times \tbyt{*}{*}{}{*} 
\]
and let $Q_0$ denote its projection to $G_0$. We denote the unipotent radical of $Q$ (resp. $Q_0$) by $N$ (resp. $N_0$) and define the Shalika subgroup $S_0 \subset G_0$ to be the subgroup of all matrices of the form
\[
\tbyt{h}{hX}{}{h}
\]
where $h \in \opn{GL}_n$ and $X \in M_n$ (the space of $n$-by-$n$ matrices). Let $\psi_0 : \QQ_{\ell} \to \mathbb{C}^{\times}$ denote the standard additive character on $\mbb{Q}_\ell$ which is trivial on $\mbb{Z}_\ell$ and define $\psi$ to be the character $S_0(\mbb{Q}_\ell) \to \mbb{C}^{\times}$ given by 
\[
\psi \tbyt{h}{hX}{}{h} = \psi_0(\opn{tr}X).
\]

\begin{definition}
A \emph{Shalika model} for $\sigma$ is a non-zero (and hence injective)  $G_0(\mbb{Q}_\ell)$-equivariant homomorphism
\[
\invs{S}_\bullet \colon \sigma \to \opn{Ind}_{S_0(\mbb{Q}_\ell)}^{G_0(\mbb{Q}_\ell)} (\psi). 
\]   
\end{definition}
Here $ \mathrm{Ind}_{S_{0}(\QQ_{\ell})}^{G_{0}(\QQ_{\ell})}(\psi) $ is the representation induced from $ \psi $, i.e. the set of all locally constant functions $ f \colon G_{0}(\QQ_{\ell})  \to \CC $ such that $ f(x g) = \psi(x) f(g) $ for all $ x \in S_{0}(\QQ_{\ell}) $, $ g \in G_{0}(\QQ_{\ell}) $,  considered as a smooth left representation of $ G_{0}(\QQ_{\ell}) $ under right-translation of the argument. Shalika models are unique (up to scalar) when they exist (see \cite{ULP}) and their existence is closely related to whether $\sigma$ is $H_0(\mbb{Q}_\ell)$-distinguished or not. Indeed, we have the following:

\begin{proposition}  \label{DistImpliesShalika}
If $\opn{Hom}_{H_0}(\sigma, \mbb{C}) \neq 0$, then $\sigma$ admits a Shalika model.
\end{proposition}
\begin{proof}
Recall that $\sigma$ was assumed to be generic. The result then follows from \cite[Corollary 1.1]{Matringe}.
\end{proof}

From now on, we assume that $\sigma$ admits a Shalika model. This will allow us to define certain \emph{zeta integrals} that compute the local $L$-factors of $\sigma$ (and of all its twists by unramified characters). We fix $dx$ to be the unique Haar measure on $\GL_{n}(\QQ_{\ell})$ normalised so that the volume of $\GL_{n}(\ZZ_{\ell})$ is equal to $1$. 

\begin{proposition} \label{DefinitionOfZetaIntegral}
Let $\chi \colon \mbb{Q}_\ell^{\times} \to \mbb{C}^\times$ be any unramified character and $ \varphi \in \sigma $. Then for $\opn{Re}(s)$ large enough, the integral
\[
Z(\varphi, \chi; s) \defeq \int_{\GL_{n}(\QQ_{\ell})} \invs{S}_{\varphi}\tbyt{x}{}{}{1} \chi(\opn{det}x) |\opn{det}x|^{s-1/2} dx
\]
converges absolutely. Furthermore, the function $ s \mapsto Z(\varphi, \chi; s)$ can be analytically continued to a meromorphic function on all of $\mbb{C}$, and  there exists a (unique) holomorphic function $s \mapsto \tilde{Z}(\varphi, \chi; s)$ such that
\[
Z(\varphi, \chi; s) = \tilde{Z}(\varphi, \chi; s) \cdot L(\sigma \otimes \chi, s)
\]
for all $s \in \mbb{C}$. Moreover, there exists a spherical vector $\varphi_0$ such that $\invs{S}_{\varphi_0}(1) = 1$ and   $\tilde{Z}(\varphi_0, \chi; s) = 1 $ for all $ s \in  \CC  $.    
\end{proposition}
\begin{proof}
See \cite[Proposition 3.3]{DJR18} (which is a slightly more general restatement of \cite[Propositions 3.1, 3.2]{FriedbergJacquet1993}) for the first two claims.  Note that, since we are assuming $\sigma$ is unitary up to a twist and that it admits a Shalika model, it is in fact unitary. Furthermore, we are assuming that $\sigma$ is generic, hence it admits a Whittaker model with respect to the character $\psi_0$ (for representations of $G_0(\mbb{Q}_\ell) = \opn{GL}_{2n}(\mbb{Q}_{\ell})$,   we are free to choose any non-trivial character for the Whittaker model). The  third   claim about the  existence of a spherical vector $\varphi_0 \in \sigma$ such that $\mathcal{S}_{\varphi_0}(1) = 1$ follows from \cite{GrobnerMatringe}, and the fact that $\tilde{Z}(\varphi_0, \chi; s) = 1$ follows from \cite[Proposition 3.2]{FriedbergJacquet1993}.
\end{proof}

The holomorphy factor between the zeta integral and the $L$-factor above provides us with a linear model for $\sigma$ by fixing $ s = 1/2 $ and varying the input $ \varphi \in \sigma  $.    

\begin{proposition} \label{BasisElement} 
Let $\chi \colon \QQ_{\ell}^{\times} \to \mathbb{C}^{\times}$ be a finite-order unramified character. Then the $ \CC$-linear map given by $\varphi \mapsto \tilde{Z}(\varphi, \chi;  1/2)$ constitutes a basis of 
\begin{equation} \label{ChiHomSpace}
\Hom_{H}(\pi \boxtimes \chi^{-1} ,  \mathbb{C}). 
\end{equation}
\end{proposition} 
\begin{proof} 
Since the underlying spaces for $\sigma$, $\pi$ and $\pi \boxtimes \chi ^{-1}$ are the same, the linear map in the statement of the proposition is well-defined. Fix $ h = (h_{1}, h_{2})  \in H $ and  $ \varphi \in \sigma $. It  is  straightforward to verify that  
\[
Z(h \cdot \varphi, \chi; s) = \chi(\nu(h))\left|\nu(h)\right|^{s-1/2} Z(\varphi, \chi; s)
\]
for any $ s \in \CC $ with $ \mathrm{Re}(s) $
sufficiently large  (so that the two zeta integrals on each side are convergent\footnote{Here, the bound on  $ \mathrm{Re}(s) $   may depend on our fixed $h $ and $ \varphi $.}).  Dividing by $L(\sigma \otimes \chi, s)$ on both sides gives us an analogous equality with $ Z $ replaced by $ \tilde{Z} $.   Since $ s \mapsto  \tilde{Z}(\varphi, \chi; s) $ is holomorphic, this identity involving $ \tilde{Z} $ holds for all $ s \in \CC $. Evaluating at $ s =   1/2 $  then  gives the desired $ H $-equivariance property with respect to  arbitrary $ h \in H $, $ \varphi \in \sigma $. The resulting map $\varphi  \mapsto  \tilde{Z}(\varphi, \chi; 1/2)$ is  non-zero  as   $\tilde{Z}(\varphi_0, \chi; 1/2) = 1 $ by the  last  part of Proposition   \ref{DefinitionOfZetaIntegral}, and therefore constitutes a non-zero element of (\ref{ChiHomSpace}).   Finally, note that
\[
\Hom_{H}(\pi\boxtimes \chi^{-1}, \mathbb{C}) = \Hom_{H_0}(\sigma, \chi \circ \nu) 
\]
and $ \chi  \circ \nu $ is the product of two characters $ \gamma _{0} $, $ \gamma_{1} $, one for each factor of $H_{0}(\mbb{Q}_\ell)$, each given by the composition of $ \det $ with $ \chi ^{-1}$, $ \chi$ respectively. The result now follows from \cite{Chen2019}; since $\chi$ has finite-order, the character $ \gamma _{0} \boxtimes \gamma_{1} $ is always ``good'' and therefore the Hom space is at most one-dimensional (see the terminology preceding Theorem C and the discussion that follows in \emph{op.cit.}). 
\end{proof}    

\subsection{Recap on lattices and flag varieties}

In this subsection,    we introduce the objects that will be used throughout the rest of the section. Let $M$ denote the submonoid of all matrices in $\opn{GL}_n(\mbb{Q}_\ell)$ with entries in $\mbb{Z}_\ell$, and let $M^{\geq r}$ be the subset of all matrices in $M$ whose determinant has $\ell$-adic valuation $\geq r$. For brevity, we set $K_n = \opn{GL}_n(\mbb{Z}_\ell)$.

\begin{definition} \label{DefOfLattice}
Let $\invs{L}_n$ denote the set of all lattices inside $\mbb{Q}_{\ell}^n$, i.e. all $\mbb{Z}_\ell$-submodules which are free of rank $n$. The set $\invs{L}_n$ can be identified with the quotient space $K_n \backslash \opn{GL}_n(\mbb{Q}_\ell)$ via the map
\begin{align*}
    K_n \backslash \opn{GL}_n(\mbb{Q}_\ell) & \xrightarrow{\sim} \invs{L}_n \\
    K_n \cdot g & \mapsto L[g].
\end{align*}
where $L[g]$ is the lattice spanned by the rows of $g$ (we view elements of $\mbb{Q}_\ell^n$ as row vectors). For a lattice $L \in \invs{L}_n$, let 
\[
M_L = \{ g \in \opn{GL}_n(\mbb{Q}_\ell) : L[g] \subset L \}
\]
denote the subset of all matrices in $\opn{GL}_n(\mbb{Q}_\ell)$ whose rows span a sublattice inside $L$. This contains the same amount of information as the original lattice $L$, but it will be useful to consider the subset $M_L$ later on. In particular, $M_{L[g]}$ is simply equal to the translate $Mg$.
\end{definition}

Let $\invs{L}^{\circ}_n \subset \invs{L}_n$ denote the finite subset of all lattices satisfying 
\[
\ell \mbb{Z}_{\ell}^n \subset L \subsetneq \mbb{Z}_{\ell}^n
\]
which is closed under intersection. This can be of course identified with all proper subspaces of $\mbb{F}_\ell^n$, a perspective we will frequently take. 

\begin{definition} \label{DefOfChain}
A chain $C$ of lattices in $\invs{L}_n$ of length $r$ is a sequence
\[
L_r \subsetneq L_{r-1} \subsetneq \cdots \subsetneq L_0.
\]
We let $C_{\mathrm{min}} = L_r$ denote the minimal element of the chain $C$.
\end{definition}

Let $B_n$ and $T_n$ denote the standard Borel subgroup and torus in $\opn{GL}_n$ respectively, and let $X_*(T_n)$ denote the cocharacter group $\opn{Hom}(\mbb{G}_m, T_n)$. Fixing this choice of Borel subgroup and torus determines a subset $X_*(T_n)^+ \subset X_*(T_n)$ of dominant cocharacters. Explicitly, every dominant cocharacter is of the form
\begin{align*}
\mbb{G}_m & \longrightarrow T_n \\
z & \mapsto \opn{diag}(z^{a_1}, \dots, z^{a_n})
\end{align*}
where $a_i$ are integers satisfying $a_1 \geq \cdots \geq a_n$ and therefore,  we will often refer to such a cocharacter by the tuple $(a_1, \dots, a_n)$. Recall that one has the Cartan decomposition of $\opn{GL}_n(\mbb{Q}_\ell)$ 
\[
\opn{GL}_n(\mbb{Q}_\ell) = \bigsqcup_{\chi} K_n \chi(\ell) K_n
\]
where the disjoint union is over all dominant cocharacters as above.

\begin{definition}
For a lattice $L = L[g]$ in $\invs{L}_n$, we let $\opn{inv}(L)$ denote its relative position with respect to the standard lattice $\mbb{Z}_{\ell}^n$, i.e. $\opn{inv}(L) = \chi(\ell)$ where $\chi$ is the unique dominant cocharacter as above, such that $K_n g K_n = K_n \chi(\ell) K_n$. 
\end{definition}

We introduce the following matrices which will play a key role in the rest of this section.

\begin{definition}
For $1 \leq m \leq n$, let 
\[
t_m = \opn{diag}(\ell, \dots, \ell, 1, \dots, 1) \in \opn{GL}_n(\mbb{Q}_\ell)
\]
denote the diagonal matrix whose first $m$ entries are equal to $\ell$.
\end{definition}

\begin{remark} \label{RemarkOnBoolean}
The space $\invs{L}^{\circ}_n$ can be viewed as the subset of $\invs{L}_n$ generated by all intersections of lattices of relative position $t_1$, or equivalently, as the subset of all lattices with relative position equal to $t_m$, for some $1 \leq m \leq n$.
\end{remark}

We will also need to talk about flag varieties of arbitrary signature, so we recall the definition for the convenience of the reader.

\begin{definition} \label{DefOfFlagVar}
Let $d_0 < d_1 < \cdots < d_k$ be non-negative integers satisfying $d_0 = 0$ and $d_k = m$. The flag variety $\opn{Fl}(d_1, \dots, d_k; \mbb{F}_\ell)$ of signature $(d_1, \dots, d_k)$ is the space of all flags
\[
\{0\} = V_0 \subset V_1 \subset \cdots \subset V_k = \mbb{F}_\ell^m
\]
where $V_i$ is a subspace of dimension $d_i$. 
\end{definition}

We will frequently use the fact that $\opn{Fl}(d_1, \dots, d_k; \mbb{F}_\ell)$ is a finite set of order equal to 
\[
\genfrac[]{0pt}{0}{m}{n_1 \cdots n_k}_\ell \defeq \frac{[m]_\ell !}{[n_1]_\ell ! \cdots [n_k]_\ell !}
\]
where $n_i = d_i - d_{i-1}$, with notation as in \S 
   \ref{NotationSection}.

\subsection{Lattice counting}

To study the zeta integrals introduced in Proposition \ref{DefinitionOfZetaIntegral} and their translates by certain Hecke operators, it will be necessary to formulate an inclusion-exclusion principle for the monoid $M$. Essentially, we will use this principle to count the elements in $M^{\geq 1}$ with respect to a suitable measure arising from the Shalika model for $\sigma$.

Recall that we have fixed a left-invariant Haar measure $dx$ on $\opn{GL}_n(\mbb{Q}_\ell)$ (which must also be right-invariant). Let $f \colon \opn{GL}_n(\mbb{Q}_\ell) \to \mbb{C}$ be any locally constant function such that $\int_{\opn{GL}_n(\mbb{Q}_\ell)} f(x) dx < + \infty$. Then the assignment 
\[
X \mapsto \mu_f(X) \defeq \int_X f(x) dx
\]
defines a (finite) measure on the $\sigma$-algebra generated by all compact open subsets of $\opn{GL}_n(\mbb{Q}_\ell)$.

\begin{proposition}[Inclusion-Exclusion] \label{LatticeFcn}
Let $f \colon \opn{GL}_n(\mbb{Q}_\ell) \to \mbb{C}$ be any locally constant function as above. Then
\begin{equation} \label{IncExc}
    \mu_f\left( M^{\geq 1} \right) = \mu_f\left( \bigcup_{L \in \invs{L}_n^{\circ}} M_L \right) = \sum_{C \in \ide{C}} (-1)^{l(C)} \mu_f \left( M_{C_{\mathrm{min}}} \right)
\end{equation}
where $\ide{C}$ denotes the set of chains in $\invs{L}_n^{\circ}$ and $l(C)$ is the length of the chain $C$ (Definition \ref{DefOfChain}). 
\end{proposition}
\begin{proof}
We follow \cite{Narushima}. For any non-empty finite subset $ S $ of $ \mathcal{L}^{\circ}_{n} $, we define a subset $ \overline{S}  \subset  \mathcal{L}^{\circ}_{n} $ as follows: set $ S_{0} = S $, and for $ i \geq 0 $, set $ S_{i+1} = S_{i}  \cup \left \{ L \cap L' | L , L ' \in  S_{i} \right \} $. Then we take $ \overline{S} = \bigcup_{i} S_{i} $. The collection of all such subsets $ \overline{S} $ obtained from any finite subset $ S $ is denoted by $ \mathcal{O}(\mathcal{L}_{n}^{\circ}) $. Recall that for any $ L  \in   \mathcal{L} ^ { \circ }_{n} $, we have set    
\[
M_{L} =  \left \{  g \in \opn{GL}_n(\mbb{Q}_\ell) \,   |    \,    L[g] \subset  L  \right \} \subset M^{\geq 1}. 
\]
Then, $ M ^ { \geq 1 } =  \bigcup _{ L  \in   \mathcal{L } ^ { \circ } _ { n }    }   M _{  L  }$. Indeed, for $g \in M^{\geq 1}$ one has $L[g] \subsetneq \mbb{Z}_\ell^n$, which implies that $L[g] + \ell \mbb{Z}_\ell^n \subsetneq \mbb{Z}_\ell^n$. Therefore, $g \in M_L$ with $L = L[g] + \ell \mbb{Z}_\ell^n$.

By the inclusion-exclusion principle for finite measures, we have 
\begin{equation}   \label{incl-excl}   
\mu_{f} \left  (M  ^ { \geq  1 }  \right )  =    \sum_{\emptyset \neq  S \subset  \mathcal{L} ^{\circ}_{n} }  ( - 1 ) ^ {  |S| - 1 }   \mu_{f}  \left (   \bigcap _{ L  \in  S }  M_{L}    \right  ) .
\end{equation} 
Fix an element $ Y \in  \mathcal{O} (\mathcal{L}_{n} ^ { \circ } ) $ and let 
\[
Y_{0} \defeq Y - \left \{ L \cap L' \, | \, L, L' \in Y,  L \not\subset L', L' \not\subset L \right \} 
\]
i.e. the subset obtained by deleting the intersection of incomparable pairs of lattices.  Set $ s = |Y_{0}| $, $ t = |Y| $, $ l = t - s  $. Clearly, for any non-empty $ S \subset \mathcal{L}^{\circ}_{n} $, we have $ \overline{S} = Y $ if and only if $ Y_{0} \subset S \subset Y $.  Therefore
\begin{align*} \sum _ { S \subset \mathcal{L}^{\circ}_{n}, \, \overline { S  }    = Y }(-1)^{|S|}  &  =  (-1) ^{s}  \sum_{ k = 0 } ^ { l } \binom{l}{k} (-1)^{k}  \\   
& = \begin{cases} (-1) ^{ s }  &   \text{ if } Y = Y_{0} \\
0  &  \text{ otherwise }
\end{cases} 
\end{align*} 
Moreover, $ Y = Y_{0} $ if and only if elements of $ Y $ form a chain $ C \in \ide{C} $. Using the fact that $ \bigcap_{L \in C} M_{L} = M_{C_{\mathrm{min}}}$, the right hand side of (\ref{incl-excl}) is equal to  
\[
\sum_{ C \in \ide{C} }  \left (  \sum _ {S: \, \overline{S} = C } (-1) ^{ |S|-1} \right )   \mu_{f}  (  M_{C_{\mathrm{min}}} )  =   \sum_{C \in \ide{C}} (-1) ^ { l(C) }  \mu_{f}(M_{C_{\mathrm{min}}} )     
\]
and the claim follows. 
\end{proof}

We obtain the following corollary.

\begin{corollary} \label{LambdaCor}
Keeping the same notation as in Proposition \ref{LatticeFcn}, suppose that $f$ is invariant under right-translation of the argument by $K_n$. For an integer $1 \leq m \leq n$ let
\[
\lambda_m  = \# \{L \in \invs{L}_n^{\circ} : \opn{inv}(L) = t_m \} \cdot \sum_{j=0}^{m-1} (-1)^j \cdot \#\left\{ \begin{array}{c} \text{ chains in } \invs{L}_n^{\circ} \text{ of length } j \\ \text{ with minimal element } L[t_m] \end{array} \right\}.
\]
Then
\begin{enumerate}
    \item $\mu_f(M^{\geq 1}) = \sum_{m=1}^n \lambda_m \mu_f\left(Mt_m\right)$.
    \item We have 
    \[
    \lambda_m \equiv \genfrac(){0pt}{0}{n}{m} (-1)^{m+1}
    \]
    modulo $\ell - 1$, which implies that $\sum_{m=1}^n \lambda_m \equiv 1$ modulo $\ell - 1$. 
\end{enumerate}
\end{corollary}
\begin{proof}
Since $f$ is invariant under right-translation by $K_n$, by Proposition \ref{LatticeFcn} and the fact that any element in $\invs{L}_n^{\circ}$ has relative position $t_m$ for some $m$ (Remark \ref{RemarkOnBoolean}), we have
\begin{eqnarray} 
\mu_f(M^{\geq 1}) & = & \sum_{m=1}^n \sum_{\opn{inv}(L) = t_m} \sum_{j=0}^{m-1} (-1)^j D_{j, m} \cdot \mu_f(M_L) \nonumber \\
 & = & \sum_{m=1}^n \lambda_m \cdot \mu_f(Mt_m) \nonumber 
\end{eqnarray}
where $D_{j, m}$ is the number of chains of length $j$ with minimal element $L[t_m]$. To calculate $D_{j, m}$ it is enough to study subspaces of $\mbb{F}_{\ell}^m$ (since $\invs{L}_n^{\circ}$ can be identified with all proper subspaces of $\mbb{F}_\ell^n$). 

Note that $D_{j, m}$ is equal to the sum over $(d_1, \dots, d_{j+1})$ of $\# \opn{Fl}(d_1, \dots, d_{j+1}; \mbb{F}_\ell)$, where the latter space is the flag variety of signature $(d_1,\dots, d_{j+1})$ associated to the vector space $\mbb{F}_\ell^m$ (Definition \ref{DefOfFlagVar}). The order of this flag variety is congruent modulo $\ell - 1$ to the multinomial coefficient 
\[
\genfrac(){0pt}{0}{m}{n_1 \; \cdots \; n_{j+1}}
\]
where $n_i = d_i - d_{i-1}$ (and $d_0 = 0$). This implies that $D_{j, m}$ is congruent to the sum of these coefficients over all $n_i$ such that $n_i \neq 0$ for all $i=1, \dots, j+1$. The total sum of all coefficients is $(j+1)^m$ by the multinomial theorem. We compute $D_{j, m}$ as follows:
\[
D_{j, m} \equiv (j+1)^m - \sum_{\text{some } n_i = 0} = (j+1)^m - \left( \sum_i \sum_{n_i = 0} - \sum_{i, j} \sum_{n_i = n_j = 0} + \cdots \pm \sum_{\text{all } n_i = 0} \right).
\]
This is equal to 
\[
D_{j, m} \equiv (j+1)^m - \left( \genfrac(){0pt}{0}{j+1}{1} j^m - \genfrac(){0pt}{0}{j+1}{2} (j-1)^m + \cdots \pm 0 \right) = \sum_{k = 0}^j (-1)^k \genfrac(){0pt}{0}{j+1}{k} (j+1 - k)^m .
\]
Consider the following variable change $(r, t) = (j-k, k)$. This implies that
\begin{eqnarray}
\sum_{j=0}^{m-1}(-1)^j D_{j, m} & \equiv & \sum_{r=0}^{m-1}(-1)^r (r+1)^m \sum_{t = 0}^{m-r-1} \bincoeff{r+t+1}{t} \nonumber \\
 & = & \sum_{r=0}^{m-1} (-1)^r (r+1)^m \bincoeff{m + 1}{m-r-1} \nonumber \\
 & = & - \sum_{s=0}^{m-1} (-1)^s \bincoeff{m+1}{s} P(s) \nonumber 
\end{eqnarray}
where $P(X) = (X - m)^m$ and for the last equality we have used the change of variables $s = m - r - 1$. Now we have $P(m) = 0$ and $P(m+1) = 1$, and since $P$ is a polynomial of degree $< m+1$, by the standard identities involving binomial coefficients we have
\[
-\sum_{s=0}^{m-1}(-1)^s \bincoeff{m+1}{s} P(s) = (-1)^{m+1} P(m+1) = (-1)^{m+1}.
\]
To complete the proof we note that $\# \{L \in \invs{L}_n^{\circ}: \opn{inv}(L) = t_m \}$ is equal to the number of points in the Grassmannian of $(n-m)$-dimensional subspaces of $\mbb{F}_\ell^n$. This quantity is congruent to 
\[
\bincoeff{n}{n-m} = \bincoeff{n}{m}
\]
modulo $\ell - 1$.
\end{proof}

\begin{example}
Let $n = 2$ and suppose that $f$ is right $K_n$-invariant. Then there are $\ell+2$ elements in $\invs{L}_2^{\circ}$; $\ell+1$ of them have relative position $t_1$ and we will denote them by $L_1, \dots, L_{\ell+1}$, the other has relative position $t_2$, and we will denote it by $S$. There are $\ell+2$ chains of length zero and $\ell+1$ chains of length $1$, namely they are all of the form $S \subset L_i$. Since $f$ is $K_n$-invariant we have 
\[
\mu_f(M^{\geq 1}) = (\ell+1) \mu_f(Mt_1) + \mu_f(Mt_2) - (\ell+1)\mu_f(Mt_2) = (\ell+1) \mu_f(Mt_1) - \ell \mu_f(Mt_2).
\]
\end{example}

\subsection{The $U_{m}$-operators} 

Let $\opn{M}_{m \times n}(\mbb{F}_\ell)$ denote the set of $(m \times n)$-matrices with coefficients in $\mbb{F}_\ell$ which can naturally be identified with the subgroup of $\opn{M}_{n \times n}(\mbb{F}_\ell)$ consisting of matrices with the last $n-m$ rows equal to zero. We also view $\opn{M}_{n \times n}(\mbb{F}_\ell) \subset \opn{M}_{n \times n}(\mbb{Z}_\ell)$ via the Teichm\"{u}ller lift. We emphasise that when talking about the ranks of these matrices, we always mean the rank as a linear map $\mbb{F}_\ell^n \to \mbb{F}_\ell^m$ (as the Teichm\"{u}ller lift doesn't preserve ranks).  

We introduce some operators. For $1 \leq m \leq n$ we let 
\[
U_m = \sum_{X} \tbyt{1}{X}{}{1} \tbyt{t_m}{}{}{1} \; \; \in \; \; \mbb{Z}[G_0(\mbb{Q}_\ell)]
\]
where the sum is over all matrices in $\opn{M}_{m \times n}(\mbb{F}_\ell)$ (the blocks in the above matrices are of size $n \times n$). These operators interact with the Shalika model as follows. 

\begin{lemma} \label{UmactiononShalika}
Let $\varphi_0 \in \sigma$ be a spherical vector. Then
\[
\invs{S}_{U_m \cdot \varphi_0}\tbyt{x}{}{}{1} = \left\{ \begin{array}{cc} 0 & \text{ if } x \notin M \\ \ell^{m n} \invs{S}_{\varphi_0} \tbyt{xt_m}{}{}{1} & \text{ if } x \in M \end{array} \right.
\]
Consequently, we have
\[
Z(\chi(\ell)^m \ell^{n(n-m)} U_m \cdot \varphi_0, \chi; s) = \ell^{n^2 + m(s-1/2)}\int_{Mt_m} S_{\varphi_0}\tbyt{x}{}{}{1} \chi(\opn{det}x)  |\opn{det}x|^{s-1/2} dx
\]
for $\opn{Re}(s)$ large enough so that the integrals converge. 
\end{lemma}
\begin{proof}
Note that the Iwasawa decomposition for $\opn{GL}_{2n}(\mbb{Q}_{\ell})$ implies that
\[
G_0(\mbb{Q}_\ell) = S_0(\mbb{Q}_\ell) \cdot \left\{ \tbyt{x}{}{}{1} : x \in \opn{GL}_n(\mbb{Q}_{\ell}) \right\} \cdot G_0(\mbb{Z}_{\ell})
. \]
As $ \mathcal{S}_{\bullet} $ is injective, we see from the decomposition  that any non-zero vector $\varphi_0 \in \sigma$ fixed by $G_0(\mbb{Z}_\ell)$ satisfies $\mathcal{S}_{\varphi_0}\tbyt{x}{}{}{1} \neq 0$ for some $x \in \opn{GL}_n(\mbb{Q}_{\ell})$. Let $M_n(\mbb{Z}_\ell)$ denote the set of $n \times n$-matrices with coefficients in $\mbb{Z}_\ell$. By right-translating the function $\mathcal{S}_{\varphi_0}(-)$ by $\tbyt{1}{X}{}{1}$ for any $X \in M_n(\mbb{Z}_\ell)$, we find that
\[
\mathcal{S}_{\varphi_0}\tbyt{x}{}{}{1} = \psi_0(\opn{tr}(x X)) \mathcal{S}_{\varphi_0}\tbyt{x}{}{}{1}, \quad \quad x \in \opn{GL}_n(\mbb{Q}_\ell) .
\]
Since for any $x \in \opn{GL}_n(\mbb{Q}_\ell) - \left( \opn{GL}_n(\mbb{Q}_\ell) \cap M_n(\mbb{Z}_\ell) \right)$ we can always find an $X \in M_n(\mbb{Z}_\ell)$ such that $\psi_0(\opn{tr}(x X)) \neq 1$, the support of the function   $x \mapsto \mathcal{S}_{\varphi_0}\tbyt{x}{}{}{1}$ is contained in
 $M = \opn{GL}_n(\mbb{Q}_\ell) \cap M_n(\mbb{Z}_\ell)$.  Furthermore, a direct calculation shows that
\[
\invs{S}_{U_m \cdot \varphi_0}\tbyt{x}{}{}{1} = \sum_{X} \psi_0(\opn{tr}(x \cdot X) ) \cdot  \invs{S}_{\varphi_0}\tbyt{xt_m}{}{}{1}
\]
which implies that the support of 
 $ x \mapsto \mathcal{S}_{U_{m}\cdot \varphi_{0}} \tbyt{x}{}{}{1} $ is  contained in $ Mt_m^{-1}$. Now for each  fixed $x \in Mt_m^{-1}$,  the function $X \mapsto \psi_0(\opn{tr}(x \cdot X) )$ defines an  additive  character $\opn{M}_{m \times n}(\mbb{F}_\ell) \to \mbb{C}^{\times}$ which is trivial if and only if $x \in M$. Therefore, the  first claim  follows from the usual character orthogonality relations. The second  claim   follows from the change of variable $x \mapsto xt_m^{-1}$. 
\end{proof}

\begin{remark}
The action of $U_m$ on the function $\invs{S}_{\varphi_0}(-)$ differs from the action of the more common $U_{\ell}$-operator associated with the parabolic of type $(m, n-m, n)$ by a power of $\ell$. However, it is more convenient for our purposes to use the above definition instead. Since we are eventually going to consider congruences modulo $\ell - 1$ this makes no difference to the end result. 
\end{remark}

We now apply the combinatorics of the previous section.
\begin{theorem} \label{ESRelation}
Let $\chi: \mbb{Q}_\ell^{\times} \to \mbb{C}^{\times}$ be a finite-order unramified character and let $\varphi_0$ be a spherical vector normalised such that $\invs{S}_{\varphi_0}(1) = 1$. Then there exist integers $\lambda_1, \dots, \lambda_n$ (independent of $\chi$) such that 
\[
\tilde{Z}\left((\sum_{m=1}^n \lambda_m \ell^{n(n-m)} \chi(\ell)^m U_m) \cdot \varphi_0, \chi; 1/2 \right) = \ell^{n^2} \left[\tilde{Z}(\varphi_0, \chi; 1/2) - L(\sigma \otimes \chi, 1/2)^{-1} \right].
\]
Furthermore, the sum $\ell^{n^2} - \sum_{m=1}^n \ell^{n(n-m)} \lambda_m$ is divisible by $\ell - 1$. 
\end{theorem}
\begin{proof}
For any $s \in \mbb{C}$, set $f_s(x) = \invs{S}_{\varphi_0}\tbyt{x}{}{}{1} \chi(\opn{det}x) |\opn{det}x|^{s-1/2}$. Then $f_s$ is right $K_n$-invariant and $\int_{\opn{GL}_n(\mbb{Q}_\ell)} f_s(x) dx < + \infty$ for $\opn{Re}(s) \gg 0$. By Corollary \ref{LambdaCor}, there exist integers $\lambda_1, \dots, \lambda_n$ which are independent of $\chi$ and $s$ and satisfy
\[
\mu_{f_s}(M^{\geq 1}) = \sum_{m=1}^n \lambda_m \mu_{f_s}(Mt_m), \quad \quad \opn{Re}(s) \gg 0.
\]
Therefore, by Lemma \ref{UmactiononShalika}, we have 
\begin{eqnarray} 
Z\left((\sum_{m=1}^n \lambda_m \ell^{n(n-m)} \chi(\ell)^m U_m) \cdot \varphi_0, \chi; s \right) & = & \sum_{m=1}^n \lambda_m \ell^{n^2 + m(s-1/2)}\mu_{f_s}(Mt_m) \nonumber \\
 & = & \ell^{n^2}\left( \sum_{m=1}^n \lambda_m (\ell^{m(s-1/2)} - 1) \mu_{f_s}(Mt_m) + \mu_{f_s}(M^{\geq 1}) \right) \nonumber
\end{eqnarray}
for $ \mathrm{Re}(s)  \gg    0 $. But  $ s \mapsto \mu_{f_s}(Mt_m)L(\sigma \otimes \chi, s)^{-1} =  \ell^{-n^2 -m (s-1/2)} \tilde{Z}( \chi(\ell)^{m} \ell^{n(n-m)} U_{m} \cdot \varphi_{0} , \chi ; s ) $ is holomorphic in $s$ and can therefore  be evaluated at $ s = 1/2 $. Therefore, multiplying the  displayed equality  above by $L(\sigma \otimes \chi, s)^{-1}$ and passing to the limit as $s \to 1/2$,  we find that  
   
\[
\tilde{Z}\left((\sum_{m=1}^n \lambda_m \ell^{n(n-m)} \chi(\ell)^m U_m) \cdot \varphi_0, \chi; 1/2 \right) = \ell^{n^2} \lim_{s \to 1/2} \left[ \mu_{f_s}(M^{\geq 1}) L(\sigma \otimes \chi, s)^{-1} \right].
\]
Now to conclude the proof, we note that 
\[
\mu_{f_s}(M^{\geq 1}) = \mu_{f_s}(M) - \mu_{f_s}(K_n) = Z(\varphi_0, \chi; s) - 1, \quad \quad \quad \opn{Re}(s) \gg 0.
\]
The second part of the theorem just follows from Corollary \ref{LambdaCor}. 
\end{proof}

\subsection{Congruences modulo $\ell - 1$} 
We now seek to find the element $\phi$ in Theorem \ref{MainTheorem}. For this, we must prove a congruence relation for the operators $U_m$. As before, we will express the result in terms of the extended group $\gt$, so we introduce some additional notation:

\begin{definition} 
For an element $g \in G_0(\mbb{Q}_\ell)$ let $\widetilde{g}$ denote the element
\[
(1, g, \opn{det}g^{-1}) \in \gt(\mbb{Q}_\ell).
\]
We also let $\widetilde{U}_m$ denote the corresponding operator in $\mbb{Z}[\gt(\mbb{Q}_\ell)]$.
\end{definition}

We note that we have a left action of $H(\mbb{Z}_\ell) = \opn{GL}_1(\mbb{Z}_\ell) \times \opn{GL}_n(\mbb{Z}_\ell) \times \opn{GL}_{n}(\mbb{Z}_\ell)$ on $\opn{M}_{n \times n}(\mbb{F}_\ell)$ given by
\[
(\alpha, A, B) \cdot X = AXB^{-1}
\]
which will be used in the proof of the following proposition. Let $c_{r, m}$ denote the number of rank $r$ matrices in $\opn{M}_{m \times n}(\mbb{F}_\ell)$; explicitly we have 
    \begin{equation} \label{NumberOfRankrMatrices}
    c_{r, m} = \prod_{j=0}^{r-1} \frac{(\ell^m - \ell^j)(\ell^n - \ell^j)}{\ell^r - \ell^j}
    \end{equation}
for $1 \leq r \leq m$, and $c_{0, m} = 1$. 

\begin{proposition} \label{IntegralProposition}
We retain the notation of Theorem \ref{MainTheorem}. Let $\varphi_0 \in \pi$ be a spherical vector. Let $\psi_m = \sum_{r=0}^m c_{r, m} \opn{ch}\left( (g_r, 1) K \right)$, which is a $\mbb{Z}$-valued right $K$-invariant smooth compactly supported function on $\gt(\mbb{Q}_\ell)$. Then
\begin{enumerate}
    \item For any unramified $\chi$ and $\ide{z} \in \opn{Hom}_H(\pi \boxtimes \chi^{-1}, \mbb{C})$
    \[
    \ide{z}( \widetilde{U}_m \cdot \varphi_0 ) = \ide{z}( \psi_m \cdot \varphi_0 ).
    \]
    \item For $r \geq 0$ we have 
    \[
    \opn{Stab}_{H(\mbb{Z}_\ell)} ( X_r ) = V_r \defeq (g_r, 1) K (g_r, 1)^{-1} \cap H(\mbb{Q}_\ell)
    \]
    which implies that $[H(\mbb{Z}_\ell) : V_r] = c_{r, n}$. Furthermore, we have $[V_r : V_{1, r}] = \ell - 1$ when $0 \leq r \leq n-1$ and $V_n = V_{1, n}$. 
\end{enumerate}
\end{proposition}
\begin{proof}
Let $g_X   :  = \tbyt{1}{X}{}{1}\tbyt{t_m}{}{}{1}$ be an element appearing the sum $U_m$. Since $\varphi_0$ is spherical, we have 
\[
g_X \cdot \varphi_0 = \tbyt{t_m A^{-1}}{}{}{B^{-1}} \tbyt{1}{t_m^{-1} (t_m A t_m^{-1})XB^{-1}}{}{1} \cdot \varphi_0
\]
for   arbitrary     $A, B \in \opn{GL}_n(\mbb{Z}_\ell)$. We can choose $A, B$ such that $t_m A t_m^{-1} \in \opn{GL}_n(\mbb{Z}_\ell)$ and $(t_m A t_m^{-1})XB^{-1} \equiv X_r$ modulo $\ell$, where $r$ is the rank of $X$ (viewed as a linear map $\mbb{F}_\ell^n \to \mbb{F}_\ell^m$).\footnote{Working modulo $\ell$, we can first apply column operations to ensure that $X$ has its last $(n-r)$ columns equal to zero. Then apply row operations to ensure that $X$ is only non-zero in its top-left $(r \times r)$ block (such an $A$ exists for this step because $r \leq m$). This block must be invertible, so we can apply column operations to put $X$ in the form $X_r$, as required.} Hence we have 
\[
\ide{z}(\widetilde{g}_X \cdot \varphi_0) = \ide{z}((g_r, 1) \cdot \varphi_0).
\]
Indeed, this follows from that fact that $(1, (t_mA^{-1}, B^{-1}), \ell^{-m} \opn{det}A \opn{det}B^{-1}) \in H(\mbb{Q}_\ell)$ and $\chi(\opn{det}A) = \chi(\opn{det}B) = 1$. This proves the first part of the proposition.

The  second    part of the proposition follows from an explicit calculation and the orbit--stabiliser theorem. Furthermore the induced map 
\[
\nu \colon V_r \to \mbb{Z}_\ell^{\times}
\]
given by $\nu(\alpha, A, B) = \opn{det}A^{-1}\opn{det}B$ is surjective when $r < n$ (for example, take $A=1$ and $B$ equal to a diagonal matrix with the first $n-1$ entries equal to $1$),     so we have $[V_r : V_{1, r}] = \ell - 1$.
\end{proof}

\subsection{Proof of Theorem \ref{MainTheorem}}

We begin with the following lemma:

\begin{lemma}
For $1 \leq r \leq n$ set 
\[
b_r' = \sum_{m=r}^n \ell^{n(n-m)} \lambda_m c_{r, m}
\]
where $\lambda_m$ and $c_{r, m}$ are the integers in Corollary \ref{LambdaCor} and Proposition \ref{IntegralProposition} respectively. Then $(\ell - 1) \cdot c_{r, n}$ divides $b_r'$ when $1 \leq r \leq n-1$, and $b_n' = \lambda_n c_{n, n}$.
\end{lemma}
\begin{proof}
We first show that $c_{r, n}$ divides $\lambda_m c_{r, m}$. Indeed the number of $\mbb{F}_\ell$-points of the Grassmannian of $m$-dimensional subspaces of $n$-dimensional space is equal to $\genfrac[]{0pt}{0}{n}{m}_\ell$, so by Corollary \ref{LambdaCor} the integer $\lambda_m$ is divisible by $\genfrac[]{0pt}{0}{n}{m}_\ell$. Using the formula in (\ref{NumberOfRankrMatrices}) we have 
\[
\genfrac[]{0pt}{0}{n}{m}_\ell \cdot \frac{c_{r, m}}{c_{r, n}} = \genfrac[]{0pt}{0}{n}{m}_\ell \cdot \genfrac[]{0pt}{0}{m}{r}_\ell \cdot \genfrac[]{0pt}{0}{n}{r}^{-1}_\ell = \genfrac[]{0pt}{0}{n-r}{n-m}_\ell .
\]
Furthermore, by Corollary \ref{LambdaCor} we have $\lambda_m \equiv \genfrac(){0pt}{0}{n}{m}(-1)^{m+1}$ modulo $\ell - 1$, which implies that
\[
\ell^{n(n-m)} \cdot \lambda_m \cdot c_{r, m} \cdot c_{r, n}^{-1} \equiv (-1)^{m+1} \genfrac(){0pt}{0}{n-r}{n-m} \text{ modulo } \ell -1 
\]
and thus by the binomial formula, $b_r' c_{r, n}^{-1} \equiv 0 \pmod{\ell - 1}$ in the case $r < n$. This completes the proof of the lemma. 
\end{proof}

We now construct the element $\phi$ claimed in Theorem \ref{MainTheorem}. We set
\[
b_0 = (\ell - 1)^{-1} \left( \ell^{n^2} - \sum_{m=1}^n \lambda_m \ell^{n(n-m)} \right)
\]
which is an integer by Theorem \ref{ESRelation}, and for $r \geq 1$ we set $b_r = -(\ell - 1)^{-1} b_r'$ which is an integer by the above lemma and the fact that $\ell - 1$ divides $c_{n, n}$. Then we set 
\[
\phi \defeq \sum_{r=0}^n b_r \opn{ch}\left( (g_r, 1) K \right) \; \; \in \; \;  \invs{H}\left( \gt(\mbb{Q}_\ell), \mbb{Z} \right)
\]
which naturally acts on the smooth representation $\pi \boxtimes \chi^{-1}$ of $\gt(\mbb{Q}_\ell)$. Let $\varphi_0$ be a spherical vector normalised so that $\mathcal{S}_{\varphi_0}(1) = 1$ and set $\ide{z}(-)   : = \tilde{Z}(-, \chi; 1/2)$. Then we have 
\begin{eqnarray}
\ide{z}( (\ell - 1) \phi \cdot \varphi_0 ) & = & \ide{z} \left( \left( \ell^{n^2} - \sum_{m=1}^n \lambda_m \ell^{n(n-m)} \right)\varphi_0 - \sum_{r = 1}^n b_r' (g_r, 1) \cdot \varphi_0 \right) \nonumber \\
 & = & \ide{z} \left( \left( \ell^{n^2} - \sum_{m=1}^n \lambda_m \ell^{n(n-m)} \right)\varphi_0 - \sum_{r=1}^n \sum_{m=r}^n \ell^{n(n-m)} \lambda_m c_{r, m} (g_r, 1) \cdot \varphi_0  \right) \nonumber \\
 & = & \ide{z} \left( \ell^{n^2} \varphi_0 - \sum_{m=1}^n \ell^{n(n-m)} \lambda_m \sum_{r=0}^m c_{r, m} (g_r, 1) \cdot \varphi_0 \right) \nonumber \\
 & = & \ide{z}\left( \ell^{n^2} \varphi_0 - \sum_{m=1}^n \ell^{n(n-m)} \lambda_m \cdot \widetilde{U}_m \cdot \varphi_0 \right)  \label{cstoUs} \\
 & = & \tilde{Z}\left( \ell^{n^2} \varphi_0 - (\sum_{m=1}^n \lambda_m \ell^{n(n-m)} \chi(\ell)^m U_m)\cdot \varphi_0, \chi; 1/2 \right)  \label{tildetonormal} \\
 & = & \ell^{n^2} L(\sigma \otimes \chi, 1/2)^{-1} \tilde{Z}(\varphi_0, \chi; 1/2) \label{toLvalue} 
\end{eqnarray}
where the equalities in (\ref{cstoUs}), (\ref{tildetonormal}) and (\ref{toLvalue}) follow from Proposition \ref{IntegralProposition}, the fact that $\widetilde{U}_m \cdot \varphi_0 = \chi(\ell)^m U_m \cdot \varphi_0$ and Theorem \ref{ESRelation} respectively. The latter equality also uses the fact that $\tilde{Z}(\varphi_0, \chi; 1/2) = 1$. 

Since $ \ide{z}(-) = \tilde{Z}(-, \chi; 1/2)$ is a basis of the hom-space $\opn{Hom}_H(\pi \boxtimes \chi^{-1}, \mbb{C})$ by  Proposition    \ref{BasisElement}, this completes the proof of the first part of Theorem \ref{MainTheorem}.   Furthermore, by construction $\phi$ is an integral linear combination of indicator functions $\opn{ch}((g_r, 1) K)$, so we have also proven the second part of the theorem.     Finally, by Proposition \ref{IntegralProposition} we have 
\[
(\ell - 1) \cdot b_r \cdot [H(\mbb{Z}_\ell) : V_{1, r}]^{-1} = \left\{ \begin{array}{cc} -b_r' \cdot c_{r, n}^{-1} \cdot (\ell - 1)^{-1} & \text{ if } 1 \leq r \leq n-1 \\ -b_n' \cdot c_{n, n}^{-1} & \text{ if } r=n \end{array} \right.
\]
which in both cases was shown to be an integer. For $r = 0$ the statement follows from the fact that $b_0$ is an integer. 


\section{Vertical norm relations at split primes} \label{VerticalNormRelations}

In this section we construct classes in the ``$p$-direction'' following the general method outlined in \cite{Loeffler19}. As a result we will obtain cohomologically trivial cycles which are compatible as one goes up in the \emph{anti-cyclotomic} tower. We will also need results of this section in the construction of the Euler system map. We do not require $n$ to be odd in this section. 

\subsection{Group theoretic set-up} \label{VerticalSetUp}

Let $p$ be an odd prime which splits in $E/\mbb{Q}$. Then we can consider the local versions of the groups in section \ref{TheGroups} at the prime $p$ -- we will denote these groups by unscripted letters. By Remark \ref{RemIdentifyGroups}, we have identifications 
\begin{eqnarray}
G = \Gb_{\QQ_{p}} & \cong & \opn{GL}_1 \times \opn{GL}_{2n} \nonumber \\
H = \Hb_{\QQ_{p}} & \cong & \opn{GL}_1 \times \opn{GL}_n \times \opn{GL}_n \nonumber \\
T = \Tb_{\QQ_{p}} & \cong & \opn{GL}_1 \nonumber 
\end{eqnarray}
Since all of these groups are unramified, they all admit a reductive model over $\mbb{Z}_p$ which we fix and, by abuse of notation, denote by the same letters. We recall some notation from section \ref{BranchingLawsAndReps}.

\begin{definition}
Let $Q$ denote the Siegel parabolic subgroup of $G$ given by
\[
Q = \opn{GL}_1 \times \tbyt{*}{*}{}{*} 
\]
and let $J \subset G(\mbb{Z}_p)$ denote the parahoric subgroup associated with $Q$ (i.e. all elements in $G(\mbb{Z}_p)$ which lie in $Q$ modulo $p$). We identify the Levi of $Q$ with the subgroup $\iota(H)$. 
\end{definition}

The automorphic representations that we will consider will be ``ordinary'' with respect to a certain Hecke operator associated with the parabolic subgroup $Q \times T \subset \widetilde{G}$, which we will now define. Let $\tau$ and $\tilde{\tau}$ be the elements 
\[
\tau =  1 \times \tbyt{p}{}{}{1} \in G(\mbb{Q}_p) \quad \quad \quad \tilde{\tau} =  1 \times \tbyt{p}{}{}{1}  \times 1 \in \widetilde{G}(\mbb{Q}_p). 
\]
Then for any sufficiently small compact open subgroup $\tilde{K} \subset \Gt(\mbb{A}_f)$,   we let $\invs{T}$ denote the Hecke operator associated with the double coset $[\tilde{K} \tilde{\tau}^{-1} \tilde{K}]$. More concretely, $\invs{T}$ is the operator associated with the following correspondence:
\[
\begin{tikzcd}
   & \opn{Sh}_{\Gt}(\tilde{K} \cap \tilde{\tau} \tilde{K} \tilde{\tau}^{-1}) \arrow[ld]  & \arrow[l, "{[\tilde{\tau}^{-1}]}", swap] \opn{Sh}_{\Gt}(\tilde{K} \cap \tilde{\tau}^{-1} \tilde{K} \tilde{\tau}) \arrow[rd] &   \\
\opn{Sh}_{\Gt}(\tilde{K}) &          &       & \opn{Sh}_{\Gt}(\tilde{K})
\end{tikzcd}
\]
Explicitly, if $p_1$ and $p_2$ denote the left-hand and right-hand projection maps respectively, then $\invs{T} = p_{2,*} \circ [\tilde{\tau}^{-1}]^* \circ p_1^*$ (see also section \ref{HeckeCorrs}). In all choices $\tilde{K}' \subset \tilde{K}$ that we will consider, this operator will be compatible with the pushforward morphism $(\opn{pr}_{\tilde{K}', \tilde{K}})_*$, where $\opn{pr}_{\tilde{K}', \tilde{K}} \colon \opn{Sh}_{\Gt}(\tilde{K}') \to \opn{Sh}_{\Gt}(\tilde{K})$ is the natural map; so we will often omit the level when referring to it (see \cite[\S 4.5]{Loeffler19}). Note that if $\tilde{K} = K \times C $ is decomposable, then $\invs{T}$ only acts on the cohomology of the first factor of $\opn{Sh}_{\Gt}(\tilde{K}) \cong \opn{Sh}_{\mbf{G}}(K) \times \opn{Sh}_{\mbf{T}}(C)$. We will denote the Hecke correspondence induced on $ \Sh_{\mbf{G}} ( K ) $ also by $ \mathcal{ T } $. 

\begin{remark} \label{RemAJAppendixRephrase}
We can rephrase this in the language of Appendix \ref{AbelJacobi} as follows. Let $K \subset \mbf{G}(\mbb{A}_f)$ be a sufficiently small compact open subgroup and, for this remark only, let $\invs{T}_{K}$ denote the Hecke operator $[K \tau^{-1} K]$ defined above, acting on either the cohomology group $\opn{H}^i_{\et}\left(\opn{Sh}_{\mbf{G}}(K), \mathscr{F}_{K} \right)$ or $\opn{H}^i_{\et}\left(\opn{Sh}_{\mbf{G}}(K)_{\overline{E}}, \mathscr{F}_{K} \right)$ for a lisse equivariant \'{e}tale $\Lambda$-sheaf $\mathscr{F}$ (with $\Lambda = \mbb{Z}_p$ or $\mbb{Q}_p$). To ease notation, we will set $K_{\tau} \defeq K \cap \tau^{-1} K \tau$. Then the endomorphism $\invs{T}_{K}$ is induced from the \'{e}tale correspondence 
\begin{equation} \label{EtCorForTK}
\left( \opn{Sh}_{\mbf{G}}(K_{\tau}), [\tau^{-1}]_{K_{\tau}, K}, \opn{pr}_{K_{\tau}, K}, \phi_{K} \right) \in \opn{\acute{E}tCor}\left(\opn{Sh}_{\mbf{G}}(K), \mathscr{F}_{K} \right)
\end{equation}
with morphism $\phi_{K} \colon [\tau^{-1}]_{K_{\tau}, K}^*\mathscr{F}_{K} \to \opn{pr}_{K_{\tau}, K}^* \mathscr{F}_{K}$ given by the composition 
\[
[\tau^{-1}]_{K_{\tau}, K}^*\mathscr{F}_{K} \to \mathscr{F}_{K_{\tau}} \xleftarrow{\sim} \opn{pr}_{K_{\tau}, K}^* \mathscr{F}_{K}
\]
where both maps arise from the equivariant structure on $\mathscr{F}$. Let $\mbb{N} = \{0, 1, 2, ... \}$ viewed as a discrete monoidal category. Then one obtains a monoidal functor $F_K \colon \mbb{N} \to \opn{\acute{E}tCor}\left(\opn{Sh}_{\mbf{G}}(K), \mathscr{F}_{K} \right)$ by defining $F_K(1)$ to be the correspondence in (\ref{EtCorForTK}). If we set $R = \Lambda[\mbb{N}]$ (which is identified with the polynomial algebra over $\Lambda$ in one variable, then one obtains $R$-module structures on $\opn{H}^i_{\et}\left(\opn{Sh}_{\mbf{G}}(K), \mathscr{F}_{K} \right)$ and $\opn{H}^i_{\et}\left(\opn{Sh}_{\mbf{G}}(K)_{\overline{E}}, \mathscr{F}_{K} \right)$.

Now suppose that $\mathscr{F}$ is an equivariant sheaf associated with a continuous representation of $\mbf{G}(\mbb{Q}_p)$ (or of an appropriate open submonoid of $\mbf{G}(\mbb{Q}_p)$ containing $\tau^{-1}$). Let $g \in \mbf{G}(\mbb{A}_f^p)$ which we view as element of $\mbf{G}(\mbb{A}_f)$ in the natural way. Suppose that $L \subset \mbf{G}(\mbb{A}_f)$ is a sufficiently small compact open subgroup such that $g^{-1}Lg \subset K$, and assume that $K = K^pK_p$, $L = L^pL_p$ with $L_p = K_p$. Then we obtain a commutative diagram of correspondences
\[
\begin{tikzcd}
\opn{Sh}_{\mbf{G}}(L) \arrow[d, "{[g]}"'] & \opn{Sh}_{\mbf{G}}(L_{\tau}) \arrow[d, "{[g]}"] \arrow[l, "{[\tau^{-1}]}", swap] \arrow[r, "\opn{pr}"] & \opn{Sh}_{\mbf{G}}(L) \arrow[d, "{[g]}"] \\
\opn{Sh}_{\mbf{G}}(K)                     & \opn{Sh}_{\mbf{G}}(K_{\tau}) \arrow[l, "{[\tau^{-1}]}", swap] \arrow[r, "\opn{pr}"]                    & \opn{Sh}_{\mbf{G}}(K)                   
\end{tikzcd}
\]
because $g$ and $\tau$ commute. Both squares are Cartesian because the degrees of the horizontal maps only depend on $L_p = K_p$ and the degrees of the vertical maps are all equal to $[K^p : g^{-1} L^p g]$ (c.f. the proof of Lemma \ref{Heckecommutes} for a similar calculation). Furthermore, under the canonical identification $[g]^*\mathscr{F}_K = \mathscr{F}_L$ (recall that the action of $\mbf{G}(\mbb{A}_f)$ factors through $\mbf{G}(\mbb{Q}_p)$) one has $\phi_{L} = [g]^*\phi_K$, because of the cocycle condition for the equivariant sheaf $\mathscr{F}$ and the fact that $g$ and $\tau^{-1}$ commute. This implies that the monoidal functors $F_L$ and $F_K$ are compatible under $[g]$ (in the sense of Definition \ref{DefOfCompatUnderPi}) and so the Hecke operators $\invs{T}_K$ and $\invs{T}_L$ commute with $[g]^*$ and $[g]_*$.

One has a similar description for the group $\Gt$ and the Hecke operator $[\tilde{K} \tilde{\tau}^{-1} \tilde{K}]$ as above -- in this case, if the level $\tilde{K}$ is of the form $\tilde{K}^p \tilde{K}_p$ with $\tilde{K}_p = K_p \times C_p$, then the Hecke operator $\invs{T}$ is compatible with varying the levels $\tilde{K}^p$ and $C_p$.
\end{remark}

\subsection{The flag variety} \label{TheFlagVariety}

We now consider the spherical variety that underlies the whole construction in this section. Let $\overline{Q}$ denote the opposite of $Q$ (with respect to the standard torus) with unipotent radical $\overline{N}$. We will also abusively write $\overline{N}$ for the unipotent radical of $\overline{Q} \times T$. Then we consider the flag variety $\invs{F} = \widetilde{G}/(\overline{Q} \times T) = G/\overline{Q}$ over $\opn{Spec}\mbb{Z}_p$, which has a natural left action of $\widetilde{G}$ (which factors through the projection to $G$). Let $\tilde{u} \in \widetilde{G}(\mbb{Z}_p)$ denote the element 
\[
\tilde{u} = 1 \times \tbyt{1}{1}{}{1} \times 1 
\]
and consider the following reductive subgroup $L^\circ \triangleleft \iota(H) \times T$ given by
\[
L^\circ = \iota(H) \times 1 = \opn{GL}_1 \times \left\{ \tbyt{X}{}{}{Y} : X, Y \in \opn{GL}_n \right\} \times 1 .
\]
Then we verify the following properties.
\begin{lemma} \label{flag}  
Keeping the same notation as above, the following are satisfied:
\begin{enumerate}
 \item The stabiliser $\opn{Stab}_{H}([\tilde{u}]) = (\iota, \nu)(H) \cap  \tilde{u} (\overline{Q} \times T) \tilde{u}^{-1}$ is contained in $\tilde{u} \cdot L^\circ \cdot \tilde{u}^{-1}$. 
 \item The $H$-orbit of $\tilde{u}$ in $\invs{F}$ (over $\opn{Spec}\mbb{Z}_p$) is Zariski open. 
\end{enumerate}
\end{lemma}
\begin{proof}
The first part is a routine check that is left to the reader. For the second part we note that we have the following equality
\[
(\iota, \nu)(H) \cap \tilde{u}(\overline{Q} \times T)  \tilde{u}^{-1} =  \GL_{1} \times \left\{ \tbyt{X}{}{}{X} : X \in \opn{GL}_n \right\} \times 1
\]
which implies that the orbit of $[\tilde{u}]$ can be identified with a copy of $\opn{GL}_n \subset \invs{F}$. One can check that this is open.
\end{proof}

We define the following level subgroups
\begin{itemize}
 \item $U_r = \{ g \in \tilde { G } (\mbb{Z}_p) : g \pmod{p^r} \in L^\circ   \text{ and }    \tilde{\tau}^{-r} g \tilde{\tau}^{r} \in \tilde{G} (\ZZ_{p} )  \}    $ 
 \item $V_r = \tilde{\tau}^{-r}U_r \tilde{\tau}^r $
\end{itemize}
for $r \geq 0$. We let $D_{p^r} \subset T(\mbb{Z}_p)$ denote the group of all elements $x \equiv 1$ modulo $p^r$. Then $V_r$ and $U_r$ are compact open subgroups of $J \times D_{p^r}$ for $r \geq 1$ (so in the notation of section \ref{FunctorsFromEt}, we have $V_r, U_r \in \Upsilon_{\gt_p}$).    

\begin{remark}
The conditions defining these subgroups are slightly stronger than in \cite[\S 4.4, \S 4.6]{Loeffler19}, however this does not affect the validity of \emph{op.cit.} because we have a stronger property on the flag variety (Lemma \ref{flag}(1)). Moreover, the second condition in the definition of $U_r$ is superfluous in our case, however we have included it to keep in line with the notation in \emph{op.cit.}
\end{remark}

The following lemma ensures that the action of the Hecke operator $\invs{T}$ will be independent of the choice of level subgroup.

\begin{lemma}   \label{Iwahori}
Let $r \geq 1$ and set $V_{r}' = V_{r} \cap \tilde{\tau} V_{r} \tilde{\tau}^{-1}$. Then 
\begin{enumerate}[(a)]
\item $V_{r}' \backslash V_{r} / V_{r+1}$ is a singleton and $V_{r}' \cap V_{r+1} = V_{r+1}'$. 
\item Set $(J \times D_{p^r})' = \tilde{\tau} (J \times D_{p^r}) \tilde{\tau}^{-1} \cap (J \times D_{p^r})$. The double coset space 
\[
( J \times D_{p^{r}})' \backslash J \times  D_{p^{r}} / V_{r}
\]
is a singleton and $V_{r} \cap ( J \times  D_{p^{r}})' = V_{r}'$.    
\end{enumerate}
\end{lemma}    
\begin{proof} 
In terms of $ U_{r} $, the first claim in part (a) amounts to showing that 
\[
( U_{r}  \cap  \tilde{\tau} U_{r} \tilde{\tau}^{-1}   ) \backslash      U_{r} / \tilde{\tau}^{-1} U_{r+1} \tilde{\tau} 
\]
is a singleton. Notice that we have the Iwahori decomposition 
\[
 U_{r} = \overline{N}_{r} \cdot L_{r} \cdot   N_{r}
\]
where $ N$, $\overline{N}$ are the unipotent radicals of $ Q \times T $, $ \overline{Q} \times T $ respectively, $ L = \iota(H) \times T \subset G \times   T $ is the Levi, and $ N_{r} \defeq \tilde{\tau}^{r} N(\mbb{Z}_p) \tilde{\tau}^{-r} $, $ \overline{N}_{r} \defeq \tilde{\tau}^{-r}   \overline{N}(\mbb{Z}_p)   \tilde{\tau}^{r}   $ and $ L_{r} = L^{\circ}(\ZZ_{p}) (1 \times D_{p^{r}}) \subset G(\ZZ_{p}) \times T(\ZZ_{p})$. One then easily sees that 
\[
U_{r} \cap \tilde{\tau} U_{r} \tilde{\tau}^{-1} =  \overline{N}_{r} \cdot L_{r} \cdot N_{r+1}  ,  \quad  \quad  \tilde{\tau} ^{-1} U_{r+1}  \tilde{\tau}  =   \overline{N}_{r+2} \cdot L_{r+1} \cdot  N_{r}   
\]
which implies that $U_r = \left(U_r \cap \tilde{\tau}U_r \tilde{\tau}^{-1}\right) \cdot \left(\tilde{\tau}^{-1}U_{r+1}\tilde{\tau}\right)$ as required. As 
\[
V_{r}' = \overline{N}_{2r} \cdot L_{r} \cdot N_{1}, \quad  \quad  V_{r+1} = \overline{N}_{2r+2} \cdot L_{r+1} \cdot N_0 
\]
their intersection equals $\overline{N}_{2r+2} \cdot L_{r+1} \cdot N_{1} = V_{r+1}' $. The proof of part (b) is similar.    
\end{proof}

\subsection{Abstract localisations} \label{AbstractLocalisations}

To obtain classes in the cohomology of $ \Sh_{\Gb}$ that are cohomologically trivial, it turns out it is enough to construct universal norms out of them in the ``$p$-direction''. Once we have accomplished this, we can apply the constructions in Appendix \ref{AbelJacobi} to map to the cohomology of the Galois representation attached to $\pi$. We obtain these universal norms by localising the cohomology of $\opn{Sh}_{\Gt} $ at a suitable maximal ideal of the ring $\mbb{Z}_p[\invs{T}]$. Informally, this process will cut out the part of the cohomology on which $\invs{T}$ acts through a scalar, which is congruent to a fixed element of $\overline{\mbb{F}}_p^{\times}$ modulo $p$. We describe this abstract process in this section.

Let $\mbb{Z}_p[X]$ denote the polynomial algebra in one variable over $\mbb{Z}_p$. This is a unique factorisation domain and its quotients satisfy the following property:

\begin{lemma} \label{LemmaRingDecomposition}
Let $h(X) \in \mbb{Z}_p[X]$ denote any monic polynomial (of positive degree) and set $\mbb{T} = \mbb{Z}_p[X]/h(X)$. Then one has a canonical isomorphism of rings:
\[
\mbb{T} \cong \bigoplus_{\ide{m} \subset \mbb{T}} \mbb{T}_{\ide{m}}
\]
where the sum is over all maximal ideals of $\mbb{T}$. 
\end{lemma}
\begin{proof}
Since $\mbb{Z}_p$ is a Henselian local ring, this follows from \cite[Tag 04GG, Lemma 10.153.3 (10)]{stacks-project}.
\end{proof}

For any element $\alpha \in \overline{\mbb{Z}}_p$ we let $\ide{m}_{\alpha}$ denote the kernel of the composition:
\[
\mbb{Z}_p[X] \to \mbb{F}_p[X] \xrightarrow{X \mapsto \alpha} \overline{\mbb{F}}_p .
\]
This defines a maximal ideal of $\mbb{Z}_p[X]$. We obtain the following corollary:

\begin{corollary} \label{CorFinGenSplitting}
Let $M$ be a $\mbb{Z}_p[X]$-module which is finitely-generated over $\mbb{Z}_p$. Then
\begin{enumerate}
    \item For any maximal ideal $\ide{n} \subset \mbb{Z}_p[X]$, the localisation 
    \[
    M_{\ide{n}} = M \otimes_{\mbb{Z}_p[X]} \mbb{Z}_p[X]_{\ide{n}}
    \]
    is finitely-generated over $\mbb{Z}_p$. Moreover, the natural map $M \to M_{\ide{n}}$ is surjective and admits a $\mbb{Z}_p[X]$-linear splitting.
    \item For $\alpha \in \overline{\mbb{Z}}_p$ and a $\overline{\mbb{Z}}_p[X]$-module $N$ on which $X$ acts through the scalar $\alpha$, any $\mbb{Z}_p[X]$-linear morphism $f \colon M \to N$ factors through the natural map $M \twoheadrightarrow M_{\ide{m}_{\alpha}}$.
\end{enumerate}
\end{corollary}
\begin{proof}
Let $h(X)$ be a monic polynomial (of positive degree) which kills $M$ (such a polynomial exists by the Cayley--Hamilton theorem) and set $\mbb{T} = \mbb{Z}_p[X]/h(X)$. By Lemma \ref{LemmaRingDecomposition}, we have the following canonical decomposition
\[
M = M \otimes_{\mbb{T}} \mbb{T} = \bigoplus_{\ide{m} \subset \mbb{T}} (M \otimes_{\mbb{T}} \mbb{T}_{\ide{m}} )
\]
where the sum is over all maximal ideals of $\mbb{T}$. If $h(X) \not\in \ide{n}$ then $M_{\ide{n}} = 0$ and the claim is vacuous. Otherwise, there is a one-to-one correspondence between maximal ideals $\ide{n}$ containing $h(X)$ and maximal ideals of $\mbb{T}$. Since the action of $\mbb{Z}_p[X]$ on $M$ factors through $\mbb{T}$, we have the following identification
\[
M_{\ide{n}} = M \otimes_{\mbb{T}} \mbb{T}_{\ide{m}}
\]
where $\ide{m} \subset \mbb{T}$ is the maximal ideal corresponding to $\ide{n}$. The decomposition above immediately gives part (1).

The second part of the corollary follows from the fact that for any $g(X) \in \mbb{Z}_p[X] - \ide{m}_{\alpha}$, one has $g(\alpha) \in \overline{\mbb{Z}}_p^{\times}$. 
\end{proof}

\begin{remark} \label{RemRationalLocalisation}
One also has a rational version of this as follows. For $\alpha \in \Qpb$ let $\ide{t}_{\alpha}$ denote the kernel of the evaluation-at-$\alpha$ map $\mbb{Q}_p[X] \to \Qpb$. Then if $N$ is a $\mbb{Q}_p[X]$-module which is finite-dimensional over $\mbb{Q}_p$, then the localisation
\[
N_{\ide{t}_{\alpha}} = N \otimes_{\mbb{Q}_p[X]} \mbb{Q}_p[X]_{\ide{t}_{\alpha}}
\]
is finite-dimensional over $\mbb{Q}_p$ and is naturally a direct summand of $N$ (if $\alpha \in \mbb{Q}_p$ then this is simply the generalised eigenspace where $X = \alpha$).

If $\alpha \in \overline{\mbb{Z}}_p$, $M$ is a $\mbb{Z}_p[X]$-module which is finitely-generated over $\mbb{Z}_p$, and $f \colon M \to N$ is a $\mbb{Z}_p[X]$-linear map, then one has an induced map $M_{\ide{m}_{\alpha}} \to N_{\ide{t}_{\alpha}}$ which commutes with the splittings $M_{\ide{m}_{\alpha}} \hookrightarrow M$ and $N_{\ide{t}_{\alpha}} \hookrightarrow N$.
\end{remark}

\subsection{Application to cohomology functors}

In this subsection we will freely use the notation from \S \ref{BranchingLawsAndReps}--\ref{FunctorsFromEt}. By a slight abuse of notation, if a $\mbb{Z}_p$-module $M$ carries a $\mbb{Z}_p$-linear action of the Hecke operator $\invs{T}$ then we will write $M$ is a $\mbb{Z}_p[\invs{T}]$-module to indicate that $M$ is a $\mbb{Z}_p[X]$-module with $X$ acting through $\invs{T}$ (so in particular, we do \emph{not} use the notation $\mbb{Z}_p[\invs{T}]$ to indicate the image of $\mbb{Z}_p[X]$ in the endomorphism ring of $M$). Similarly, if we define $\invs{U}_S' \defeq \mu(\tau)\invs{T}$ where $\mu$ is the highest weight of the algebraic representation $V$, then we will write $N$ is a $\mbb{Q}_p[\invs{U}'_S]$-module to indicate that $N$ is a $\mbb{Q}_p[X]$-module with $X$ acting through $\invs{U}'_S$.

We fix the level in the $ G_{p} $ component to be $ J  \in  \Upsilon_{G_{p}} $ in this section (although the results apply for any fixed level $ K_{p}   $). We note that each $ M_{\tilde{G}, \ZZ_{p}}(K^pJ \times C )$ is endowed with an action of $ \mathcal{T}$, and these actions commute with pullbacks and pushforwards. More precisely

\begin{lemma} \label{LemmaCohFuncHecke}
Let $ M_{\tilde{G}, \ZZ_{p}}^{p} \colon \mathcal{P}_{\tilde{G}^{p}} \times \mathcal{P}_{T_{p}} \to \ZZ_{p}  \textbf{-Mod} $ denote the cohomology functor 
\[
(K^p \times C^p, C_p) \mapsto  M_{\tilde{G},\ZZ_{p}}(K^pJ \times C^pC_p).
\]
Then $\mathcal{T}$ is an endo-functor of $ M_{\tilde{G},\ZZ_{p}}^{p} $. In other words, the functor $ M_{\tilde{G}, \ZZ_{p}}^{p} $ is valued in $ \ZZ_{p}[\mathcal{T}] \textbf{-Mod}$. 
\end{lemma}  
\begin{proof}   
This is a consequence of Lemma \ref{Heckecommutes} (see also Remark \ref{RemAJAppendixRephrase}).
\end{proof}  

We now localise this functor at a maximal ideal of the ring $\mbb{Z}_p[\invs{T}]$ and pass to the inverse limit.

\begin{lemma} 
Let $\alpha \in \overline{\mbb{Z}}_p$.
\begin{enumerate} 
\item The mapping $ M_{\tilde{G},\ZZ_{p}, \alpha} \colon  \mathcal{P}_{\tilde{G}^{p}} \times \mathcal{P}_{T_{p}}  \to  \mathcal{\ZZ}_{p}\textbf{-Mod} $ given by 
\[
(K^p \times C^p, C_p)   \mapsto   M_{\tilde{G},\ZZ_{p}} ( K^pJ \times C^pC_p )_{\ide{m}_{\alpha}} 
\]
is a $ \ZZ_{p}[\mathcal{T}]_{\ide{m}_{\alpha}}$-valued CoMack functor, where $\ide{m}_{\alpha}$ is the maximal ideal defined in \S \ref{AbstractLocalisations}.   
\item The mapping $ M_{\tilde{G}, \Iw, \ZZ_{p}, \alpha} \colon \mathcal{P}_{\tilde{G}^{p}} \to \ZZ_{p}[\mathcal{T}]_{\ide{m}_{\alpha}}$-\textbf{Mod} given by  
\[
K^{p} \times C^p \mapsto \varprojlim_{ C_{p}} M_ { \tilde { G } , \ZZ_{p}, \alpha} ( K^pJ \times C^p C_{p}  )
\]
defines a CoMack functor.
\end{enumerate}
\end{lemma} 

\begin{proof}
This follows from the fact that the Hecke operator $\invs{T}$ commutes with pushforwards and pullbacks arising from morphisms in $\invs{P}_{\gt^p} \times \invs{P}_{T_p}$ (see Lemma \ref{LemmaCohFuncHecke}).
\end{proof}  

Let $ M_{ H , \ZZ_{p}} ^ { p } \colon \mathcal { P } _ { H ^ { p } } \to   \ZZ_{p} $-\textbf{Mod} denote the CoMack functor $ U ^ {p } \mapsto M _ { H ,  \ZZ_{p} } ( U ^ { p } H ( \ZZ_{p} ) ) $. 

\begin{proposition} \label{pushingintofs} 
Let $\alpha \in \overline{\mbb{Z}}_p^{\times}$. There is a (Cartesian) pushforward of cohomology functors
\[
\iota_{ *, \alpha} \colon M_ { H ,  \ZZ_{p}} ^ { p } \to M_{\tilde{G}, \Iw , \ZZ_{p}, \alpha } .
\]
where the latter is viewed as a functor into $\mbb{Z}_p$\textbf{-Mod} by forgetting the extra structure.
\end{proposition}   
\begin{proof} 
We apply the results of \cite{Loeffler19} to the pushforward $\iota_* \colon M_{H,\ZZ_{p}}\to M_{\tilde{G},\ZZ_{p}}$. Let $ V_{r} \in \mathcal{P} _ { \tilde {G } _ { p } } $ and $ D _ { p ^ { r} } \in \mathcal { P } _ { T _ { p } } $ be as in section \ref{TheFlagVariety}. By Lemma \ref{flag} and Theorem 4.5.3 in \emph{op.cit.}, we have a (Cartesian) pushforward 
\begin{equation} \label{LPush}
M _ { H , \ZZ_{p}} ^ { p } \to \displaystyle { \varprojlim _ { r } M _ { \tilde { G } , \ZZ_{p}} ( - \cdot V_{r} )_{\ide{m}_{\alpha}} }
\end{equation}
where the target is a functor on $ \mathcal { P } _ { \tilde { G } ^ { p } } $.  Explicitly, let $\eta_r = \tilde{u} \tilde{\tau}^r$ and $U^p = H^p \cap K^p$. For all $r \geq 1$, we have a map $ \theta_{r, K^{p}}   $ given by the composition:
\begin{eqnarray}
M_{H,\ZZ_{p}}(U^p \cdot H(\mbb{Z}_p) ) & \xrightarrow{\opn{pr}^*} & M_{H,\ZZ_{p}}(U^p \cdot (H_p \cap \eta_r V_r \eta_r^{-1}) ) \nonumber \\
 & \xrightarrow{\iota_*} & M_{\tilde{G},\ZZ_{p}}(K^p \cdot (\eta_r V_r \eta_r^{-1} ) ) \nonumber \\
 & \xrightarrow{[\eta_r]_*} & M_{\tilde{G},\ZZ_{p}}(K^p \cdot V_r) \nonumber \\
 & \xrightarrow{\invs{T}^{-r}} & M_{\tilde{G},\ZZ_{p}}(K^p \cdot V_r)_{\ide{m}_{\alpha}} \nonumber 
\end{eqnarray}
since $\invs{T} \not\in \ide{m}_{\alpha}$. For any $ r $, the maps $ \theta_{r, K^{p}} $, $ \theta_{r+1, K^{p}} $ commute with the pushforward map from level $  K^{p}  \cdot  V_{r+1} $ to $ K^{p}  \cdot  V_{r} $ --  this amounts to showing the commutativity of the bottom left square of the diagram in  \S 4.5 of \emph{op.cit.}, which is where the Cartesian property of $ \iota_{*} $ is invoked (here we use Lemma \ref{Iwahori} (a) and the Cartesian property of $ M_{\tilde{G} , \ZZ_{p}}  $). As the dependence on $K^{p} $ is functorial by Lemma \ref{Heckecommutes}, this gives the pushforward claimed in (\ref{LPush}).  

Now $V_r \subset J \times D_{p^r}$, so there also exist projection maps   
\[
\opn{pr}_* \colon M_{\tilde{G}}(- \cdot V_r) \longrightarrow M_{\tilde{G}}(- \cdot (J \times D_{p^r}) ).
\]
Again, by Lemma \ref{Heckecommutes}, this is a pushforward of cohomology functors, and moreover commutes with the action of $\invs{T}$ by Lemma \ref{Iwahori} (b), whence also  with  the  localisation with respect to $\ide{m}_{\alpha}$. We therefore obtain a pushforward  of cohomology functors 
\[
\iota_{*, \alpha} \colon   M_{H,\ZZ_{p}}^{p} \to  \varprojlim _ { r  } M_{\tilde{G} , \ZZ_{p}}  \left(-\cdot (J \times D_{p^{r}})\right)_{\ide{m}_{\alpha}}  
\]
by composing (\ref{LPush}) with $\opn{pr}_*$ (after localising at $ \ide{m}_{\alpha} $), as required.
\end{proof}

\begin{remark} \label{RemRationalpushinginto}
If we let $M^p_{H, \mbb{Q}_p} \colon \invs{P}_{H^p} \to \mbb{Q}_p$\textbf{-Mod} and $M_{\gt, \opn{Iw}, \mbb{Q}_p, \alpha} \colon \invs{P}_{\gt^p} \to \mbb{Q}_p[\invs{U}_S']_{\ide{t}_{\alpha}}$\textbf{-Mod} denote the cohomology functors 
\[
M_{H, \mbb{Q}_p}^p(U^p) = M_{H, \mbb{Q}_p}(U^p H(\mbb{Z}_p)) \quad \quad \quad M_{\gt, \opn{Iw}, \mbb{Q}_p, \alpha}(K^p \times C^p) = \varprojlim_{C_p} M_{\gt, \mbb{Q}_p}(K^pJ \times C^p C_p)_{\ide{t}_{\alpha}}
\]
then (for $\alpha \in \Qpb^{\times}$) we obtain a pushforward of cohomology functors $\iota_{*, \alpha} \colon M_{H, \mbb{Q}_p}^p \to M_{\gt, \opn{Iw}, \mbb{Q}_p, \alpha}$ given by the composition

\begin{eqnarray}
M_{H,\mbb{Q}_{p}}(U^p \cdot H(\mbb{Z}_p) ) & \xrightarrow{\opn{pr}^*} & M_{H,\mbb{Q}_{p}}(U^p \cdot (H_p \cap \eta_r V_r \eta_r^{-1}) ) \nonumber \\
 & \xrightarrow{\iota_*} & M_{\tilde{G},\mbb{Q}_{p}}(K^p \cdot (\eta_r V_r \eta_r^{-1} ) ) \nonumber \\
 & \xrightarrow{\mu(\tau)^r [\eta_r]_*} & M_{\tilde{G},\mbb{Q}_{p}}(K^p \cdot V_r) \nonumber \\
 & \xrightarrow{(\invs{U}_S')^{-r}} & M_{\tilde{G},\mbb{Q}_{p}}(K^p \cdot V_r)_{\ide{t}_{\alpha}} \nonumber 
\end{eqnarray}
where, as usual, $\mu$ is the highest weight of the algebraic representation $V$. By Lemma \ref{LemmaCompatibleAction}, whenever $\alpha \in \overline{\mbb{Z}}_p^{\times}$, one has a commutative diagram of cohomology functors
\[
\begin{tikzcd}
{M_{H, \mbb{Q}_p}^p} \arrow[r, "{\iota_{*, \opn{Iw}, \alpha}}"]           & {M_{\gt, \opn{Iw}, \mbb{Q}_p, \alpha}}           \\
{M_{H, \mbb{Z}_p}^p} \arrow[u] \arrow[r, "{\iota_{*, \opn{Iw}, \alpha}}"] & {M_{\gt, \opn{Iw}, \mbb{Z}_p, \alpha}} \arrow[u]
\end{tikzcd}
\]
where the vertical arrows are the natural ones.
\end{remark}

\subsection{Pushforward to Iwasawa  cohomology} 
For a compact open subgroup $ C \in \Upsilon_T $, we consider the finite group 
\[
\Delta(C) \defeq \mbf{T}(\mbb{Q}) \backslash \mbf{T}(\mbb{A}_f) / C 
\]
which carries a transitive action of $\Gal(E^{\mathrm{ab}}/E)$ via translation by the surjective character $\kappa _ { C } = \kappa $, where we set $\kappa_{ C } (\sigma) = \bar{s}_f/s_f$ for any idele $s \in \mbb{A}_E^{\times}$ satisfying $\opn{Art}_E(s) = \sigma$. We let $E(C ) $ denote the finite abelian extension of $E$ corresponding to the kernel of this character. We note the following application of Shapiro's lemma.
\begin{lemma} \label{ShapiroLemma}
We have an isomorphism 
\[
\opn{H}^1\left(E, M_{\tilde{G}, \et, \ZZ_{p}}\left( K \times C \right) \right) \cong \opn{H}^1\left(E(C), M_{G, \et, \ZZ_{p}} \left( K \right) \right)
\]
which intertwines the action of $\Gal(E(C)/E)$ and $\mbf{T}(\mbb{A}_f)$ via the character $\kappa_{C}^{-1}$ (explicitly, the action of an element $\sigma \in \Gal(E(C)/E)$ on the right-hand side corresponds to the action of $\kappa_C(\sigma)^{-1}$ on the left-hand side). The same statement holds after localising at a maximal ideal of $\mbb{Z}_p[\invs{T}]$. 
\end{lemma}
\begin{proof}
One sees from the definition of the Shimura datum for $\mbf{T}$ that $\Gal(E^{\mathrm{ab}}/E)$ acts on $\opn{Sh}_{\mbf{T}}(C)_{\bar{E}} = \Delta(C)$ via translation by the character $\kappa = \kappa_C$. One has an isomorphism 
\begin{eqnarray} 
\opn{H}^0_{\et}\left( \opn{Sh}_{\mbf{T}}(C) \times_{\opn{Spec}E} \opn{Spec}\overline{E}, \ZZ_{p} \right) = \opn{Maps}\left(\opn{Sh}_{\mbf{T}}(C)(\mbb{C}), \mbb{Z}_p \right) & = & \ZZ_{p} [\Gal(E(C)/E)] \label{GroupRingEqn} \\
 \phi & \mapsto & \sum_{\sigma \in \Gal(E(C)/E)} \phi(\kappa(\sigma^{-1})) \cdot \sigma \nonumber 
\end{eqnarray}
which intertwines the action of $t \in \mbf{T}(\mbb{A}_f)$ on the left-hand side (given by right-translation of the argument $(t \cdot \phi) (-) = \phi(- \cdot t)$) with the action of $\sigma \in \Gal(E(C)/E)$ on the right-hand side (the action is by left-multiplication), where $\sigma$ is the unique element such that the image of $t$ in $\Delta(C)$ equals $\kappa(\sigma)$. 

We have the following isomorphisms
\begin{eqnarray}
\opn{H}^1\left(E, M_{\tilde{G}, \et, \ZZ_{p}}\left( K \times C \right) \right) & \cong & \opn{H}^1\left(E, M_{G, \et, \mbb{Z}_p}(K) \otimes_{\mbb{Z}_p} \opn{H}^0_{\et}\left( \opn{Sh}_{\mbf{T}}(C) \times_{\opn{Spec}E} \opn{Spec}\overline{E}, \ZZ_{p} \right) \right) \label{KunnethEqn} \\
 & \cong & \opn{H}^1\left(E, M_{G, \et, \mbb{Z}_p}(K) \otimes_{\mbb{Z}_p} \ZZ_{p} [\Gal(E(C)/E)] \right) \label{abovefact} \\
 & \cong & \opn{H}^1\left(E(C), M_{G, \et, \ZZ_{p}} \left( K \right) \right) \label{Shapiroeqnline}
\end{eqnarray}
where (\ref{KunnethEqn}), (\ref{abovefact}) and (\ref{Shapiroeqnline}) follow from the K\"{u}nneth formula, (\ref{GroupRingEqn}) and Shapiro's lemma respectively. The action of an element $\sigma \in \Gal(E(C)/E)$ on the right-hand side of (\ref{Shapiroeqnline}) is intertwined with left multiplication by $\sigma^{-1}$ on $\mbb{Z}_p[\Gal(E(C)/E)]$ appearing on the right-hand side of (\ref{abovefact}) (see \cite[\S 8.2]{nekovar-selmer}). By (\ref{GroupRingEqn}), this is precisely intertwined with the action of an element $t^{-1} \in \mbf{T}(\mbb{A}_f)$ on the left-hand side of (\ref{KunnethEqn}), where the image of $t$ in $\Delta(C)$ coincides with $\kappa(\sigma)$. 

The same result holds after localising at a maximal ideal of $\mbb{Z}_p[\invs{T}]$ because the Hecke operator only acts on the cohomology of $\opn{Sh}_{\mbf{G}}(K)$ (and commutes with the action of $\mbf{T}$).
\end{proof}

\begin{definition}
Let $\alpha \in \overline{\mbb{Z}}_p$. We let $N_{\Iw, \ZZ_{p}, \alpha} \colon \mathcal{P}_{\tilde{G}^{p}} \to \ZZ_{p} [ \mathcal{T} ]_{\ide{m}_{\alpha}} $\textbf{-Mod} denote the map 
\[
K^p \times C^p \mapsto \varprojlim_{r} \opn{H}^1\left(E, M_{\tilde{G}, \et, \ZZ_{p}}(K^p J \times C^p D_{p^r})_{\ide{m}_{\alpha}} \right).
\]
\end{definition}
\begin{lemma} \label{LemmaNIwisCoMack}
$ N _ { \Iw, \ZZ_{p}, \alpha} $ is a CoMack functor.    
\end{lemma} 
\begin{proof} 
The functor $ N _ { \Iw, \ZZ_{p}, \alpha} $ is obtained by composing the functor 
\[
K^p \times C^p \mapsto M _ { \tilde { G } , \et , \ZZ_{p}} ( K ^ { p } J \times C^{p} D _ { p^ { r} } )
\]
with $(-)_{\ide{m}_{\alpha}}$, applying $ \opn{H} ^ { 1 } ( E , - ) $ and then taking the inverse limit over $ r $. The localisation commutes with pullbacks and pushforwards arising from morphisms in $\invs{P}_{\gt^p}$, and the latter two operations are  additive   endomorphisms of the category of $ \ZZ_{p}[ \mathcal { T } ]_{\ide{m}_{\alpha}} $-modules.   
\end{proof} 

We now obtain the vertical norm relations for our Euler system. Recall from Remark \ref{RemRationalpushinginto} that we have defined $ M_{H,\QQ_{p}}^p \defeq M_{H,\ZZ_{p}}^{p} \otimes_{\ZZ_{p}} \QQ_{p} $. We set $ N_{\Iw,\QQ_{p}, \alpha} \colon \invs{P}_{\gt^p} \to \mbb{Q}_p[\invs{U}'_S]$\textbf{-Mod} to be the cohomology functor given by 
\[
N_{\Iw,\QQ_{p}, \alpha}(K^p \times C^p) \defeq \varprojlim_{r} \opn{H}^1\left(E, M_{\tilde{G}, \et, \mbb{Q}_{p}}(K^p J \times C^p D_{p^r})_{\ide{t}_{\alpha}} \right).
\]
This functor is CoMack by the same proof as in Lemma \ref{LemmaNIwisCoMack}, but note that $N_{\Iw,\mbb{Z}_{p}, \alpha} \otimes_{\mbb{Z}_p} \mbb{Q}_p \neq N_{\Iw,\QQ_{p}, \alpha}$.

\begin{theorem}[Vertical norm relations] \label{VerticalRelThm} 
Let $\alpha \in \overline{\mbb{Z}}_p^{\times}$. For $R = \mbb{Z}_p$ or $\mbb{Q}_p$, there is a (Cartesian) pushforward of cohomology functors 
\[
\iota _ { * , \Iw, \alpha } \colon M _ { H , R} ^ { p } \to N _ { \Iw , R, \alpha} 
\]
where the latter is viewed as a functor into $R$\textbf{-Mod} by forgetting the extra structure, which is compatible with respect to base change of coefficients (as in the paragraph preceding Corollary \ref{CommDiagramCor}). 
\end{theorem} 

\begin{proof}
Set $E_r(C^p) = E(C^pD_{p^r})$ and $E_{\infty}(C^p) = \bigcup E_r(C^p)$. We first note that, by Lemma \ref{ShapiroLemma}, the functor $N_{\mathrm{Iw},\ZZ_{p}, \alpha}$ is isomorphic to the functor 
\[
K^p \times C^p \mapsto \varprojlim_r \opn{H}^1\left(E_r(C^p), M_{G, \et,\ZZ_{p}}(K^p J)_{\ide{m}_{\alpha}} \right)
\]
where the inverse limit is with respect to corestriction. Since localisation is exact, one has localised versions of the Hochschild--Serre spectral sequence (see Appendix \ref{AbelJacobi}):
\begin{align*}
    {_{\mbb{Z}_p}}\opn{E}^{i, j}_2 \defeq \opn{H}^{i}\left( E, \opn{H}^j_{\et}\left( \opn{Sh}_{\Gt}(K^p J \times C^p D_{p^r})_{\overline{\mbb{Q}}}, \tilde{\mathscr{T}}(n) \right) \right)_{\ide{m}_{\alpha}} &\Rightarrow \opn{H}^{i+j}_{\et}\left( \opn{Sh}_{\Gt}(K^pJ \times C^pD_{p^r}), \tilde{\mathscr{T}}(n) \right)_{\ide{m}_{\alpha}} \\
    {_{\mbb{Q}_p}}\opn{E}^{i, j}_2 \defeq \opn{H}^{i}\left( E, \opn{H}^j_{\et}\left( \opn{Sh}_{\Gt}(K^p J \times C^p D_{p^r})_{\overline{\mbb{Q}}}, \tilde{\mathscr{V}}(n) \right) \right)_{\ide{t}_{\alpha}} &\Rightarrow \opn{H}^{i+j}_{\et}\left( \opn{Sh}_{\Gt}(K^pJ \times C^pD_{p^r}), \tilde{\mathscr{V}}(n) \right)_{\ide{t}_{\alpha}}
\end{align*}
where we have suppressed the dependence on $r$, $C^p$, $K^p$ and $\alpha$ from the notation. The Hochschild--Serre spectral sequence is functorial with respect to the variety, so the spectral sequences above are functorial with respect to pullbacks/pushforwards arising from morphisms in $\invs{P}_{\gt^p} \times \invs{P}_{T_p}$ (combine Proposition \ref{PropFuncPropsForAJet} and Remark \ref{RemAJAppendixRephrase}). Moreover, there is a natural $\mbb{Z}_p$-linear map of spectral sequences $ {_{\mbb{Z}_p}}\opn{E}^{i, j}_2 \to {_{\mbb{Q}_p}}\opn{E}^{i, j}_2$ (Remark \ref{RemRationalLocalisation}) which commutes with pullbacks/pushforwards up to rescaling the action of $\tilde{\tau}^{-1}$ (see Lemma \ref{LemmaCompatibleAction}). In particular, by Lemma \ref{ShapiroLemma} (the rational version follows from the same proof) and the above functoriality properties, one obtains ``Abel--Jacobi'' maps
\begin{align*}
    \opn{AJ}_{\opn{Iw}, \mbb{Z}_p, \alpha} \colon \varprojlim_r M_{\gt, \mbb{Z}_p}\left( K^pJ \times C^p D_{p^r} \right)_{\ide{m}_{\alpha}, 0} &\to \varprojlim_r \opn{H}^1\left(E_r(C^p), M_{G, \et, \mbb{Z}_p}(K^pJ)_{\ide{m}_{\alpha}} \right) \\
    \opn{AJ}_{\opn{Iw}, \mbb{Q}_p, \alpha} \colon \varprojlim_r M_{\gt, \mbb{Q}_p}\left( K^pJ \times C^p D_{p^r} \right)_{\ide{t}_{\alpha}, 0} &\to \varprojlim_r \opn{H}^1\left(E_r(C^p), M_{G, \et, \mbb{Q}_p}(K^pJ)_{\ide{t}_{\alpha}} \right)
\end{align*}
arising from edge maps in the Hochschild--Serre spectral sequences above, where
\begin{align*}
    M_{\gt, \mbb{Z}_p}\left( K^pJ \times C^p D_{p^r} \right)_{\ide{m}_{\alpha}, 0} &\defeq \opn{ker}\left( M_{\gt, \mbb{Z}_p}\left( K^pJ \times C^p D_{p^r} \right)_{\ide{m}_{\alpha}} \to \opn{H}^0\left(E_r(C^p),  M_{G, \et, \mbb{Z}_p}(K^pJ) \right)_{\ide{m}_{\alpha}} \right)  \\
    M_{\gt, \mbb{Q}_p}\left( K^pJ \times C^p D_{p^r} \right)_{\ide{t}_{\alpha}, 0} &\defeq \opn{ker}\left( M_{\gt, \mbb{Q}_p}\left( K^pJ \times C^p D_{p^r} \right)_{\ide{t}_{\alpha}} \to \opn{H}^0\left(E_r(C^p),  M_{G, \et, \mbb{Q}_p}(K^pJ) \right)_{\ide{t}_{\alpha}} \right)
\end{align*}
denote the ``cohomologically trivial'' classes. These Abel--Jacobi maps are the inverse limit of the maps $\opn{AJ}_{\et, \ide{m}_{\alpha}}$ and $\opn{AJ}_{\et, \ide{t}_{\alpha}}$ in the notation of Definition \ref{AJetDef}, and are functorial with respect to $K^p$ and $C^p$ by Proposition \ref{PropFuncPropsForAJet}. The maps $\opn{AJ}_{\opn{Iw}, \mbb{Z}_p, \alpha}$ and $\opn{AJ}_{\opn{Iw}, \mbb{Q}_p, \alpha}$ are also compatible with each other.

Since $M_{G, \et, \mbb{Z}_p}(K^pJ)$ is finitely-generated over $\mbb{Z}_p$ (\cite[Corollaire 1.10, Chapitre 7]{SGA4demi}) so is its localisation $M_{G, \et, \mbb{Z}_p}(K^pJ)_{\ide{m}_{\alpha}}$ (Corollary \ref{CorFinGenSplitting}). In particular, by \cite[Lemma 3.2, Appendix B]{rubin} and the fact that $E_{\infty}(C^p)/E_1(C^p)$ is a Galois extension with Galois group $\mbb{Z}_p$, one sees that 
\begin{equation} \label{VanishingIwasawaH0}
\varprojlim_r \opn{H}^0\left(E_r(C^p),  M_{G, \et, \mbb{Z}_p}(K^pJ) \right)_{\ide{m}_{\alpha}} \hookrightarrow \varprojlim_r \opn{H}^0\left(E_r(C^p),  M_{G, \et, \mbb{Z}_p}(K^pJ)_{\ide{m}_{\alpha}} \right) = 0 .
\end{equation}
The first map is injective because localisation is exact and taking invariants under $\Gal(\overline{E}/ E_r(C^p))$ is left-exact. Now consider the exact sequence
\[
0 \to M_{\gt, \mbb{Z}_p}\left( K^pJ \times C^p D_{p^r} \right)_{\ide{m}_{\alpha}, 0} \to M_{\gt, \mbb{Z}_p}\left( K^pJ \times C^p D_{p^r} \right)_{\ide{m}_{\alpha}} \to \opn{H}^0\left(E_r(C^p),  M_{G, \et, \mbb{Z}_p}(K^pJ) \right)_{\ide{m}_{\alpha}} .
\]
Using the left-exactness of $\varprojlim_r$ applied to this sequence and the vanishing in (\ref{VanishingIwasawaH0}), we obtain the following equality
\[
\varprojlim_r M_{\gt, \mbb{Z}_p}\left( K^pJ \times C^p D_{p^r} \right)_{\ide{m}_{\alpha}, 0} = \varprojlim_r M_{\gt, \mbb{Z}_p}\left( K^pJ \times C^p D_{p^r} \right)_{\ide{m}_{\alpha}} .
\]
Therefore, for $R = \mbb{Z}_p$ or $\mbb{Q}_p$, the image of the pushforward $\iota_{*, \alpha}$ is contained in the inverse limit of the subspaces of cohomologically trivial elements (the claim for $R = \mbb{Q}_p$ follows from the case $R=\mbb{Z}_p$ by Remark \ref{RemRationalpushinginto}), so we can compose this with the Abel--Jacobi map above to obtain a (Cartesian) pushforward of cohomology functors
\[
\iota_{*, \opn{Iw}, \alpha} \defeq \opn{AJ}_{\opn{Iw}, R, \alpha} \circ \; \iota_{*, \alpha} \colon M^p_{H, R} \to N_{\opn{Iw}, R, \alpha}
\]
as required.
\end{proof}

\begin{remark} \label{RemVerticalRelThm}
For $R = \mbb{Z}_p$ or $\mbb{Q}_p$, if we define $N_{\opn{Iw}, R} \colon \invs{P}_{\gt^p} \to R$\textbf{-Mod} as the functor
\[
K^p \times C^p \mapsto \varprojlim_r \opn{H}^1\left( E, M_{\gt, \et, R}(K^pJ \times C^p D_{p^r}) \right)
\]
then $N_{\opn{Iw}, R}$ is a CoMack cohomology functor and we obtain pushforwards $\iota_{*, \opn{Iw}, \alpha} \colon M^p_{H, R} \to N_{\opn{Iw}, R}$ by composing the pushforward in Theorem \ref{VerticalRelThm} with the splittings
\[
M_{G, \et, \mbb{Z}_p}(K^pJ)_{\ide{m}_{\alpha}} \hookrightarrow M_{G, \et, \mbb{Z}_p}(K^pJ) \quad \quad \quad M_{G, \et, \mbb{Q}_p}(K^pJ)_{\ide{t}_{\alpha}} \hookrightarrow M_{G, \et, \mbb{Q}_p}(K^pJ).
\]
For notational simplicity, we will work with this pushforward (rather than the one in Theorem \ref{VerticalRelThm}), however we stress that to obtain (potentially) non-zero Euler system classes we will still need to make an ordinarity assumption on the automorphic representation with respect to $\alpha$ (see Assumption \ref{MainAssumption}(5) ).
\end{remark}


\section{The anticyclotomic Euler system}

We now have all the ingredients needed to define the Euler system.

\subsection{Automorphic assumptions}

Let $\pi_0$ be a cuspidal automorphic representation of $\mbf{G}_0$ such that $\pi_{0, \infty}$ lies in the discrete series, and let $\pi$ be a lift to $\mbf{G}$. From now on, we make the following assumption:

\begin{assumption} \label{MainAssumption}
We assume that:
\begin{enumerate}
    \item $\pi_{0, f}$ admits a $\mbf{H}_0(\mbb{A}_f)$-linear model. In particular, we can choose a lift $\pi$ with trivial central character (see section \ref{DiscreteSeriesReps}). 
    \item $n$ is odd and $\pi$ satisfies Assumption \ref{KeyAssumption}. In particular, let $K$ be a compact open subgroup of $\mbf{G}(\mbb{A}_f)$ such that $\pi_f^K \neq 0$.     
    \item Let $S$ be the finite set of primes containing all primes that ramify in $E$ and such that we have a decomposition $K = K_S \cdot \prod_{\ell \not\in S} K_\ell$ with $K_S \subset \mbf{G}(\mbb{A}_S)$ compact open and $K_\ell \subset \mbf{G}(\mbb{Q}_\ell)$ hyperspecial. We assume that $p \notin S$ and splits in $E/\mbb{Q}$. In this case we set $S' = S \cup \{p\}$.

    For any prime $\ell \notin S'$ which splits in $E/\mbb{Q}$, we assume that $K_\ell = \opn{GL}_1(\mbb{Z}_\ell) \times \opn{GL}_{2n}(\mbb{Z}_\ell)$ under the identification $\mbf{G}_{\mbb{Q}_\ell} = \opn{GL}_{1, \mbb{Q}_\ell} \times \opn{GL}_{2n, \mbb{Q}_{\ell}}$ induced from the fixed embedding $E \hookrightarrow \overline{\mbb{Q}}_{\ell}$.
    \item We assume that $K_S$ satisfies the following property: for any compact open subgroups $L_p \subset \mbf{G}(\mbb{Q}_p)$ and $L^{S'} \subset \mbf{G}(\mbb{A}_f^{S'})$, the subgroup $L_p K_S L^{S'} \subset \mbf{G}(\mbb{A}_f)$ is sufficiently small (see Definition \ref{DefinitionOfSufficientlySmall}).
    \item Let $\mbf{c} = (0; c_1, \dots, c_{2n})$ denote the weight corresponding to the discrete series $L$-packet containing $\pi_{\infty}$, and set $\lambda = \sum_{i=1}^n c_i$. We assume that $\pi$ is ``Siegel ordinary'', i.e. there exists an eigenvector 
    \[
    \varphi = \left( \bigotimes_{\ell \notin S'} \varphi_\ell \right) \otimes \varphi_S \otimes \varphi_p \in \bigotimes_{\ell \notin S'} \pi_\ell^{K_\ell} \otimes \pi_S^{K_S} \otimes \pi_p^J = \pi_f^{K^p J}
    \]
    of the Hecke operator $\invs{U}_S \defeq p^\lambda [J \tau J]$ with eigenvalue $\alpha$ a $p$-adic unit. Here, $J$ is the Siegel parahoric subgroup and $\tau$ is the element defined in section \ref{VerticalSetUp}. In particular, this implies that $\alpha$ occurs as an eigenvalue for the transpose operator $\invs{U}'_S \defeq p^\lambda [J \tau^{-1} J]$ on $(\pi_f^\vee)^{K^pJ}$. 
\end{enumerate}

\end{assumption}

\begin{remark}[Justification of Assumption \ref{MainAssumption}, (5)]
Suppose that $\pi_{0,p}$ is \emph{generic}, in the sense that $\pi_{0,p} \cong I(\chi)$, where $I(-)$ denotes the (normalised) induction from the (standard) Borel to $\mbf{G}_{0}(\mbb{Q}_p) \cong \opn{GL}_{2n}(\mbb{Q}_p)$, and $\chi = \chi_1 \boxtimes \cdots \boxtimes \chi_{2n} \colon (\mbb{Q}_p^\times)^{2n} \to \mbb{C}^{\times}$ is an unramified character, naturally viewed as a character of the (standard) torus. Let $\beta_i = \chi_i(p)$ ($i=1, \dots, 2n$) denote the Satake parameters of $\pi_{0, p}$. For any rational number $s$, we let $p^s$ denote the unique element of $\overline{\mbb{Q}}$ which maps to $e^{s\opn{log}p}$ under the fixed embedding $\overline{\mbb{Q}} \hookrightarrow \mbb{C}$. Via the fixed embedding $\overline{\mbb{Q}} \hookrightarrow \Qpb$, we can also view $p^s$ as an element of $\Qpb$.

Let $v | p$ be the prime of $E$ fixed by the embedding $\overline{\mbb{Q}} \hookrightarrow \Qpb$, and set $\rho_v \defeq \rho_{\pi}|_{E_v}$. Then by Theorem \ref{ExistenceOfGaloisRep}(3), the Galois representation $\rho_v$ is crystalline and one has an equality of polynomials
\[
\opn{det}\left(1 - \varphi \cdot X | \dcris{\rho_v} \right) = \prod_{i=1}^{2n}(1 - p^{n-1/2}\beta_i \cdot X ).
\]
Let $v_p$ denote the $p$-adic valuation on $\Qpb$, normalised so that $v_p(p) = 1$, and order the Satake parameters so that $\gamma_i \defeq p^{n-1/2}\beta_i$ satisfy $v_p(\gamma_1) \leq \cdots \leq v_p(\gamma_{2n})$. For any subset $I \subset \{1, \dots, 2n\}$ of size $n$, set $\gamma_{I} = \prod_{i \in I} \gamma_i$.

Recall that attached to $D = \dcris{\rho_v}$ one has the Hodge polygon $\opn{Hodge}(D)$, which is defined to be the convex hull of the points 
\[
\{(0, 0) \} \cup \left\{ \left( i, \sum_{j=1}^i c_{2n+1-j} + (j-1) \right) : i=1, \dots, 2n \right\}
\]
using the explicit recipe for the jumps in Hodge--Tate filtration in Theorem \ref{ExistenceOfGaloisRep}(3). One also has the Newton polygon $\opn{Newt}(D)$, which is defined to be the convex hull of the points
\[
\{(0, 0)\} \cup \left\{ \left(i, \sum_{j=1}^i v_p(\gamma_j) \right) : i=1, \dots, 2n \right\}. 
\]
Since $D$ is weakly admissible, the Newton polygon lies on or above the Hodge polygon (and has equal endpoints). In particular, this implies that 
\begin{equation} \label{NewtHodgeIneq}
v_p(\gamma_I) \geq \sum_{j=1}^n v_p(\gamma_j) \geq \sum_{j=1}^n c_{2n+1-j} + (j-1) = \sum_{j=1}^n c_{2n+1-j} + n(n-1)/2.
\end{equation}
where the second inequality uses the regularity of the Hodge--Tate weights. On the other hand, we can express the eigenvalues of $\invs{U}_S$ acting on $\pi_p^J$ in terms of $\gamma_I$, using the results of \cite[\S 5.1.1]{OST19}. Indeed, note that we have $\pi_p^J = \pi_{0, p}^{J_0} \cong I(\chi)^{J_0}$ and, in the notation of \emph{loc.cit.}, the action of the operator $[J\tau J]$ corresponds to the action of the operator $I_\chi(\mathbbm{1}_{\check{\gamma}})$, with $\check{\gamma}(p) = \opn{diag}(p, \dots, p, 1, \dots, 1)$ (there are $n$ lots of $p$). Therefore, one has an equality of polynomials
\[
\opn{det}\left(1 - p^{-n^2/2}[J\tau J] \cdot X | \pi_p^J \right) = \opn{det}\left(1 - p^{-n^2/2}I_{\chi}(\mathbbm{1}_{\check{\gamma}}) \cdot X | I(\chi)^{J_0} \right) = \prod_{I} ( 1 - \beta_I \cdot X)
\]
where the product runs over all subsets $I \subset \{1, \dots, 2n\}$ of size $n$, and $\beta_I = \prod_{i \in I} \beta_i$. This implies that the eigenvalues of $\invs{U}_S$ acting on $\pi_p^J$ are of the form
\begin{equation} \label{OSTEquality}
p^\lambda p^{n^2/2} \beta_I = p^{\lambda + n^2/2 - n(n - 1/2)} \gamma_I = p^{\lambda - n(n-1)/2} \gamma_I.
\end{equation}
It turns out that, under Assumption \ref{MainAssumption}(1), the algebraic representation with highest weight $(c_1, \dots, c_{2n})$ is self-dual (see Lemma \ref{SelfDualCoefficients} below), which implies that $\lambda = -\sum_{j=1}^n c_{2n+1-j}$. Therefore, combining (\ref{NewtHodgeIneq}) and (\ref{OSTEquality}), it is reasonable (at least generically) to assume that $\invs{U}_S$ may have an eigenvalue which is a $p$-adic unit. In particular, if $\pi_p$ is Borel ordinary (or equivalently, $\opn{Newt}(D) = \opn{Hodge}(D)$) then $\pi_p$ is Siegel ordinary. 

In fact, the above argument shows that Assumption \ref{MainAssumption}(5) is equivalent to the middle inequality of (\ref{NewtHodgeIneq}) being an equality, i.e. the Hodge and Newton polygons share the same middle point. Furthermore, in this case there is a unique eigenvalue (with multiplicity one) of $\invs{U}_S$ which is a $p$-adic unit, so the local vector $\varphi_p$ is uniquely determined (up to $\mbb{C}^{\times}$).
\end{remark}

Let $V^*$ be the algebraic representation of $\mbf{G}_{\mbb{C}}$ (which arises from an algebraic representation of $\mbf{G}_E$) parameterising the discrete series $L$-packet of $\pi_\infty$. By \cite[\S II.5]{BorelWallach}, this implies that $\pi$ is cohomological with respect to the representation $V$. 

\begin{lemma} \label{SelfDualCoefficients}
The representation $V$ is self-dual. 
\end{lemma}
\begin{proof}
Let $\ell \notin S'$ be a split prime. Since $\pi_{0, f}$ admits a $\mbf{H}_{0}(\mbb{A}_f)$-linear model, the representation $\pi_{0, \ell}$ is $\Hb_{0} (\mbb{Q}_\ell)$-distinguished, therefore by \cite{ULP} 
\[
\pi_{0, \ell}^\vee \cong \pi_{0, \ell}.
\]
Let $\Pi = \mbf{1} \boxtimes \Pi_0$ be a weak base-change of $\pi$, which we have assumed to be cuspidal. Then, since duality is preserved under base-change, we have $\Pi_{0,\lambda} \cong \Pi_{0,\lambda}^{\vee}$ for any prime $\lambda$ of $E$ dividing $\ell$ (see Proposition \ref{WeakBCProp}). This implies that $\Pi_{0}$ is self-dual by the mild Chebotarev density theorem proven in \cite[Corollary B]{Ramakrishna}. In particular, this implies that $V$ is self-dual since the algebraic representation of $\opn{Res}_{\mbb{C}/\mbb{R}}\opn{GL}_{2n}$ corresponding to $\Pi_0$ has weight $(\mbf{c}, -\mbf{c}')$, where $\mbf{c} = (c_1, \dots, c_{2n})$ is the weight of $V^*$ and $\mbf{c}' = (c_{2n}, \dots, c_1)$ (with notation as in Proposition \ref{WeakBCProp}).
\end{proof}

\subsection{Ring class fields}

Let $\invs{R}$ be the set of all square-free products of rational primes that lie outside the set $S'$ and split in the extension $E/\mbb{Q}$ (we allow $1 \in \invs{R}$). For any positive integer $m$, let $\widehat{\OO}_{m} ^{\times}  \subset \Ab_{E, f}^{\times}$ be the compact open subgroup consisting of the units of the profinite completion of $ \OO_{m} = \ZZ +  m \OO_{E} $, and let $ D[m]  \subset \Tb(\Ab_{f})  $ denote the image of $  \widehat{\OO}_{m} ^ {\times } $ under the map
\begin{eqnarray} 
\Nm \colon \mbb{A}_{E, f}^{\times} & \to & \Tb(\Ab_{f}) \nonumber \\
 z & \mapsto & \bar{z}/z \nonumber 
\end{eqnarray}
By Lemma \ref{ACCFT}, $ D[m] $ is a compact open subgroup.

\begin{remark}
If $m \in \invs{R}$ and we set $D_\ell = \Nm \left( (\ordd_m \otimes \mbb{Z}_\ell)^{\times} \right)$ for $\ell | m$, then $D[m]$ decomposes as $\prod_{\ell | m} D_\ell \times \prod_{\ell \nmid m} \Nm \left( (\ordd_E \otimes \mbb{Z}_\ell)^{\times} \right)$. Furthermore, for any integer $r \geq 1$, the compact open subgroup $D[mp^r]$ decomposes as $D[mp^r]^p D_{p^r}$, where $D[mp^r]^p \subset \mbf{T}(\mbb{A}^p_f)$ and $D_{p^r} \subset \mbf{T}(\mbb{Q}_p)$ is as in \S \ref{TheFlagVariety}.
\end{remark}

Let $ E[\infty] $ be the compositum of all ring class extensions $ E[m] $ for $ m \geq 1 $ and let $ \mathrm{Res}_{E[\infty]} : \Gal(E^\mathrm{ab}/E) \to \Gal(E[\infty]/E) $ denote the restriction map. Since the infinite place of $ E $ is complex, the restriction of the Artin map $ \Art_{E}$ to the finite ideles is surjective onto the abelianisation of $G_E$. 

\begin{lemma} \label{TorusVariationLemma}
There exists a continuous surjective homomorphism 
\[
\mathrm{Art}_{0} \colon \Tb (\Ab_{f})  \to \Gal(E[\infty]/E) 
\]
with kernel $ \Tb(\QQ)$, which satisfies the following properties:   
\begin{itemize} 
\item One has the following equality of compositions
\[
\mathrm{Res}_{E[\infty]} \circ \Art = \Art_{0} \circ \Nm   
\]
\item For each $ m \geq 1 $, the map $ \Art_{0}$ induces an isomorphism $  \Tb(\Ab_{f})/\Tb(\QQ ) D[m] \cong  \Gal(E[m]/E) $.   
\item If $ \ell \nmid m $ is split in $ E $, then under the identification $ \Tb(\QQ_{\ell}) \cong \QQ_{\ell}^{\times}$, the map $\Art_{0}$ sends $\ell$ to the geometric Frobenius at $\lambda$ in $\Gal(E[m]/E)$, where $\lambda$ is the prime above $\ell$ distinguished by the fixed embedding $\overline{\mbb{Q}} \hookrightarrow \overline{\mbb{Q}}_\ell$ (the identification is also with respect to this fixed embedding). 
\end{itemize}
\end{lemma}
\begin{proof}  
It is well known (e.g. see \cite[Lemma 2]{khne2017intersections}) that $E[\infty]$ is the fixed field of the transfer map $\opn{Ver} \colon G_{\QQ}^{\mathrm{ab}} \to G_{E}^{\mathrm{ab}}$, so have a commutative diagram
\begin{center}   
          \begin{tikzcd}  1  \arrow[r] & \mbb{A}_f^{\times}  \arrow[r]   \arrow[d, "\Art_{\QQ}"]  &   \mbb{A}_{E, f}^{\times}  \arrow[r, "\Nm"]   \arrow[d, "\Art_{E}"] &  \Tb(\Ab_ {f} )    \arrow[r]    \arrow[d, "\Art_{0}"]   &   1 \\
          1  \arrow[r]  &  G_{\QQ}^{\mathrm{ab}}  \arrow[r, "\mathrm{Ver}"] &  G_{E} ^ {\mathrm{ab}}   \arrow[r, "\Res"]   &  \Gal(E[\infty]/E))  \arrow[r] & 1      
\end{tikzcd} 
\end{center}
with exact rows, by Lemma \ref{ACCFT}. The surjectivity of $ \Art_{0} $ follows from that of $ \Art_{E} $. The isomorphism in the second assertion is induced from restricting the isomorphism $  \Gal(E[m]/E)  \cong  \mbb{A}_{E, f}^\times /  E^{\times}   \widehat{\OO}_{m}^{\times} $ via $ \Nm $ and $ \Res $. The third assertion follows by tracking the uniformiser under $ \Res \circ \Art_{E}$.
\end{proof}

\subsection{Lattices and completions of pushforwards}

Let $V$ be the representation of $\mbf{G}_E$ with highest weight $\mbf{c} = (0; c_1, \dots, c_{2n})$ associated with $\pi$ as above, and let $\tilde{V} \defeq V \boxtimes \mbf{1}$ denotes its trivial extension to $\Gt_E$. Since $V$ is self-dual, we put ourselves in the situation of \S \ref{BranchingLawsAndReps}--\ref{FunctorsFromEt} and \S \ref{VerticalNormRelations} and we will freely use notation from these sections. In particular, we let $\mathscr{V}$, $\tilde{\mathscr{V}}$, $\mathscr{T}$ and $\tilde{\mathscr{T}}$ denote the \'{e}tale equivariant sheaves associated with the representations $V$, $\tilde{V}$, $T$ and $\tilde{T}$ respectively. The fixed parahoric vector in Assumption \ref{MainAssumption} gives rise to a ``modular parametrisation", i.e. an equivariant linear map
\[
v_{\varphi} \circ \pr_{\pi^{\vee}}  \colon   \mathrm{H}^{2n-1}_{\et}(\Sh_{\Gb,\overline{\QQ}}, \mathscr{V}(n)) \otimes \overline{\QQ}_{p}  \twoheadrightarrow W_{\pi}^*(1-n) 
\]
where $v_{\varphi}$ is evaluation at $\varphi$. Here we have used the fact that $V$ is self-dual, so $\mathscr{V}^* \cong \mathscr{V}$. Since we have assumed that $\pi$ is Siegel ordinary, this map is non-trivial when pre-composed with the map
\[
\opn{H}^{2n-1}_{\et}\left( \opn{Sh}_{\mbf{G}}(K^pJ)_{\overline{\mbb{Q}}}, \mathscr{V}(n) \right)_{\ide{t}_{\alpha}} \hookrightarrow \opn{H}^{2n-1}_{\et}\left( \opn{Sh}_{\mbf{G}}(K^pJ)_{\overline{\mbb{Q}}}, \mathscr{V}(n) \right) \to \opn{H}^{2n-1}_{\et}\left( \opn{Sh}_{\mbf{G},\overline{\mbb{Q}}}, \mathscr{V}(n) \right)
\]
where the first map is the splitting in Remark \ref{RemRationalLocalisation}. This is a necessary condition for the Euler system classes to be non-zero, but is not required for the actual construction of the classes because we will work with the pushforward in Remark \ref{RemVerticalRelThm}.

\begin{definition} \label{DefOfGStableLattice}
We let $T_{\pi}^*(1-n) \subset W_{\pi}^{*}(1-n)$ denote the image of the composition 
\begin{equation} \label{EqnLatticeDef}
\mathrm{H}^{2n-1}_{\et}(\Sh_{\Gb,\overline{\QQ}}, \mathscr{T}(n)) \otimes \overline{\mbb{Z}}_{p} \to \mathrm{H}^{2n-1}_{\et}(\Sh_{\Gb,\overline{\QQ}}, \mathscr{V}(n)) \otimes \overline{\mbb{Q}}_{p} \xrightarrow{v_{\varphi} \circ \pr_{\pi^{\vee}}} W_{\pi}^{*}(1-n)
\end{equation}
which is by construction a Galois stable $\overline{\mbb{Z}}_p$-submodule inside $ W_{\pi}^{*}(1-n) $. Let $\Phi$ be a finite extension of $\mbb{Q}_p$ such that:
\begin{itemize}
    \item $W^*_{\pi}(1-n)$ has a model over $\Phi$ (i.e. there exists finite-dimensional vector space $W^*_{\pi}(1-n)_{\Phi}$ over $\Phi$ with a continuous action of $\Gal(\overline{E}/E)$ such that $W^*_{\pi}(1-n)_{\Phi} \otimes_{\Phi} \Qpb = W^*_{\pi}(1-n)$)
    \item The map $v_{\varphi} \circ \opn{pr}_{\pi^\vee}$ is defined over $\Phi$.
\end{itemize} 
Let $\mathcal{O}$ denote the ring of integers of $\Phi$. Then we define $T^*_{\pi}(1-n)_{\ordd}$ to be the image of the composition:
\[
\mathrm{H}^{2n-1}_{\et}(\Sh_{\Gb,\overline{\QQ}}, \mathscr{T}(n)) \otimes \ordd \to \mathrm{H}^{2n-1}_{\et}(\Sh_{\Gb,\overline{\QQ}}, \mathscr{V}(n)) \otimes \Phi \xrightarrow{v_{\varphi} \circ \pr_{\pi^{\vee}}} W_{\pi}^{*}(1-n)_{\Phi} .
\]
This defines a Galois stable $\ordd$-lattice inside $W^*_{\pi}(1-n)_{\Phi}$ and satisfies the property $T^*_{\pi}(1-n)_{\ordd} \otimes_{\ordd} \overline{\mbb{Z}}_p = T^*_{\pi}(1-n)$. Since such an extension $\Phi$ always exists, this implies that $T^*_{\pi}(1-n)$ is a Galois stable $\overline{\mbb{Z}}_p$-lattice. For the majority of the time, we will work with Galois modules over $\Qpb$ or $\overline{\mbb{Z}}_p$, however sometimes we will need to work with their models over $\Phi$ or $\ordd$ (see Corollary \ref{CorForThmA} for example). We will always indicate the latter with an appropriate subscript as above.
\end{definition}

\begin{remark}
Note that $T^*_{\pi}(1-n)$ depends on $\varphi$ and the lattice $T \subset V_{\mbb{Q}_p}$, but by rescaling, it is enough to establish the main result of this article (Theorem \ref{TheoremA}) for a single Galois-stable lattice. Therefore, we have suppressed the dependence on $\varphi$ and $T$ from the notation.
\end{remark}

\begin{notation}
We denote the (Galois-equivariant) map 
    \[
    M_{G, \et, \mbb{Z}_p}\left(K^p J \right) \to T^*_{\pi}(1-n),
    \]
    defined as the composition of the pullback $M_{G, \et, \mbb{Z}_p}\left(K^p J \right) \to \mathrm{H}^{2n-1}_{\et}(\Sh_{\Gb,\overline{\QQ}}, \mathscr{T}(n)) \otimes \overline{\mbb{Z}}_p$ and the map in (\ref{EqnLatticeDef}), by $v_{\varphi} \circ \opn{pr}_{\pi^\vee}$.
\end{notation}

For every (sufficiently small) compact open subgroup $U \subset \mbf{H}(\mbb{A}_f)$, fix elements 
\[
\mbf{1}_{U} \in \opn{H}^0_{\mathrm{mot}}\left(\opn{Sh}_{\mbf{H}}(U), \mbb{Z} \right)
\]
that are compatible under pullbacks, such that their common image $\mbf{1} \in \opn{H}^0_{\mathrm{mot}}\left(\opn{Sh}_{\mbf{H}}, \mbb{Z} \right)$ is fixed by the action of $\mbf{H}(\mbb{A}_f)$. By abuse of notation, we also let $\mbf{1}_{U}$ denote the \'{e}tale realisation in $\opn{H}^0_{\et}\left(\opn{Sh}_{\mbf{H}}(U), \mbb{Z}_p \right)$.  

\begin{example}
One has $\opn{H}^0_{\mathrm{mot}}\left( \opn{Sh}_{\mbf{H}}(U), \mbb{Z} \right) = \opn{CH}^0\left(\opn{Sh}_{\mbf{H}}(U) \right)$ where the latter group is the abelian group of $0$-dimensional cycles up to rational equivalence. We can take
\[
\mbf{1}_U = [\opn{Sh}_{\mbf{H}}(U)]
\]
where $[\opn{Sh}_{\mbf{H}}(U)]$ denotes the fundamental cycle as in \cite[\S 1.5]{FultonIntTheory}. Indeed, if $h \in \mbf{H}(\mbb{A}_f)$ and $V$, $U$ are sufficiently small compact open subgroups satisfying $h^{-1}V h \subset U$, then the map $h \colon \opn{Sh}_{\mbf{H}}(V) \to \opn{Sh}_{\mbf{H}}(U)$ given by right-translation by $h$ is finite \'{e}tale and one has an equality $h^*[\opn{Sh}_{\mbf{H}}(U)] = [\opn{Sh}_{\mbf{H}}(V)]$ (see Lemma 1.7.1 in \emph{loc.cit.}). This implies that $\mbf{1}_U$ are compatible under pullbacks and their common image $\mbf{1}$ is invariant under the action of $\mbf{H}(\mbb{A}_f)$.
\end{example}

\begin{definition} \label{TwistedIwClassesDef}
Let $m \in \invs{R}$ and $g \in \tilde{G}^p$, and set $U_g = U_g^p \cdot \mbf{H}(\mbb{Z}_p)$ where 
\[
U_g^p \defeq g \left(K^p \times D[m]^p \right) g^{-1} \cap H^p.
\]
We define the following class 
\[
z_{g, m, \mathrm{Iw}}^{\circ} \defeq (v_{\varphi} \circ \opn{pr}_{\pi^\vee}) \circ \; [g]_* \circ \iota_{*, \mathrm{Iw}, \alpha} \left( \mbf{1}_{U_g} \right) \in \opn{H}^1_{\opn{Iw}} \left( E[m p^{\infty}], T_{\pi}^*(1-n) \right)    
\]
If $\phi = \sum_i a_i \opn{ch}(g_i (K^p \times D[m]^p) ) \in \invs{H}(\gt^p, \mbb{Z}_p)$ is an arbitrary element invariant under the action of $K^p \times D[m]^p$, then we have a well-defined element 
\[
z_{\phi, m, \mathrm{Iw}}^{\circ} \defeq \sum_i a_i z_{g_i, m, \mathrm{Iw}}^{\circ} \in \opn{H}^1_{\opn{Iw}} \left( E[m p^{\infty}], T_{\pi}^*(1-n) \right).  
\]
We let $z_{\phi, m, \mathrm{Iw}}$ denote the image of $z_{\phi, m, \mathrm{Iw}}^{\circ}$ in $\varprojlim_r \opn{H}^1\left(E[mp^r], W^*_{\pi}(1-n) \right)$. Note that although the level subgroup $U_g$ only depends on the projection of $g$ to $G^p$, the pushforward $[g]_*$ does depend on the component of $g$ in $T^p$.
\end{definition} 

\begin{remark} \label{RemAboutModelOfESclass}
Suppose that $\Phi$ is a finite extension of $\mbb{Q}_p$ with ring of integers $\mathcal{O}$, as in Definition \ref{DefOfGStableLattice}. Then the class $z_{\phi, m, \opn{Iw}}^{\circ}$ lies in the cohomology group $\opn{H}^1_{\opn{Iw}}\left(E[mp^{\infty}], T^*_{\pi}(1-n)_{\ordd} \right)$.
\end{remark}

The following proposition will allow us to prove norm compatibility relations for our Euler system classes.

\begin{proposition} \label{TwistedClassIwProp}
Keeping the same notation as in Definition \ref{TwistedIwClassesDef}, there exists a $\mbf{H}(\mbb{A}_f^p) \times \Gt(\mbb{A}_f^p)$-equivariant linear map
\[
\mathcal{L}_{m} \colon \invs{H}(\tilde{G}^p)^{1 \times D[m]^{p}} \longrightarrow \pi_f^\vee \; \otimes \;  \varprojlim_r \opn{H}^1\left( E[mp^r], W_{\pi}^* (1-n) \right) 
\]
where 
\begin{itemize}
    \item The action of $(h, \tilde{g}) \in \mbf{H}(\mbb{A}_f^p) \times \Gt(\mbb{A}_f^p)$ on $\xi \in \invs{H}(\tilde{G}^p)^{1 \times D[m]^{p}}$ is given by $(h, \tilde{g}) \cdot \xi(-) = \xi(h^{-1} \cdot - \cdot \tilde{g})$
    \item If we write $\tilde{g} = g \times t$ for $g \in \mbf{G}(\mbb{A}_f^p)$ and $t \in \mbf{T}(\mbb{A}_f^p)$, then $(h, \tilde{g})$ acts on 
    \[
    x \otimes y \in \pi_f^\vee \; \otimes \;  \varprojlim_r \opn{H}^1\left( E[mp^r], W_{\pi}^* (1-n) \right) 
    \]
    as $(h, \tilde{g}) \cdot (x \otimes y) = g \cdot x \otimes \sigma^{-1} \cdot y$, where $\sigma \in \Gal(E[mp^{\infty}]/E)$ is the unique element such that its image under the map 
    \begin{align*}
    \kappa \colon \Gal(E[mp^{\infty}]/ E) &\xrightarrow{\sim} \mbf{T}(\mbb{Q}) \backslash \mbf{T}(\mbb{A}_f) / D[m]^p \\
    \sigma = \opn{Art}_E(s) &\mapsto \overline{s}_f/s_f 
    \end{align*}
    is equal to the image of $t$ in the right-hand side
\end{itemize}
satisfying: 
\begin{enumerate}[(a)]
\item  For any $\phi^{(m)} = \sum_i a_i \opn{ch}\left(g_i (K^p \times D[m]^p) \right)$ in the Hecke algebra $\invs{H}(\tilde{G}^p, \mbb{Z}_{p})$, one has an equality
\[
z_{\phi^{(m)}, m, \mathrm{Iw}} = v_{\varphi} \circ \mathcal{L}_m (\psi^{(m)})
\]
where $\psi^{(m)} = \sum \frac{a_i}{\opn{Vol}(U_{g_i})} \opn{ch}\left(g_i (K^p \times D[m]^p) \right)$.
\item  For a prime $\ell$ such that $\ell m \in \invs{R}$, one has
\[
\opn{cores}^{E[\ell m p^{\infty}]}_{E[mp^{\infty}]} \invs{L}_{\ell m}(\psi^{(\ell m)}) = \invs{L}_m( \pr_{*} (\psi^{(\ell m)} )).
\]
where $ \pr _{ * }   : \mathcal{H}(\widetilde{G}^{p}) ^{1 \times D[\ell m]^{p} }   \to  \mathcal{ H } (  \widetilde{G} ^{p} ) ^ { 1 \times  D [m ] ^{ p } } $ is the natural norm map.  
\end{enumerate}
\end{proposition}

\begin{proof}
Since $1 \times D[m]^p$ lies in the centre of $\tilde{G}^p$, the completed pushforward induced from the pushforward in Remark \ref{RemVerticalRelThm} gives rise to a map
\begin{equation} \label{CompPushComposition}
\hat{\iota}_{*, \mathrm{Iw}, \alpha} \colon \widehat{M}_{H, \mbb{Q}_p}^p \otimes \invs{H}(\tilde{G}^p)^{1 \times D[m]^p} \longrightarrow \overline{N}_{\mathrm{Iw}, \mbb{Q}_p} \left( 1 \times D[m]^p \right) \to \varprojlim_r \opn{H}^1\left( E[mp^r], \varinjlim_{L^p} M_{G, \et, \mbb{Q}_p}(L^p J) \right) 
\end{equation}
where the last map is the natural one and $\alpha$ is the eigenvalue in Assumption \ref{MainAssumption}(5). We consider the following bilinear map:
\[
\invs{S}_m \colon \widehat{M}_{H, \mbb{Q}_p}^p \otimes \invs{H}(\tilde{G}^p)^{1 \times D[m]^p} \to \pi_f^\vee \; \otimes \;  \varprojlim_r \opn{H}^1\left( E[mp^r], W_{\pi}^* (1-n) \right) 
\]
defined as the composition
\begin{align*}
\widehat{M}_{H, \mbb{Q}_p}^p \otimes \invs{H}(\tilde{G}^p)^{1 \times D[m]^p} &\xrightarrow{\hat{\iota}_{*, \opn{Iw}, \alpha}} \varprojlim_r \opn{H}^1\left( E[mp^r], \varinjlim_{L^p} M_{G, \et, \mbb{Q}_p}(L^p J) \right)  \\
&\longrightarrow \varprojlim_r \opn{H}^1\left( E[mp^r], \opn{H}^{2n-1}_{\et}\left(\opn{Sh}_{\mbf{G}, \overline{\mbb{Q}}}, \mathscr{V}(n) \right) \otimes \Qpb \right)  \\
&\xrightarrow{\opn{pr}_{\pi^\vee}} \pi_f^\vee \; \otimes \;  \varprojlim_r \opn{H}^1\left( E[mp^r], W_{\pi}^* (1-n) \right)  
\end{align*}
where the middle map is induced from the pullback map to the direct limit. The map $\invs{S}_m$ forms an intertwining map in the sense of Definition \ref{DefinitionOfIntertwiningMap}, and we define $\invs{L}_m(\phi) \defeq \invs{S}_m(\mbf{1} \otimes \phi)$. The equivariance properties for $\invs{L}_m$ follow immediately from the fact that $\mbf{1}$ is invariant under the action of $\mbf{H}(\mbb{A}_f^p)$. Moreover, for part (b), notice that there are induced maps  $\overline{N}_{\mathrm{Iw}, \mbb{Q}_p} \left( 1 \times D[\ell m]^p \right) \to \overline{N}_{\mathrm{Iw}, \mbb{Q}_p} \left( 1 \times D[m]^p \right)$ coming from the pushforward of $ N_{\Iw,\QQ_{p}} $ in the $ T^{p}$ component. By Lemma \ref{ShapiroLemma}, the norm map corresponds to corestriction under the second morphism in (\ref{CompPushComposition}).

Let $\phi^{(m)} = \opn{ch}(g (K^p \times D[m]^p)) \in \invs{H}( \gt^p, \mbb{Z}_p )$. By Corollary \ref{CommDiagramCor} and Definition \ref{DefOfGStableLattice}, one obtains a commutative diagram:

\[
\begin{tikzcd}
\quad  \quad  \quad      &{M_{H, \mbb{Z}_p}^p\left( U_g^p \right)} \arrow[d, "{\iota_{*, \opn{Iw}, \alpha}}"'] \arrow[r]  & {\widehat{M}_{H, \mbb{Q}_p}^p} \arrow[ddd, "{\hat{\iota}_{*, \opn{Iw}, \alpha}}"] \arrow[dddd, to path={..controls ++(320:2)  and ++(10:5.5) ..   node[near start, sloped,above] {\scalebox{0.8}{$\mathcal{S}_{m}(  - \otimes  \phi^{(m)})$}}   (\tikztotarget)}] \\
& {N_{\opn{Iw}, \mbb{Z}_p}\left( g (K^p \times D[m]^p ) g^{-1} \right)} \arrow[d, "{[g]_*}"']& \\
&{N_{\opn{Iw}, \mbb{Z}_p}(K^p \times D[m]^p)} \arrow[d, equals]& \\
&{\opn{H}^1_{\opn{Iw}}\left( E[mp^{\infty}], M_{G, \et, \mbb{Z}_p}(K^p J) \right)} \arrow[dd, "v_{\varphi} \circ \opn{pr}_{\pi^\vee} "'] \arrow[r, "\opn{Vol}(U_g)"] & {\varprojlim_r \opn{H}^1\left( E[mp^r], \varinjlim_{L^p} M_{G, \et, \mbb{Q}_p}(L^p J) \right) } \arrow[d]    \\                      && {\pi_f^\vee \; \otimes \;  \varprojlim_r \opn{H}^1\left( E[mp^r], W_{\pi}^* (1-n) \right)} \arrow[d, "v_{\varphi}"]             \\
&{\opn{H}^1_{\opn{Iw}}\left(E[mp^{\infty}], T^*_{\pi}(1-n) \right)} \arrow[r, "\opn{Vol}(U_g)"]          & {\varprojlim_r \opn{H}^1\left( E[mp^r], W_{\pi}^* (1-n) \right) }                 
\end{tikzcd}
\]
Hence, we must have $z_{\phi^{(m)}, m, \opn{Iw}} = \opn{Vol}(U_g)^{-1} (v_{\varphi} \circ \invs{L}_m)(\phi^{(m)})$. The proof for general $\phi^{(m)}$ follows immediately from linearity.
\end{proof}

\subsection{The Euler system} \label{TheEulerSystemSection}

We continue with the same notation as in the previous sections. We consider the following test data for our Euler system: for $m \in \invs{R}$, define an element $\phi^{(m)} \in \invs{H}(\tilde{G}^p, \mbb{Z})^{K^p \times D[m]^p}$ as the product
\[
\phi^{(m)} \defeq \left( \bigotimes_{\substack{ \ell \notin S \\ \ell \nmid m }} \phi_{\ell, 0} \right) \otimes \left( \bigotimes_{\ell | m} \phi_\ell \right) \otimes \phi_S
\]
where 
\begin{itemize}
    \item $\phi_S \in \invs{H}(\tilde{G}_S, \mbb{Z})^{K_S \times \Nm ((\ordd_E \otimes \mbb{Z}_S)^{\times})}$ is fixed.
    \item $\phi_{\ell, 0} = \opn{ch}\left(K_{\ell} \times \mbf{T}(\mbb{Z}_\ell) \right) \in \invs{H}(\tilde{G}_\ell, \mbb{Z})$ (note that $\Nm((\ordd_E \otimes \mbb{Z}_\ell)^{\times}) = \mbf{T}(\mbb{Z}_\ell)$ for primes $\ell$ that do not ramify in $E/\mbb{Q}$, by the proof of Lemma \ref{ACCFT}).
    \item Using the notation in Theorem \ref{MainTheorem}, we define $\phi_\ell = \sum_{i=0}^n a_i \opn{ch}\left( (g_i, 1) (K_{\ell} \times D_\ell) \right)$, where the element $g_i$ is 
    \[
    g_i = 1 \times \tbyt{1}{\ell^{-1} X_i}{}{1} \; \; \in \; \; \opn{GL}_1(\mbb{Q}_\ell) \times \opn{GL}_{2n}(\mbb{Q}_\ell) = \mbf{G}(\mbb{Q}_\ell).
    \]
    The integers $a_i$ are defined to be 
    \[
    a_i \defeq (\ell - 1) \cdot b_i \cdot [\mbf{H}(\mbb{Z}_\ell) : V_{1, i}]^{-1}
    \]
    where $V_{1, i} = (g_i, 1) (K_\ell \times D_\ell ) (g_i, 1)^{-1} \cap H_\ell$ and we have chosen the model of $\mbf{G}_{\mbb{Q}_\ell}$ over $\mbb{Z}_\ell$ so that $K_{\ell} = \mbf{G}(\mbb{Z}_\ell)$ (this quantity $a_i$ is an integer by part (3) of Theorem \ref{MainTheorem}). 
\end{itemize}

\begin{definition}
Let $\pi_0$ be a cuspidal automorphic representation of $\mbf{G}_0$ satisfying Assumption \ref{MainAssumption}. For $m \in \invs{R}$, we define the \emph{Iwasawa Euler system class} to be 
\[
\invs{Z}^{\circ}_{m, \mathrm{Iw}} \defeq m^{-n^2} \cdot z^{\circ}_{\phi^{(m)}, m, \mathrm{Iw}} \
\; \; \;  \in \; \; \; \opn{H}^1_{\mathrm{Iw}}\left(E[mp^\infty], T_\pi^* (1-n) \right) .
\]
We set $\invs{Z}_{m, \mathrm{Iw}} \defeq m^{-n^2} \cdot z_{\phi^{(m)}, m, \mathrm{Iw}}$ which equals the image of $\invs{Z}^{\circ}_{m, \mathrm{Iw}}$ in $\varprojlim_r \opn{H}^1\left(E[mp^r], W^*_{\pi}(1-n) \right)$.    
\end{definition}
We now prove our main result of this article.      
\begin{theorem} \label{RelationForIwClasses}
Let $\ell, m \in \invs{R}$ such that $\ell$ is prime and their product satisfies $\ell m \in \invs{R}$. The Iwasawa Euler system classes satisfy
\begin{equation} \label{CoresRelationMainThm}
\opn{cores}^{E[\ell mp^\infty]}_{E[ m p^\infty]} \invs{Z}_{\ell m, \mathrm{Iw}} = P_{\lambda}(\opn{Fr}_\lambda^{-1}) \cdot \invs{Z}_{m, \mathrm{Iw}}
\end{equation}
where $\lambda$ is the unique prime of $E$ lying above $\ell$ fixed by the embedding $\overline{\mbb{Q}} \hookrightarrow \overline{\mbb{Q}}_\ell$, the polynomial $P_{\lambda}(X) = \opn{det}\left(1 - \opn{Frob}_{\lambda}^{-1} X | \rho_{\pi}(n) \right)$ is the reverse characteristic polynomial of a geometric Frobenius at $\lambda$, and $\opn{Fr}_{\lambda} \in \Gal\left(E[m p^{\infty}]/E \right)$ denotes the arithmetic Frobenius at $\lambda$.  
\end{theorem}
\begin{proof}
We will perform a series of reduction steps to reduce the proof to that of Theorem \ref{MainTheorem}.   \\

\noindent \textit{Step 1.} By Proposition  \ref{TwistedClassIwProp} (b), there exists a commutative diagram   
\begin{center}   
\begin{tikzcd} 
 \mathcal{H} ( \widetilde{G} ^{p} ) ^ { 1 \times D[ \ell m ]^{p} }  \arrow[r,  "\mathcal{L}_{\ell m}" ]   \arrow[d, "\pr_{*}", swap] &   \pi_f^{\vee} \otimes \varprojlim_r \opn{H}^1\left(E[\ell mp^r], W_{\pi}^*(1-n) \right)    \arrow[d,  "{\opn{cores}^{E[\ell mp^\infty]}_{E[ m p^\infty]}}"   ]   \\
\mathcal{H}(\widetilde{G}^{p} ) ^ { 1 \times D [ m ] ^{p} }  \arrow[r,  "\mathcal{L}_{m}" ]  &  \pi_{f}^{\vee} \otimes   \varprojlim_{r}  \mathrm{H}^{1} ( E[mp^{r} ] , W_{\pi}^{*}(1-n))   
\end{tikzcd} 
\end{center} 

and functions $\psi^{(\ell m)}$ and $\psi^{(m)}$ such that:
\begin{itemize}
    \item $\invs{Z}_{\ell m, \mathrm{Iw}} = \ell^{-n^2} m^{-n^2} (v_{\varphi} \circ \invs{L}_{\ell m})(\psi^{(\ell m)})$
    \item $\invs{Z}_{m, \mathrm{Iw}} = m^{-n^2} (v_{\varphi} \circ \invs{L}_{m})(\psi^{(m)})$.
\end{itemize}
By the diagram above, and cancelling the $m^{-n^2}$ factor, this reduces the theorem to proving the statement 
\begin{equation} \label{ReductionAfterStep1}
\left(v_{\varphi} \circ \invs{L}_m \right)(  
\pr_{*}  (\psi^{(\ell m)} ) ) = \ell^{n^2} P_{\lambda}(\opn{Fr}_{\lambda}^{-1}) \cdot \left(v_{\varphi} \circ \invs{L}_m \right)(\psi^{(m)} ).
\end{equation}
We note that both functions $\pr _{ * } ( \psi^{(\ell m)}  )$ and $\psi^{(m)}$ decompose as $ 
\pr_{*} ( \psi^{(\ell m)}_\ell   ) \otimes \psi^\ell   $ and $\psi^{(m)}_\ell \otimes \psi^\ell$ for some fixed element $\psi^\ell \in \invs{H}(\tilde{G}^{\ell p})$, where
\begin{eqnarray}
\opn{pr}_*\left(\psi^{(\ell m)}_{\ell}\right) & = & \sum_{i=0}^n \frac{a_i}{\opn{Vol}(V_{1, i})} \opn{ch}\left( (g_i, 1) (K_\ell \times \mbf{T}(\mbb{Z}_\ell)) \right) \quad \in \quad \invs{H}(\tilde{G}_{\ell}) \nonumber \\
\psi^{(m)}_\ell & = & \frac{1}{\opn{Vol}\left(\mbf{H}(\mbb{Z}_\ell ) \right)} \opn{ch}(K_\ell \times \mbf{T}(\mbb{Z}_\ell)) \quad \in \quad \invs{H}(\tilde{G}_\ell ) \nonumber 
\end{eqnarray}
using the explicit formulae in Proposition \ref{TwistedClassIwProp}(a) and at the start of section \ref{TheEulerSystemSection}, and the fact that $(K_\ell \times \mbf{T}(\mbb{Z}_\ell) ) \cap H_\ell = \mbf{H}(\mbb{Z}_\ell )$ (by our choice of models over $\mbb{Z}_\ell$).  \\
 
\noindent \textit{Step 2.} We now simplify the expression in (\ref{ReductionAfterStep1}) further by cutting out by a finite-order character. Since it is enough to prove the relation (\ref{ReductionAfterStep1}) at each finite layer of the inverse limit $\varprojlim_r \opn{H}^1\left(E[mp^r], W_{\pi}^*(1-n) \right)$, it is enough to show the relation in (\ref{ReductionAfterStep1}) after post-composing both sides by any Galois equivariant morphism
\begin{equation} \label{CharacterProj}
\varprojlim_r \opn{H}^1\left(E[mp^r], W_{\pi}^*(1-n) \right) \to \Qpb(\chi)
\end{equation}
where $\chi$ is a finite-order character of $\Gal(E[mp^{\infty}]/E)$. 

More precisely, let $\chi$ be a finite-order character of $\Gal(E[mp^{\infty}]/E)$. Since the Galois action and the action of $\mbf{T}(\mbb{A}_f^p) = T^p$ are intertwined by the character $\kappa^{-1}$ (see Lemma \ref{ShapiroLemma}), we can view $\chi$ as a character of $T^p$ by setting $\chi(t) = \chi(\sigma)$, where $\sigma \in \Gal(E[mp^\infty]/E)$ is an element such that $\kappa(\sigma^{-1})$ equals the image of $t$ in $\varprojlim_r \Delta(D[m]^p D_{p^r})$. This character factorises as $\chi_{\ell}\cdot \chi^{\ell}$, where $\chi_\ell$ is a character of $T_\ell$ and $\chi^{\ell}$ is a character of $T^{\ell p} = \mbf{T}(\mbb{A}_f^{\ell p})$.

Furthermore, the Frobenius element $\opn{Fr}_{\lambda}^{-1}$ satisfies $\kappa(\opn{Fr}_{\lambda}^{-1}) = \ell^{-1}$, so the action of $P_{\lambda}(\opn{Fr}_{\lambda}^{-1})$ is intertwined with multiplication by 
\[
P_{\lambda}(\chi_\ell(\ell) ) = Q_\lambda(\chi_\ell(\ell) \ell^{-n}) = L(\pi_\ell \otimes \chi_\ell, 1/2)^{-1}
\]
under the map in (\ref{CharacterProj}), where $Q_\lambda(X)$ is as in Theorem \ref{ExistenceOfGaloisRep}.

Let $\invs{L}_m^{\chi}$ denote the composition of $\invs{L}_m$ with the map $1 \otimes $(\ref{CharacterProj}). This is a $\mbf{H}(\mbb{A}^p_f) \times \Gt(\mbb{A}^p_f)$-equivariant linear map $\invs{H}(\tilde{G}^p)^{1 \times D[m]^p} \to (\pi_f \boxtimes \chi^{-1})^\vee$, and to prove the theorem, we are required to prove the statement
\begin{equation} \label{ReductionAfterStep2}
\left(v_{\varphi} \circ \invs{L}^{\chi}_m \right)(
\psi^{(\ell m)}) = \ell^{n^2}  L(\pi_\ell \otimes \chi_\ell, 1/2)^{-1} \left(v_{\varphi} \circ \invs{L}^{\chi}_m \right)(\psi^{(m)} )
\end{equation}
for all finite-order characters $\chi$. \\

\noindent \textit{Step 3.} We now reduce the statement in (\ref{ReductionAfterStep2}) to one in local representation theory. The map 
\begin{align*} 
\ide{Z} \colon \invs{H}(\tilde{G}_\ell) ^ { 1 \times D[m] ^{p } }  &\longrightarrow (\pi_\ell \boxtimes \chi_\ell^{-1})^\vee \\
 \upsilon &\mapsto \left( v_{\varphi^\ell} \circ \invs{L}_m^{\chi}\right)(\upsilon \otimes \psi^{\ell} )
\end{align*}
where $v_{\varphi^\ell} \colon (\pi_f^\ell \boxtimes (\chi^{\ell})^{-1})^\vee \otimes (\pi_\ell \boxtimes \chi_\ell^{-1})^\vee \to (\pi_\ell \boxtimes \chi_\ell^{-1})^\vee$ is the evaluation map at the vector $\varphi^\ell = \bigotimes_{q \notin S' \cup \{\ell\}} \varphi_q \otimes \varphi_S \otimes \varphi_p$, is $H_\ell \times \tilde{G}_\ell$-equivariant by construction.  Note that we have 
\[
(\ell - 1)^{-1} \cdot \opn{Vol}\left( \mbf{H}(\mbb{Z}_\ell) \right) \cdot \frac{a_i}{\opn{Vol}(V_{1, i})} = b_i
\]
where $b_i$ are the integers appearing at the start of section \ref{TheEulerSystemSection}. So by multiplying the relation in (\ref{ReductionAfterStep2}) by $(\ell - 1)^{-1} \cdot \opn{Vol}\left( \mbf{H}(\mbb{Z}_\ell) \right)$ and using the explicit formulae in Step 1, we are reduced to proving the relation:
\begin{equation} \label{ReductionAfterStep3}
\left( v_{\varphi_\ell} \circ \ide{Z} \right) \left( \sum_i b_i \opn{ch}\left( (g_i, 1) (K_\ell \times \mbf{T}(\mbb{Z}_\ell)) \right) \right) = \frac{\ell^{n^2}}{\ell - 1} L(\pi_\ell \otimes \chi_\ell, 1/2)^{-1} \cdot \left( v_{\varphi_\ell} \circ \ide{Z} \right)\left( \opn{ch}(K_\ell \times \mbf{T}(\mbb{Z}_\ell)) \right).
\end{equation} \\

\noindent \textit{Step 4.} We now apply Frobenius reciprocity. Let $\upsilon$ denote the element $\sum_i b_i \opn{ch}\left( (g_i, 1) (K_\ell \times \mbf{T}(\mbb{Z}_\ell)) \right)$. By Lemma \ref{Frobenius} (b),  the map $\ide{Z}$ corresponds to an element $ \mathfrak{z} \in \opn{Hom}_H(\pi_\ell \boxtimes \chi_\ell^{-1}, \mbb{C})$ (implicitly using the identification of $\Qpb$ and $\mbb{C}$), and the relation in (\ref{ReductionAfterStep3}) takes the form
\[
\mathfrak{z} (\upsilon \cdot \varphi_\ell ) = \frac{\ell^{n^2}}{\ell - 1} L(\pi_\ell \otimes \chi_\ell, 1/2)^{-1} \cdot \mathfrak{z} (\varphi_\ell)
\]
noting that $\varphi_\ell$ is spherical. Note that $\chi_{\ell}$ is unramified and, by Proposition \ref{WeakBCProp}(2), the local representation $\pi_{0, \ell}$ appears as the local component $\Pi_{0, \lambda}$ of the cuspidal automorphic representation $\Pi_0$ of $\opn{GL}_{2n}(\mbb{A}_E)$. By the work of Shalika \cite{ShalikaMultOne}, any such local component is generic. The relation above now follows from Theorem \ref{MainTheorem} and this completes the proof of Theorem \ref{RelationForIwClasses}.
\end{proof}

\begin{definition}
For $m \in \invs{R}$ and $r \geq 0$, we let $d_{mp^r}^{\circ}$ (resp. $d_{mp^r}$) denote the image of $\invs{Z}_{m, \mathrm{Iw}}^{\circ}$ (resp. $\invs{Z}_{m, \mathrm{Iw}}$) under the natural map 
\begin{align*}
\opn{H}^1_{\mathrm{Iw}}\left(E[mp^\infty], T_\pi^* (1-n) \right) &\to \opn{H}^1\left(E[mp^r], T_\pi^* (1-n) \right) \\ 
\Biggl( {\text{ resp. }} \varprojlim_r \opn{H}^1\left(E[mp^r], W_\pi^* (1-n) \right)  &\to \opn{H}^1\left(E[mp^r], W_\pi^* (1-n) \right) \Biggr)
\end{align*}
\end{definition}

As a corollary of Theorem \ref{RelationForIwClasses}, we obtain Theorem \ref{TheoremA} in the introduction:

\begin{corollary}[Theorem \ref{TheoremA}] \label{CorForThmA}
For $m \in \invs{R}$ and $r \geq 0$, there exist classes $c_{mp^r} \in \opn{H}^1\left( E[mp^r], T^*_{\pi}(1-n) \right)$ satisfying 
\[
\opn{cores}^{E[\ell m p^r]}_{E[mp^r]} c_{\ell m p^r} = \left\{ \begin{array}{cc} P_{\lambda} \left( \opn{Fr}_\lambda^{-1} \right) \cdot c_{m p^r} & \text{ if } \ell \neq p \text{ and } \ell m \in \invs{R} \\ c_{m p^r} & \text{ if } \ell = p  \end{array} \right.
\]
where $P_{\lambda}(X) \in \overline{\mbb{Z}}_p[X]$ and $\lambda$ are as in Theorem \ref{RelationForIwClasses}. Moreover, there exists a constant $C \in \mbb{Z}_{>0}$ such that $c_{mp^r} = C \cdot d_{mp^r}^\circ$ for all $m \in \invs{R}$ and $r \geq 0$. 
\end{corollary}
\begin{proof}
Let $\Phi$ be a finite extension of $\mbb{Q}_p$ with ring of integers $\mathcal{O}$, as in Definition \ref{DefOfGStableLattice}. Then by Remark \ref{RemAboutModelOfESclass}, the classes $d^{\circ}_{mp^r}$ and $d_{mp^r}$ are elements of the cohomology groups $\opn{H}^1\left(E[mp^r], T^*_{\pi}(1-n)_{\ordd} \right)$ and $\opn{H}^1\left(E[mp^r], W^*_{\pi}(1-n)_{\Phi} \right)$ respectively, and $d_{mp^r}$ is the image of $d_{mp^r}^{\circ}$ under the map
\begin{equation} \label{EqnIntCohtoRatCoh}
\opn{H}^1\left(E[mp^r], T^*_{\pi}(1-n)_{\ordd} \right) \to \opn{H}^1\left(E[mp^r], W^*_{\pi}(1-n)_{\Phi} \right)
\end{equation}
Since $W^*_{\pi}(1-n)_{\Phi}$ is absolutely irreducible (by assumption) and of $\Phi$-dimension $\geq 2$, one must have $\opn{H}^0\left(E^{\mathrm{ab}}, W^*_{\pi}(1-n)_{\Phi} \right) = 0$. Indeed, suppose for the sake of contradiction that  $v \in W^*_{\pi}(1-n)_{\Phi}$ is a non-zero vector fixed by $\Gal( \overline{E}/E^{\mathrm{ab}})$. As $W^*_{\pi}(1-n)_{\Phi}$ is irreducible, the $\Gal(\overline{E}/E)$-span of any non-zero vector (in particular, $v$) must be all of $W^*_{\pi}(1-n)_{\Phi}$. But this implies that the action of $\Gal(\overline{E}/E)$ on $W^*_{\pi}(1-n)_{\Phi}$ factors through $\Gal(E^{\mathrm{ab}}/E)$, and hence $ W^*_{\pi}(1-n) = W^*_{\pi}(1-n)_{\Phi} \otimes_{\Phi} \Qpb$ contains a one-dimensional subrepresentation -- a contradiction.

In particular, this implies that the cohomology group $\opn{H}^0\left(E^{\mathrm{ab}}, W^*_{\pi}(1-n)_{\Phi}/T^*_{\pi}(1-n)_{\ordd} \right)$ is finite of order $C \in \mbb{Z}_{>0}$ say (using the compactness of $T^*_{\pi}(1-n)_{\ordd} - \varpi T^*_{\pi}(1-n)_{\ordd}$, where $\varpi$ is a uniformiser of $\ordd$, one can show that elements of $W^*_{\pi}(1-n)_{\Phi}$ that are fixed by the Galois action modulo $T^*_{\pi}(1-n)_{\ordd}$ have to be universally bounded). Passing to the long exact sequence in group cohomology for the short exact sequence
\[
0 \to T^*_{\pi}(1-n)_{\ordd} \to W^*_{\pi}(1-n)_{\Phi} \to W^*_{\pi}(1-n)_{\Phi}/T^*_{\pi}(1-n)_{\ordd} \to 0
\]
one sees that $C$ kills the kernel of the map in (\ref{EqnIntCohtoRatCoh}). By Theorem \ref{RelationForIwClasses}, the classes $d_{mp^r}$ satisfy the Euler system relations (the relations in the statement of the corollary) so we can define $c_{mp^r} \defeq C \cdot d_{mp^r}^{\circ}$. 
\end{proof}

\subsection{Motivic interpretation}

It will come as no surprise to the reader that the Euler system classes $c_{mp^r}$ arise from elements in the motivic cohomology (see section \ref{NotationSection}) of $\opn{Sh}_{\Gt}$. More precisely, we consider the following definition.
\begin{definition} \label{DefOfTwistedMotivic}
Let $g \in \Gt(\mbb{A}_f)$ and $L \subset \Gt(\mbb{A}_f)$ a sufficiently small compact open subgroup. Set $U \defeq gLg^{-1} \cap \mbf{H}(\mbb{A}_f)$. Recall the definition of the embedding ``$\opn{br}$'' from Notation \ref{NoteFixedEmbedding}. Then we define 
\[
\invs{Z}_{g, \mathrm{mot}} \defeq ([g]_* \circ (\iota, \nu)_* \circ \opn{br} )(\mbf{1}_U) \in \opn{H}^{2n}_{\mathrm{mot}}\left(\opn{Sh}_{\Gt}(L), \tilde{\mathscr{V}} (n) \right)
\]
where, to simplify notation, we set $\tilde{\mathscr{V}} = \opn{Anc}_{\Gt, E}(\tilde{V})$. This definition can be extended by linearity to elements $\invs{Z}_{\phi, \mathrm{mot}}$ where $\phi \in \invs{H}(\Gt(\mbb{A}_f))^L$, by the same formula in Definition \ref{TwistedIwClassesDef}.
\end{definition}

Unwinding the definition of the classes $\invs{Z}_{m, \mathrm{Iw}}$ we have the following result (c.f. the proof of Theorem \ref{VerticalRelThm}):

\begin{proposition}
For $r \geq 1$, set $\phi_{mp^r} \defeq \mu(\tau)^r \phi^{(m)} \otimes \opn{ch}\left(\eta_r (J \times D_{p^r}) \right)$ where $\eta_r$ is the element defined in the proof of Proposition \ref{pushingintofs}. Then
\begin{enumerate}
    \item The image of $\invs{Z}_{\phi_{mp^r}, \mathrm{mot}}$ under the composition 
    \begin{eqnarray}
    \opn{H}^{2n}_{\mathrm{mot}}\left( \opn{Sh}_{\Gt}(K^pJ \times D[m]^pD_{p^r}), \tilde{\mathscr{V}} (n) \right) & \xrightarrow{r_{\et}} & \opn{H}^{2n}_{\et}\left( \opn{Sh}_{\Gt}(K^pJ \times D[m]^pD_{p^r}), \tilde{\mathscr{V}} (n) \right) \nonumber \\
    & \xrightarrow{m^{-n^2}(\invs{U}_S')^{-r}} & \opn{H}^{2n}_{\et}\left( \opn{Sh}_{\Gt}(K^pJ \times D[m]^pD_{p^r}), \tilde{\mathscr{V}} (n) \right)_{\ide{t}_{\alpha}} \nonumber 
    \end{eqnarray}
    is cohomologically trivial, where $r_{\et}$ is the \'{e}tale regulator and $\alpha$ is the eigenvalue in Assumption \ref{MainAssumption}(5).
    \item The element $d_{mp^r}$ equals the image of $m^{-n^2}(\invs{U}_S')^{-r} r_{\et}\left(\invs{Z}_{\phi_{mp^r}, \mathrm{mot}} \right)$ under the composition 
    \begin{eqnarray}
        \opn{H}^{2n}_{\et}\left( \opn{Sh}_{\Gt}(K^pJ \times D[m]^pD_{p^r}), \tilde{\mathscr{V}} (n) \right)_{\ide{t}_{\alpha}, 0} & \xrightarrow{\opn{AJ}_{\et, \ide{t}_\alpha}} & \opn{H}^1\left( E[mp^r], \opn{H}^{2n-1}_{\et}\left( \opn{Sh}_{\mbf{G}}(K^pJ), \mathscr{V} (n) \right)_{\ide{t}_{\alpha}} \right) \nonumber \\
        & \longrightarrow & \opn{H}^1\left( E[mp^r], \opn{H}^{2n-1}_{\et}\left( \opn{Sh}_{\mbf{G}}(K^pJ), \mathscr{V} (n) \right) \right) \nonumber \\
        & \xrightarrow{\opn{pr}_{\pi^\vee}} & \pi_f^{\vee} \otimes \opn{H}^1\left(E[mp^r], W_{\pi}^* (1-n) \right) \nonumber  \\
        & \xrightarrow{v_{\varphi}} & \opn{H}^1\left(E[mp^r], W_{\pi}^* (1-n) \right) \nonumber 
    \end{eqnarray}
    where the first map is the Abel--Jacobi map (and Shapiro's lemma) and the second is induced from the splitting $(\cdots)_{\ide{t}_{\alpha}} \hookrightarrow (\cdots)$.
\end{enumerate}
\end{proposition}

The fact that these classes come from motivic cohomology will allow us to investigate the image of our Euler system classes in other cohomology theories, e.g. in Deligne cohomology or syntomic cohomology as defined by Nekov\'{a}\v{r} and {Nizio\l} (\cite{NN}). We will investigate these questions in a future paper. We end with the following result showing that, under suitable conditions, the anticyclotomic Euler system classes actually live in the Bloch--Kato Selmer group (after inverting $p$).

\begin{proposition} \label{UnramifiedClassProp}
Let $\pi_0$ be a cuspidal automorphic representation satisfying Assumption \ref{MainAssumption} and let $S_{\mathrm{ns}} \subset S$ denote the subset of all primes which are either inert or ramified in $E/\mbb{Q}$. Then the image of the anticyclotomic Euler system classes $c_{mp^r}$ in $\opn{H}^1(E[mp^r], W_{\pi}^* (1-n) )$ are unramified at all primes $\lambda$ of $E[mp^r]$ which do not lie above a prime in $S_{\mathrm{ns}} \cup \{p\}$, i.e. the image of $c_{mp^r}$ under the map 
\[
\opn{H}^1(E[mp^r], W_{\pi}^* (1-n) ) \to \opn{H}^1(E[mp^r]_{\lambda}, W_{\pi}^* (1-n) ) \to \opn{H}^1(I_{\lambda}, W_{\pi}^* (1-n) )
\]
is trivial, where $I_{\lambda} \subset G_{E[mp^r]}$ denotes the inertia subgroup at $\lambda$. 

Furthermore, for any prime $\ide{P}$ of $E[mp^r]$ dividing $p$, the restriction of the class $c_{mp^r}$ at $\ide{P}$ lies in the Bloch--Kato local group $\opn{H}^1_f(E[mp^r]_{\ide{P}}, W_{\pi}^* (1-n) )$.
\end{proposition}
\begin{proof}
It is enough to prove the statement for $r \geq 1$ and for the classes $d_{mp^r}$. In this case, if $\lambda$ is a prime not dividing a rational prime in 
\[
\widetilde{S} \defeq S \cup \{ \text{ primes dividing } mp^r \}
\]
then the class $d_{mp^r}$ is constructed from sub-Shimura varieties which have hyperspecial level outside $\widetilde{S}$. By \cite{LanArithmetic}, these Shimura varieties have smooth integral models over $\ordd_E[\widetilde{S}^{-1}]$, which implies that the class $d_{mp^r}$ is unramified at $\lambda$.

If $\lambda$ is a prime lying above $\ell \in \widetilde{S}$, with $\ell \neq p$ and split in $E/\mathbb{Q}$, then the decomposition group of $\lambda$ in $E[mp^{\infty}]/E[mp^r]$ is infinite. Indeed, it is enough to check that the prime of $E$ lying below $\lambda$ doesn't split completely in the anticyclotomic $\mbb{Z}_p$-extension of $E$, and this is standard. In this case the result follows from \cite[Corollary B.3.5]{rubin}, since the classes $d_{mp^r}$ are universal norms in this extension. 

The comparison of syntomic cohomology and \'{e}tale cohomology in \cite{NN} and the fact that the classes come from motivic classes imply that, at primes dividing $p$, the classes $d_{mp^r}$ lie in the local group $H^1_g$. A criterion for when $H^{1}_g$ and $H^1_f$ coincide is that $p^{-1}$ is not an eigenvalue of the crystalline Frobenius $\varphi$ on $\dcris{W_{\pi}^* (1-n)}$. By part (3) of Theorem \ref{ExistenceOfGaloisRep} the eigenvalues of $\varphi$ are given by $p^{n-1}\alpha_i^{-1}$ where 
\[
L(\Pi_v, s + (1-2n)/2 )^{-1} = \prod_{i=1}^{2n} (1 - \alpha_i p^{-s} ).
\]
Now $\alpha_i$ are Weil numbers of weight $2n-1$, because the representation $\Pi$ satisfies the Ramanujan--Petersson conjecture (see \cite[Corollary 8.4.9]{Morel}). This implies that the eigenvalues of $\varphi$ are Weil numbers of weight $-1$, hence $p^{n-1}\alpha_i^{-1}$ cannot equal $p^{-1}$. 
\end{proof}

\begin{corollary} \label{BKClassCor}
Suppose that the weak base-change $\Pi$ is unramified outside a finite set of primes that split in $E/\mbb{Q}$. Then 
\[
c_{mp^r} \in \opn{H}^1_f\left(E[mp^r], W_{\pi}^*(1-n) \right)
\]
i.e. the classes live in the Bloch--Kato Selmer group after inverting $p$.
\end{corollary}
\begin{proof}
Let $\mathfrak{L}$ be a prime of $E[mp^r]$ lying above a prime $\ell$ that ramifies in $E/\mbb{Q}$, and let $\lambda$ denote the prime of $E$ lying below $\mathfrak{L}$. By assumption, the representation $\Pi_{\lambda}$ is unramified, so \cite[Theorem 3.2.3]{ChenevierHarris} implies that the representation $\rho_{\pi}$ is unramified at $\lambda$ and the eigenvalues $\beta_1, \dots, \beta_{2n}$ of $\opn{Frob}_{\lambda}^{-1}$ acting on $\rho_{\pi}^*(1-n)$ satisfy
\[
L(\Pi_{\lambda}, s + (1-2n)/2)^{-1} = \prod_{i=1}^{2n} \left(1 - \beta_i^{-1} (\opn{Nm} \lambda)^{s + n -1}\right) = \prod_{i=1}^{2n} \left(1 - \beta_i^{-1} \ell^{s + n -1}\right).
\]
As in the proof of Proposition \ref{UnramifiedClassProp}, the quantities $\ell^{n-1}\beta^{-1}_i$ are Weil numbers of weight $2n-1$, and since $\opn{Frob}_{\mathfrak{L}}$ is a power of $\opn{Frob}_{\lambda}$, we see that the eigenvalues of $\opn{Frob}_{\mathfrak{L}}$ acting on $\rho_{\pi}^*(1-n)$ cannot possibly equal $1$. This implies that
\begin{equation} \label{H0Vanish}
\opn{H}^0\left( E[mp^r]_{\mathfrak{L}}, W_{\pi}^*(1-n) \right) = 0 .
\end{equation}
Since $\Pi \cong \Pi^{\vee} \cong \Pi^c$ we have $W_{\pi}(n) \cong W_{\pi}^*(1-n)$ (see the proof of Lemma \ref{SelfDualCoefficients} and Theorem \ref{ExistenceOfGaloisRep}). Combining this with (\ref{H0Vanish}), local Tate duality (\cite[Theorem I.4.1]{rubin}) and the Euler--Poincar\'{e} characteristic for local fields (\cite[Theorem 7.3.1]{Neukirch} and \cite[Corollary 4.6.10]{nekovar-selmer}) implies 
\[
\opn{H}^1\left( E[mp^r]_{\mathfrak{L}}, W_{\pi}^*(1-n) \right) = 0 
\]
so in particular, any local condition at $\mathfrak{L}$ is vacuous. The result then follows from Proposition \ref{UnramifiedClassProp}.  
\end{proof}

\begin{remark}
The conditions in Corollary \ref{BKClassCor} are completely analogous to the case of Heegner points. Indeed, in this setting one considers an elliptic curve whose conductor is divisible only by primes that split in the imaginary quadratic extension. 
\end{remark}


\appendix


\section{Continuous \'{e}tale cohomology}   \label{TheAppendix} 
In this appendix, we explicate in a general setting two constructions in Jannsen's continuous \'{e}tale cohomology that we have made use of in this article. The first of these are the ``pushforward functors'' -- compositions of certain Gysin and trace morphisms. In \cite[\S 2.1]   {KLZ2015}, these were defined in cases when continuous \'{e}tale cohomology coincides with the usual definition of ($p$-adic) \'{e}tale cohomology. The second is localised versions of the ($p$-adic) \'{e}tale Abel--Jacobi map arising from the Hochschild-Serre spectral sequence, where we establish certain functoriality properties of these maps.

\subsection{Pushforwards maps}   

For any abelian category $\mathcal{A}$, we let $ D^{+}(\mathcal{A}) $ denote the derived category of left-bounded complexes in $ \mathcal{A}$, and if $ \mathcal{C} $ is any category, we denote by $ \mathcal{C} ^{\mathbb{N}} $ the category of inverse systems of objects in $ \mathcal{C} $. If $ \mathcal{C} $ is Grothendieck  abelian, so is $ \mathcal{C}^{\mathbb{N}}  $, in which case $ D ^{+}(\mathcal{C}^{\mathbb{N}})  =  D^{+}(\mathcal{C})^{\mathbb{N}}$.  Moreover, if $   F :  \mathcal{C}  \to  \mathcal{D} $ is a left exact functor, then so is $ F^{\mathbb{N}  } : \mathcal{C} ^{\mathbb{N}}  \to \mathcal{D}  ^ { \mathbb{N}} $, and if $  \mathcal{C} $ has enough injectives, so does $ \mathcal{C}^{\mathbb{N}} $, in which case the right derived functors of $ F $ and $ F^{\mathbb{N} } $ are related by $ R ^{i} ( F^{\mathbb{N}  } )  =  ( R ^{i}  F   ) ^ {  \mathbb{N} }  $  \cite[Proposition 1.2]{Jannsen1988}.  If $ X $ is a scheme, we denote by $ \mathbf{Sh}(X_{\et}) $ the category of \'{e}tale sheaves of abelian groups on the small \'{e}tale site $ X_{\et} $. We let $ \Gamma_{X} : \mathbf{Sh}(X_{\et})^ {\mathbb{N}}  \to \mathbf{Ab} $ denote the functor $ \mathscr{F} = (\mathscr{F}_{n}) \mapsto \varprojlim_{n} H^{0}(X,\mathscr{F}_{n}) $. Then, following \cite[\S 3]{Jannsen1988}, the continuous \'{e}tale cohomology of $\mathscr{F}$ is defined to be 
\[
\cohom{j}{\et}(X,\mathscr{F}) \defeq R^{j}\Gamma_{X} (\mathscr{F}).
\]

\begin{lemma}   \label{A.1}         
Let $ X $, $ Y $ be quasi-compact, quasi-separated (qcqs) schemes. For any quasi-finite separated morphism $ f \colon  X \to Y $, there exists a pair of functors 
\[
f_! \colon \mathbf{Sh}(X_{\et})^{\mathbb{N}}  \to  \mathbf{Sh}(Y_{\et})^{\mathbb{N}} \quad \quad  f^{!}  \colon    \mathbf{Sh}(Y_{\et})^{\mathbb{N}} \to  \mathbf{Sh}(X_{\et})^{\mathbb{N}} 
\]
satisfying the following properties:    
\begin{enumerate}
\item $ f_{!} $ is exact and  $ f_{!}  \dashv f^{!}  $ i.e. for any $ \mathscr{F} \in   \mathbf{Sh}({X_{\et})}^{\mathbb{N}} $ and $  \mathscr{G} \in  \mathbf{Sh}(Y_{\et})^{\mathbb{N}} $,  there are functorial isomorphisms   
\[
\Hom (f_{!} \mathscr{F},    \mathscr{G}  )   \xrightarrow{\sim}   \Hom ( \mathscr{F} ,   f^{!} \mathscr{G}). 
\]
\item If $ f =  g \circ h $ with $ g $, $ h $ both quasi-finite and separated, then there are canonical isomorphisms $  f_{!} \cong g_{!} \circ h_{!} $ and $ f^{!} \cong h ^{!} \circ g^{!} $. Moreover, these isomorphisms are compatible with any third such composition. 
\item $ f_{!} $ is a sub-functor of $f_{*} $ and if $ f $ is proper, $ f_{!} = f_{*}   $.
\item  $ f^{!} $ is a sub-functor of $f^{*} $ and if $ f $ is \'{e}tale, $ f^{!} = f^{*} $.
\end{enumerate}
\end{lemma}    

\begin{proof} 
See Propositions 3.1.4 and 3.1.8 in \cite[Exp. XVIII]{SGA4} and Proposition 6.1.4 and Theorem 5.1.8 in \cite[Exp. XVII]{SGA4}. In a more general setting, these are also found in  \cite[\href{https://stacks.math.columbia.edu/tag/0F4U}{Tag 0F4U}]{stacks-project}. 
\end{proof}  

In \cite{EkedahlLadic}, Ekedahl introduced an adic formalism in which the derived versions of the aforementioned functors behave well. We will need this formalism to establish certain functoriality properties of our constructions. For the rest of this section, we take $S$ to be a regular qcqs scheme of dimension $0$ or $1$, and we let $p$ denote a prime number which is invertible on $S$. We denote by $\opn{Sch}_S$ the category of schemes that are separated and of finite-type over $S$. Morphisms in $\opn{Sch}_S$ are just morphisms of schemes over $S$.

\begin{definition} \label{DefOfEtCategory}
Let $X \in \opn{Sch}_S$. 
\begin{enumerate}
\item The \emph{Artin--Rees category} $\mbf{Sh}(X_{\et})^{\mbb{N}}_{\opn{AR}}$ of $\mbf{Sh}(X_{\et})^{\mbb{N}}$ is the category whose objects are the same as those of $\mbf{Sh}(X_{\et})^{\mbb{N}}$, but whose morphisms are given by 
\[
\opn{Hom}_{\mbf{Sh}(X_{\et})^{\mbb{N}}_{\opn{AR}}}(\mathscr{F}, \mathscr{G}) \defeq \varinjlim_{n \geq 1} \opn{Hom}_{\mathbf{Sh}(X_{\et})^{\mathbb{N}}}\left( \mathscr{F}\{ n \}, \mathscr{G} \right) .
\]
Here, $(\mathscr{F}\{ n \})_i \defeq \mathscr{F}_{n+i}$ with the natural induced transition maps.
\item An  \emph{\'{e}tale $ \mathbb{Z}_{p} $-sheaf}, or  a \emph{$ p $-adic \'{e}tale sheaf} $ \mathscr{F}  \in  \mathbf{Sh}  (X _{\et }   ) ^   {   \mathbb{N }  }   $ is an inverse system $ \left\{ \mathscr{F}_ n\right\} _{n\geq 1} $ 
where   
\begin{itemize} 
\item $\mathscr{F}_ {n} $ is a constructible $ \mathbb{Z}/p^ n\mathbb{Z}$-module on $X_{\et} $, 
\item the transition maps $ \mathscr{F}_{n+1}\to \mathscr{F}_ n $  induce isomorphisms  $  \mathscr{F}_{n+1} \otimes _{(\mathbb{Z}/ p   ^{n+1}\mathbb{Z})} (\mathbb{Z}/ p ^ n\mathbb{Z} ) \cong \mathscr{F}_ n    $. 
\end{itemize}
We  say  that  $  \mathscr{F} $ is  \emph{lisse} or that $ \mathscr{F} $ is  a \emph{local system}  if each $ \mathscr{F}_{n} $ is  locally  constant. A morphism of such sheaves is just a morphism of these  objects in  $\mathbf{Sh}(X_{\et})^{\mathbb{N}}$. We denote the category of \'{e}tale $\mbb{Z}_p$-sheaves on $X$ by $\opn{\acute{E}t}(X)_{\mbb{Z}_p}$, and its isogeny category by $\opn{\acute{E}t}(X)_{\mbb{Q}_p}$. The objects of the latter will be referred to as \'{e}tale $\mbb{Q}_p$-sheaves.\footnote{In concrete terms, the objects of $\opn{\acute{E}t}(X)_{\mbb{Q}_p}$ are the same as those in $\opn{\acute{E}t}(X)_{\mbb{Z}_p}$ but the morphisms are $\opn{Hom}_{\opn{\acute{E}t}(X)_{\mbb{Q}_p}}\left(\mathscr{F}, \mathscr{G}\right) \defeq \opn{Hom}_{\opn{\acute{E}t}(X)_{\mbb{Z}_p}}\left(\mathscr{F}, \mathscr{G}\right) \otimes_{\mbb{Z}_p} \mbb{Q}_p$. The continuous cohomology of a sheaf $\mathscr{F} = (\mathscr{F}_n)_{n \geq 1} \in \opn{\acute{E}t}(X)_{\mbb{Q}_p}$ is then defined to be the $\mbb{Q}_p$-vector space $\opn{H}^i_{\et}\left(X, \mathscr{F} \right) \defeq R^i \Gamma_X (\mathscr{F}) \otimes_{\mbb{Z}_p} \mbb{Q}_p$.}
\end{enumerate}
\end{definition}

Let $\opn{\acute{E}t}(X)_{\mbb{Z}_p, \opn{AR}}$ denote the full subcategory of $\mbf{Sh}(X_{\et})^{\mbb{N}}_{\opn{AR}}$ spanned by projective systems $\mathscr{F}$ which are isomorphic in $\mbf{Sh}(X_{\et})^{\mbb{N}}_{\opn{AR}}$ to an \'{e}tale $\mbb{Z}_p$-sheaf. This is an abelian category by \cite[Proposition I.12.11]{FKWeil}. It is easily seen that the natural functor $\opn{\acute{E}t}(X)_{\mbb{Z}_p} \to \opn{\acute{E}t}(X)_{\mbb{Z}_p, \opn{AR}}$ is an isomorphism of categories (see \cite[p. 123]{FKWeil}), so by transport of structure, we view $\opn{\acute{E}t}(X)_{\mbb{Z}_p}$ as an abelian category. Similarly, $\opn{\acute{E}t}(X)_{\mbb{Q}_p}$ is isomorphic to $\opn{\acute{E}t}(X)_{\mbb{Q}_p, \opn{AR}}$ (the isogeny category of $\opn{\acute{E}t}(X)_{\mbb{Z}_p, \opn{AR}}$).

Let $D^b_c(X, \mbb{Z}_p)$ denote the bounded ``derived category'' of (constructible) $p$-adic \'{e}tale sheaves as constructed in \cite[Theorem 6.3]{EkedahlLadic} (see also \cite{DeligneLAdic, BhattScholzeLAdic}). This has a canonical $t$-structure whose heart is $\opn{\acute{E}t}(X)_{\mbb{Z}_p}$ (with abelian structure as above) but note that this is \emph{not} the bounded derived category of $\opn{\acute{E}t}(X)_{\mbb{Z}_p}$ in the usual sense. Nevertheless, $D^b_c(-, \mbb{Z}_p)$ satisfies a $6$-functor formalism and the functors $f_*, f^*, f_!, f^!, \Gamma_{-}$ induce derived functors $Rf_*, Rf^*, Rf_!, Rf^!, R\Gamma_{-}$ on $D^b_c(-, \mbb{Z}_p)$. Moreover, $R^i\Gamma_X = \opn{Hom}_{D^b_c(X, \mbb{Z}_p)}(\mbb{Z}_p, -[i])$ applied to $\mathscr{F} \in \opn{\acute{E}t}(X)_{\mbb{Z}_p}$ recovers Jannsen's continuous \'{e}tale cohomology of $\mathscr{F}$ in degree $i$ (see \cite[Lemma 4.1]{HuberLAdic}). We have similar properties for $D^b_c(X, \mbb{Q}_p)$ (the localisation of $D^b_c(X, \mbb{Z}_p)$ at the full subcategory of objects that are annihilated by some power of $p$).

\begin{definition}  
Let $ X $, $ Y $ be schemes, $ \mathscr{F}  \in  \mathbf{Sh}(X_{\et})^{\mathbb{N}}$ and $\mathscr{G}    \in   \mathbf{Sh}(Y_{\et})^{\mathbb{N}}$.  A \emph{pushforward} $ (f,\phi)_{*}   \colon (X,\mathscr{F}) \to (Y, \mathscr{G}) $ is a morphism $ f \colon X \to Y $ of schemes and a morphism $ \phi : \mathscr{F} \to f^{*} \mathscr{G} $ of sheaves on $ X $ (in the category $\mbf{Sh}(X_{\et})^{\mbb{N}}$).  
\end{definition}           

\begin{definition}   \label{pushschemes}   
Let $ S $ be as above. An \emph{\'{e}tale smooth $ S $-pair of codimension $c $} is a morphism $ f \colon X \to Y $ in $\opn{Sch}_S$ of smooth schemes $X, Y$ over $S$ satisfying the following condition: there exists a scheme $\bar{Y} \in \opn{Sch}_S$ smooth over $S$ and a factorisation $ f : X \xrightarrow{h}  \bar{Y} \xrightarrow{g} Y $ in $\opn{Sch}_S$ such that $ h $ is a closed immersion with fibres over each point of $ S $ of  pure codimension $ c $ in $  \bar{Y}  $ and $ g $ is \'{etale}. If $ f ' \colon X'  \to  Y' $ is another such pair, then a \emph{morphism} from $ f $ to $ f ' $ is a pair of \'{e}tale maps $ p \colon X ' \to X $ and $ q \colon Y' \to Y $ in $\opn{Sch}_S$ that commute with $ f $, $ f' $. 
\end{definition}

\begin{prop}   \label{contpush}  
Let $f \colon X \to Y $ be an \'{e}tale smooth $ S $-pair of codimension $ c $ and let $ \mathscr{F} $, $\mathscr{G}$ be \'{e}tale  $ \ZZ_{p} $-sheaves on $ X $, $ Y $ respectively. Then for any pushforward $ (f,\phi)_{*} : (X,\mathscr{F}) \to (Y,\mathscr{G}) $ and any $ j \in \ZZ_{\geq 0} $, there is an induced ``pushforward'' on cohomology  
\[
(f, \phi)_{*}   \colon   \cohom{j}{\et} (X, \mathscr{F} )  \to  \cohom{j+2c}{\et}(Y, \mathscr{G}(c))  
\]
that is functorial and Cartesian: 
\begin{itemize}
    \item(Functoriality) if $ f' \colon X ' \to Y' $ is another such pair and $ (p,q) \colon f \to f' $ is a morphism, then we have a commutative diagram 
\begin{center}
    \begin{tikzcd}   \cohom{j}{\et}(X', p^{*} \mathscr{F} )  \arrow[r,"{(f',\phi')_*}"]   \arrow[d,"\mathrm{Tr}_{p}", swap]  &   \cohom{j+2c}{\et}(Y', q^{*}\mathscr{G}(c))  \arrow[d,"{\mathrm{Tr}_{q}}"] \\
    \cohom{j}{\et}(X,\mathscr{F})  \arrow[r,"{(f,\phi)_*}"]  &  \cohom{j+2c}{\et}(Y, \mathscr{G}(c))   
\end{tikzcd}
\end{center}   
where $ \phi' = p^{*} \phi $.
\item (Cartesian) if $ q : Y' \to  Y  $  is  any  finite  \'{e}tale   morphism,     $ X' =  X \times_{Y}  Y' $, $ f ' : X' \to Y ' $, $ p : X' \to  X   $ are the natural morphisms, then we have a commutative diagram
\begin{center}
    \begin{tikzcd}   \cohom{j}{\et} (X',p^{*} \mathscr{F}) \arrow[r,"{(f',\phi')_{*}}"]  &   \cohom{j+2c}{\et}(Y', q^{*}\mathscr{G}(c))    \\
       \cohom{j}{\et} (X, \mathscr{F})  \arrow[u, "p^{*}"] \arrow[r,"{(f,\phi)_{*}}"]  &   \arrow[u,"q^{*}",swap]   \cohom{j+2c}{\et}(Y,\mathscr{G}(c))   
    \end{tikzcd}
\end{center}
where $ \phi' = p^{*} \phi $.
\end{itemize}
\end{prop} 
When $ f $ is a closed immersion, $ (f, \phi)_{*}   $    is defined in \cite[Theorem 3.17]{Jannsen1988}. The proof below is along similar lines except that we account for the independence of the factorisation of $ f $ by working in $D^b_c(-, \mbb{Z}_p)$.       
\begin{proof}  
Recall that $ p $ is assumed to be invertible on $S $. Let $ f = g \circ h  : X \xrightarrow{h}  \bar{Y}  \xrightarrow{g}  Y $ be a factorisation.  As $ g $ is   \'{e}tale,  $ g^{!} = g^{*} $ by Lemma \ref{A.1}   (4).     By \cite[Eqn.  3.20]{Jannsen1988},  $  R h^{!}  \ZZ_{p}(c) = \ZZ_{p} [-2c]   $   canonically. By equation (3.19) of \emph{op.cit.}, for any $  \bar{\mathscr{G}} \in \opn{\acute{E}t}(\bar{Y})_{\mbb{Z}_p}$, $ R h ^{!}  \bar{\mathscr{G}}(  c )    =  h  ^{*}   \bar { \mathscr{G } } (c)   \otimes   Rh^{!} \ZZ_{p} (c) =   h ^{*} \bar{\mathscr{G}} [-2c]  $. Applying this to $ \bar{\mathscr{G}} = g ^{*}  \mathscr{G} $, we see   that        
\[
Rf^{!}  \mathscr{G}(c) = Rh^{!} \circ R g^{!} \, \mathscr{G}(c) =  R h  ^{!} (   g    ^{*}\mathscr{G}(c)) =  f^{*} \mathscr{G}[-2c] 
\]
Deriving $  \mathrm{Tr}_{f}  \colon  \Gamma_{X} f^{!} \to \Gamma _{Y} $, we obtain   a natural transformation $ R ({\mathrm{Tr} _{f} })  :  R ( \Gamma  _  { X } f^{!} ) = R \Gamma_{X} \circ Rf^{!}  \to  R \Gamma_{Y} $ between the two derived functors on $D^b_c(Y, \mbb{Z}_p)$.   Evaluating  $ R (  \mathrm{Tr} _ { f }  ) $  at   $   \mathscr{G} ^ {   * }   (  c  )    $, we obtain an induced  map $   \cohom{j}{\et} (X,f^{*}\mathscr{G}) \cong 
R^{j+2c}\Gamma_{X}(Rf^{!} (\mathscr{G}(c) )  ) \xrightarrow{\mathrm{Tr_{f}}} \cohom{j+2c}{\et}(Y, \mathscr{G}(c)   )    $  via  passage  to  cohomology.   The pushforward map is now defined to be 
$$ (f, \phi)_{*} = \cohom{j}{\et}(X,\mathscr{F}) \xrightarrow{\phi}      \cohom{j}{\et}(X,f^{*}\mathscr{G})  \xrightarrow{\mathrm{Tr}_{f}} \cohom{j+2c}{\et}(Y,\mathscr{G}(c))  .  $$  Lemma \ref{A.1} (2)   then gives the  independence of this map on the choice of $ \bar{Y} $.    The functoriality of these pushforwards follows from that of the trace, i.e. $R\mathrm{Tr}_{f} \circ R\mathrm{Tr}_{p} =  R   \mathrm{Tr}_{q} \circ R \mathrm{Tr}_{f'}$. The Cartesian property  follows  by  applying  proper  base  change (\cite[Tag 095S]{stacks-project})  to  $ q : Y' \to  Y $.   
\end{proof}

\subsection{Localised Abel--Jacobi maps} \label{AbelJacobi}

In this section we let $\Lambda = \mbb{Z}_p$ or $\mbb{Q}_p$ and take $E$ to be a field of characteristic $0$. All schemes in this section are assumed to be finite-type, separated over $\opn{Spec}E$; if $X$ is such a scheme, then we denote by $\overline{X} = X \times_{\Spec E}\Spec \overline{E}$ its base change to a fixed separable closure of $E$.

\begin{theorem}[Jannsen, Deligne] \label{JannsenDeligne} 
Let $\mathscr{F}$ be a lisse \'{e}tale $\Lambda$-sheaf on $X$. There is a first quadrant cohomological spectral sequence 
\[
E^{i,j}_{2} \colon \opn{H}^i\left(E, \opn{H}^j_{\et}\left(\overline{X}, \mathscr{F} \right) \right) \Rightarrow \opn{H}^{i+j}_{\et}\left( X, \mathscr{F} \right)
\]
which we refer to as the Hochschild--Serre spectral sequence.
\end{theorem}

The Hochschild--Serre spectral sequence enjoys the follows properties:

\begin{proposition} \label{FuncPropsOfHSseq}
The Hochschild--Serre spectral sequence is functorial with respect to morphisms of sheaves and varieties, in the following sense:
\begin{enumerate}
    \item Let $f \colon Y \to X$ be an \'{e}tale morphism of finite-type, separated schemes over $\Spec E$. Then the pullback maps 
    \[
    f^* \colon \opn{H}^i\left(E, \opn{H}^j_{\et}\left(\overline{X}, \mathscr{F} \right) \right) \to \opn{H}^i\left(E, \opn{H}^j_{\et}\left(\overline{Y}, f^*\mathscr{F} \right) \right)
    \]
    commute with the differentials in $E_2^{i, j}$ and induce the pullback $f^* \colon \opn{H}^{i+j}_{\et}\left( X, \mathscr{F} \right) \to \opn{H}^{i+j}_{\et}\left( Y, f^*\mathscr{F} \right)$ on the abutment. 
    \item Let $f \colon Y \to X$ be a finite \'{e}tale morphism of finite-type, separated schemes over $\Spec E$. Then the pushforward maps 
    \[
    f_* \colon \opn{H}^i\left(E, \opn{H}^j_{\et}\left(\overline{Y}, f^* \mathscr{F} \right) \right) \to \opn{H}^i\left(E, \opn{H}^j_{\et}\left(\overline{X}, \mathscr{F} \right) \right)
    \]
    commute with the differentials in $E_2^{i, j}$ and induce the pushforward $f_* \colon \opn{H}^{i+j}_{\et}\left( Y, f^*\mathscr{F} \right) \to \opn{H}^{i+j}_{\et}\left( X, \mathscr{F} \right)$ on the abutment. 
    \item If $\phi \colon \mathscr{F} \to \mathscr{G}$ is a morphism of lisse \'{e}tale $\Lambda$-sheaves on $X$, then the induced maps 
    \[
    \phi \colon \opn{H}^i\left(E, \opn{H}^j_{\et}\left(\overline{X}, \mathscr{F} \right) \right) \to \opn{H}^i\left(E, \opn{H}^j_{\et}\left(\overline{X}, \mathscr{G} \right) \right)
    \]
    commute with the differentials in $E_2^{i, j}$ and induce the map $\phi \colon \opn{H}^{i+j}_{\et}\left( X, \mathscr{F} \right) \to \opn{H}^{i+j}_{\et}\left( X, \mathscr{G} \right)$ on the abutment.
\end{enumerate}
\end{proposition}
\begin{proof}
The first part follows from the fact that we have a natural transformation of derived functors $R\Gamma_X \to R\Gamma_Y \circ f^*$ induced from the unit of the adjunction $f^* \dashv f_*$, which is compatible with base-change to $\Spec \overline{E}$. Similarly, the second part follows from the natural transformation $R \Gamma_Y \circ f^* \to R\Gamma_X$ induced from the counit of the adjunction $f_! \dashv f^!$ (recall that $f$ is finite \'{e}tale, so $f^! = f^*$ and $f_! = f_*$), which again is compatible with base-change to $\Spec \overline{E}$. The third part follows immediately from the fact that the Hochschild--Serre spectral is the Grothendieck spectral sequence obtained from $R\Gamma(G_E, -) \circ R\Gamma_{\overline{X}} = R\Gamma_X$, where $G_E$ denotes the absolute Galois group of $E$ and $R\Gamma(G_E, -)$ is the derived functor computing continuous group cohomology (it is the special case of continuous \'{e}tale cohomology for the variety $\Spec E$).
\end{proof}

\begin{remark} \label{RemSpecCaseHSFunc}
A special case of this functoriality is as follows. Let $F/E$ be a finite Galois extension of fields (with $F \subset \overline{E}$) and set $Y = X \times_{\Spec E} \Spec F$. Then we have 
\[
\opn{H}_{\et}^i\left(Y \times_{\Spec E} \Spec \overline{E}, \mathscr{F} \right) = \opn{Hom}_{\Lambda}\left( \Lambda[\Gal(F/E)], \opn{H}_{\et}^i\left(\overline{X}, \mathscr{F} \right) \right).
\]
Let $p \colon Y \to X$ denote the natural map. By Shapiro's lemma, one has morphisms of spectral sequences:
\[
\begin{tikzcd}
{\opn{H}^i\left(E, \opn{H}^j_{\et}\left( \overline{X}, \mathscr{F} \right) \right)} \arrow[d, "\opn{res}"'] \arrow[r, Rightarrow] & {\opn{H}^{i+j}_{\et}\left(X, \mathscr{F} \right)} \arrow[d, "p^*"] &  & {\opn{H}^i\left(F, \opn{H}^j_{\et}\left( \overline{X}, \mathscr{F} \right) \right)} \arrow[d, "\opn{cores}"'] \arrow[r, Rightarrow] & {\opn{H}^{i+j}_{\et}\left(Y, \mathscr{F} \right)} \arrow[d, "p_*"] \\
{\opn{H}^i\left(F, \opn{H}^j_{\et}\left( \overline{X}, \mathscr{F} \right) \right)} \arrow[r, Rightarrow]                         & {\opn{H}^{i+j}_{\et}\left(Y, \mathscr{F} \right)}                  &  & {\opn{H}^i\left(E, \opn{H}^j_{\et}\left( \overline{X}, \mathscr{F} \right) \right)} \arrow[r, Rightarrow]                           & {\opn{H}^{i+j}_{\et}\left(X, \mathscr{F} \right)}                 
\end{tikzcd}
\]
where, by abuse of notation we have written $\mathscr{F}$ for $p^*\mathscr{F}$, and $\opn{res}$ and $\opn{cores}$ denote restriction and corestriction respectively.
\end{remark}

The functoriality properties in Proposition \ref{FuncPropsOfHSseq} play an important role in our discussion of cohomology functors in \S \ref{VerticalNormRelations}. Furthermore in \emph{loc.cit.}, we also need to localise the Hochschild--Serre spectral sequence at a suitable maximal ideal arising from correspondences on the variety to produce ``cohomologically trivial'' elements, which we will now explain in general. The following definition is inspired by \cite[\S 2.3]{LiuCubic}:

\begin{definition} \label{DefOFEtCorXF}
Let $X$ denote a finite-type, separated scheme over $\Spec E$ and let $\mathscr{F}$ be a lisse \'{e}tale $\Lambda$-sheaf on $X$. An \emph{\'{e}tale correspondence} on $X$ is a tuple $(X', f, g, \phi)$ where $(X', f, g)$ refers to a diagram
\[
X \xleftarrow{f} X' \xrightarrow{g} X
\]
with $f$ and $g$ finite \'{e}tale morphisms, and $\phi \colon f^* \mathscr{F} \to g^* \mathscr{F}$ is a morphism of sheaves. The composition of two such correspondences is defined to be $(X_3, f_3, g_3, \phi_3) = (X_2, f_2, g_2, \phi_2) \circ (X_1, f_1, g_1, \phi_1)$ where $(X_3, f_3, g_3)$ is defined via the diagram 
\[
\begin{tikzcd}
  &                                          & X_3 \arrow[ld, "p"'] \arrow[rd, "q"] \arrow[lldd, "f_3"', bend right] \arrow[rrdd, "g_3", bend left] &                                          &   \\
  & X_1 \arrow[ld, "f_1"'] \arrow[rd, "g_1"] &                                                                                                      & X_2 \arrow[rd, "g_2"] \arrow[ld, "f_2"'] &   \\
X &                                          & X                                                                                                    &                                          & X
\end{tikzcd}
\]
with the middle square Cartesian, and $\phi_3$ is given by the composition:
\[
f_3^*\mathscr{F} = p^*f_1^* \mathscr{F} \xrightarrow{p^*\phi_1} p^*g_1^*\mathscr{F} = q^*f_2^* \mathscr{F} \xrightarrow{q^*\phi_2} q^*g_2^*\mathscr{F} = g_3^*\mathscr{F} .
\]
We denote the category of all such correspondences by $\opn{\acute{E}tCor}(X, \mathscr{F})$ which forms a monoidal category with composition as above and unit elements corresponding to those correspondences $(X', f, g, \phi)$ where $f$, $g$ and $\phi$ are isomorphisms.

If $T = (X', f, g, \phi)$ is an \'{e}tale correspondence on $X$, then one obtains an endomorphism of $\opn{H}^i_{\et}\left(X, \mathscr{F} \right)$ (which we also denote by $T$) given by the composition:
\[
\opn{H}^i_{\et}\left(X, \mathscr{F} \right) \xrightarrow{f^*} \opn{H}^i_{\et}\left(X', f^*\mathscr{F} \right) \xrightarrow{\phi} \opn{H}^i_{\et}\left(X', g^*\mathscr{F} \right) \xrightarrow{g_*} \opn{H}^i_{\et}\left(X, \mathscr{F} \right)
\]
where the last map is induced from the counit of the adjunction $g_! \dashv g^!$ (recall $g^! = g^*$ and $g_* = g_!$). We have a similar endomorphism on $\opn{H}^i_{\et}\left(\overline{X}, \mathscr{F} \right)$. Note that if $T_1$ and $T_2$ are two such correspondences, then the composition of $T_1 \in \opn{End}(\opn{H}^i_{\et}\left(X, \mathscr{F} \right))$ followed by $T_2 \in \opn{End}(\opn{H}^i_{\et}\left(X, \mathscr{F} \right))$ is the endomorphism induced by $T_2 \circ T_1$ (this is the opposite convention to \cite[\S 2.3]{LiuCubic}).
\end{definition}

Consider the monoid $\mbb{N} = \{0, 1, 2, \dots \}$ (under addition) which we regard as a discrete monoidal category, and suppose we have a monoidal functor $\mbb{N} \to \opn{\acute{E}tCor}(X, \mathscr{F})$. Set $R = \Lambda[\mbb{N}]$ which is identified with the polynomial algebra over $\Lambda$ in one variable. Then the cohomology groups $\opn{H}^i_{\et}\left(X, \mathscr{F}\right)$ and $\opn{H}^i_{\et}\left(\overline{X}, \mathscr{F}\right)$ carry the structure of an $R$-module via the monoidal functor above. Let $\ide{m} \subset R$ be a prime ideal. Then by the functoriality properties in Proposition \ref{FuncPropsOfHSseq} and the fact that localisation is exact, we obtain a localised Hochschild--Serre spectral sequence:
\[
E_{2, \ide{m}}^{i, j} \colon \opn{H}^i\left(E, \opn{H}^j_{\et}\left(\overline{X}, \mathscr{F} \right) \right)_{\ide{m}} \Rightarrow \opn{H}^{i+j}_{\et}\left(X, \mathscr{F} \right)_{\ide{m}} .
\]
We denote by $\opn{H}^{i}_{\et}\left(X, \mathscr{F} \right)_{\ide{m}, 0}$ the kernel of the natural base-change map $\opn{H}^{i}_{\et}\left(X, \mathscr{F} \right)_{\ide{m}} \to \opn{H}^0\left(E, \opn{H}^i_{\et}\left(\overline{X}, \mathscr{F} \right)\right)_{\ide{m}}$, which we refer to as the subspace of \emph{cohomologically trivial} elements. We have a similar definition for the non-localised Hochschild--Serre spectral sequence.

\begin{definition}[Abel--Jacobi maps] \label{AJetDef}
Let $\opn{AJ}_{\et}$ and $\opn{AJ}_{\et, \ide{m}}$ denote the edge maps 
\begin{eqnarray}
\opn{AJ}_{\et} \colon \opn{H}^i_{\et}\left(X, \mathscr{F} \right)_0 & \to & \opn{H}^1\left(E, \opn{H}^{i-1}_{\et}\left(\overline{X}, \mathscr{F} \right) \right) \nonumber \\
\opn{AJ}_{\et, \ide{m}} \colon \opn{H}^i_{\et}\left(X, \mathscr{F} \right)_{\ide{m}, 0} & \to & \opn{H}^1\left(E, \opn{H}^{i-1}_{\et}\left(\overline{X}, \mathscr{F} \right) \right)_{\ide{m}} \to \opn{H}^1\left(E, \opn{H}^{i-1}_{\et}\left(\overline{X}, \mathscr{F} \right)_{\ide{m}} \right) \label{LocAJDef}
\end{eqnarray}
arising from the Hochschild--Serre and localised Hochschild--Serre spectral sequences respectively, where the last map in (\ref{LocAJDef}) is the natural one.
\end{definition}

The Abel--Jacobi maps satisfy certain compatibility properties with respect to the scheme and the sheaf. To be able to formulate this precisely, we introduce the following definition:

\begin{definition} \label{DefOfCompatUnderPi}
Let $\pi \colon Y \to X$ be a finite \'{e}tale morphism of finite-type, separated schemes over $\Spec E$ and let $T = (X', f, g, \phi)$ and $U = (Y', a, b, \psi)$ be correspondences in $\opn{\acute{E}tCor}(X, \mathscr{F})$ and $\opn{\acute{E}tCor}(Y, \pi^*\mathscr{F})$ respectively. We say that $T$ and $U$ are compatible under $\pi$ if there exists a finite \'{e}tale morphism $\sigma \colon Y' \to X'$ such that the diagram
\[
\begin{tikzcd}
Y \arrow[d, "\pi"'] & Y' \arrow[d, "\sigma"] \arrow[l, "a"'] \arrow[r, "b"] & Y \arrow[d, "\pi"] \\
X                   & X' \arrow[l, "f"'] \arrow[r, "g"]                     & X                 
\end{tikzcd}
\]
is commutative with Cartesian squares, and $\psi = \sigma^*\phi$. If this is the case, we write $U = \pi^* T$. A lengthy computation involving several Cartesian squares shows that this property is preserved under composition, i.e. if $U_1 = \pi^*T_1$ and $U_2 = \pi^* T_2$ then $U_2 \circ U_1 = \pi^*(T_2 \circ T_1)$.

If $F \colon \mbb{N} \to \opn{\acute{E}tCor}(X, \mathscr{F})$ and $G \colon \mbb{N} \to \opn{\acute{E}tCor}(Y, \pi^*\mathscr{F})$ are monoidal functors, then we say that $F$ and $G$ are compatible under $\pi$ if $G(n) = \pi^*F(n)$ for all $n \in \mbb{N}$. By compatibility with composition, this is equivalent to requiring $G(1) = \pi^*F(1)$.
\end{definition}

\begin{proposition} \label{PropFuncPropsForAJet}
Suppose that $\pi \colon Y \to X$ is a finite \'{e}tale morphism of finite-type, separated schemes over $\Spec E$ and $\mathscr{F}$ is a lisse \'{e}tale $\Lambda$-sheaf on $X$. 
\begin{enumerate}
    \item Suppose $T$ and $U$ are correspondences in $\opn{\acute{E}tCor}(X, \mathscr{F})$ and $\opn{\acute{E}tCor}(Y, \pi^*\mathscr{F})$ respectively, which are compatible under $\pi$. Then
    \[
    \pi^* \circ T = U \circ \pi^* \quad \quad \quad \pi_* \circ U = T \circ \pi_*
    \]
    where, as usual, $T$ and $U$ denote the induced endomorphisms on cohomology.
    \item Suppose that $F \colon \mbb{N} \to \opn{\acute{E}tCor}(X, \mathscr{F})$ and $G \colon \mbb{N} \to \opn{\acute{E}tCor}(Y, \pi^*\mathscr{F})$ are monoidal functors which are compatible under $\pi$. Let $\ide{m} \subset R = \Lambda[\mbb{N}]$ be a prime ideal. Then we have:
    \[
    \pi^* \circ \opn{AJ}_{\et, \ide{m}} = \opn{AJ}_{\et, \ide{m}} \circ \; \pi^* \quad \quad \quad \pi_* \circ \opn{AJ}_{\et, \ide{m}} = \opn{AJ}_{\et, \ide{m}} \circ \; \pi_*
    \]
    where pullbacks/pushforwards induce maps between the localisations $(...)_{\ide{m}}$ by part (1).
\end{enumerate}
\end{proposition}
\begin{proof}
Throughout the course of this proof, we will continually use the fact that pullbacks and pushforwards commute for Cartesian squares of finite \'{e}tale morphisms. For the first part, write $T=(X', f, g, \phi)$ and $U = (Y', a, b, \psi)$ and recall that the endomorphisms $T$ and $U$ are given by the formulae $T = g_* \circ \phi \circ f^*$ and $U = b_* \circ \psi \circ a^*$. Then 
\[
\pi^* \circ T = \pi^* \circ g_* \circ \phi \circ f^* = b_* \circ \sigma^* \circ \phi \circ f^* = b_* \circ \psi \circ \sigma^* \circ f^* = b_* \circ \psi \circ a^* \circ \pi^* = U \circ \pi^*
\]
where we have used that $\psi$ is the pullback of $\phi$ along $\sigma$. The proof for pushforwards is similar.

The second part follows from the functoriality properties in Proposition \ref{FuncPropsOfHSseq} and the fact that $\pi^*$ and $\pi_*$ are $R$-linear, by part (1).
\end{proof}


\section{Shimura--Deligne  varieties}  \label{PureAppendix}
 
In this article, we have worked with varieties that are strictly speaking not Shimura varieties in the sense of \cite{DeligneSVs}. Specifically, axiom 2.1.1.3 of \emph{op.cit.}, which we will refer to as (SD3), fails for the morphism $ h \colon \mathbb{S} \to \Hb $ considered in \S \ref{TheShimuraData} (since $ \Hb^{\mathrm{ad}} $ has a compact $ \QQ $-simple factor $ \mathrm{SU}(n) $). More generally, for a Shimura datum arising from PEL moduli problems, (SD3) usually does not hold, e.g. see \cite[\S 1.1]{Morel}, or \cite[Remark 2.5.8]{Lan-Toroid}. The primary reason for imposing it is that it allows one to apply strong approximation, which is used in describing the reciprocity law on the geometric  connected  components \cite[\S 3.4]{DeligneTS}.  Another     application is in reducing the problem of the existence of canonical models to the case of connected Shimura varieties \cite[\S 2.7]{DeligneSVs}. Assuming (SD3)  however  excludes the case of so-called Shimura sets, for example the Gross curve in \cite[3.1]{Wei-Zhang}, which  are instrumental in applications to Euler systems.      

The purpose of this appendix is to record results that continue to hold for datum that do not necessarily satisfy (SD3). We make no claim of originality here  as   most proofs carry over verbatim. The proofs that we have chosen to include are primarily for purposes of exposition of the originals. We shall freely use results from \cite{DeligneTS} when they do not invoke the hypothesis in \S 2.1 where (SD3) assumed. We note that a more general framework for Shimura varieties (which in particular does not require assumption (SD3)) has also been proposed in \cite{SemplinerTaylor}.      

\begin{notation}
We fix an algebraic closure $ \CC $ of $  \RR $, and take $  \overline{\QQ} $ to be the algebraic closure of $ \QQ $ in $ \CC $. We denote the Deligne torus by $  \delT = \Res_{\CC/\RR}  \GG_{m} $ and let $ w \colon \GG_{m, \RR} \to  \delT  $ be defined via the \emph{inverse} of the inclusion $ \RR^ { \times  }    \hookrightarrow  \CC ^ { \times } $. 
We fix the identification $  \delT_{\CC} \cong \GG_{m} \times \GG_{m} $ such that the inclusion $ \delT(\RR) \to \delT (\CC) $ is given by $ z \mapsto (z, \overline{z} ) $, and we take $ \mu \colon \GG_{m,\CC} \to \delT_{\CC} $ to be the cocharacter $ z  \mapsto  (z, 1) $.  For an  algebraic  group $ \Gb $, we denote by $ \mathbf{Z}_{\Gb} $ its centre, $  \Gb ^{\mathrm{der}} $ its derived group, and $ \Gb^{\mathrm{ad}} $ its quotient by $  \mathbf{Z}  _{  \Gb }  $. A superscript ``$+$'' (e.g. $ \Gb (\RR) ^{+}$) denotes the connected component of the identity in the analytic topology.  
\end{notation}

\subsection{Preliminaries} 

\begin{definition}   \label{SD}            
A \emph{Shimura--Deligne datum} is a pair $  (\Gb, X) $  consisting of  a  connected   reductive algebraic group $ \Gb $ over $ \QQ $ and a $ \Gb(\RR)$-conjugacy class $ X $ of homomorphisms $ h \colon \mathbb{S}   \to  \Gb_{\RR} $ satisfying
\begin{enumerate}
    \item [(SD1)] For all $ h \in X $, the Hodge  bigrading  of the complex vector space $ \mathrm{Lie}(\Gb)_{\mathbb{C}}    $ under the adjoint action of $  \mathbb{S} _ { \CC }  $ is of type $ \left \{ (-1,1), (0,0), (1,-1) \right \}  $.  In particular, the cocharacter $   h \circ w : \GG_{m} \to \Gb_{\RR}  $ is central and independent of $ h $. 
    \item[(SD2)]  For any $ h \in X $,  $ \opn{ad} \left(h(\sqrt{-1})\right) \colon  \Gb _ {\RR}  \to \Gb _{ \RR }  $ (which is an involution by (SD1)) is a Cartan involution of $ \Gb^{\mathrm{der}}_{\RR} $, i.e. the real Lie group 
    \[
    \left \{ g \in \Gb^{\mathrm{der}}(\mathbb{C}) \, | \,   h(\sqrt{-1}) \overline{g} h(\sqrt{-1})^{-1} = g    \right \}  
    \]
    is compact. 
\end{enumerate}
A  \emph{morphism} $  (\Gb_{1} , X_{1} ) \to  (\Gb_{2}, X_{2}) $ of Shimura--Deligne datum is a homomorphism $ u \colon \Gb_{1} \to \Gb_{2} $ such that $ u(X_{1} ) \subset X_{2}    $. We say that such a morphism is \emph{injective} if $ u $ is.   
\end{definition}

\begin{remark} 
These are the axioms (1.5.1), (1.5.2), (1.5.3) of \cite{DeligneTS}, or (2.1.1.1), (2.1.1.2) of \cite{DeligneSVs}.  We  have  borrowed the   above terminology from \cite{Ngo-Genestier}.    
\end{remark} 

\begin{remark}  It suffices to verify the axiom for one $ h \in X $, and we may write $ (\Gb, h) $ for the Shimura--Deligne datum. Notice  also  that  since  $ \Gb^{\mathrm{der}} \to \Gb ^{  \mathrm{ad} } $ is a (central) isogeny, (SD2) is equivalent to requiring  that $ \opn{ad} \left( h ( \sqrt{-1} ) \right) $ induces a Cartan involution of $ \Gb  ^ { \mathrm{ad} }_{\mbb{R}} $.
\end{remark}   

Fix $ h_{0} \in X  $ and let $ K_{ \infty  } $ be the centraliser of $ h_{0} $ in $    \Gb(\RR) $, so that $ X =   \Gb ( \RR )  / K_{\infty} $. Then $ K_{\infty} $ contains the centre of $ \Gb(\RR) $ and $ \mathrm{Lie}(K_{\infty})_{\CC} $ coincides with $ \mathrm{Lie}(\Gb)^{(0,0)}  $. Since $   h_{0}(\sqrt{-1}) $ acts as $ -1 $ on $ \mathrm{Lie}(\Gb)/ \mathrm{Lie}(K_{\infty}) $, we  see that $ h_{0} $ is central if and only if $ h_{0} ( \sqrt{-1} ) $ is. Since Cartan involutions are unique up to conjugacy, this is equivalent  to $ \Gb ^ { \mathrm{der}} (\RR) $ being compact.  Such a  pair will be referred to as a \emph{trivial Shimura--Deligne datum}. 

Let $ X ^{0} $ be the connected component of $ X $ containing $ h_{0} $. By  \cite[Corollary 1.1.17]{DeligneSVs}, $ X ^ { 0  }    $ is either a singleton (which happens if and only if $(\Gb, h_{ 0 }) $ is trivial) or  a Hermitian symmetric domain. More  precisely, if  $ \Gb  ^ { \mathrm{ad} } _{ \RR }   =  \Gb_{1} \times \ldots   \Gb _{k}   $ is  the decomposition into $ \RR $-simple factors,  then $  X ^ { 0 }  = X_{1} \times \cdots  \times  X_{k}  $ where each $ X_{i} $ is a quotient of $  \Gb_{i}(\RR)  ^ { +  }    $ by  a  maximal  compact subgroup.    

\begin{lemma} \label{CptIntLemma}
The Lie group $ K _{ \infty }  ' =   K_{\infty }  \cap  \Gb ^ { \mathrm{der}} ( \RR ) ^{+} $ is connected, and  $ K _ {\infty } =  \mathbf{Z}_{\mbf{G}} (  \RR)  \cdot   K_{\infty}   ' $.  
\end{lemma}

\begin{proof}  
As $ K_{\infty}'$ is a maximal compact subgroup in a connected Lie group, it is connected. Let $ \Gb ^ { * } $ denote the compact real form of $ \Gb ^{  \mathrm { \mathrm{ad} } }  _ { \RR }    $  defined by  $ h( \sqrt{-1} ) $. Let $ h_{0}'  \colon \delT  \to \Gb ^ { \mathrm{ad}} $ be the   induced  map, and  let $ \mathbf{C} $ be the centraliser of $ h_{0}'$ in $ \Gb ^  \mathrm{ad} _{\RR} $. Then $ \mathbf{C}  $ is a Cartan subgroup of $ \Gb ^ {\mathrm{ad}} $, and is therefore (Zariski) connected. Moreover, $  \mathbf{C} ( \RR ) $ is a closed    subgroup of $    \Gb ^ {  * } ( \RR )  $, so $  \mathbf{C} $ is $ \RR $-anisotropic  and hence     $  \mathbf{C}  (\RR )$ is  connected by \cite[14.3]{Borel-Tits}. As $ K_{\infty} $ lands in $   \mathbf{C} ( \RR ) \subset  \Gb ^ { \mathrm{ad}} ( \RR )^{+} $ and $ \Gb ^ { \mathrm{der}} ( \RR )^{+} $ surjects onto $  \Gb ^ {\mathrm{ad}   } ( \RR)  ^ { +  }  $, the second  claim  follows.  
\end{proof}

\begin{definition}[{\cite[\S 0.1]{PinkThesis}}] \label{DefinitionOfSufficientlySmall}
Let $g \in \mbf{G}(\mbb{A}_f)$ and let $(g_p)_p \in \opn{GL}_n(\mbb{A}_f)$ denote its image under some faithful representation $\mbf{G} \hookrightarrow \opn{GL}_{n, \mbb{Q}}$. We say that $g$ is neat if 
\[
\bigcap_p \left( \overline{\mbb{Q}}^{\times} \cap \Gamma_p \right)_{\mathrm{tors}} = \{1\}
\]
where $\Gamma_p$ is the subgroup of $\Qpb^{\times}$ generated by the eigenvalues of $g_p$ (and the intersection is independent of the embedding $\overline{\mbb{Q}} \hookrightarrow \Qpb$). A compact open subgroup $K \subset \mbf{G}(\mbb{A}_f)$ is said to be \emph{neat} or \emph{sufficiently small} if all of its elements are neat. In this case, $\Gamma \defeq K \cap \mbf{G}(\mbb{Q})$ is a neat subgroup in the sense that, for every $g \in \Gamma$, the subgroup of $\overline{\mbb{Q}}^{\times}$ generated by the eigenvalues of $g$ is torsion-free. In particular, $\Gamma$ is torsion-free.
\end{definition}

\begin{definition}  
Given a Shimura--Deligne datum $ (\Gb, X ) $, and a compact open subgroup  $ K \subset  \Gb(\Ab_{f} ) $, we define the \emph{Shimura--Deligne variety} to be the double quotient 
\[
{\Gb (\QQ ) \backslash     [ X   \times (  \Gb ( \Ab_{f} ) / K )]}
\]
which we denote by  $ \Sh_{\Gb}(X,K)(\mathbb{C})$ or just $ \Sh_{\Gb}(K)(\mathbb{C}) $.
\end{definition}    

For a neat compact open subgroup $K \subset \mbf{G}(\mbb{A}_f)$, arguing as in \cite[\S 2.5]{Lan-Toroid} (or  \cite[Lemma 4.6.1]{Ngo-Genestier}), $ \Sh_{\Gb}(K)   (   \mathbb{C}  ) $  is a disjoint union of quotients of Hermitian symmetric domains/finite discrete sets by torsion-free arithmetic subgroups,  and  therefore  are the complex points of a  normal quasi-projective  $\mathbb{C}$-variety $ \Sh_{\Gb}(K)_{\mathbb{C} }  $  by the theorem of Baily--Borel. By \cite[Proposition 3.3(b)]{PinkThesis}, these varieties are also smooth. As $ K $ gets smaller, these varieties form a projective system whose limit is a quasi-compact separated scheme carrying a continuous action of $ \Gb(\Ab_{f} )   $ \cite[\S 1.8]{DeligneTS}. We shall denote this scheme by $ \Sh_{\Gb}(X)_{\mathbb{C}} $ or just $ \Sh_{\Gb, \mathbb{C} } $.

Let $ u \colon (\Gb_{1 }  , X_{1} ) \to  ( \Gb_{2} , X_{2}  ) $  be a morphism of Shimura--Deligne data. For compact open subgroups $ K_{i} \subset \Gb_{i } ( \Ab_{f} )   $  such  that  $  u ( K_{1}) \subset K_{2}  $, we have an induced $ \mathbb{C} $-morphism of corresponding $ \mathbb{C} $-varieties by Borel's Theorem (see \cite[\S 1.14]{DeligneTS}). If one starts with  a closed immersion of reductive groups, this morphism is  finite unramified for sufficiently small compatible compact opens, and factors as a composition of a closed  immersion  followed by a finite \'{e}tale morphism \cite[Proposition 1.15]{DeligneTS}.

\subsection{Canonical models}    

Let $ \Gb $ be a reductive $ \QQ $-group. For any $ \QQ $-algebra $ R $, the group $ \Gb ( R  )$ acts on the left on the space $  \Hom_{R}(\GG_{m,R} ,   \Gb_{R} )$ of algebraic group homomorphisms over $ R $, by conjugation on the target. Let $Y = Y _ { \Gb  }$ be (the fppf sheafification of)   the  functor 
\[
\QQ\text{-}\mathrm{algebras} \to  \mathbf{Sets}  \quad  \quad   R \mapsto \Gb(R) \backslash \Hom_{R}(\GG_{m,R} , \Gb_{R})   . 
\]
If $ F / \QQ $ is a finite Galois extension over which $\Gb$ splits, then restricted to $ F $-algebras, $ Y $ is  a  constant  functor, hence representable as $ \bigsqcup \Spec F $. By Galois descent, one sees  that it is representable by an \'{e}tale $ \QQ $-scheme. 

Suppose moreover that  $ (\Gb, X)  $ is a Shimura--Deligne datum. The cocharacters $ h_{\CC} \circ \mu \colon \GG_{m, \CC} \to  \Gb_{\CC} $ for $ h \in  X $ lie in the same $ \Gb(\CC)$-conjugacy class (as $ h $ lies in a single $\Gb(\RR)$-conjugacy class). Therefore, one obtains a geometric point $ \mu_{X}  \in     Y(\CC)  =  Y ( \Qbar )   \subset Y  _ { \overline{  \QQ  } }   $.        

\begin{definition} 
The \emph{reflex field} $ E(\Gb,X) $ of a Shimura--Deligne datum $ (\Gb, X) $ is the field of definition of the point $ \mu_{X} \in  Y_{\Gb}(\overline{ \QQ })$. 
\end{definition}

\begin{example}  \label{ReflexFieldUnitaryExample}
Let $ \Gb  = \mathrm{GU}(p,q)  $ be the unitary group defined in \S \ref{TheGroups} using  the imaginary quadratic field $ E $ and let $ h \colon \delT  \to  \Gb_{\RR} $ be the map $ z \mapsto (z, \ldots, z , \bar{z} , \ldots , \bar{z} ) $ where there are $ p $ copies of $ z $, $ q $ copies of $ \bar{z} $, and $p+q = n$. Then $ (\Gb, \left\{h \right\} ) $ is a Shimura--Deligne datum. The cocharacter $ \mu_{h} $   associated with $ h $ over $ \CC $ is given  by    $  \mu _ { h  } \colon  \GG  _  { m  }  \to   \GG_{m}  \times  \GL_{n, \CC },   \, \,   z  \mapsto   \left ( z ,  \mathrm{diag}(  z ,  , \ldots ,    z ,  1 ,  \ldots  1  )   \right               )   $,
which  is defined over $ E $. If $ \sigma \in  \Gal  ( E / \QQ ) $ denotes complex conjugation, then
$ \sigma (  \mu _ { h } ) \colon z   \mapsto  \left(  z  ,   \mathrm{diag} ( 1,  \ldots,  1 , z  ,   \ldots  , z  ) \right) $.  These are in two different conjugacy classes if and only if $ p \neq q $, i.e. the reflex field is $ E $ if $ p \neq q $ and $ \QQ $ otherwise.    
\end{example}

For the notion of \emph{models} of $ \Sh_{\Gb, \mathbb{C} } $ over subfields $ E \subset  \mathbb{C}   $,   we  refer the  reader to \cite[\S 3]   {DeligneTS}. If a model exists, we denote it by $ \Sh_{\Gb, E} $.   As  in \cite[\S 11.4]{PinkThesis}, one can define a  model of $ \Sh_{\Gb, \CC} $ over its reflex field when $ \Gb $ is a torus (and $ X = \left \{ h \right \} $ is a singleton).   
Moreover, if $ (\Gb_{1} , X_{1} ) \to  (\Gb_{2}, X_{2} )$ is  a   morphism, then $ E(\Gb_{1}, X_{1})  \supset E (\Gb_{2}, X_{2}) $ because one obtains a morphism $ Y_{\Gb_{1}} \to Y_{\Gb_{2}} $ with $ \mu_{X_{1}} $ mapping to $  \mu_{X_{2} } $. This motivates the following definition.

\begin{definition} 
A \emph{canonical model} of $ \Sh_{\Gb,\CC} $ is a model $ \Sh_{\Gb,E} $ defined over the reflex field $ E  =  E (\Gb, X) $ such  that    for every Shimura--Deligne datum $  ( \Tb, X ' )$ with reflex field $ E' $, where $ \Tb $ is a torus, and any injective morphism $ (\Tb, X ' )  \to  (\Gb, X  ) $, the induced map 
\[
\Sh_{ \Tb,  \CC } \to \Sh_{\Gb, \CC} 
\]
is the pullback of a morphism $\Sh_{\Tb,  E'  } \to \Sh_{\Gb, E   } \times_{ \Spec E}  \Spec E'$.    
\end{definition}    

\begin{lemma}    \label{Special Points}   
Let $ (\Gb,  X ) $ be a Shimura--Deligne datum and $  E =  E(\Gb,X) $ its reflex field. Let $ S $ be the  collection of all subfields $  F \subset \CC $ that arise as follows: there exists an injective  morphism of Shimura--Deligne datum $ (\Tb,  X' ) \to  (\Gb, X) $ such that $ F = E( \Tb, X'   ) $.   Then $ S $ is non-empty and      $ \bigcap _ { F  \in  S  }  F  =  E     $.         
\end{lemma}     

This result is \cite[Theorem 5.1]{DeligneTS} and is the key input of the next theorem. While (SD3) is a running hypothesis in \S 5 of \emph{op.cit.}, the proof of this result does not need it. As this might not be immediately obvious, we provide an  exposition of the key steps. 

\begin{proof} 
Let $ Y_{\Gb} $ be the  $ \QQ  $-scheme as above and let $ Y _{0}  \cong  \Spec  E $ be the finite \'{e}tale $ \QQ $-subscheme whose $\overline{\mbb{Q}}$-points consist of the Galois translates of $ \mu_{X} \in Y_{ \Gb}(\overline{\mbb{Q}})$. Let $ V $  be   the    $ \QQ  $-scheme  of  regular elements in $ \mathrm{Lie}(\Gb) $, i.e. points in $V(R)$ are elements $ v_{R}  \in  \mathrm{Lie}(\Gb)_{R} $  whose centraliser is a maximal torus $ \Tb_{v_{R}} $ in $ \Gb_{R}   $. Let $ W $  be (the fppf sheafification of) the functor that associates to each $ \QQ $-algebra $ R   $ the set of all  triplets $  (\Tb_{R} ,  v_{R}  , \lambda_{R}) $ where 
\begin{itemize} 
\item $ \Tb_{R} $ is a (fibrewise) maximal torus in  $ \Gb_{R} $. 
\item $ v_{R}  \in  \mathrm{Lie} (\Tb)_{R}  $  is a regular element  of  $  \mathrm{Lie}(\Gb)_{R} $, i.e. the centraliser of $ v $ in $\Gb_{R}$ is $ \Tb_{R} $. 
\item $ \lambda_{R} \colon \GG_{m,R} \to \Tb_{R} $ is a cocharacter such that the class of the induced morphism $  \GG_{m,R} \to \Gb_{R} $ is in $ Y_{0,R} $. 
\end{itemize}
Then $ W _ { \overline{ \QQ} } \to V_ { \overline{ \QQ} } $ has constant finite fibres. By Galois descent, $ W $ is representable by a finite-type $ \QQ $-scheme admitting a finite \'{e}tale surjective map $ f \colon W \to V  $ and a map $ p \colon  W  \to Y_{0} $. This makes $ W $ a $ E $-scheme. 
\begin{equation}  \label{WVDiagram}  
\begin{tikzcd}   W  \arrow[r, "f"]    \arrow[d, "p"]  &  V  \arrow[d]  \\
               Y_{0}   \arrow[r]  &  \Spec \QQ 
\end{tikzcd}    
\end{equation}
For all $ v \in V(\RR) $, $ w \in W(\CC) $ with $ f(w) = v $ and $ \lambda_{w} \colon  \GG_{m, \mbb{C}} \to \Tb_{v,\CC} $ the cocharacter part of $ w $, there is a unique $ h_{w} \colon  \delT  \to  \Tb_{\RR} $ such that $ h_{w, \CC}  \circ   \mu     = \lambda_{w} $;  on $ \RR$-points, it is given by 
\[
h_{ w  }(z )  =  \lambda_{w} (z)  \cdot \overline{ \lambda_{w}   (  z  )    }  ,  \quad  \quad      z \in   \delT(\RR) =  \CC ^{\times} .
\]
We note that if $ v \in V ( \QQ ) $ is such that 
\begin{itemize}   
\item  the $ \QQ $-scheme $  Z_{v}  : =    f^{-1}(v) $ is the spectrum of a field $ E_{v} $,
\item  for some $  w  \in  Z_{v} ( \CC  )  $, $ h _{w} \in  X $
\end{itemize}
then $ E_{v} \in S  $.  Indeed, the second point means that there is an injective morphism $ (\Tb_{v} , \left \{h_{w} \right \} )  \hookrightarrow  (\Gb, X) $ of Shimura--Deligne datum, and  the first  guarantees  that the cocharacter $  \lambda_{w} =  \mu \circ h_{w,\CC}  $ associated with the datum $ ( \Tb_{v},  \left \{ h  \right \}  )   $ has field of definition $ E_{v} $.   The goal therefore is to show that there are many such $ v  $.  \\

\noindent \textit{Step 1.} $ W $ is geometrically irreducible as an $ E $-scheme.  \\
\\
Let  $W'$ be the geometric fibre of $ p $ over $ \mu_{X} $   and $ C $ the scheme over $ \overline{\QQ} $ of cocharacters of $ \Gb_{\overline{\QQ}} $ whose conjugacy class is $ \mu_{X} $. Let $ q \colon W' \to C $ be the natural map. Then $ C $, being a quotient of $\Gb_ { \overline { \QQ } } $,    is irreducible and the fibers of $ q $ are isomorphic. Let $ W''  \subset   W'    $ be the fibre of $ q $ over a given $ \lambda \colon \GG_{m}  \to  \Gb_{\overline{\QQ} }  \in C ( \overline { \QQ } ) $ and $ D $   the scheme of all maximal tori of $ \Gb _ { \overline {\QQ } }  $ that contain the image of $ \lambda  $. Let  $ r \colon W ''  \to D $ be the natural  map. Any maximal torus coming from $ D $ centralizes $ \lambda ( \GG_{m} ) $ and therefore must be contained in the centraliser $ Z_{\Gb _ { \overline {  {  \QQ  } }  } } ( \lambda(\GG_{m}) ) $. Now, $ Z_{\Gb _ { \overline { \QQ } }   } ( \lambda  ( \GG_{m} )  )  $ is connected by \cite[Theorem 17.38]{Milne-Algebraic} and $ D $ is a quotient of $ Z_{\mbf{G}_{ \overline{\QQ} }} ( \lambda ( \GG_{m}  ) )$  by \cite[Theorem  17.10]{Milne-Algebraic}. Therefore, $ D $ is irreducible. As $  W'' \to D $ is open in the vector bundle over $ D $ associated with the  Lie  algebra of the universal torus over $ D $, it has geometrically irreducible fibers over $ D $, and is therefore irreducible too.  Collecting these statements together, we see that the fibers of $ q $ are irreducible, whence $ W ' $ is irreducible. As the geometric fibres of $ p $ are isomorphic and $ Y_{0} $ is irreducible, the claim follows.   \\    

\noindent  \textit{Step 2.}   In the analytic  topology of $  V (  \RR  )   $, there is a  non-empty open subset $ U $  such that for all $ v \in U $, $  \Tb_{v} $ contains the image of some $ h \in X $.\\
\\
We let $ U  \subset V(\RR) $ denote the subset of  all  $ v \in  V (\RR ) $ such that $  (  \Tb_{v}/   \mathbf{Z}_{\mbf{G}, \RR}  ) (\RR)   $ is compact. Then $ U $ is open since the association $ v \mapsto \Tb_{v}/\mathbf{Z}_{\mbf{G}, \RR} $ is continuous and the set of anisotropic (equivalently compact) tori is an open subspace\footnote{The space of conjugacy classes of maximal tori in $ \Gb^{\mathrm{ad}}(\RR) $ form finitely many connected components, each of which is open, and any two in the same component are conjugate by $ \Gb ^ {\mathrm{ad}} (\RR) ^{+}    $.  See  the  footnote in   \cite[p. 117]{Milne}.} of the moduli space of all maximal  tori of $ \Gb ^{\mathrm{ad}}_{\RR} $. Moreover, $ U $ is non-empty as follows: for  any $ h  \in X $, the map $ h^  {\mathrm{ad} } :  \delT  \to  \Gb ^ {  \mathrm{ad} } _ { \RR }  $ factors  via $ h ^ { \mathrm{ad} } :  \U_{1} \to \Gb^{\mathrm{ad}}_{\RR} $. Hence any maximal torus $\mbf{T}'$ of $ \Gb ^{\mathrm{ad}}(\RR) $ containing the image of $ h^{\mathrm{ad}}$ is contained in the centraliser $ \mathrm{Stab} _ { \Gb } ( h  ^ { \mathrm{ad} }   )    $, and is  therefore  compact. Any regular element in the Lie algebra of the inverse image in $ \Gb $ of $ \Tb' $  is then the desired element.    It now  remains to show that for \emph{all}   $ v \in U  $, $ \Tb_{v} $ contains the image of some $ h \in X $. To this end, fix   $ h_{0} \in X $ and let $ K_{\infty} $ be the stabiliser of $ h_{0} $ in $ \Gb(\RR) $.  Then   $ K _ { \infty } '   =   K_{\infty}  \cap  \Gb^{\mathrm{der}} (\RR) ^{+} $ is a maximal compact Lie  subgroup  of   $ \Gb^{\mathrm{der}}(\RR)^{+} $, and  $  K_{\infty} / \mathbf{Z}_{\mbf{G}}(\RR) $ is a maximal compact subgroup  of  $\Gb ^ { \mathrm{ad} } ( \RR ) ^  {+ } $ by Lemma \ref{CptIntLemma}.  As any two maximal compact subgroups   in  $ \Gb ^ {  \mathrm{ad} } ( \RR  ) ^ { + } $  are conjugate and $  \Tb_{v}/ \mathbf{Z}_{\mbf{G}, \RR}  ( \RR  )   $ is compact, there exists $ g \in  \Gb^{\mathrm{der} }( \RR ) ^ { + } $ such that $ g  \Tb_{v}  (\RR)    g^{-1}  \subseteq  K_{\infty} $. Since   $ g  \Tb_v (\RR) g ^{-1} $ is a  maximal torus in  $ K _ {\infty  } $, it contains the centre of $ K_{\infty} $, and in particular the image of $ h_{0} $.   Therefore, $ \Tb_{v} $ contains the image of $ g  ^{-1}  h_{0}  g   $.   \\    

\noindent  \textit{Step 3.}  For any finite extension  $ E' $ of $ E $, there exists $ v \in  V( \QQ) \cap U $ such that  $ f^{-1} (v) $ is the spectrum of a field linearly  disjoint  from  $ E' $.    \\
\\
This is (a slightly general form of) Hilbert's irreducibility theorem applied to the diagram in (\ref{WVDiagram}), and is proved in \cite[Lemma 5.1.3]{DeligneTS}. 
\end{proof}    

\begin{remark}  
In fact, the proof shows that there are two extensions in $ S $ whose intersection is $ E $. 
\end{remark} 

\begin{theorem}[Functoriality]  
Let $( \Gb_{1} , X_{1} ) \to  (\Gb_{2} , X_{2} )  $ be a morphism of Shimura--Deligne datum and, for $i=1, 2$, let $ E_{i}$ be the reflex field of $ (\Gb_{i}, X_{i} ) $. Suppose $ \Sh_{\Gb_{i} , E_{i} } $ is a canonical model over $ E_{i} $ of $ \Sh_{\Gb_{i} ,\CC} $. Then the morphism 
\[
\Sh_{\Gb_{1}, \CC}  \to  \Sh_{\Gb _{2} , \CC }  
\]
is the pullback of a morphism  $ \Sh_{\Gb _{1} , E_{1} } \to  \Sh_{\Gb_{2} , E_{2} } \times_{ \Spec E_{2} }  \Spec   E_{1}  $.
\end{theorem}    
\begin{proof} 
This is \cite[Corollary 5.4] {DeligneTS} whose proof relies on Propositions 5.2, 5.3, and Theorem 5.1 in \emph{op.cit}. Proposition 5.2 of \emph{op.cit.} holds in our case as we still have real approximation, and Lemma 5.3 is a general fact about descent data of schemes (see the explanation in \cite[Lemma 11.8]{PinkThesis} and the reference therein). The key input of Theorem 5.1 is that the intersection of the reflex fields of tori admitting a map to $ (\Gb_{1}, X_{1}) $ is $ E _ { 1  }    $ (without assuming (SD3)), and this was shown in Lemma \ref{Special Points} above.
\end{proof}    

\begin{corollary}[Uniqueness] \label{UniquenessCM}
Canonical models are unique up to a unique isomorphism.
\end{corollary}

\begin{proof} 
If $ \Sh_{\Gb,E} $ and $ \Sh_{\Gb,E}' $ are two canonical models over the reflex field $ E $ of a datum $ (\Gb, X )$, then the identity map on $ \Sh_{\Gb,\CC} $ is the pullback of a morphism $ \Sh_{\Gb,E}' \to \Sh_{\Gb,E} $.
\end{proof}

\begin{corollary}[Existence criterion] \label{ExistenceCriterion}
Let $ (\Gb, X)  \to  (\Gb' , X' ) $ be an injective morphism of Shimura--Deligne datum. If $ (\Gb' , X') $ admits a canonical model, then so does $ (\Gb, X) $. 
\end{corollary}           

\begin{proof}  
This is \cite[Corollary 5.7]{DeligneTS} and the proof carries over verbatim.       
\end{proof}    

\begin{definition}    
Let $ (\Gb , X ) $ be a Shimura--Deligne datum. In addition to the axioms introduced in Definition \ref{SD}, some additional axioms are often imposed which we list here: 
\begin{enumerate}
    \item [(SD3)] $\Gb^{\mathrm{ad}}(\RR) $ has no $ \QQ $-simple factors that are $ \RR $-anisotropic.  
    \item[(SD4)] The \emph{weight morphism}  $ w_{X} \defeq h \circ w \colon \GG_{m , \RR } \to \Gb_{\RR} $, a priori defined over $ \RR $, is defined over $ \QQ $.
    \item[(SD5)]  $ \Zb _{\Gb} ( \QQ ) $ is discrete in $ \Zb _{\Gb} ( \Ab_{f} ) $. 
\end{enumerate}  
\end{definition}  

\begin{remark} 
When $ (\Gb,X) $ satisfies (SD3), we arrive at the common definition of \emph{Shimura datum} in the literature. We note that (SD5) is equivalent to the statement that  the maximal anisotropic  $ \QQ $-subtorus of $ \Zb_{\Gb} $ remains anisotropic (equivalently, compact) over $ \RR $, i.e. $ \Zb_{\Gb} $ is an almost direct product of a split and a compact type torus over $ \QQ   $  \cite[Theorem 5.36]{Milne}.  (SD5) implies (SD4) (see \cite[\S 5.4]{Pink1992}) and is therefore  a practical means to check the latter.
\end{remark}

One particularly useful application of (SD5) is the following:

\begin{lemma} \label{LemSD5index}
Let $(\mbf{G}, X_{\mbf{G}})$ be a Shimura--Deligne datum which satisfies (SD5). Let $L \subset K \subset \mbf{G}(\mbb{A}_f)$ be sufficiently small compact open subgroups (Definition \ref{DefinitionOfSufficientlySmall}) such that $L$ is normal in $K$. Then the map of $\mbb{C}$-varieties 
\[
\opn{Sh}_{\mbf{G}}(L)_{\mbb{C}} \twoheadrightarrow \opn{Sh}_{\mbf{G}}(K)_{\mbb{C}}
\]
is Galois with Galois group $K/L$.
\end{lemma}
\begin{proof}
By \cite[Theorem 5.28]{Milne} and axiom (SD5), one has 
\[
S \defeq \opn{Sh}_{\mbf{G}}(\mbb{C}) = \mbf{G}(\mbb{Q}) \backslash X_{\mbf{G}} \times \mbf{G}(\mbb{A}_f) 
\]
and $\opn{Sh}_{\mbf{G}}(L)(\mbb{C}) = S/L$, $\opn{Sh}_{\mbf{G}}(K)(\mbb{C}) = S/K$. Since $K$ is sufficiently small, the action of $K$ on $S$ is free. Indeed, the stabiliser $\opn{Stab}_K(s)$ of a point $s \in S$ under the action of $K$, satisfies
\[
g\opn{Stab}_K(s)g^{-1} = \opn{Stab}_{\mbf{G}(\mbb{R})}(x) \cap \mbf{G}(\mbb{Q}) \cap g K g^{-1}
\]
where we have written $s = [x, g]$ for $x \in X_{\mbf{G}}$ and $g \in \mbf{G}(\mbb{A}_f)$. Since $\opn{Stab}_{\mbf{G}(\mbb{R})}(x)$ is compact-modulo-centre and $\Gamma \defeq \mbf{G}(\mbb{Q}) \cap g K g^{-1}$ is a discrete subgroup of $\mbf{G}(\mbb{R})$, for any $\gamma \in g\opn{Stab}_K(s)g^{-1}$ there exists an integer $n > 0$ such that 
\[
\gamma^n \in \mbf{Z}_{\mbf{G}}(\mbb{R}) \cap \Gamma = \mbf{Z}_{\mbf{G}}(\mbb{Q}) \cap \Gamma \subset \mbf{Z}_{\mbf{G}}(\mbb{Q}) \cap gKg^{-1}.
\]
Since $\mbf{Z}_{\mbf{G}}(\mbb{Q})$ is discrete in $\mbf{Z}_{\mbf{G}}(\mbb{A}_f)$ this implies that $\mbf{Z}_{\mbf{G}}(\mbb{Q}) \cap gKg^{-1}$ is finite, and so every element of $g\opn{Stab}_K(s)g^{-1}$ has finite-order. But $K$ is neat, so $\Gamma$ is torsion-free, which implies that $g\opn{Stab}_K(s)g^{-1}$ (and hence $\opn{Stab}_K(s)$) must be trivial. 

Similarly, the action of $L$ on $S$ is free. This implies that the map $S/L \twoheadrightarrow S/K$ is Galois and has degree $[K:L]$ as required. 
\end{proof}

\subsection{PEL-type Shimura--Deligne varieties}   
For the notion of (semisimple) PEL-datum, we refer the reader to  \cite[Definition 7.2]{Torzewski2019}. Here, we simply recall that a PEL datum is a tuple $ ( B ,  * , V , \langle \cdot , \cdot  \rangle , h  )  $ where 
\begin{itemize} 
\item $ B $ is a finite-dimensional semisimple $ \QQ $-algebra
\item $ * $ is a positive anti-involution of $ B $
\item $ V $ is  a finite-dimensional $ B $-module
\item $ \langle \cdot,  \cdot  \rangle  \colon V \times V \to \QQ $ a non-degenerate alternating (i.e. skew Hermitian) pairing 
\item $ h \colon  \CC  \to \mathrm{End} _{B_{\RR} } V_{\RR} $ is an $ \RR $-algebra homomorphism
\end{itemize}
satisfying the conditions in \emph{loc.cit.}. Given such a PEL datum as above, we can associate to it an algebraic group $ \Gb $ (over $\mbb{Q}$) whose $ R $-points equal  
\[
\Gb ( R) = \left \{   \mathrm{Aut}_{B_{R}}(V_{R}) \, | \, \exists \mu(g) \in R^{\times} \text{ such that } \langle g u , gv  \rangle  =  \mu(g)  \langle u , v \rangle  \right \} . 
\]
The group $ \Gb $ is reductive and $ \mu \colon \Gb \to \GG_{m} $ is a homomorphism, which we call the \emph{similitude factor}. We let $ \Gb_{1} $ denote the kernel of $ \mu $. Then $ \Gb_{1,  \RR} $ splits into a product of unitary (A), symplectic (C) and orthogonal factors (D) \cite[Lemma 7.5 (i)]{Torzewski2019}. In particular, if $ B $ is simple then only one type can occur. The $ \RR $-algebra homomorphism $ h \colon \CC \to \mathrm{End}_{B_{\RR}}(V_{\RR} ) $ gives rise to a morphism $ \delT \to \Gb_{\mbb{R}} $, which we still denote by $ h $.    

\begin{lemma}  
If $ \Gb_{1} $ has no factors of type $ (D) $, then the pair $ (\Gb, h) $ is a Shimura--Deligne datum. 
In addition, it satisfies (SD4). 
\end{lemma}    

\begin{proof}  
This is \cite[Lemma 7.5 (ii)]{Torzewski2019}, with the only modification being that we include non-central factors of compact type.
The second statement follows by  Lemma 7.5 (iii) of \emph{op.cit}.
\end{proof}

\begin{definition} \label{PELDatumDef}  
A \emph{PEL-type Shimura--Deligne datum} is a pair $ (\Gb, h ) $ that arises from a choice of a PEL datum with no orthogonal factors as above.
\end{definition}

To any PEL-datum as above, and for any compact open subgroup $K \subset \mbf{G}(\mbb{A}_f)$, we have an associated moduli problem $\invs{M}_K$ over the reflex field which classifies abelian varieties with corresponding PEL structure as in \cite[\S 1.2]{Lan-Toroid}. If $K$ is \emph{neat}, then $\invs{M}_K$ is representable by a smooth quasi-projective scheme.
 
\begin{remark} Here, we are assuming that such a \emph{rational} PEL-datum arises from a choice of \emph{integral} PEL-datum as in \cite[\S 1.1]{Lan-Toroid}. As we are only interested in moduli problems over reflex fields in the sequel, the choice of an integral PEL datum does not matter by Corollary 1.4.3.7 of \cite{LanArithmetic}.
\end{remark} 

\begin{lemma}   \label{Hasse}       Let $ (\Gb, X) $ be a PEL-type Shimura--Deligne datum arising from $  ( B  , * ,  V , \langle  \cdot , \cdot  \rangle,  h)  $.  Set 
\[
\ker ^{1}  (\QQ, \Gb )  \defeq    \opn{ker}\left( \opn{H}^1(\mbb{Q}, \mbf{G}) \to \bigoplus_{v \leq \infty} \opn{H}^1(\mbb{Q}_v, \mbf{G}) \right) .  
\]
Then $  \ker ^{1} ( \Gb, \QQ ) $ is finite and classifies  isomorphism classes of skew-Hermitian $ B $-modules $ (V',  \langle \cdot,  \cdot  \rangle ' ) $ such that the $ B_{\Ab} $-module $ (V'_{\Ab},  \langle \cdot, \cdot  \rangle' ) $ is  isomorphic to $ (V_{\Ab} , \langle \cdot, \cdot   \rangle  ) $.
\end{lemma} 
\begin{proof}  That  $ \ker ^{1} ( \QQ, \Gb ) $ is finite follows from \cite[\S 7]{Borel-Serre}.    By \cite[Proposition 3.50]{Milne-Algebraic}),  $ \cohom{1}{} (\QQ, \Gb) $ (resp. $ \cohom{1}{}(\QQ_{v}, \Gb_{v}) $) classifies isomorphism classes of $ \Gb $-torsors over $ \QQ $ (resp. $\Gb_{v} $ torsors over $ \QQ_{v} $), and so $ \ker ^{1}(\QQ, \Gb) $ classifies isomorphism classes of $ \Gb $ torsors over $ \QQ $ that are locally trivial. Now $ \Gb $-torsors are just symplectic $ B $-modules $ V' $ that are isomorphic to $ V $ over $ \overline{\QQ} $.  
\end{proof}     

\begin{proposition} \label{ComplexModuliProp}
Let $(\mbf{G}, X)$ be a PEL-type Shimura--Deligne datum associated with the tuple $(B, *, V, \langle \cdot, \cdot \rangle, h)$, and fix an integral PEL-datum satisfying \cite[Condition 1.2.5]{Lan-Toroid}. For any neat compact open subgroup $K \subset \mbf{G}(\mbb{A}_f)$, one has an open and closed embedding of $\mbb{C}$-varieties
\begin{equation} \label{ComplexModuliEmbedding}
\opn{Sh}_{\mbf{G}, \mbb{C}}(K) \hookrightarrow \invs{M}_{K, \mbb{C}}.  
\end{equation}
If the group $\mbf{G}$ satisfies the Hasse principle, i.e. $ \ker ^{1} ( \QQ, \Gb )  = 0 $, 
then the embedding in (\ref{ComplexModuliEmbedding}) is an isomorphism.
\end{proposition} 
\begin{proof}
The embedding in (\ref{ComplexModuliEmbedding}) follows from \cite[Lemma 2.5.6]{Lan-Toroid} and the paragraph following its proof. More precisely, it is shown that $\opn{Sh}_{\mbf{G}}(K)(\mbb{C})$ is isomorphic to the locus in $\invs{M}_K(\mbb{C})$ parameterising all PEL abelian varieties whose first (rational) homology is isomorphic to $V$ (as polarised symplectic spaces with $B$-structure). Note that for any point in $\invs{M}_K(\mbb{C})$, the first homology (over $\mbb{A}$) of the associated abelian variety is always isomorphic to $V \otimes_{\mbb{Q}} \mbb{A}$ (as polarised symplectic spaces with $B$-structure), by definition, since $(\mbf{G}, X)$ is a PEL-type Shimura--Deligne datum associated with the tuple $(B, *, V, \langle \cdot, \cdot \rangle, h)$. So to show (\ref{ComplexModuliEmbedding}) is an isomorphism, it is enough to show that there is a unique polarised symplectic $B$-module $W$ such that $W \otimes_{\mbb{Q}} \mbb{A} \cong V \otimes_{\mbb{Q}} \mbb{A}$. This follows by Lemma \ref{Hasse}.      
\end{proof}    

\begin{remark}  The discussion above can also be found in \cite[\S 8]{Kottwitz} under the assumption that $ B $ is simple.
\end{remark}    
\begin{remark}   \label{Hassefortorus}    The computation of $ \ker^{1}(\QQ, \Gb) $ can be reduced to $ \ker^{1}(\QQ, \Gb/\Gb^{\mathrm{der}}) $ by \cite[Proposition 25.71]{Milne-Algebraic}, as $ \Gb^{\mathrm{der}} $ is always simply connected for PEL-type Shimura Deligne datum (\cite[Proposition 8.7]{Milne}).
\end{remark}

\begin{corollary}
Keeping the same notation as in Proposition \ref{ComplexModuliProp}, suppose that $\mbf{G}$ satisfies the Hasse principle. Then $\opn{Sh}_{\mbf{G}, \mbb{C}}(K)$ has a canonical model isomorphic to $\invs{M}_{K, E}$, where $E = E(\mbf{G}, X)$ is the reflex field. 
\end{corollary}
\begin{proof} 
Any PEL-type Shimura--Deligne datum has an embedding into the standard Siegel Shimura datum, so $\opn{Sh}_{\mbf{G}, \mbb{C}}(K)$ has a unique canonical model by Corollaries \ref{ExistenceCriterion} and \ref{UniquenessCM}. The fact that $\invs{M}_{K, E}$ is a canonical model follows from \cite[Proposition 14.12]{Milne} with the identity map (the proof of this proposition is purely a statement about CM abelian varieties). 
\end{proof}

\begin{example} \label{UnitaryPELModEx}
Let $(\mathscr{G}, \ide{X}) \in \{(\mbf{H}, X_{\mbf{H}}), (\mbf{G}, X_{\mbf{G}}) \}$ be the Shimura--Deligne datum introduced in section \ref{TheShimuraData}. Then $(\mathscr{G}, \ide{X})$ is of PEL-type and satisfies (SD5) and \cite[Condition 1.2.5]{Lan-Toroid}. Indeed
\begin{itemize}
    \item If $\mathscr{G} = \mbf{G}$, then the semisimple algebra is just $B = E$ and $V$ is the $2n$-dimensional Hermitian space with Hermitian form corresponding to the matrix $J_{1, 2n-1}$ (see section \ref{TheGroups}). In particular, this arises from an integral PEL-datum with semisimple algebra $\ordd_E$.
    \item If $\mathscr{G} = \mbf{H}$, then we take $B = E \times E$ and consider the product of Hermitian spaces $V = W_1 \oplus W_2$ where $W_1$ (resp. $W_2$) has Hermitian form given by the matrix $J_{1, n-1}$ (resp. $J_{0, n}$). As above, this arises from an integral PEL-datum with semisimple algebra $\ordd_E \times \ordd_E$. 
\end{itemize}
In addition to this, the group $\mbf{G}$ satisfies the Hasse principle, for the following reason. By Remark \ref{Hassefortorus}, it is enough to deduce it for the the maximal abelian quotient of $\mathscr{G}$, denoted $\mbf{T}_{\mathscr{G}}$. In either case, $\mbf{T}_{\mathscr{G}}$ is a product of factors $\opn{U}(1)$, $\opn{GU}(1)$ or $\mbb{G}_m$. The Hasse principle is known to hold for these three groups; for $ U(1) $, see  (\cite[\S 7]{Kottwitz}), while for $ \GU(1) $, $ \GG_{m} $, it follows from Hilbert's Theorem 90.      
\end{example}


\section{Ancona's construction for Shimura--Deligne varieties} \label{AnconaAppendix}

In this appendix, we document certain results, particularly the functoriality of motivic lifts, of \cite{Ancona2015} and \cite{Torzewski2019} that hold in the absence of (SD3). The techniques used by these authors involve mixed Shimura varieties, which are a  generalisation of the usual (pure) Shimura varieties. The definition of a mixed Shimura datum includes a counterpart of axiom (SD3); axiom (vii) of \cite[Definition 2.1]{PinkThesis}), and as in the case of pure Shimura data, this axiom is usually invoked when strong approximation is needed (for example, when describing the connected components of mixed Shimura varieties in \cite[Proposition 3.9]{PinkThesis}). It also plays a key role in reducing many statements to ones involving pure Shimura data. By replacing this condition with an alternative assumption (Assumption \ref{OneDimAssump}) however, many proofs in \cite{PinkThesis} carry over verbatim (building on the results in Appendix \ref{PureAppendix}). 

We have not attempted to be thorough here -- instead, we content ourselves with a summary of the general results we need from \cite{PinkThesis}, and only use them in the restricted setting of \S  \ref{Uni-Extensions} that suffices for the purposes of this article. We continue with the notation of Appendix \ref{PureAppendix} until \S \ref{LiftsCanConst}, after which we specialise to the particular groups defined in \S \ref{TheGroups}.

\subsection{Mixed Shimura--Deligne data}   \label{MixedSD}     
Consider the following collection of data
\begin{itemize}
    \item $\mathbf{P} $ a connected linear algebraic group over $ \QQ $ \item $\mathbf{W} $ the unipotent radical of $ \mathbf{P} $
    \item  $ \mathbf{U} $ a subgroup of $ \mathbf{W} $ over $ \QQ $ that is  normal in $ \mathbf{P} $.
    \item  $ \mathcal{X} $ a left homogenous space under the group $ \mathbf{P}(\RR) \cdot \mathbf{U}(\CC) \subset \mathbf{P}(\CC)$,
    \item   $ h \colon \mathcal{X}  \to \Hom(\delT_{\CC}, \mathbf{P}_{\CC}) $ a $  \mathbf{P}(\RR) \cdot \mathbf{U}(\CC) $-equivariant map.
\end{itemize} 
Set $ \mathbf{V} = \mathbf{W} / \mathbf{U} $, $ \Gb = \mathbf{P}/\mathbf{W} $ and $ \pi \colon \mathbf{P} \to \mathbf{G} $, $ \pi' \colon \mathbf{P} \to \mathbf{P}/\mathbf{U} $ the natural projections. Then such a collection is said to be \emph{mixed Shimura--Deligne datum} if it satisfies axioms (i)--(viii) of \cite[Definition 2.1]{PinkThesis}, except possibly (vii). If these axioms are satisfied, the data is determined by the triple $ (\mathbf{P}, \mathcal{X} , h) $ since $ \mathbf{U} $ is characterised by its action on $ \mathrm{Lie}(\mathbf{P})$. We suppress the dependency on $ h $ when a choice has been made, and denote such a datum by $ (\mathbf{P}, \mathcal{X}) $ only.  

A \emph{morphism} of mixed Shimura--Deligne datum $ (\mathbf{P}_{1},\mathcal{X}_{1}, h_{1}) \to (\mathbf{P}_{2}, \mathcal{X}_{2}, h_{2})  $ is a morphism $ \phi \colon \mathbf{P}_{1} \to \mathbf{P}_{2} $ and a $ \mathbf{P}_{1}(\RR) \cdot  \mathbf{U}_{1}(\CC) $-equivariant map $ \Psi \colon  \mathcal{X}_{1} \to \mathcal{X}_{2} $ satisfying the commutativity property of \cite[Definition 2.3]{PinkThesis}. We call a morphism \emph{injective} if both $ \phi $ and $ \Psi $ are.

Given a mixed Shimura--Deligne datum $ (\mathbf{P}, \mathcal{X}, h) $ and a compact open subgroup $ K \subset \mathbf{P}(\Ab_{f}) $, we define the corresponding \emph{mixed Shimura--Deligne variety} to be the double coset 
\[
\Sh_{\mathbf{P}}(\mathcal{X},K)(\CC) \defeq \mathbf{P}(\QQ) \backslash \left [ \mathcal{X}  \times  \mathbf{P}(\Ab_{f})/K \right ] .
\]
As in \cite[Proposition 2.19]{PinkThesis}, every connected component of $ \mathcal{X} $ is a holomorphic complex vector bundle on a Hermitian symmetric domain, or simply a complex vector space. Therefore, if $K$ is neat, then $\Sh_{\mbf{P}}(\mathcal{X}, K)(\mbb{C})$ inherits the structure of a complex analytic variety. A morphism of mixed Shimura--Deligne data induces a morphism on the corresponding mixed Shimura--Deligne varieties when the compact open subgroups are chosen as in \cite[Proposition 3.8]{PinkThesis}.

\subsection{Hodge structures} 

Let $(\Gb, X)$ be a Shimura--Deligne datum as in Definition \ref{SD}, and let $ F $ be a number field. Let $\opn{Rep}_F(\Gb)$ denote the category of all finite-dimensional algebraic representations $ (\rho,  \mathbf{V}) $ of $\Gb_{F}$ and set $ V = \mathbf{V}(F) $, the underlying $ F $-vector space. Such a representation can equivalently be viewed as an algebraic representation of $ \Gb $ defined over $ \QQ $ of dimension $ \dim_{F}(V) \cdot [F : \QQ ] $, together with a $ \QQ $-algebra homomorphism  $ F  \hookrightarrow  \mathrm{End}_{\Gb(\QQ)}(V) $. 

For any $(\rho, \mathbf{V}) \in \opn{Rep}_F(\Gb)$ and $ h  \in  X $, we obtain a Hodge structure on $ V \otimes_{\mbb{Q}} \mbb{C}$ via the map $\rho \circ h$. In particular, $ V \otimes_{\QQ} \CC $ decomposes as a direct sum of $\mbb{C}$-subspaces $ V \otimes_{\QQ} \CC = \bigoplus V^{p_i, q_i}$ where $\mbb{C}^{\times} \subset \mbb{S}_{\mbb{C}}$ acts on $V^{p_i, q_i}$ via the character $z \mapsto z^{-p_i}\bar{z}^{-q_i}$, with $p_i, q_i \in \mbb{Z}$ -- the corresponding Hodge bigrading of the subspace $V^{p_i, q_i}$ is given by the pair $(p_i, q_i)$, and its weight is defined to be $ p_{i} + q_{i} $. The \emph{Hodge type} of $ (\rho, \mathbf{V}) $ is the collection of all such pairs that appear in this decomposition. If $p_i + q_i = n$ for all $i$, then we say that $ \mathbf{V} $ is pure of weight $n$. In this case, for any $h \in X$, the weight homomorphism $\rho \circ h \circ w \colon \GG_{m, \mbb{R}} \to \opn{GL}(V \otimes_{\mbb{Q}} \mbb{R})$ is equal to the character $ \lambda \to \lambda^{n} $ (recall our conventions on $w$ at the start of Appendix \ref{PureAppendix}).  

\begin{definition} We let $ \mathrm{Rep}_{F}(\Gb)^{\mathrm{AV}} $ denote the full sub-category of $ \mathrm{Rep}_{F}(\Gb) $ whose objects have Hodge type contained in $ \left \{ (-1,0) , (0,-1) \right\} $. 
\end{definition} 

\subsection{Unipotent extensions}   \label{Uni-Extensions} 
We shall primarily be concerned with a specific sub-class of mixed Shimura--Deligne datum, namely those which are unipotent extensions and of a prescribed Hodge type.  

\begin{definition} \label{MixedUESDDef}
Let $(\Gb, X)$ be a Shimura--Deligne datum satisfying (SD5). For a representation $\mathbf{V}$ in $\opn{Rep}_F(\mathbf{G})^{\mathrm{AV}}$, consider the pair $(\mathbf{P}, \mathcal{X})$, where
\begin{itemize}
    \item $ \mathbf{P} =  \Res_{F/\QQ} ( \mathbf{V} ) \rtimes \mathbf{G}$ is the unipotent extension of $\Gb$, with natural map $ \pi \colon \mathbf{P} \to  \Gb$
    \item $ \mathcal{X} $  is the subset of all $t \in \opn{Hom}(\mbb{S}_{\mbb{C}}, \mathbf{P}_{\mbb{C}})$ which are defined over $\mbb{R}$ and satisfy $ \pi  \circ t \in X _{\mbb{C}} = \left \{ h_{\CC} : h \in X  \right \}  $.
\end{itemize}
Then the pair $(\mathbf{P}, \mathcal{X})$ satisfies axioms (i)--(viii) in \cite[2.1]{PinkThesis}, excluding axiom (vii) (which is axiom (SD3)) and therefore determines a mixed Shimura--Deligne datum. We refer to such a datum as a \emph{unipotent extension} of the pure datum $(\Gb,X) $.  
\end{definition}

Under an additional assumption below, we will explain how one can associate mixed Shimura--Deligne varieties to the unipotent extensions in Definition \ref{MixedUESDDef} which satisfy good functoriality properties in both $(\Gb, X)$ and $\mbf{V}$. Essentially, all the results in \cite{PinkThesis} that we will need hold by replacing (SD3) with the following assumption. Let $\mbb{S}^1 \subset \mbb{S}$ denote the circle group.

\begin{assumption}[c.f. {\cite[1.12]{PinkThesis}}] \label{OneDimAssump}
Let $(\Gb, X)$ be a Shimura--Deligne datum satisfying (SD5) and let $(\mbf{P}, \mathcal{X})$ be a unipotent extension as in Definition \ref{MixedUESDDef}. Let $\mbf{G}_1$ be a normal subgroup of $\Gb$ which contains the image of $h(\mbb{S}^1)$, for any $h \in X$. Let $(\rho, \mbf{M})$ be an algebraic representation of $\mbf{P}$ (over $\mbb{Q}$) that is pure of weight $n$, i.e. for any $h \in \mathcal{X}$, the action of $\rho \circ h \circ w(z)$ on $\mbf{M}_{\mbb{C}}$ is given by multiplication by $z^{n}$. We assume that there exists
\begin{itemize}
    \item A \emph{one-dimensional} algebraic representation $\mbf{N}$ of $\mbf{P}$, defined over $\mbb{Q}$, which factors through $\Gb/\Gb_1$ and is pure of Hodge type $(n, n)$.
    \item A $\mbf{P}$-equivariant non-degenerate pairing $\Psi \colon \mbf{M} \otimes_{\mbb{Q}} \mbf{M} \to \mbf{N}$, and 
    \item For every $h \in \mathcal{X}$, a morphism of rational Hodge structures $\lambda_h \colon \mbf{N} \to \mbb{Q}(-n)$, such that $\lambda_h \circ \Psi$ is a polarisation for the Hodge structure on $\mbf{M}$ defined by $h$. 
\end{itemize}
\end{assumption}

\begin{remark} \label{ClarificationOnAssumption}
Assumption \ref{OneDimAssump} is an assumption on $(\mbf{P}, \invs{X}, \rho, \mbf{M})$. By abuse of language, we will say that Assumption \ref{OneDimAssump} holds for $(\mbf{P}, \invs{X})$ if it holds for $(\mbf{P}, \invs{X}, \rho, \mbf{M})$, for every pure algebraic representation $(\rho, \mbf{M})$ of $\mbf{P}$. 
\end{remark}

\begin{remark}
If $(\Gb, X)$ is a Shimura--Deligne datum satisfying (SD3) and (SD5) then Assumption \ref{OneDimAssump} holds for irreducible $\mbf{M}$. If, in addition to this, $\opn{dim}\left( \Gb / \Gb_1 \right) \leq 1$, then Assumption \ref{OneDimAssump} holds for all $\mbf{M}$ (see \cite[1.13]{PinkThesis}).
\end{remark}

In the absence of (SD3), Assumption \ref{OneDimAssump} provides an alternative way to deduce \cite[Lemma 5.5]{Torzewski2019}. In particular, it holds for Shimura--Deligne data of PEL-type satisfying (SD5).

\begin{lemma} \label{UnitaryMixed}
Let $(\Gb, X)$ be a PEL-type Shimura--Deligne datum as in Definition \ref{PELDatumDef} that satisfies (SD5). Then Assumption \ref{OneDimAssump} holds for any unipotent extension of $(\Gb, X)$ as in Definition \ref{MixedUESDDef}. In fact, one can take $\mbf{N}$ to be a power of the similitude character.
\end{lemma}
\begin{proof} 
Let $(\mbf{P}, \invs{X})$ be any unipotent extension, as in Definition \ref{MixedUESDDef}, and take any $t \in \invs{X}$ and $h \in X$ such that $\pi \circ t = h_{\mbb{C}}$. Take $\Gb_1$ to be the kernel of the similitude character. Then this is a normal subgroup of $\Gb$ which contains $h(\mbb{S}^1)$, so satisfies the conditions in Assumption \ref{OneDimAssump}. Furthermore, we have $\Gb/\Gb_1 \cong \mbb{G}_m$.

The homomorphism $t \circ w \colon \mbb{G}_{m, \mbb{C}} \to \mbf{P}_{\mbb{C}}$ defines a weight filtration $W_\bullet$ on the category $\opn{Rep}_{\mbb{C}}(\mbf{P})$ as in the end of \S 1 in \cite{MilneCanonicalMixed} (note that we are using ascending, rather than descending, filtrations -- see Remark 1.8 in \emph{loc.cit.}). By the assumptions on $(\mbf{P}, \invs{X})$, the algebraic group $\mbf{P}_{\mbb{C}}$ stabilises the filtration $W_\bullet$ (see \cite[Proposition 1.4]{PinkThesis}). By \cite[Proposition 1.7]{MilneCanonicalMixed}, this implies that $W_0\mbf{P}_{\mbb{C}} = \mbf{P}_{\mbb{C}}$ and $W_{-1}\mbf{P}_{\mbb{C}} = \opn{Res}_{F/\mbb{Q}}(\mbf{V})_{\mbb{C}}$. As a consequence, we see that any pure algebraic representation of $\mbf{P}$ must factor through $\mbf{G}$. 

Combining this with the fact that the pairing $\Psi$ is required to be $\mbf{P}$-equivariant, it is enough to verify that Assumption \ref{OneDimAssump} holds in the case $(\mathbf{P}, \mathcal{X}) = (\mathbf{G}, X)$, which we place ourselves in for the remainder of this proof. It is also enough to show that for any irreducible representation $\mathbf{M}$ that is pure of weight $n$, the assumption is satisfied with $\mbf{N}$ equal to the $n$-th power of the similitude character. Indeed, the category $\opn{Rep}_{\mbb{Q}}(\Gb)$ is semisimple, so we can easily extend the result to arbitrary pure representations after showing this. 

Let $\mathbf{M}$ be an irreducible algebraic representation of $\Gb$. Any such representation is still irreducible after restricting to $\mathbf{G}_1$, and since the category $\opn{Rep}_{\mbb{Q}}(\mbf{G}_1)$ is semisimple, we must have 
\[
\opn{dim}_{\mbb{Q}} \left( (\mathbf{M} \otimes_{\mbb{Q}} \mathbf{M})_{\mbf{G}_1} \right) \leq 1
\]
where $(-)_{\Gb_1}$ denotes coinvariants by $\Gb_1$. Furthermore this dimension is equal to $1$ precisely when $\mathbf{M}|_{\Gb_1} \cong \mathbf{M}^*|_{\Gb_1}$, by Schur's lemma. Another way of saying this is that there exists an integer $r$ such that
\[
\mathbf{M}^* \cong \mathbf{M} \otimes \mu^r
\]
where $\mu$ is the similitude character of $\mathbf{G}$. But the character $\mu$ gives rise to a Hodge structure that is pure of weight $-2r$, so the above isomorphism implies that $r = n$. The rest of the lemma now follows from the proof of \cite[1.12]{PinkThesis} (using the fact that conjugation by $h(i)$ is a Cartan involution for $\mathbf{G}_{1, \mbb{R}}$).
\end{proof}

\subsection{Summary of results on mixed Shimura--Deligne varieties} \label{SummarySDResults}

Let $(\Gb, X)$ be a Shimura--Deligne datum that satisfies (SD5) and Assumption \ref{OneDimAssump} for any unipotent extension $(\mathbf{P}, \invs{X})$ as in Definition \ref{MixedUESDDef} (see Remark \ref{ClarificationOnAssumption}). This assumption will allow us to apply the majority of the results in \cite[\S 1--\S 3, \S 11]{PinkThesis} and \cite{Torzewski2019}, which we summarise below:
\begin{enumerate}
    \item Let $\mathbf{V}$ be an algebraic representation of $\mathbf{G}_F$ in $\opn{Rep}_F(\Gb)^{\mathrm{AV}}$ which we view as an algebraic representation of $\Gb$ with an $F$-structure, and let $(\mathbf{P} = \mathrm{Res}_{F/\QQ}(\mathbf{V}) \rtimes \mbf{G}, \invs{X})$ denote the unipotent extension as in Definition \ref{MixedUESDDef}. For a neat compact open subgroup $K' \subset \mbf{P}(\mbb{A}_f)$, we set
    \[
    \opn{Sh}_{\mbf{P}}(K')(\mbb{C}) \defeq \mbf{P}(\mbb{Q}) \backslash \left[ \invs{X} \times \mbf{P}(\mbb{A}_f)/K' \right]
    \]
    which we view with the quotient topology induced from $\invs{X}$. By \cite[3.2--3.3]{PinkThesis} and Baily--Borel this set carries the structure of a complex algebraic variety, which we call the associated mixed Shimura--Deligne variety. Such an example of a neat compact open subgroup is as follows. Let $K \subset \mbf{G}(\mbb{A}_f)$ be a neat compact open subgroup, and let $L \subset \mbf{V}(\mbb{A}_f)$ denote a (full rank) $\widehat{\mbb{Z}}$-lattice that is stable under $K$. By \cite[Lemma 5.4]{Torzewski2019}, $L \rtimes K$ is a neat compact open subgroup of $\mbf{P} \defeq \opn{Res}_{F/\mbb{Q}}(\mbf{V}) \rtimes \mbf{G}$. 
    \item Let $K', L' \subset \mbf{P}(\mbb{A}_f)$ be neat compact open subgroups satisfying $\sigma^{-1}L' \sigma \subset K'$ for some $\sigma \in \mbf{P}(\mbb{A}_f)$. Then we have a finite \'{e}tale map
    \[
    \opn{Sh}_{\mbf{P}}(L')(\mbb{C}) \xrightarrow{[\sigma]} \opn{Sh}_{\mbf{P}}(K')(\mbb{C})
    \]
    constructed in exactly the same way as for (pure) Shimura--Deligne varieties (see \cite[3.4]{PinkThesis}). We can also define morphisms of mixed Shimura--Deligne datum in the usual way, and this defines morphisms between the associated Shimura--Deligne varieties (see \cite[3.8]{PinkThesis}). 
    \item For $\mbf{V}$ in $\opn{Rep}_F(\mbf{G})^{\mathrm{AV}}$ the mixed Shimura--Deligne variety $\opn{Sh}_{\mbf{P}}(L \rtimes K)(\mbb{C})$ has the structure of an abelian scheme over $\opn{Sh}_{\mbf{G}}(K)(\mbb{C})$ (c.f. \cite[3.22]{PinkThesis} -- we use Assumption \ref{OneDimAssump} here). In particular, these abelian schemes satisfy functoriality properties with respect to the datum $(\mbf{G}, X)$ and the representation $\mbf{V}$ (see \cite[Lemma 5.7]{Torzewski2019}). 
    \item There is a natural way to define reflex fields and canonical models associated to mixed Shimura--Deligne data (see \cite[\S 11]{PinkThesis}). In our setting, the reflex field $E$ of any unipotent extension $(\mbf{P}, \invs{X})$ as in Definition \ref{MixedUESDDef} is equal to the reflex field of $(\mbf{G}, X)$ (11.2 in \emph{op.cit.}). Suppose that $\opn{Sh}_{\mbf{G}}(K)$ admits a canonical model over $E$ (which is unique by Corollary \ref{UniquenessCM}), then the results of \S 11 and the reduction lemma (Lemma 2.26) in \emph{op.cit.} imply that $\opn{Sh}_{\mbf{P}}(K')$ admits a unique canonical model over $E$.
    
    Furthermore, the morphisms discussed in (2) descend to morphisms between the canonical models, as well as the functoriality properties in (3) (as long as the pure Shimura--Deligne data admit canonical models). Note that we are permitted to apply the reduction lemma in \emph{op.cit.} by Assumption \ref{OneDimAssump}. 
\end{enumerate}

\begin{corollary} \label{CorollaryOneDimAssumption}
Take $\mathscr{G} \in \{\mbf{G}, \mbf{H}, \mbf{T}, \Gt \}$ to be any of the groups defined in section \ref{TheGroups}, and ${X}$ the associated symmetric space. Then $(\mathscr{G}, {X})$ is a Shimura--Deligne datum which satisfies (SD5) and Assumption \ref{OneDimAssump} for any unipotent extension $(\mathscr{P}, \invs{X})$ of $(\mathscr{G}, {X})$ as in Definition \ref{MixedUESDDef} (including, of course, the datum $(\mathscr{G}, {X})$ itself). In particular, (1)--(4) above hold for any such $(\mathscr{P}, \invs{X})$.
\end{corollary}
\begin{proof}
The only potentially problematic group is $\mathscr{G} = \mbf{H}$, but we have shown that Assumption \ref{OneDimAssump} is satisfied in this case (see Lemma \ref{UnitaryMixed}). The existence of a canonical model for $(\mbf{H}, X_{\mbf{H}})$ follows from Corollary \ref{ExistenceCriterion} (the other groups gives rise to Shimura data in the usual sense).  
\end{proof}

\subsection{Lifts of the canonical construction} \label{LiftsCanConst}

We now put ourselves in the situation of Corollary \ref{CorollaryOneDimAssumption}, so $(\mathscr{G}, {X}) \in \{(\mbf{G}, X_{\mbf{G}}), (\mbf{H}, X_{\mbf{H}}), (\mbf{T}, X_{\mbf{T}}), (\Gt, X_{\Gt}) \}$. Let $p$ be a prime that splits in $E/\mbb{Q}$ and recall that we have fixed an embedding $E \hookrightarrow \Qpb$, which distinguishes a prime $\ide{P}$ of $E$ lying above $p$ satisfying $E_{\ide{P}} \cong \mbb{Q}_p$. By the general procedure described in section \ref{TheCanonicalConstructionEtaleCoh}, for any neat compact open subgroup $K \subset \mathscr{G}(\mbb{A}_f)$ we have a $\mbb{Q}_p$-linear tensor functor
\[
\mu_{\mathscr{G}, K} \colon \opn{Rep}_{\mbb{Q}_p}(\mathscr{G}) \to \opn{\acute{E}t}\left( \opn{Sh}_{\mathscr{G}}(K) \right)_{\mbb{Q}_p}
\]
from the category of finite-dimensional algebraic representations of $\mathscr{G}_{\mbb{Q}_p}$ to \'{e}tale $\mbb{Q}_p$-sheaves on $\opn{Sh}_{\mathscr{G}}(K)$. 

\begin{theorem}
Let $\mathscr{G} \in \{ \mbf{G}, \mbf{H}, \mbf{T}, \Gt \}$ as above. Then the results of \cite[\S 10]{Torzewski2019}, excluding Lemma 10.6, hold in this setting. If $\mathscr{G} \in \{ \mbf{G}, \mbf{H} \}$, then Lemma 10.6 does hold, in addition to the results of \S 8--\S 9 of \emph{op.cit.} where Betti cohomology is replaced with $p$-adic cohomology and the functoriality statements are with respect to the embedding $\mbf{H} \hookrightarrow \mbf{G}$ (as in section \ref{TheGroups}).
\end{theorem}
\begin{proof}
We will justify that all the proofs hold in our setting, even though we have not assumed (SD3). Note that (SD5) and Assumption \ref{OneDimAssump} hold for (any unipotent extension of) the Shimura--Deligne data that we are considering, and that the morphism $\mbf{H} \hookrightarrow \mbf{G}$ is admissible in the sense of \cite[Definition 9.1]{Torzewski2019}. The key references for the proofs are \cite{PinkThesis}, \cite{Pink1992} and \cite[\S I--II]{Wilde}. We justify the use of the results in \cite{Wilde} (the remaining results from \cite{PinkThesis} and \cite{Pink1992} are justified in section \ref{SummarySDResults}).

As we are working in the special case of mixed Shimura--Deligne datum in Definition \ref{MixedUESDDef}, as remarked at the end of page 10 in \cite[\S II]{Wilde}, one can check that we are in the situation of \cite[\S I.3]{Wilde} (except for geometric connectedness). Also \cite[3.13]{PinkThesis} is valid in our setting, so Lemma 1.6 in \cite[\S II]{Wilde} holds (the fibres of the map $\opn{Sh}_{\mathscr{P}}(L \rtimes K)_{\overline{E}} \to \opn{Sh}_{\mathscr{G}}(K)_{\overline{E}}$ are ``unipotent $K(\pi, 1)$s''). These observations mean that \cite[\S II.4]{Wilde} is valid in our setting. Indeed, we first check that all the cited results in this section hold:
\begin{itemize}
    \item The reference \cite[Proposition 3.3.3]{Pink1992}. It is an easy check that this proof does not depend on the axiom (SD3).
    \item Theorem 4.3 in \cite[\S II]{Wilde} (which is a consequence of Propositions 5.5.4, 5.8.2 and 5.6.1 in \cite{Pink1992}). For part (a) there is nothing to check, because any torus satisfies (SD3). For parts (b) and (c), the proofs in \cite[\S 5.6]{Pink1992} still hold verbatim (we do not need to consider models and compactifications). The arguments involving special points are also valid in our setting, because the results in \cite[\S 11]{PinkThesis} are valid by replacing any referenced result in \cite{DeligneTS} with the appropriate analogue in Appendix \ref{PureAppendix}. 
    \item The results in \cite[\S I]{Wilde} are valid by the above discussion. 
\end{itemize}
Then one can check that none of the remaining arguments in \cite[\S II.4]{Wilde} rely on the axiom (SD3). We note that the opposite convention of left/right actions for the functor $\mu_{\mathscr{G}}$ is used in \emph{loc.cit.}, however this does not affect the validity of the results. 

One can now check that the arguments in \cite{Torzewski2019} carry over into our situation. Furthermore, Ancona's construction is valid in our setting, as the results in \cite{Ancona2015} do not depend on (SD3) (in fact the majority of the paper doesn't even involve Shimura varieties). 
\end{proof}

\begin{remark}
Some of the arguments in \cite{Torzewski2019} do involve passing to connected components of the mixed Shimura varieties. However, this is only ever used in an abstract way; an explicit description of the connected components is not needed, which would require (SD3).
\end{remark}


\newcommand{\etalchar}[1]{$^{#1}$}
\renewcommand{\MR}[1]{}
\providecommand{\bysame}{\leavevmode\hbox to3em{\hrulefill}\thinspace}
\providecommand{\MR}{\relax\ifhmode\unskip\space\fi MR }
\providecommand{\MRhref}[2]{%
  \href{http://www.ams.org/mathscinet-getitem?mr=#1}{#2}
}
\providecommand{\href}[2]{#2}


\Addresses

\end{document}